\documentclass[9pt,reqno]{amsart}
\usepackage[T1]{fontenc}
\usepackage{lmodern}
\usepackage{microtype}
\usepackage[
  paperwidth=6.75in,paperheight=9.5in,
  inner=0.82in,outer=0.72in,top=0.75in,bottom=0.82in
]{geometry}
\usepackage{amsmath,amssymb,amsthm,mathtools,mathrsfs}
\usepackage{enumitem}
\usepackage{booktabs}
\usepackage{array}
\usepackage{longtable}
\usepackage{xcolor}
\usepackage{tikz}
\usetikzlibrary{arrows.meta,positioning,calc,fit}
\usepackage[colorlinks=true,linkcolor=blue!55!black,citecolor=blue!55!black,
            urlcolor=blue!55!black]{hyperref}

\numberwithin{equation}{section}
\allowdisplaybreaks
\newtheorem{theorem}{Theorem}[section]
\newtheorem{proposition}[theorem]{Proposition}
\newtheorem{lemma}[theorem]{Lemma}
\newtheorem{corollary}[theorem]{Corollary}
\theoremstyle{definition}
\newtheorem{definition}[theorem]{Definition}
\newtheorem{construction}[theorem]{Construction}

\theoremstyle{remark}
\newtheorem{remark}[theorem]{Remark}

\newcommand{\K}{\mathbb{K}}
\newcommand{\N}{\mathbb{N}}
\newcommand{\Q}{\mathbb{Q}}
\newcommand{\Bcal}{\mathcal{B}}
\newcommand{\Kcal}{\mathcal{K}}
\newcommand{\Cal}{\mathop{\rm Cal}}
\newcommand{\Env}{\operatorname{Env}}
\newcommand{\dist}{\operatorname{dist}}
\newcommand{\supp}{\operatorname{supp}}
\newcommand{\ran}{\operatorname{ran}}
\newcommand{\rank}{\operatorname{rank}}
\newcommand{\norm}[1]{\left\lVert #1\right\rVert}
\newcommand{\abs}[1]{\left\lvert #1\right\rvert}
\newcommand{\ilim}{\varprojlim\nolimits^{\,b}}

\title[Inverse-Limit Detection Envelopes as Calkin Algebras]
{Inverse-Limit Detection Envelopes as Calkin Algebras}
\hypersetup{
  pdftitle={Inverse-limit detection envelopes as Calkin algebras},
  pdfsubject={Detection envelopes and an oracle-recursive Bourgain--Delbaen Calkin construction},
  pdfkeywords={Calkin algebra, detection envelope, bounded inverse limit,
  Banach algebra, Bourgain--Delbaen construction, oracle recursion}
}

\author{Rui Liu}
\address{School of Mathematical Sciences and LPMC, Nankai University, Tianjin
300071, P.R. China}
\email{ruiliu@nankai.edu.cn}

\author{Jie Shen}
\address{School of Mathematical Sciences and LPMC, Nankai University, Tianjin
300071, P.R. China}
\email{1710064@mail.nankai.edu.cn}

\date{September 26, 2026}


\begin{document}

  \begin{abstract}
    For a unital Banach algebra \(A\) with \(\norm{1_A}=1\) and a decreasing sequence \((I_r)\) of proper closed two-sided ideals of finite codimension, we introduce the bounded detection envelope
    \[ \widehat A:=\Env^b_{(I_r)}(A)=\ilim(A/I_r,\pi_r^s),\qquad \norm{(a_r)_r}=\sup_r\norm{a_r}. \]
    Combining vector-valued Bourgain--Delbaen extensions, oracle Turing-machine recursion, and Argyros--Haydon operator analysis, we construct a separable Banach space \(X\) such that, isometrically as unital Banach algebras,
    \[ \Cal(X)=\Bcal(X)/\Kcal(X)\simeq\widehat A. \]
    We also characterize these envelopes as the unital Banach algebras with norm-one units admitting an isometric predual and countably many jointly faithful weak-star-to-norm continuous finite-dimensional unital representations.

    Applications include arbitrary finite-dimensional unital Banach algebras with norm-one units and their countable \(\ell_\infty\)-sums, as well as unitized coordinatewise algebras over normalized \(1\)-unconditional boundedly complete Schauder bases.
    These include incidence algebras, \(\ell_\infty\), and the unitization of Tsirelson's original space.
    Further applications concern analytic operator algebras, Fourier--Stieltjes and Fourier multiplier algebras of countable discrete groups, and classical and quantum convolution algebras.
    In particular, we realize \(H^\infty(\mathbb D)\), \(\operatorname{Mult}(H_2^2)\), \(\mathcal L_2\), \(M_{\mathrm{cb}}A(\mathbb F_2)\), and \(C^u(O_3^+)^*\) as Calkin algebras, with their canonical norms and products.
  \end{abstract}

  \maketitle
  \tableofcontents


\section{Introduction}\label{sec:introduction}\label{sec:detector-filtrations}
For a Banach space \(Z\), write
\[ \Cal(Z)=\Bcal(Z)/\Kcal(Z), \]
where \(\Bcal(Z)\) is the algebra of bounded operators on \(Z\) and \(\Kcal(Z)\) is the ideal of compact operators on \(Z\).
For \(T\in\Bcal(Z)\), write \([T]=T+\Kcal(Z)\) for its Calkin class.
Throughout, \(\K=\mathbb R\) or \(\mathbb C\), \(\N=\{0,1,\ldots\}\), and \(\N_+=\{1,2,\ldots\}\).
The name comes from Calkin's 1941 study of the Hilbert-space quotient \cite{Calkin1941}.

The main result of this paper is a realization theorem for bounded inverse limits of finite-dimensional unital Banach algebras.
The norm and multiplication are those prescribed by the inverse system.

Let \(\mathscr A=(A_r,\pi_r^s)_{r\leq s}\) be a countable inverse system of finite-dimensional unital Banach algebras, with bonding maps \(\pi_r^s:A_s\to A_r\) for \(r\leq s\).
We assume that the bonding maps are unital homomorphisms and metric quotients, that is,
\begin{align}
  \norm{1_{A_r}}&=1,\qquad
  \pi_r^s(1_{A_s})=1_{A_r},\qquad
  \pi_r^s\pi_s^t=\pi_r^t,                         \label{eq:inverse-system}\\
  \norm{x_r}
  &=\inf\{\norm{x_s}:\pi_r^sx_s=x_r\}
  \qquad(r\leq s).                            \label{eq:metric-quotient}
\end{align}
Its bounded inverse limit is
\begin{equation}\label{eq:bounded-limit}
  \widehat A=\ilim(A_r,\pi_r^s)
  =\bigl\{(a_r)_r:\pi_r^sa_s=a_r\ (r\leq s),\quad
  \sup_r\norm{a_r}<\infty\bigr\},
\end{equation}
with coordinatewise multiplication and norm \(\norm a=\sup_r\norm{a_r}\).
The boundedness condition makes \(\widehat A\) a Banach algebra.

\begin{theorem}[Inverse-limit realization]\label{thm:main}
  Let \(\mathscr A=(A_r,\pi_r^s)_{r\leq s}\) be a countable inverse system of finite-dimensional unital Banach algebras satisfying \eqref{eq:inverse-system}--\eqref{eq:metric-quotient}, and put \(\widehat A=\ilim(A_r,\pi_r^s)\).
  There are a separable Banach space \(X_{\mathscr A}\) and a linear map \(\widehat A\to\Bcal(X_{\mathscr A})\), \(a\mapsto M_a\), such that
  \begin{equation}\label{eq:representation}
    M_1=I,
    \qquad
    M_aM_b=M_{ab},
    \qquad
    \norm{M_a}=\norm a.
  \end{equation}
  Moreover, we have
  \begin{equation}\label{eq:classification}
    \Bcal(X_{\mathscr A})
    =\{M_a:a\in\widehat A\}+\Kcal(X_{\mathscr A}),
  \end{equation}
  and
  \begin{equation}\label{eq:essential-norm}
    \norm{[M_a]}_{\Cal(X_{\mathscr A})}=\norm a.
  \end{equation}
  The coefficient in \eqref{eq:classification} is unique.
  Consequently,
  \[ \widehat A\longrightarrow\Cal(X_{\mathscr A}),\qquad a\longmapsto[M_a], \]
  is an isometric unital Banach-algebra isomorphism.
\end{theorem}

Let \(A\) be a unital Banach algebra with \(\norm{1_A}=1\).
Consider a decreasing sequence
\begin{equation}\label{eq:detector-filtration}
  A\supseteq I_0\supseteq I_1\supseteq\cdots
\end{equation}
of proper closed two-sided ideals of finite codimension.
Give \(A_r=A/I_r\) its quotient norm and let \(\rho_r:A\to A_r\) be the quotient map. Then we have
\[ \rho_r(a)=a+I_r,\qquad \pi_r^s(a+I_s)=a+I_r\quad(r\leq s). \]
These maps satisfy \eqref{eq:inverse-system}--\eqref{eq:metric-quotient}.

\begin{definition}[Bounded detection envelope]\label{def:detection-envelope}
  The bounded detection envelope of \((A,(I_r))\) is \(\Env^b_{(I_r)}(A)=\ilim(A/I_r,\pi_r^s)\).
  Its canonical map is
  \begin{equation}\label{eq:canonical-detector-map}
    \jmath:A\longrightarrow\Env^b_{(I_r)}(A),
    \qquad \jmath(a)=(\rho_r(a))_r.
  \end{equation}
\end{definition}

Let \(\tau_{(I_r)}\) be the locally convex topology on \(A\) generated by the quotient seminorms \(a\mapsto\norm{\rho_r(a)}\), \(r\in\N\).
The filtration \((I_r)\) is \emph{separating} if \(\bigcap_rI_r=\{0\}\) and \emph{norming} if \(\norm a=\sup_r\norm{\rho_r(a)}\).
We give the envelope its coordinate topology.
For the algebra \(c\) of convergent sequences, the first \(r+1\) coordinates give a separating norming filtration whose bounded envelope is \(\ell_\infty\).

\begin{corollary}[Filtered-algebra realization]
  \label{cor:filtered-realization}
  Let \((I_r)\) be a decreasing sequence of proper closed two-sided ideals of finite codimension in \(A\), and let \(\mathscr A=(A/I_r,\pi_r^s)\) be the associated inverse system.
  The map \(\jmath\) in \eqref{eq:canonical-detector-map} has coordinate dense range through uniformly bounded approximants.
  It is injective when \((I_r)\) is separating, isometric when \((I_r)\) is norming, and onto when every closed norm ball is compact in \(\tau_{(I_r)}\).
  In particular, if the filtration is separating and every closed norm ball is compact in \(\tau_{(I_r)}\), then we have
  \[ \Cal(X_{\mathscr A})\simeq A \]
  isometrically as unital Banach algebras.
\end{corollary}

The finite-detection condition also has a dual formulation.
By Theorem~\ref{thm:finite-detection-duality}, we know that a unital Banach algebra is isometrically a bounded inverse limit as above if and only if it has an isometric predual \(E\) with increasing finite-dimensional subspaces \(E_r\) such that
\[ E=\overline{\bigcup_r E_r},\qquad E_r^\perp\text{ is a proper two-sided ideal for every }r. \]
Equivalently, it has a jointly faithful countable family of weak-star-to-norm continuous finite-dimensional representations.
These formulations characterize the algebras to which Theorem~\ref{thm:main} applies.

The applications also reach Fourier multiplier algebras and convolution duals beyond the universal compact quantum-group case.
For the Fourier--Stieltjes and Fourier multiplier terminology, we use \cite{Eymard1964,Knudby2014Semigroups}.
For every countable discrete group \(\Gamma\), the algebras \(B(\Gamma)\), \(MA(\Gamma)\), and \(M_{\mathrm{cb}}A(\Gamma)\) are realized with pointwise multiplication and their respective Fourier--Stieltjes, multiplier, and completely bounded multiplier norms.
In particular, the Herz--Schur multiplier algebra \(M_{\mathrm{cb}}A(\mathbb F_2)\) strictly contains \(B(\mathbb F_2)\), as noted in Corollary~\ref{cor:fourier-multipliers}.
Using unital completely positive (UCP) coproducts, Proposition~\ref{prop:coalgebra} also realizes the convolution dual of every separable compact quantum hypergroup in the sense of Chapovsky--Vainerman \cite{ChapovskyVainerman1999} (Corollary~\ref{cor:compact-quantum-hypergroups}).
For any separable compact quantum-group completion in the sense of \cite{KyedSoltan2012}, including reduced and exotic ones, whose coproduct extends to the minimal tensor product, the convolution dual is realized when its counit is bounded, while its standard \(\ell_1\)-unitization is realized without a counit assumption (Corollary~\ref{cor:quantum-completions}).

The construction uses the Tsirelson and Bourgain--Delbaen methods reviewed below, together with a finite-extension recursion.

\subsection{The Argyros--Haydon construction and related work}
\label{sec:intro-ah}

One line begins in 1974 with Tsirelson's construction of a reflexive space containing neither \(c_0\) nor any \(\ell_p\) \cite{Tsirelson1974}.
Figiel and Johnson gave the dual Tsirelson space its modern version of that norm, providing the recursive method used in many later constructions \cite{FigielJohnson1974}.
Maurey and Rosenthal constructed a normalized weakly null sequence with no unconditional subsequence, and their argument was an early source of the coding methods used later to suppress unconditional structure \cite{MaureyRosenthal1977}.
Schlumprecht adapted recursive norming in 1991 to construct the first arbitrarily distortable space \cite{Schlumprecht1991}.
Odell and Schlumprecht subsequently proved that \(\ell_2\) is arbitrarily distortable \cite{OdellSchlumprecht1994}.

In 1993, Gowers and Maurey combined recursive norming, weighted averages, and special coding to solve the unconditional basic sequence problem.
Their construction gave the first hereditarily indecomposable (HI) space \cite{GowersMaurey1993}.
Here HI means that no infinite-dimensional closed subspace decomposes as a direct sum of two infinite-dimensional closed subspaces.
Gowers later constructed an HI space with an asymptotic unconditional basis \cite{Gowers1995}, and his dichotomy showed that every infinite-dimensional Banach space contains either an unconditional basic sequence or an HI subspace \cite{Gowers1996}.
Ferenczi showed in 1997 that HI can coexist with uniform convexity \cite{Ferenczi1997UC}.
He also proved that, for a complex HI space \(X\), every operator from a subspace of \(X\) into \(X\) is a scalar multiple of the inclusion plus a strictly singular operator \cite{Ferenczi1997Operators}.
Gowers and Maurey developed related coding methods further to construct Banach spaces with small operator algebras \cite{GowersMaurey1997}.

In parallel, asymptotic and mixed-Tsirelson methods developed along a related line.
Maurey, Milman, and Tomczak-Jaegermann introduced a general asymptotic framework and studied asymptotic \(\ell_p\)-spaces \cite{MaureyMilmanTomczak1995}.
Argyros and Deliyanni constructed asymptotic \(\ell_1\)-spaces through mixed-Tsirelson norms \cite{ArgyrosDeliyanni1997}.
Argyros, Deliyanni, Kutzarova, and Manoussakis developed modified mixed-Tsirelson spaces, including arbitrarily distortable and HI examples \cite{ArgyrosDeliyanniKutzarovaManoussakis1998}.
The work of Odell and Schlumprecht on Krivine sets and block finite universality provided the early form of the method later called saturation under constraints \cite{OdellSchlumprecht1995,OdellSchlumprecht2000}.
Argyros and Felouzis used interpolation and HI methods in quotient and operator-factorization problems \cite{ArgyrosFelouzis2000}.
In 2004, Argyros and Tolias gave a general method for constructing HI spaces using special convex combinations, rapidly increasing sequences (RIS), the basic inequality, and operator analysis \cite{ArgyrosTolias2004}.
These methods were further developed in the construction of the first known reflexive space with the hereditary invariant subspace property \cite{ArgyrosMotakis2014}.
Argyros and Motakis later introduced a dual method covering both reflexive HI spaces and HI spaces without reflexive subspaces \cite{ArgyrosMotakis2016}.

The second line comes from the theory of \(\mathcal{L}_\infty\)-spaces, developed in the work of Lindenstrauss and Pe\l czy\'nski \cite{LindenstraussPelczynski1968} and Lindenstrauss and Rosenthal \cite{LindenstraussRosenthal1969}.
In 1980, Bourgain and Delbaen introduced a rank-by-rank finite-dimensional extension scheme and constructed new \(\mathcal{L}_\infty\)-spaces whose duals are isomorphic to \(\ell_1\) and which do not contain \(c_0\) \cite{BourgainDelbaen1980}.
Bourgain and Pisier gave a complementary construction which embeds every separable Banach space into a separable \(\mathcal{L}_\infty\)-space with a quotient having the Radon--Nikod\'ym and Schur properties \cite{BourgainPisier1983}.
Alspach established the small Szlenk index of a Bourgain--Delbaen space \cite{Alspach2000}.
Haydon showed that every infinite-dimensional closed subspace of the original \(X_{a,b}\) contains a further subspace isomorphic to some \(\ell_p\), \(1<p<\infty\) \cite{Haydon2000}.
Freeman, Odell, and Schlumprecht later proved that every Banach space with separable dual embeds into an \(\mathcal{L}_\infty\)-space whose dual is isomorphic to \(\ell_1\) \cite{FreemanOdellSchlumprecht2011}.

In 2011, Argyros and Haydon combined the Bourgain--Delbaen scheme with mixed-Tsirelson restrictions, RIS estimates, and dependent sequences.
They constructed an HI space \(X_{\mathrm{AH}}\) on which every operator is scalar plus compact, and hence \(\Cal(X_{\mathrm{AH}})\simeq\K\) \cite{ArgyrosHaydon2011}.
This answered the scalar-plus-compact question recorded as Question~1 by Lindenstrauss \cite{Lindenstrauss1976}.
Zisimopoulou subsequently introduced vector-valued Bourgain--Delbaen \(\mathcal{L}_\infty\)-sums in 2014 \cite{Zisimopoulou2014}.
Later work developed the Bourgain--Delbaen and Argyros--Haydon methods in several directions.
Argyros, Freeman, Haydon, Odell, Raikoftsalis, Schlumprecht, and Zisimopoulou showed that every separable uniformly convex space embeds into a scalar-plus-compact space \cite{ArgyrosEtAl2012}.
Argyros, Gasparis, and Motakis proved that every infinite-dimensional separable \(\mathcal{L}_\infty\)-space is isomorphic to a Bourgain--Delbaen space \cite{ArgyrosGasparisMotakis2016}.
Manoussakis, Pelczar-Barwacz, and \'Swi\k{e}tek constructed an unconditionally saturated Bourgain--Delbaen space with the scalar-plus-compact property \cite{ManoussakisPelczarSwietek2017}.
Argyros and Motakis later constructed a scalar-plus-compact space containing no infinite-dimensional reflexive subspace \cite{ArgyrosMotakis2019}.

The Argyros--Haydon result was followed by realizations of other prescribed Calkin algebras.
In 2012, Tarbard first moved beyond the one-dimensional algebra.
For every \(k\geq2\), he constructed a space whose Calkin algebra is the truncated-polynomial algebra \(\K[t]/(t^k)\), a unital algebra with nilpotent radical \cite{Tarbard2012}.
In his thesis, Tarbard also constructed a space whose Calkin algebra is isometrically isomorphic to the convolution algebra \(\ell_1(\N)\) \cite{Tarbard2013}.

In 2016, Motakis, Puglisi, and Zisimopoulou used Bourgain--Delbaen sums to realize \(C(K)\) for every countable compact metric space \(K\) \cite{MotakisPuglisiZisimopoulou2016}.
Motakis later proved that every compact metric space \(K\) admits a realization whose Calkin algebra is isometrically isomorphic to \(C(K)\) \cite{Motakis2024}.

Kania and Laustsen used \(X_{\mathrm{AH}}\) and one of its subspaces to obtain a three-dimensional triangular Calkin algebra in their 2017 work.
This algebra has two scalar characters and a one-dimensional square-zero off-diagonal radical \cite{KaniaLaustsen2017}.
Their note added in proof also realizes every finite-dimensional semisimple complex algebra \cite[Note added in proof]{KaniaLaustsen2017}. Corollary~\ref{cor:finite-dimensional-algebras} below allows arbitrary finite-dimensional unital Banach algebras with norm-one units and preserves their given norms.

Argyros, Deliyanni, and Tolias developed a method for Banach algebras of diagonal operators and constructed an HI example on which every diagonal operator is scalar plus compact \cite{ArgyrosDeliyanniTolias2011}.
Building on this work and on Argyros--Haydon sums, Motakis, Puglisi, and Tolias proved that the algebra of scalar-plus-compact diagonal operators associated with every real Banach space with a Schauder basis is a Calkin algebra.
Their applications include Banach-algebra structures on the James spaces \(J_p\) and their duals.
They also proved, at the Banach-space level, that every nonreflexive real Banach space with an unconditional basis is isomorphic to a Calkin algebra \cite{MotakisPuglisiTolias2020} in 2020.

Motakis and Pelczar-Barwacz showed that a Calkin algebra can be infinite-dimensional and reflexive, even isomorphic to a Hilbert space.
More generally, they realized the unitization of \(U\), with coordinatewise multiplication, when \(U\) has a normalized unconditional basis with no \(c_0\) asymptotic version \cite{MotakisPelczar2025}.
Motakis and Puglisi realized the unitization of \(\Kcal(c_0)\).
Their space is an Argyros--Haydon sum of copies of one Argyros--Haydon space, with the internal and external construction parameters kept separate \cite{MotakisPuglisi2025}.

Another related example comes from the reflexive space \(X_{\rm AM}\).
Every operator on this space is scalar plus strictly singular, its strictly singular ideal is nonseparable, and the product of two strictly singular operators is compact \cite{ArgyrosMotakis2016}.
Thus, as observed in \cite{Skillicorn2015}, its Calkin algebra has a nonseparable zero-product radical.
This radical is a quotient of operator ideals.

These results clarify the scope of the present theorem.
Theorem~\ref{thm:main} applies to a unital Banach algebra when it is the bounded inverse limit of its finite-dimensional quotient algebras, with their quotient norms.
In particular, Corollary~\ref{cor:boundedly-complete-unitization} realizes the standard unitization of the coordinatewise algebra of a space with a normalized \(1\)-unconditional boundedly complete basis.
The lower estimate in \cite[Theorem~7.1]{MotakisPelczar2025} implies bounded completeness under their hypothesis, as shown in Remark~\ref{rem:mpb-bounded-completeness}.
This class is strictly larger, as witnessed by \(T^*\) in Corollary~\ref{cor:tsirelson-original}. The resulting realization and its unital algebra lift are isometric for the standard unitization norm.

\subsection{Recursion theory and connections with other areas of mathematics}
\label{sec:intro-recursion}

Recursion theory studies effective procedures and their dependence on information.
Besides decidability, it provides tools for comparing the information contained in mathematical objects and for organizing constructions by successive finite requirements \cite{Odifreddi1989,Soare2016}.
Turing's oracle machines formalize procedures whose fixed rules may consult a specified oracle \cite[Section~4]{Turing1939}.
The oracle can supply noncomputable information, while each completed computation uses only finitely many queries.
Here a fixed oracle program schedules finite requirements and extends the construction records using the finite-dimensional data of \(\mathscr A\).

In group theory, Higman's embedding theorem says that a finitely generated group embeds into a finitely presented group if and only if it admits a recursively enumerable set of defining relations \cite{Higman1961}.
Thus a condition on effective presentation has an algebraic characterization by embeddings.

In number theory, the Davis--Putnam--Robinson--Matiyasevich theorem identifies the recursively enumerable subsets of \(\N\) with the Diophantine sets \cite{Matiyasevich1970}.
Membership in such a set can be expressed by the existence of natural-number solutions to a fixed polynomial equation with integer coefficients, with the set element as a parameter.
Since some recursively enumerable sets are undecidable, this yields the negative solution to Hilbert's tenth problem: no algorithm decides whether an arbitrary polynomial equation with integer coefficients has an integer solution.

In Banach-space theory, Pour-El and Richards developed computability structures compatible with linear operations, norms, and effective limits \cite[Chapters~2--3]{PourElRichards1989}.
Their First Main Theorem \cite[Chapter~3]{PourElRichards1989} applies to a closed linear operator between such spaces whose domain contains a computable sequence with dense linear span in the source space and whose images of that sequence form a computable sequence.
Under these hypotheses, the operator sends every computable vector in its domain to a computable vector if and only if it is bounded.
Bri\"et, Naor and Regev used constructions of locally decodable codes from theoretical computer science to prove that, for every \(1<p<\infty\), there is a Banach space \(X\) of finite cotype such that \(\ell_p\widehat\otimes_\pi X\) fails to have finite cotype \cite{BrietNaorRegev2012}.
Gunes and Naor determined the separation modulus of \(M_n(\mathbb R)\) equipped with an arbitrary unitarily invariant norm up to universal constant factors \cite{GunesNaor2026}.
They also obtained a deterministic oracle-polynomial-time algorithm that approximates this modulus within universal constant factors using norm queries.

In real analysis, algorithmic randomness gives a precise description of the points at which computable monotone functions are differentiable.
Brattka, Miller, and Nies proved that \(z\in(0,1)\) is computably random if and only if every computable nondecreasing function \(f:[0,1]\to\mathbb R\) is differentiable at \(z\) \cite[Theorem~4.3]{BrattkaMillerNies2016}.
This refines the classical almost-everywhere differentiability theorem for monotone functions by characterizing the points that work simultaneously for all computable functions in this class.

In geometric measure theory, the point-to-set principle of Jack H.~Lutz and Neil Lutz \cite{LutzLutz2018} gives
\begingroup
\postdisplaypenalty=10000
\[ \dim_{\mathrm H}(E) =\min_{A\subseteq\N}\sup_{x\in E}\dim^A(x) \qquad(\varnothing\ne E\subseteq\mathbb R^n), \]
\endgroup
where \(\dim^A(x)\) is the effective Hausdorff dimension of the point \(x\) relative to the oracle \(A\).
The formula turns the task of proving \(\dim_{\mathrm H}(E)\ge s\) into a pointwise problem: for every oracle \(A\) and every \(\varepsilon>0\), one seeks a point \(x\in E\) such that \(\dim^A(x)\ge s-\varepsilon\).

Using algorithmic fractal dimension and this principle, Jack H.~Lutz, Renrui Qi, and Liang Yu \cite{LutzQiYu2024} proved that, for every \(s\in[0,1]\), there is a Hamel basis \(B\) of \(\mathbb R\) over \(\Q\) with \(\dim_{\mathrm H}(B)=s\).
This is a classical geometric and algebraic existence theorem proved by computability-theoretic methods.

\subsection{Outline}
\label{sec:intro-outline}

There are also obstructions to Calkin realization.
Horv\'ath and Kania constructed unital Banach algebras that are not isomorphic to Calkin algebras of separable Banach spaces \cite{HorvathKania2021}.
Acuaviva and Acuaviva constructed a unital Banach algebra that is not isomorphic to the Calkin algebra of any Banach space \cite{AcuavivaAcuaviva2026}.
The theorem above concerns the dual Banach algebras characterized by the finite-dimensional quotient condition in Theorem~\ref{thm:finite-detection-duality}.

The proof combines the following constructions and estimates.
The rank-by-rank extension scheme is the Bourgain--Delbaen scheme \cite{BourgainDelbaen1980}.
The evaluation analysis, the RIS and basic-inequality method, and the exact-pair and dependent-sequence method follow the Argyros--Haydon line \cite{ArgyrosHaydon2011,Tarbard2013,Motakis2024}.
The auxiliary norm is a mixed Tsirelson norm in the sense of Argyros and Deliyanni \cite{ArgyrosDeliyanni1997}.
Vector-valued BD sums provide related background for the use of finite-dimensional fibres \cite{Zisimopoulou2014}.
The fair scheduling of finite requirements uses the recursion-theoretic finite-extension scheme described in \cite[Section~V.2]{Odifreddi1989}.
For its Baire-category formulation in terms of dense open sets, see \cite[Section~V.3]{Odifreddi1989}.

The proof of Theorem~\ref{thm:main} has two parts.
Part~I constructs the space by a monotone finite-extension scheme relative to the local structure of \(\mathscr A\) and the corresponding admissibility data.
The formal language, transition rules, and scheduler are fixed independently of \(\mathscr A\), as described in the oracle formulation in Remark~\ref{rem:oracle-firewall}.

The finite-dimensional fibres \(A_r^*\) and \(A_r\) carry the module actions, while covariant packets make \(a\mapsto M_a\) linear and multiplicative and give uniform bounds.
The finite-extension and fairness properties realize every admissible finite requirement cofinally.
Part~I also constructs probes and establishes interval restriction and common-stem comparison.
The consequences used later are collected in Theorem~\ref{thm:datum}.

Part~II first proves the local orbit statement in Theorem~\ref{thm:local-orbit}: for every \(S\in\Bcal(X)\) and every RIS \((x_k)\), we have
\[
  \dist\bigl(Sx_k,\{M_bx_k:b\in\widehat A\}\bigr)\longrightarrow0.
\]
The basic inequality gives the upper estimate for dependent sequences.
If the displayed distances stay away from zero, separation gives a packet that detects the operator error.
A root perturbation removes its error on the multiplier orbit, and the finite-extension property places the resulting data in a dependent sequence.
The lower estimate for that sequence then contradicts the upper estimate.

The local statement allows the approximating multiplier to depend on \(k\).
To obtain one coefficient for the whole operator, Proposition~\ref{prop:globalization} first reads finite-dimensional coefficients from unit probes and then compares them using paired probes.
The comparison is uniform over detector indices, so the limits form one bounded compatible family \(a=(a_r)_r\in\widehat A\).
A further pairing argument proves that \(S-M_a\) tends to zero on the whole unit and core probe fibres.

Finally, Lemma~\ref{lem:no-ghost} compares an arbitrary RIS with core probes carrying the same finite-dimensional error.
The difference of the two evaluations annihilates every multiplier image.
On every RIS, the local orbit statement therefore forces
\[
  (S-M_a)x_k\to0.
\]
The compactness test in Lemma~\ref{lem:compactness-test} makes \(S-M_a\) compact, and unit probes give the exact essential norm.

Readers interested only in the proof strategy may read Theorem~\ref{thm:datum}, then Theorem~\ref{thm:local-orbit}, then Proposition~\ref{prop:globalization} and Lemma~\ref{lem:no-ghost}, and finally Theorem~\ref{thm:global-classification}.

The realization theorem has several immediate algebraic consequences.
We record two of them here.
The remaining consequences are stated and proved in Section~\ref{sec:examples}.

\begin{corollary}[Unconditional sums of algebras]
  \label{cor:intro-coordinatewise}
  Let \(U\) have a normalized \(1\)-unconditional boundedly complete Schauder basis \((u_n)\).
  Suppose that each \(B_n\) is isometrically a bounded inverse limit of a countable system of finite-dimensional Banach algebras with metric-quotient bonding homomorphisms.
  Give \(J=(\bigoplus_{n\geq1}B_n)_U\) coordinatewise multiplication, where
  \[
    \begin{aligned}
      b=(b_n)\in J
      \quad\Longleftrightarrow\quad
      \sum_n\norm{b_n}u_n\text{ converges in }U,\qquad \text{and}\qquad
      \norm b_U=\left\lVert\sum_n\norm{b_n}u_n\right\rVert_U.
    \end{aligned}
  \]
  Then \(J\oplus\K I\) is isometrically the Calkin algebra of a separable Banach space, with
  \[
    \begin{aligned}
      (b+\lambda I)(c+\mu I)
      =bc+\lambda c+\mu b+\lambda\mu I,\qquad
      \norm{b+\lambda I}=\norm b_U+\abs\lambda.
    \end{aligned}
  \]
  For \(B_n=\K\), the algebra is \(U\oplus\K I\), and the realizing space can be chosen to have a Schauder basis.
\end{corollary}

\begin{corollary}[\(\ell_\infty\)-sums]
  \label{cor:intro-products}
  Let \((F_k)_{k\in\N}\) be finite-dimensional unital Banach algebras with \(\norm{1_{F_k}}=1\).
  Their \(\ell_\infty\)-sum, with coordinatewise multiplication, is isometrically a Calkin algebra:
  \[
    \begin{aligned}
      \left(\bigoplus_{k\in\N}F_k\right)_{\ell_\infty}
      =\left\{(a_k)_k:a_k\in F_k,\ \sup_k\norm{a_k}<\infty\right\} \simeq\Cal(X),\qquad
      \norm{(a_k)_k}=\sup_k\norm{a_k}.
    \end{aligned}
  \]
  In particular, \(\ell_\infty\) is isometrically the Calkin algebra of a separable Banach space.
\end{corollary}

The \(U\)-sum and \(\ell_\infty\)-sum assertions follow from Corollaries~\ref{cor:unconditional-algebra-sums} and~\ref{cor:finite-products}.
For \(B_n=\K\), Corollary~\ref{cor:boundedly-complete-unitization} also gives the Schauder basis assertion.
This scalar case has the coordinatewise multiplication in \cite{MotakisPelczar2025}, with the norm \(\norm u+\abs\lambda\) used here.

Constant component towers give arbitrary finite-dimensional summands.
For nonzero unital summands with normalized units, Proposition~\ref{prop:unconditional-sum-converse} gives a converse for the finite-dimensional quotient representation.
Remark~\ref{rem:fatou-algebra-sums} extends the construction to solid sequence spaces with the Fatou property.

Table~\ref{tab:realization-examples} lists examples from Section~\ref{sec:examples}, with their products, norms, and realizing corollaries.
Here \(F\) is finite dimensional with \(\norm{1_F}=1\), and \(U\) has a normalized \(1\)-unconditional boundedly complete Schauder basis.
The \(B_n\)'s satisfy Corollary~\ref{cor:intro-coordinatewise}, which defines \(\norm b_U\), and \(I\) denotes the adjoined unit.
\clearpage
\begingroup
\small
\renewcommand{\arraystretch}{1.16}
\setlength{\tabcolsep}{3pt}
\setlength{\LTleft}{0pt}\setlength{\LTright}{0pt}
\setlength{\LTpre}{7pt}\setlength{\LTpost}{7pt}
\newcommand{\AIExampleMark}{\hyperref[fn:ai-table-examples]{\textsuperscript{\ref*{fn:ai-table-examples}}}}
\begin{longtable}{@{}>{\raggedright\arraybackslash}m{0.25\textwidth}
  >{\centering\arraybackslash}m{0.24\textwidth}
  >{\centering\arraybackslash}m{0.13\textwidth}
  >{\centering\arraybackslash}m{0.24\textwidth}
  >{\centering\arraybackslash}m{0.075\textwidth}@{}}
  \caption[Examples of isometric Calkin realizations]{Examples of isometric Calkin realizations\protect\footnotemark}\label{tab:realization-examples}\\
  \toprule
  \textbf{Example}&\textbf{Algebra}&\textbf{Product}&\textbf{Norm}&\textbf{Cor.}\\
  \midrule
  \endfirsthead
  \multicolumn{5}{@{}l}{\small\itshape Table~\thetable\ (continued)}\\[3pt]
  \toprule
  \textbf{Example}&\textbf{Algebra}&\textbf{Product}&\textbf{Norm}&\textbf{Cor.}\\
  \midrule
  \endhead
  \midrule
  \multicolumn{5}{r@{}}{\small\itshape Continued on the next page}\\
  \endfoot
  \bottomrule
  \endlastfoot
  \multicolumn{5}{@{}l}{\itshape Finite algebras and sequence constructions}\\[2pt]
  Finite-dimensional algebras\footnotetext{\label{fn:ai-table-examples}The concrete examples marked by this superscript were suggested by GPT-5.6 Sol (OpenAI) and GPT-6 Astra (OpenAI). Their contribution to the choice of examples is described in the \hyperref[sec:ai-assistance]{AI assistance statement}.}&\(F\)&Given&\(\norm{\,\cdot\,}_F\)&\ref{cor:finite-dimensional-algebras}\\
  Finite incidence algebras\AIExampleMark&\(\operatorname{IA}(P)\)&Incidence&\(\max_x\sum_{y\geq x}\abs{a_{xy}}\)&\ref{cor:finite-incidence}\\
  Bounded scalar sequences&\(\ell_\infty=C(\beta\N)\)&Pointwise&\(\norm{\,\cdot\,}_\infty\)&\ref{ex:ell-infinity}\\
  \(\ell_\infty\)-sums of matrices&\(\bigl(\bigoplus_{k\geq1}M_k(\mathbb C)\bigr)_{\ell_\infty}\)&Pointwise matrix&\(\sup_k\norm{a_k}_{\rm op}\)&\ref{cor:finite-products}\\
  Unitized \(U\)-sums&\(\bigl(\bigoplus_nB_n\bigr)_U\oplus\K I\)&Unitization\(^{\dagger}\)&\(\norm b_U+\abs\lambda\)&\ref{cor:unconditional-algebra-sums}\\
  Unitized basis algebras&\(U\oplus\K I\)&Unitization\(^{\dagger}\)&\(\norm u_U+\abs\lambda\)&\ref{cor:boundedly-complete-unitization}\\
  \(\ell_p\) spaces&\(\ell_p\oplus\K I\)&Unitization\(^{\dagger}\)&\(\norm u_p+\abs\lambda\)&\ref{cor:block-lp-unitizations}\\
  Tsirelson's original space&\(T^*\oplus\K I\)&Unitization\(^{\dagger}\)&\(\norm u_{T^*}+\abs\lambda\)&\ref{cor:tsirelson-original}\\
  \addlinespace[5pt]
  \multicolumn{5}{@{}l}{\itshape Discrete convolution and incidence algebras}\\[2pt]
  Free-monoid convolution\AIExampleMark&\(\ell_1(\mathbb F_d^+)\)&Convolution&\(\sum_w\abs{a_w}\)&\ref{cor:free-semigroup}\\
  Dirichlet convolution\AIExampleMark&\(\ell_1(\N_\times)\)&Dirichlet convolution&\(\sum_{n\geq1}\abs{a_n}\)&\ref{cor:dirichlet-convolution}\\
  \(c_0\)-incidence algebras\AIExampleMark&\(\operatorname{IA}(P,c_0)\)&Incidence&\(\sup_x\sum_{y\geq x}\abs{a_{xy}}\)&\ref{cor:lower-finite-incidence}\\
  Upper-triangular matrices\AIExampleMark&\(\operatorname{UT}_\infty\)&Matrix&\(\sup_i\sum_{j\geq i}\abs{a_{ij}}\)&\ref{cor:upper-triangular}\\
  \addlinespace[5pt]
  \multicolumn{5}{@{}l}{\itshape Function and operator algebras}\\[2pt]
  H\"older algebras\AIExampleMark&\(C^{0,\theta}([0,1])\)&Pointwise&\(\norm f_\infty+[f]_\theta\)&\ref{cor:scalar-holder}\\
  Hardy algebra\AIExampleMark&\(H^\infty(\mathbb D)\)&Pointwise&\(\norm{\,\cdot\,}_\infty\)&\ref{cor:hardy-oracles}\\
  Matrix RKHS multipliers\AIExampleMark&\(M_q(\operatorname{Mult}(\mathcal H))\)&Pointwise matrix&\(\text{multiplier norm}\)&\ref{cor:rkhs-multipliers}\\
  Atomic nest algebras\AIExampleMark&\(\operatorname{Alg}\mathcal N\)&Composition&\(\text{operator norm}\)&\ref{cor:omega-atomic-nest}\\
  Free semigroupoid algebras\AIExampleMark&\(\mathfrak L_G\)&Composition&\(\text{operator norm}\)&\ref{cor:graph-semigroupoid}\\
  \addlinespace[5pt]
  \multicolumn{5}{@{}l}{\itshape Classical harmonic analysis}\\[2pt]
  Measure algebras\AIExampleMark&\(M(\mathbb T),\ M(SU(2))\)&Convolution&\(\norm\mu_{\rm TV}\)&\ref{cor:compact-measures}\\
  Fourier--Stieltjes algebra\AIExampleMark&\(B(\mathbb F_2)\)&Pointwise&\(\norm{\,\cdot\,}_{C^*(\mathbb F_2)^*}\)&\ref{cor:fourier-stieltjes}\\
  Fourier multipliers\AIExampleMark&\(MA(\Gamma)\)&Pointwise&\(\norm{M_\varphi}\)&\ref{cor:fourier-multipliers}\\
  Herz--Schur multipliers\AIExampleMark&\(M_{\mathrm{cb}}A(\Gamma)\)&Pointwise&\(\norm{M_\varphi}_{\mathrm{cb}}\)&\ref{cor:fourier-multipliers}\\
  \addlinespace[5pt]
  \multicolumn{5}{@{}l}{\itshape Quantum convolution algebras}\\[2pt]
  Universal duals\AIExampleMark&\(C^u(\mathbb G)^*\)&Convolution&\(\text{dual norm}\)&\ref{cor:compact-quantum}\\
  Quantum hypergroup duals\AIExampleMark&\(C^*\)&Convolution&\(\text{dual norm}\)&\ref{cor:compact-quantum-hypergroups}\\
  Duals with bounded counit\AIExampleMark&\(C_\mu(\mathbb G)^*\)&Convolution&\(\text{dual norm}\)&\ref{cor:quantum-completions}\\
  Unitized duals\AIExampleMark&\(C_\mu(\mathbb G)^*\oplus_1\mathbb C\)&Unitization\(^{\dagger}\)&\(\norm\omega+\abs\lambda\)&\ref{cor:quantum-completions}\\
\end{longtable}
\endgroup

\noindent\(^{\dagger}\) We use the standard Banach-algebra unitization, with 
\( (a,\lambda)(b,\mu)=\bigl(a\diamond b+\lambda b+\mu a,\lambda\mu\bigr), \)
where \(\diamond\) is coordinatewise multiplication in the four sequence-space rows and convolution in the last row.

Pointwise matrix products use \((ab)_k=a_kb_k\) and \((\Phi\Psi)(x)=\Phi(x)\Psi(x)\), respectively, with ordinary matrix multiplication.

In the H\"older row, \(0<\theta\leq1\), and we use
\[
  \begin{aligned}
    C^{0,\theta}([0,1])
    =\{f\in C([0,1]):[f]_\theta<\infty\},\qquad
    [f]_\theta
    =\sup_{0\leq s<t\leq1}\frac{\abs{f(t)-f(s)}}{\abs{t-s}^\theta}.
  \end{aligned}
\]
For \(\theta=1\), we have \(C^{0,1}([0,1])=\operatorname{Lip}([0,1])\) and \([f]_1=\operatorname{Lip}(f)\).
Here \(1\leq p<\infty\), \(d,q\in\N_+\), and \(T^*\) is Tsirelson's original space \cite{Tsirelson1974}, with the usual basis and the \(T^*\)-notation described in \cite[Section~2.4]{BaudierLancienSchlumprecht2018}.
The identification \(\ell_\infty=C(\beta\N)\) uses the Stone--\v Cech compactification of the discrete space \(\N\), as described in \cite[Section~V.6]{Conway1990}.
The free monoid \(\mathbb F_d^+\) includes the empty word, whereas \(\N_\times=\{1,2,\ldots\}\) has multiplication of integers.
The two \(\ell_1\)-algebras carry convolution, with Dirichlet convolution in the second case (see \cite[Section~3, Example~1]{Rota1964} for its divisor-poset interpretation).
In the incidence rows, \(P\) is a nonempty partially ordered set that is, respectively, finite or countable and lower-finite.
The latter means that \(\{x\in P:x\leq y\}\) is finite for every \(y\in P\).
The matrices are supported on \(x\leq y\), and incidence multiplication is the standard product from \cite[Section~3]{Rota1964}:
\[ (ab)_{xz}=\sum_{x\leq y\leq z}a_{xy}b_{yz}\qquad(x,z\in P). \]
For \(P=\N\) with its usual order, \(\operatorname{IA}(P,c_0)=\operatorname{UT}_\infty\), the upper-triangular matrices with finite displayed row norm.
For the \(c_0\)-incidence construction used here, see \cite[Definition~5.9 and Proposition~5.15]{Acuaviva2025Posets} and Corollary~\ref{cor:lower-finite-incidence}.

The space \(\mathcal H\) is a nonzero separable complex scalar reproducing-kernel Hilbert space, and the multiplier norm is that on \(\mathcal H\otimes\mathbb C^q\), with the kernel and multiplier terminology of \cite[Chapter~2]{AglerMcCarthy2002}.
The nest \(\mathcal N\) is \(\omega\)-atomic with finite-dimensional atoms \cite{Davidson1988}, and \(\mathfrak L_G\) is the canonical free semigroupoid algebra of a nonempty countable directed graph \(G\) \cite{KribsPower2004}.
Measure algebras carry convolution and the total-variation norm \cite{Dales2000}, whereas Fourier--Stieltjes and Fourier multiplier algebras carry pointwise multiplication \cite{Eymard1964,Knudby2014Semigroups}.
Here \(\Gamma\) is any countable discrete group, and \(M_\varphi\) denotes multiplication by \(\varphi\) on its Fourier algebra \(A(\Gamma)\).
In the quantum hypergroup row, \(C\) is the underlying separable \(C^*\)-algebra of a compact quantum hypergroup in the sense of Chapovsky--Vainerman \cite{ChapovskyVainerman1999}.
In the quantum-group rows, \(\mathbb G\) has separable \(C^u(\mathbb G)\), and \(C_\mu(\mathbb G)\) is a completion in a quantum group norm, so its coproduct extends to the minimal tensor product \cite{Woronowicz1998,KyedSoltan2012}.
The ununitized \(C_\mu(\mathbb G)^*\) row requires a bounded counit, while the unitized row includes reduced and exotic completions without this assumption.
All these duals carry convolution, and \((\omega,\lambda)\) has norm \(\norm\omega+\abs\lambda\) in the unitized row.


\section{Finite-dimensional quotients and their bounded limits}\label{sec:preliminaries}

We prove the assertions about \(\jmath\) in Corollary~\ref{cor:filtered-realization}, construct the predual of the bounded inverse limit, and prove the equivalences in Theorem~\ref{thm:finite-detection-duality}.
We retain the notation of Section~\ref{sec:detector-filtrations}: \(A_r=A/I_r\), \(\rho_r(a)=a+I_r\), \(\widehat A=\Env^b_{(I_r)}(A)\), and \(\jmath(a)=(\rho_r(a))_r\).

\subsection{Quotient inverse systems and the canonical map}

\begin{lemma}[The quotient inverse system]\label{lem:detector-quotient-tower}
  The maps \(\pi_r^s\) satisfy \eqref{eq:inverse-system}--\eqref{eq:metric-quotient}.
  Moreover, the infimum in \eqref{eq:metric-quotient} is attained.
\end{lemma}

\begin{proof}
  Since \(I_s\subseteq I_r\), we know that the map \(\pi_r^s\) is a well-defined surjective unital homomorphism.
  For \(r\leq s\leq t\), we have
  \[ \pi_r^s\pi_s^t(a+I_t)=a+I_r=\pi_r^t(a+I_t). \]
  If \(x_r=a+I_r\), we obtain from the quotient norms
  \[
    \begin{aligned}
      \inf_{\pi_r^sx_s=x_r}\norm{x_s}
      =\inf_{u\in I_r}\inf_{v\in I_s}\norm{a+u+v}=\inf_{w\in I_r}\norm{a+w} =\norm{x_r}.
    \end{aligned}
  \]
  Choose a minimizing sequence in the affine fibre \((\pi_r^s)^{-1}(x_r)\).
  Its norms are bounded, so finite dimensionality of \(A_s\) gives a convergent subsequence.
  The fibre is closed, so the limit is still a lift of \(x_r\), and continuity of the norm shows that it attains the infimum.
  Finally, the quotient is nonzero, and we have
  \[ 0<\norm{1_{A_r}}\leq1, \qquad \norm{1_{A_r}}\leq\norm{1_{A_r}}^2. \]
  Hence we obtain \(\norm{1_{A_r}}=1\), and prove both assertions.
\end{proof}

\begin{proposition}[The canonical homomorphism]\label{prop:canonical-detector-map}
  The bounded detection envelope \(\widehat A\) is a unital Banach algebra.
  The map \(\jmath:A\to\widehat A\) is a contractive unital algebra homomorphism with \(\ker\jmath=\bigcap_rI_r\).
  Moreover, we have
  \begin{equation}\label{eq:detected-core-norm}
    \norm{\jmath(a)}
    =\sup_r\norm{a+I_r}
    =\sup_r\dist(a,I_r).
  \end{equation}
  Thus \(\jmath\) is injective when \((I_r)\) is separating and isometric when \((I_r)\) is norming.
\end{proposition}

\begin{proof}
  The compatibility equations define a closed unital subalgebra of \((\bigoplus_rA_r)_{\ell_\infty}\), so \(\widehat A\) is a unital Banach algebra.
  For \(a,b\in A\), we have \(\rho_r(ab)=\rho_r(a)\rho_r(b)\) and \(\rho_r(1_A)=1_{A_r}\).
  Hence \(\jmath\) is a unital algebra homomorphism.
  Since every quotient map is contractive, we have
  \[ \norm{\jmath(a)}=\sup_r\norm{\rho_r(a)} =\sup_r\norm{a+I_r}\leq\norm a. \]
  Also, \(\jmath(a)=0\) if and only if \(a\in I_r\) for every \(r\), equivalently \(a\in\bigcap_rI_r\).
  Finally, we have \(\norm{a+I_r}=\inf_{u\in I_r}\norm{a+u}=\dist(a,I_r)\).
  This proves \eqref{eq:detected-core-norm} and the remaining assertions.
\end{proof}

For \(x_n\in A_n=A/I_n\), an element \(a_n\in A\) is a representative of \(x_n\) if \(a_n+I_n=x_n\), or equivalently \(\rho_n(a_n)=x_n\).

\begin{proposition}[Bounded coordinate approximation]\label{prop:bounded-detector-density}
  Let \(x=(x_r)_r\in\widehat A\) and \(\varepsilon>0\).
  There is a sequence \((a_n)\subseteq A\), with \(\rho_n(a_n)=x_n\), such that
  \[
    \norm{a_n}\leq\norm{x_n}+\varepsilon\leq\norm x+\varepsilon,
    \qquad
    \rho_r(a_n)=x_r\quad(n\geq r).
  \]
  Consequently, \(\jmath(A)\) is coordinate dense in \(\widehat A\) through uniformly bounded approximants.
\end{proposition}

This follows directly from the quotient norm and compatibility: choose \(a_n\in A\) with \(\rho_n(a_n)=x_n\) and \(\norm{a_n}\leq\norm{x_n}+\varepsilon\leq\norm x+\varepsilon\).
For \(n\geq r\), we then have \(\rho_r(a_n)=\pi_r^n\rho_n(a_n)=\pi_r^nx_n=x_r\).

The convergence in Proposition~\ref{prop:bounded-detector-density} is coordinatewise.
For example, the sequences \(a^{(n)}=\mathbf1_{\{0,\ldots,n\}}\in\ell_\infty\) satisfy
\[ a^{(n)}_r\longrightarrow1\quad(r\in\N),\qquad \text{while} \quad \norm{a^{(n)}-\mathbf1}_{\ell_\infty}=1\quad(n\in\N). \]

\begin{theorem}[Surjectivity from compactness]
  \label{thm:detector-compact-surjectivity}
  If every closed norm ball in \(A\) is compact in \(\tau_{(I_r)}\), then \(\jmath:A\to\widehat A\) is onto.
  If the filtration is also separating, then \(\jmath\) is an isometric unital algebra isomorphism.
\end{theorem}

\begin{proof}
  Let \(x\in\widehat A\), put \(R=\norm x\), and fix \(\varepsilon>0\).
  Proposition~\ref{prop:bounded-detector-density} gives \((a_n)\subseteq(R+\varepsilon)B_A\) such that \(\jmath(a_n)\to x\) in the coordinate topology.
  Thus \(x\) lies in the coordinate closure of \(\jmath((R+\varepsilon)B_A)\).
  The ball \((R+\varepsilon)B_A\) is compact in \(\tau_{(I_r)}\), and \(\jmath\) is continuous from \(\tau_{(I_r)}\) to the coordinate topology.
  Its image is therefore compact and hence closed in the Hausdorff coordinate topology.
  Thus there is \(a_\varepsilon\in A\) such that \(\jmath(a_\varepsilon)=x\) and \(\norm{a_\varepsilon}\leq R+\varepsilon\).

  Suppose that the filtration is separating.
  Then \(\jmath\) is injective, so all the elements \(a_\varepsilon\) are equal to one element \(a\).
  Hence we have \(\norm a\leq R+\varepsilon\) for every \(\varepsilon>0\), and therefore we obtain \(\norm a\leq R\).
  By contractivity of \(\jmath\), we obtain the reverse inequality \(R=\norm{\jmath(a)}\leq\norm a\).
  Thus we have \(\norm a=R\), as required.
\end{proof}

\begin{corollary}[Weak-star closed envelopes]\label{cor:weak-star-form}
  Suppose that \(A=E^*\) isometrically and every ideal \(I_r\) is weak-star closed.
  If \((I_r)\) is separating, then we have
  \[ A\simeq\Env^b_{(I_r)}(A) \]
  isometrically as unital Banach algebras.
\end{corollary}

\begin{proof}
  Put \(F_r=(I_r)_\perp =\{x\in E:a(x)=0\text{ for every }a\in I_r\}\).
  Since \(I_r\) is weak-star closed, we obtain \(I_r=F_r^\perp\) from the annihilator theorem.
  The space \(F_r\) is finite dimensional, and restriction to \(F_r\) identifies \(A/I_r\) isometrically with \(F_r^*\).
  Fix a basis \(u_1,\ldots,u_d\) of \(F_r\).
  By equivalence of finite-dimensional norms, there is \(C_r>0\) such that
  \[
    \norm{\rho_r(a^\alpha)-\rho_r(a)}
    =\norm{(a^\alpha-a)|_{F_r}}
    \leq C_r\max_{1\leq j\leq d}
      \abs{\langle a^\alpha-a,u_j\rangle}.
  \]
  If \(a^\alpha\to a\) in the weak-star topology, then the right-hand side tends to zero.
  Hence \(\rho_r:A\to F_r^*\) is weak-star-to-norm continuous.
  Thus \(\tau_{(I_r)}\) is weaker than the weak-star topology on every norm ball.
  By the Banach--Alaoglu theorem, we know that the ball is weak-star compact, and hence compact in \(\tau_{(I_r)}\).
  Then we can apply Theorem~\ref{thm:detector-compact-surjectivity}.
\end{proof}

The proof of Theorem~\ref{thm:detector-compact-surjectivity} only uses compactness for the quotient seminorms.
In Section~\ref{sec:examples}, we also verify this compactness by normal-family and equicontinuity arguments.

\subsection{The bounded-limit predual}
Let \((A_r,\pi_r^s)_{r\leq s}\) be any countable inverse system of finite-dimensional unital Banach algebras satisfying \eqref{eq:inverse-system}--\eqref{eq:metric-quotient}, and put \(\widehat A=\ilim(A_r,\pi_r^s)\).
Put \(G_r=A_r^*\) and \(J_r^s=(\pi_r^s)^*:G_r\to G_s\). 
Let \(G_{\rm alg}=\varinjlim(G_r,J_r^s)\) be the algebraic direct limit, and let \(J_r^\infty:G_r\to G_{\rm alg}\) be the canonical embeddings, so that \(J_s^\infty J_r^s=J_r^\infty\).
\begin{lemma}[Bounded-limit predual]\label{lem:bounded-limit}
\(\widehat A\) is a unital Banach algebra, and every
\(J_r^s\) is an isometry.
Equip \(G_{\rm alg}\) with the norm
\(\norm{J_r^\infty u}_{G_{\rm alg}}=\norm u_{G_r}\)
for \(u\in G_r\) and \(r\in\N\), and let \(\widehat A_*\) be its completion.
We use the same notation \(J_r^\infty\) for the embeddings of \(G_r\) into \(\widehat A_*\).
Then \(\widehat A_*^*=\widehat A\) isometrically, with the duality given by
\[
  \langle a,J_r^\infty u\rangle=\langle a_r,u\rangle.
\]
  Moreover, for every \(r\in\N\), the coordinate map
  \(q_r:\widehat A\to A_r\), given by \(q_r((a_s)_s)=a_r\),
  is a metric quotient and satisfies \(q_r(B_{\widehat A})=B_{A_r}\).
  On every norm-bounded subset of \(\widehat A\), the coordinate topology agrees with the weak-star topology induced by \(\widehat A_*\).
\end{lemma}

\begin{proof}
  The compatibility equations define a closed subspace of \((\bigoplus_r A_r)_{\ell_\infty}\).
  Coordinatewise multiplication preserves compatibility, and we have
  \[ \norm{ab}=\sup_r\norm{a_rb_r} \leq\left(\sup_r\norm{a_r}\right) \left(\sup_r\norm{b_r}\right)=\norm a\norm b. \]
  The compatible family \((1_{A_r})_r\) has norm one.
  Hence \(\widehat A\) is a unital Banach algebra.

  We next check that the maps used to form the predual preserve norms.
  Fix \(r\leq s\) and \(u\in A_r^*\).
  By contractivity of \(\pi_r^s\), we have
  \[ \norm{J_r^su} =\sup_{x_s\in B_{A_s}}\abs{\langle\pi_r^sx_s,u\rangle} \leq\sup_{x_r\in B_{A_r}}\abs{\langle x_r,u\rangle}=\norm u. \]
  Conversely, the fibre \((\pi_r^s)^{-1}(x_r)\) is closed in the finite-dimensional space \(A_s\), so the infimum in \eqref{eq:metric-quotient} is attained.
  Thus, for every \(x_r\in B_{A_r}\), there is \(x_s\in B_{A_s}\) with \(\pi_r^sx_s=x_r\) and \(\norm{x_s}=\norm{x_r}\).
  Therefore, we have \(\abs{\langle x_r,u\rangle}=\abs{\langle x_s,J_r^su\rangle} \leq\norm{J_r^su}\).
  Taking the supremum over \(x_r\in B_{A_r}\), we obtain \(\norm u\leq\norm{J_r^su}\), and hence \(\norm{J_r^su}=\norm u\).

  We now identify the dual of the completed direct limit.
  Let \(f\in\widehat A_*^*\), and restrict it to each stage.
  Since \(G_r^*=A_r^{**}=A_r\), there is a unique \(a_r\in A_r\) such that
  \[
    \langle a_r,u\rangle=f(J_r^\infty u)\qquad(u\in G_r).
  \]
  If \(r\leq s\), then we have
  \[
    \begin{aligned}
      \langle\pi_r^sa_s,u\rangle
      =\langle a_s,J_r^su\rangle
       =f(J_s^\infty J_r^su)=f(J_r^\infty u)
       =\langle a_r,u\rangle.
    \end{aligned}
  \]
  Thus \((a_r)\) is compatible.
  Moreover, we have \(\norm{a_r}\leq\norm f\), and hence \((a_r)\in\widehat A\) with \(\norm a\leq\norm f\).

  Conversely, let \(a=(a_r)\in\widehat A\).
  On the algebraic direct limit define \(f_a(J_r^\infty u)=\langle a_r,u\rangle\).
  To check that this is well defined, suppose that \(J_r^\infty u=J_s^\infty v\).
  At a common later stage \(t\geq r,s\), we have \(J_r^tu=J_s^tv\).
  Compatibility then gives
  \[
    \langle a_r,u\rangle
    =\langle a_t,J_r^tu\rangle
    =\langle a_t,J_s^tv\rangle
    =\langle a_s,v\rangle.
  \]
  Linearity follows by moving any two vectors to a common stage.
  Moreover, we have
  \[
    \abs{f_a(J_r^\infty u)}
    \leq\norm{a_r}\norm u
    \leq\norm a\norm{J_r^\infty u}.
  \]
  Therefore \(f_a\) extends to \(\widehat A_*\) and \(\norm{f_a}\leq\norm a\).
  On the other hand, we have
  \[ \norm{f_a}\geq\sup_r\sup_{u\in B_{G_r}} \abs{\langle a_r,u\rangle}=\sup_r\norm{a_r}=\norm a. \]
  Thus we obtain \(\norm{f_a}=\norm a\).
  For the family obtained from the original functional \(f\), the functionals \(f_a\) and \(f\) agree on the algebraic direct limit and hence on its closure.
  Therefore, we have \(\widehat A_*^*=\widehat A\) isometrically.

  Fix \(b\in A_r\).
  The minimum in \eqref{eq:metric-quotient} is attained because all stages are finite dimensional.
  Starting with \(b_r=b\), choose successively \(b_{s+1}\in A_{s+1}\) such that
  \[
    \pi_s^{s+1}b_{s+1}=b_s,
    \qquad
    \norm{b_{s+1}}=\norm{b_s}=\norm b
    \qquad(s\geq r).
  \]
  For \(t<r\), put \(b_t=\pi_t^rb\), so \(\norm{b_t}\leq\norm b\).
  By the inverse-system identities, we have
  \[
    \pi_t^ub_u=b_t\quad(t\leq u),
    \qquad
    \sup_t\norm{b_t}=\norm b.
  \]
  Thus \(a=(b_t)_t\in\widehat A\), with \(q_ra=b\) and \(\norm a=\norm b\).
  Thus every \(b\in A_r\) has a lift of the same norm.
  Since \(q_r\) is contractive, we obtain both the metric-quotient property and \(q_r(B_{\widehat A})=B_{A_r}\).

  It remains to compare the two topologies.
  On bounded subsets of \(\widehat A\), coordinate convergence is precisely \(\sigma(\widehat A,\widehat A_*)\)-convergence.
  Indeed, suppose that \(a^\alpha,a\in R B_{\widehat A}\) and \(a_r^\alpha\to a_r\) for every \(r\).
  For \(e\in\widehat A_*\) and \(u\in G_r\), we have
  \[
    \begin{aligned}
      \abs{\langle a^\alpha-a,e\rangle}
      \leq\abs{\langle a^\alpha-a,e-J_r^\infty u\rangle}
       +\abs{\langle a_r^\alpha-a_r,u\rangle}\leq2R\norm{e-J_r^\infty u}
      +\norm{a_r^\alpha-a_r}\norm u.
    \end{aligned}
  \]
  Given \(\varepsilon>0\), density gives \(r\) and \(u\in G_r\) such that
  \(\norm{e-J_r^\infty u}<\varepsilon/(4R+1)\).
  Keeping this \(r\) and \(u\) fixed, we obtain
  \[
    \limsup_\alpha\abs{\langle a^\alpha-a,e\rangle}
    \leq2R\norm{e-J_r^\infty u}<\varepsilon/2.
  \]
  Since \(\varepsilon\) is arbitrary, we obtain weak-star convergence.
  Conversely, weak-star convergence gives weak convergence in each finite-dimensional \(A_r\), hence norm convergence.
\end{proof}

\begin{theorem}[Duality for bounded inverse limits]
  \label{thm:finite-detection-duality}
  Let \(B\) be a unital Banach algebra with \(\norm{1_B}=1\).
  The following assertions are equivalent.
  \begin{enumerate}[label={\textup{(\roman*)}},leftmargin=27pt]
    \item There is a countable inverse system \((A_r,\pi_r^s)_{r\leq s}\) of finite-dimensional unital Banach algebras with unital metric-quotient bonding homomorphisms such that \(B\simeq\ilim(A_r,\pi_r^s)\) isometrically as unital Banach algebras.
    \item There are a Banach space \(E\), an isometric identification \(B=E^*\), and weak-star-to-norm continuous unital homomorphisms \(\rho_n:B\to F_n\), \(n\in\N\), where every \(F_n\) is finite dimensional, such that \(\bigcap_{n=0}^\infty\ker\rho_n=\{0\}\).
    \item There are a Banach space \(E\), an isometric identification \(B=E^*\), and increasing nonzero finite-dimensional subspaces \(E_0\subseteq E_1\subseteq\cdots\subseteq E\) with \(\overline{\bigcup_rE_r}=E\), such that every annihilator \(E_r^\perp=\{b\in B:\langle b,e\rangle=0 \text{ for every }e\in E_r\}\) is a proper two-sided ideal of \(B\).
  \end{enumerate}
  Whenever these assertions hold, the predual may be chosen separable.
\end{theorem}

\begin{proof}
  We prove \(\textup{(i)}\Rightarrow\textup{(ii)}\Rightarrow\textup{(iii)}\Rightarrow\textup{(i)}\).
  The first implication reads off the coordinate maps.
  The second takes their common kernels, and the third reconstructs each functional from its compatible restrictions.

  Suppose first that \textup{(i)} holds.
  Lemma~\ref{lem:bounded-limit} gives the isometric predual \(E=\overline{\bigcup_rJ_r^\infty A_r^*}\) of the bounded inverse limit.
  After transporting this predual to \(B\), let \(\rho_r:B\to A_r\) be the coordinate map.
  Under \(A_r^{**}=A_r\), this map is the adjoint of the canonical inclusion \(J_r^\infty:A_r^*\to E\).
  Hence \(\rho_r\) is weak-star-to-norm continuous.
  The coordinate maps are unital homomorphisms, and \(\bigcap_r\ker\rho_r=\{0\}\) because the coordinates determine every element of the inverse limit.
  This proves \textup{(ii)}.

  Suppose that \textup{(ii)} holds, and put \(I_r=\bigcap_{n=0}^r\ker\rho_n\) and \(E_r=(I_r)_\perp=\{e\in E:\langle b,e\rangle=0 \text{ for every }b\in I_r\}\).
  The map \(b\mapsto(\rho_0(b),\ldots,\rho_r(b))\) has finite-dimensional range and kernel \(I_r\).
  Hence \(I_r\) has finite codimension.
  Since each \(\rho_n\) is weak-star-to-norm continuous, its kernel is weak-star closed.
  Each kernel is a two-sided ideal, and unitality gives \(1_B\notin I_r\).
  Thus \(I_r\) is proper and weak-star closed.
  The annihilator identities show that
  \[
    E_r^\perp=I_r,\qquad
    E_r^*\simeq B/I_r,\qquad
    0<\dim E_r=\dim(B/I_r)<\infty.
  \]
  The ideals decrease, and therefore the spaces \(E_r\) increase.
  By joint faithfulness, we have
  \[ \left(\overline{\bigcup_rE_r}\right)^\perp =\bigcap_rE_r^\perp=\bigcap_rI_r =\bigcap_{n=0}^\infty\ker\rho_n=\{0\}. \]
  By the Hahn--Banach theorem, we now obtain \(\overline{\bigcup_rE_r}=E\), which proves \textup{(iii)}.

  Suppose next that \textup{(iii)} holds.
  Put \(I_r=E_r^\perp\) and \(A_r=B/I_r\).
  By the restriction form of the Hahn--Banach theorem, we obtain
  \[ A_r\longrightarrow E_r^*,\qquad b+I_r\longmapsto b|_{E_r}, \qquad \norm{b+I_r}=\norm{b|_{E_r}}. \]
  For \(r\leq s\), the bonding map is restriction \(E_s^*\to E_r^*\).
  It is a unital homomorphism because the \(I_r\)'s are ideals, and, by the Hahn--Banach theorem, we obtain
  \[ \norm{\phi} =\min\{\norm{\psi}:\psi\in E_s^*,\ \psi|_{E_r}=\phi\} \qquad(\phi\in E_r^*). \]
  Also, we have \(0<\norm{1_{A_r}}\leq1\) and \(\norm{1_{A_r}}\leq\norm{1_{A_r}}^2\), so \(\norm{1_{A_r}}=1\).

  Let \((\phi_r)_r\) be a bounded compatible family and put \(C=\sup_r\norm{\phi_r}\).
  On \(\bigcup_rE_r\), define
  \[ b(e)=\phi_r(e)\quad(e\in E_r), \qquad \abs{b(e)}\leq C\norm e. \]
  If \(e\in E_r\subseteq E_s\), compatibility gives \(\phi_s(e)=\phi_r(e)\), so \(b\) is well defined.
  Addition and scalar multiplication may be checked in a common \(E_s\), proving linearity.
  The displayed bound and density of \(\bigcup_rE_r\) give a unique extension \(b\in E^*=B\), with \(b|_{E_r}=\phi_r\), and we have
  \[ \sup_r\norm{\phi_r}\leq\norm b\leq C =\sup_r\norm{\phi_r}. \]
  Hence \(b\mapsto(b+I_r)_r\) is an onto isometry.
  Since all its coordinates are unital homomorphisms, we conclude that it is a unital algebra isomorphism.
  This proves \textup{(i)}.
  Finally, \(E=\overline{\bigcup_rE_r}\) is separable, as asserted.
\end{proof}

  \part{Oracle-recursive finite extensions}


\section{The block Bourgain--Delbaen construction}\label{sec:bd-extension}

The next lemma is a block form of the Bourgain--Delbaen extension estimate \cite{BourgainDelbaen1980}, and its scalar version appears in \cite[Theorem~3.4]{ArgyrosHaydon2011}.
Its bound is independent of the dimensions of the blocks.
This allows us to use the finite-dimensional algebras and their duals as fibres in Section~\ref{sec:formal-datum}.
Subscripts \(*\) denote chosen preduals, while superscripts \(*\) are also used for their correction maps and projections.

 Let \(\Delta_n\) be finite sets, let \(E_\gamma\) be finite-dimensional Banach spaces, and put
  \[ F_n=\left(\bigoplus_{\gamma\in\Delta_n}E_\gamma\right)_{\ell_\infty}, \qquad \mathcal E_*=\left(\bigoplus_{n\geq1}F_n^*\right)_{\ell_1}. \]
  Write \(\iota_n:F_n^*\to\mathcal E_*\) for the ambient injection.
  Consider a family
  \[ d_n^*\varphi=\iota_n\varphi-c_n^*\varphi,\qquad \varphi\in F_n^*, \]
  where \(c_n^*:F_n^*\to\bigoplus_{k<n}F_k^*\).
  We call \(c_n^*\varphi\) the correction functional.  In the scalar case, \(c_n^*\) is the Bourgain--Delbaen extension functional in the terminology of Motakis \cite{Motakis2024}.
  Denote by \(P_m^*\) the algebraic projection onto the first \(m\) ranges \(d_k^*(F_k^*)\), along the later \(d^*\)-blocks.
  We use the conventions \( P_0^*=0\) and \( c_1^*=0\).

\begin{figure}[htbp]
	\centering
\begin{tikzpicture}[x=1.30cm,y=0.60cm,
	every node/.style={font=\small,inner sep=2pt}]
	\foreach \r in {1,...,5} {
		\foreach \k in {1,...,5} {
			\ifnum\k<\r
				\fill[blue!7] ({\k-1},{-\r}) rectangle (\k,{-\r+1});
			\else
				\ifnum\k=\r
					\fill[green!10] ({\k-1},{-\r}) rectangle (\k,{-\r+1});
				\else
					\fill[black!2] ({\k-1},{-\r}) rectangle (\k,{-\r+1});
				\fi
			\fi
		}
	}
	\draw[black!30,step=1] (0,-5) grid (5,0);
	\node at (2.5,1.25) {Ambient blocks};
	\node[anchor=east] at (-0.18,0.50) {BD maps};
	\foreach \k/\lab in {1/{F_1^*},2/{F_2^*},3/{F_3^*},4/{\cdots},5/{F_n^*}}
		\node at ({\k-0.5},0.50) {\(\lab\)};
	\foreach \r/\lab in {1/{d_1^*},2/{d_2^*},3/{d_3^*},4/{\vdots},5/{d_n^*}}
		\node[anchor=east] at (-0.18,{-\r+0.5}) {\(\lab\)};
	\foreach \r/\k/\entry in {
		1/1/{I_{F_1^*}},1/2/{0},1/3/{0},1/4/{\cdots},1/5/{0},
		2/1/{\bullet},2/2/{I_{F_2^*}},2/3/{0},2/4/{\cdots},2/5/{0},
		3/1/{\bullet},3/2/{\bullet},3/3/{I_{F_3^*}},3/4/{\cdots},3/5/{0},
		4/1/{\vdots},4/2/{\vdots},4/3/{\vdots},4/4/{\ddots},4/5/{\vdots},
		5/1/{\bullet},5/2/{\bullet},5/3/{\bullet},5/4/{\cdots},5/5/{I_{F_n^*}}}
		\node at ({\k-0.5},{-\r+0.5}) {\(\entry\)};
\end{tikzpicture}
	\caption{The unit lower triangular BD structure.}
	\label{fig:bd-triangular}
\end{figure}

Figure~\ref{fig:bd-triangular} displays the ambient components of \(d_r^*\), with rows indexed by \(r\) and columns by the ambient rank \(k\).
The diagonal entries are identities, and each \(\bullet\) is a possibly zero component of \(-c_r^*\). All entries with \(k>r\) vanish.

We call the corrections in \eqref{eq:block-type-zero} and \eqref{eq:block-type-one} type zero and type one, respectively. Their scalar versions correspond to Types~1 and~2 in \cite[Definition~3.3]{ArgyrosHaydon2011}. Motakis \cite[Proposition~3.2]{Motakis2024} uses types~(a) and~(b) for the analogous scalar corrections without and with a predecessor term, respectively.

\begin{lemma}[Block BD lemma]\label{lem:block-bd}
 
  Let \(0\leq\theta<1/2\).
  On each summand \(E_\gamma^*\) of \(F_{n+1}^*\), assume that the correction has one of the forms
  \begin{align}
    c_{n+1}^*
    &=\beta(I-P_k^*)B^*,                                      \label{eq:block-type-zero}\\
    c_{n+1}^*
    &=\iota_jA^*+\beta(I-P_k^*)B^*,                            \label{eq:block-type-one}
  \end{align}
  where \(0\leq k\leq n\), \(0\leq\beta\leq\theta\), and \(B^*:E_\gamma^*\to(\bigoplus_{r=1}^nF_r^*)_{\ell_1}\) is a contraction.
  In the second form, \(1\leq j\leq k\) and \(A^*:E_\gamma^*\to F_j^*\) is a contraction.
  Then we have
  \[ \sup_s\norm{P_s^*}\leq M=(1-2\theta)^{-1}. \]
  The blocks biorthogonal to \(d_n^*(F_n^*)\) form a finite-dimensional decomposition (FDD) of a block BD space whose projections satisfy
  \[
    \sup_{s\geq1}\norm{P_s}\leq M,
    \qquad
    \sup_{0\leq a<b}\norm{P_{(a,b]}}\leq\kappa=2M.
  \]
\end{lemma}

\begin{proof}
  We first bound the algebraic projections on finite ambient sums, and then use their adjoints to construct the FDD.
  The change from the ambient blocks to the \(d^*\)-blocks is unit triangular on each finite initial sum.
  Therefore, we obtain
  \[
    P_s^*P_k^*=P_{\min\{s,k\}}^*,
    \qquad
    \operatorname{ran}P_s^*
    =\bigoplus_{j\leq s}\iota_j(F_j^*).
  \]
  In particular, we have \(P_s^*\iota_j=\iota_j\) for \(j\leq s\).
  We prove the estimate on the first \(n\) ambient blocks by induction on \(n\).
  The first block has zero correction.
  Suppose the estimate holds through rank \(n\), and fix \(s\leq n\).
  On a type-one summand, we have \(P_s^*d_{n+1}^*=0\), so we obtain
  \[
    P_s^*\iota_{n+1}\varphi
    =P_s^*\iota_jA^*\varphi
    +\beta(P_s^*-P_{\min\{s,k\}}^*)B^*\varphi.
  \]
  If \(s\leq k\), the second term vanishes, and the first term is controlled by the inductive bound on the first \(n\) ambient blocks.
  If \(s>k\), then \(j\leq k<s\), so the first term is \(\iota_jA^*\varphi\).
  Since \(B^*\varphi\) also lies in the first \(n\) ambient blocks, we have
  \[
    \begin{aligned}
      \norm{(P_s^*-P_k^*)B^*\varphi}
      \leq\norm{P_s^*B^*\varphi}+\norm{P_k^*B^*\varphi} \leq2M\norm{B^*\varphi}\leq2M\norm\varphi.
    \end{aligned}
  \]
  Consequently, we have
  \[
  \norm{P_s^*\iota_{n+1}\varphi} \leq \begin{cases} M\norm{A^*\varphi}\leq M\norm\varphi,&s\leq k,\\
  (1+2\beta M)\norm\varphi\leq M\norm\varphi,&s>k. \end{cases}
  \]
  On a type-zero summand, we obtain from the same calculation, with the predecessor term omitted,
  \[ \norm{P_s^*\iota_{n+1}\varphi} \leq2\beta M\norm\varphi \leq(M-1)\norm\varphi. \]
  Since \(F_{n+1}^*=(\bigoplus_{\gamma\in\Delta_{n+1}} E_\gamma^*)_{\ell_1}\), we obtain, on writing \(\varphi=(\varphi_\gamma)_\gamma\),
  \[
    \norm{P_s^*\iota_{n+1}\varphi}
    \leq\sum_\gamma\norm{P_s^*\iota_{n+1}\varphi_\gamma}
    \leq M\sum_\gamma\norm{\varphi_\gamma}
    =M\norm\varphi.
  \]
  Thus, for \(z\) in the first \(n\) ambient blocks, we have
  \[
    \begin{aligned}
      \norm{P_s^*(z+\iota_{n+1}\varphi)}
      \leq\norm{P_s^*z}+\norm{P_s^*\iota_{n+1}\varphi}\leq M(\norm z+\norm\varphi)
       =M\norm{z+\iota_{n+1}\varphi}.
    \end{aligned}
  \]
  For \(s=n+1\) the projection is the identity, completing the induction.

  Each \(P_s^*\) extends to \(\mathcal E_*\).
  If \(z_0^*\) is supported on the first \(N\) ambient blocks, then, for \(s\geq N\), we have
  \[ \norm{P_s^*z^*-z^*} \leq(M+1)\norm{z^*-z_0^*}. \]
  Hence, by density, we obtain \(P_s^*z^*\to z^*\).
  Let \(V_n:\mathcal E_*\to F_n^*\) be the \(n\)-th ambient coordinate map and define \(d_n=(V_nP_n^*)^*:F_n\to\mathcal E_*^*\).
  The top ambient coefficient of \(d_n^*\varphi\) is \(\varphi\), so we have
  \[ V_nP_n^*d_n^*=I_{F_n^*},\qquad V_nP_n^*d_k^*=0\quad(k\ne n). \]
  Hence we have \(\langle d_n^*\varphi,d_nu\rangle=\langle\varphi,u\rangle\), while every other \(d^*\)-block annihilates \(d_n(F_n)\).
  Thus \(d_n\) is injective.
  For \(u\in F_n\) and \(\psi\in F_k^*\), biorthogonality gives
  \[
    \begin{aligned}
      \langle d_k^*\psi,(P_s^*)^*d_nu\rangle
      =\langle P_s^*d_k^*\psi,d_nu\rangle=\mathbf1_{\{k\leq s\}}\langle d_k^*\psi,d_nu\rangle
       =\mathbf1_{\{n\leq s\}}\langle d_k^*\psi,d_nu\rangle.
    \end{aligned}
  \]
  Since the \(d^*\)-blocks span a dense subspace of \(\mathcal E_*\), we obtain
  \[ (P_s^*)^*d_nu=\mathbf1_{\{n\leq s\}}d_nu, \qquad \norm{(P_s^*)^*}\leq M. \]
  On \(X=\overline{\operatorname{span}}\{d_n(F_n):n\geq1\}\), the restrictions \(P_s=(P_s^*)^*|_X\) therefore converge strongly to the identity: if \(x_0\in\sum_{n\leq N}d_n(F_n)\) and \(s\geq N\), then we have \(\norm{P_sx-x}\leq(M+1)\norm{x-x_0}\).
  This proves the FDD assertion, with
  \[ \norm{P_s}\leq M, \qquad P_{(a,b]}=P_b-P_a, \qquad \norm{P_{(a,b]}}\leq2M=\kappa. \]
\end{proof}


\section{Finite modules and the oracle}\label{sec:oracle}

We use the tower \(\mathscr A=(A_r,\pi_r^s)\) and bounded limit \(\widehat A\) of Theorem~\ref{thm:main}.
As in Lemma~\ref{lem:bounded-limit}, \(G_r=A_r^*\) and \(J_r^s=(\pi_r^s)^*\) are the predual fibres and their isometric connecting maps.
We specify the finite modules and the exact statements used by the recursion.
The block estimate in Lemma~\ref{lem:block-bd} applies uniformly to these modules.

\subsection{Finite detector modules}

For \(b,c\in A_r\), write \(L_bc=bc\) and \(R_bc=cb\).
For \(a\in\widehat A\), let \(a\) act on the core fibre \(G_r=A_r^*\) by \(u\mapsto R_{a_r}^*u\), and on the unit-probe fibre \(H_r=A_r\) by \(b\mapsto a_rb\).
Both actions are contractive homomorphisms, since \(R_{a_r}^*R_{b_r}^*=(R_{b_r}R_{a_r})^*=R_{a_rb_r}^*\) and \(L_{a_r}L_{b_r}=L_{a_rb_r}\).
By the metric-quotient assertion in Lemma~\ref{lem:bounded-limit}, we have \(\{a_r1_{A_r}:a\in\widehat A\}=A_r\).

\subsection{Covariant finite packets}

For the reader's convenience, we recall some basic facts about Banach modules.

At a finite stage, an atom \(\eta\) carries a detector index \(r(\eta)\) and a finite-dimensional fibre \(E_\eta\).
We denote its module action by the contractive unital algebra homomorphism \(\Theta_\eta:A_{r(\eta)}\to\Bcal(E_\eta)\).
In operator terms, \(\Theta_\eta\) is a linear assignment of a bounded operator on \(E_\eta\) to each \(b\in A_{r(\eta)}\), satisfying
\[
  \Theta_\eta(bc)=\Theta_\eta(b)\Theta_\eta(c),\qquad
  \Theta_\eta(1)=I_{E_\eta},\qquad
  \norm{\Theta_\eta(b)}\leq\norm b.
\]
We write \(b\cdot u=\Theta_\eta(b)u\).
The first identity says that \((bc)\cdot u=b\cdot(c\cdot u)\) and this is the left module law.
We denote this space together with its action by \((E_\eta,\Theta_\eta)\).
For the two fibre types introduced above, we have
\[
  (E_\eta,\Theta_\eta(b))=
  \begin{cases}
    (G_{r(\eta)},R_b^*),&\text{for a core atom or a core probe},\\
    (H_{r(\eta)},L_b),&\text{for a unit probe}.
  \end{cases}
\]
We next describe the induced action on the dual space.
For each \(b\in A_{r(\eta)}\), the bounded operator \(\Theta_\eta(b):E_\eta\to E_\eta\) has a Banach adjoint \([\Theta_\eta(b)]^*:E_\eta^*\to E_\eta^*\).
Here the star applies to the whole operator \(\Theta_\eta(b)\).
Its action on a functional is given by composition:
\[
  \begin{aligned}
    \psi\cdot b
      :=[\Theta_\eta(b)]^*\psi=\psi\circ\Theta_\eta(b)\qquad
    (\psi\cdot b)(u)
      =\psi\bigl(\Theta_\eta(b)u\bigr).
  \end{aligned}
\]
Thus we first apply \(\Theta_\eta(b)\) to \(u\), and then evaluate the resulting vector with \(\psi\).
Using the bilinear evaluation pairing \(\langle\psi,u\rangle=\langle u,\psi\rangle=\psi(u)\),
we have, for \(b,c\in A_{r(\eta)}\), \(u\in E_\eta\), and \(\psi\in E_\eta^*\),
\[
  (\psi\cdot b)\cdot c=\psi\cdot(bc),\qquad
  \langle\psi\cdot b,u\rangle=\langle\psi,b\cdot u\rangle.
\]
The first identity makes \(E_\eta^*\) a right module, and the second expresses compatibility with the evaluation pairing.
Using the canonical identifications \(G_r^*=A_r\) and \(H_r^*=A_r^*\) with the biduals, we obtain the two dual right actions
\[
  c\cdot b=cb,\qquad (\psi\cdot b)(c)=\psi(bc).
\]

Write \(\epsilon_\eta^*\) for the injection of \(E_\eta^*\) into the ambient \(\ell_1\)-sum of fibre duals.
Fix an ordered packet representation \(\mathbf b=((c_l,\eta_l,\psi_l))_{l=1}^s\), with \(c_l\in\K\) and \(\psi_l\in E_{\eta_l}^*\), and write its value as
\begin{equation}\label{eq:packet}
  b=[\mathbf b]=\sum_{l=1}^sc_l\epsilon_{\eta_l}^*(\psi_l),\qquad
  \sum_{l=1}^s\abs{c_l}\norm{\psi_l}\leq1.
\end{equation}
Take \(r\geq r(\eta_l)\) for every entry.
For the empty packet, set \(b=0\) and \(B_{\mathbf b}^*=0\), with no restriction on \(r\).
For each entry, define \(R_l^{\mathbf b}:A_r\to E_{\eta_l}^*\) by specifying the value of \(R_l^{\mathbf b}\varphi\) on a vector:
\begin{equation}\label{eq:packet-map}
  (R_l^{\mathbf b}\varphi)(u)
  =\psi_l\bigl(\Theta_{\eta_l}(\pi_{r(\eta_l)}^r\varphi)u\bigr)
  \qquad(u\in E_{\eta_l}).
\end{equation}
Here we first apply \(\pi_{r(\eta_l)}^r\) to \(\varphi\), apply the corresponding operator to \(u\), and then evaluate with \(\psi_l\).
Equivalently, using the adjoint notation just introduced, we have
\[
  R_l^{\mathbf b}\varphi
  =[\Theta_{\eta_l}(\pi_{r(\eta_l)}^r\varphi)]^*\psi_l
  =\psi_l\cdot(\pi_{r(\eta_l)}^r\varphi).
\]
We then pass these functionals into the map \(B_{\mathbf b}^*:A_r\to\mathcal E_*\), given by
\begin{equation}\label{eq:packet-lift}
  B_{\mathbf b}^*\varphi
  =\sum_{l=1}^sc_l\epsilon_{\eta_l}^*(R_l^{\mathbf b}\varphi).
\end{equation}

\begin{lemma}[Covariant packet lemma]\label{lem:packet}
  Regard \(B_{\mathbf b}^*\) as taking values in the finite ambient sum
  \( \left(\bigoplus_{\eta\in\{\eta_1,\ldots,\eta_s\}} E_\eta^*\right)_{\ell_1}, \)
  where repeated atoms contribute to the same coordinate.
  Then \(B_{\mathbf b}^*\) is a contraction, and we have
  \( B_{\mathbf b}^*1_{A_r}=b. \)
  The adjoint \(S_l^{\mathbf b}=(R_l^{\mathbf b})^*:E_{\eta_l}\to G_r\) intertwines the left module actions, that is, for every \(a\in\widehat A\), we have
  \begin{equation}\label{eq:packet-covariance}
    S_l^{\mathbf b}\Theta_{\eta_l}(a_{r(\eta_l)})
    =R_{a_r}^*S_l^{\mathbf b}.
  \end{equation}
\end{lemma}

\begin{proof}
  Fix \(\varphi\in A_r\).
  By contractivity of the module actions and the bonding maps, we have \(\norm{R_l^{\mathbf b}\varphi}\leq\norm{\psi_l}\norm\varphi\).
  Hence we obtain
  \[
    \begin{aligned}
      \norm{B_{\mathbf b}^*\varphi}_{\ell_1}
      \leq\sum_{l=1}^s|c_l|\norm{R_l^{\mathbf b}\varphi}\leq\left(\sum_{l=1}^s|c_l|\norm{\psi_l}\right)\norm\varphi
       \leq\norm\varphi.
    \end{aligned}
  \]
  By unitality, we have \(R_l^{\mathbf b}1_{A_r}=\psi_l\), and hence \(B_{\mathbf b}^*1_{A_r}=b\).

  Fix \(a\in\widehat A\) and \(u\in E_{\eta_l}\).
  Since \(\pi_{r(\eta_l)}^r(\varphi a_r)=(\pi_{r(\eta_l)}^r\varphi)a_{r(\eta_l)}\) and \(\Theta_{\eta_l}\) is a homomorphism, we obtain
  \[
    \begin{aligned}
      \left\langle\varphi,S_l^{\mathbf b}\Theta_{\eta_l}(a_{r(\eta_l)})u\right\rangle
      &=\left\langle\psi_l,
        \Theta_{\eta_l}(\pi_{r(\eta_l)}^r\varphi)
        \Theta_{\eta_l}(a_{r(\eta_l)})u\right\rangle\\
      &=\left\langle\psi_l,
        \Theta_{\eta_l}\bigl((\pi_{r(\eta_l)}^r\varphi)a_{r(\eta_l)}\bigr)u\right\rangle\\
      &=\langle\varphi a_r,S_l^{\mathbf b}u\rangle
       =\langle\varphi,R_{a_r}^*S_l^{\mathbf b}u\rangle.
    \end{aligned}
  \]
  Since \(\varphi\) is arbitrary, we obtain \eqref{eq:packet-covariance}.
\end{proof}

The identity \(B_{\mathbf b}^*1_{A_r}=b\) will be used in the root perturbation in Lemma~\ref{lem:root}.

\subsection{The finite oracle}

Choose ordered bases in \(A_r\) and the dual bases in \(A_r^*\).
Let \(C\) be the countable set of structure constants of the units, multiplication maps, and bonding maps.
Write \(\Q_\K=\Q\) over \(\mathbb R\) and \(\Q_\K=\Q(i)\) over \(\mathbb C\).
For a subfield \(\mathbb F\subseteq\K\), let \(A_r(\mathbb F)\) and \(A_r^*(\mathbb F)\) denote the vectors with coordinates in \(\mathbb F\).
Put
\begin{equation}\label{eq:oracle-field}
  \begin{aligned}
    \mathbb F_0&=\Q_\K(C\cup\overline C),\\
    \mathbb F_{j+1}
    &=\mathbb F_j\bigl(\norm x:
      x\in A_r(\mathbb F_j)\cup A_r^*(\mathbb F_j),\ r\in\N\bigr),\\
    \mathbb F_{\mathscr A}&=\bigcup_{j\geq0}\mathbb F_j.
  \end{aligned}
\end{equation}
Here \(\overline C\) consists of the complex conjugates of the members of \(C\), and equals \(C\) in the real case.

\begin{lemma}[Coded normalization]\label{lem:coded-normalization}
  The field \(\mathbb F_{\mathscr A}\) is countable and dense in \(\K\).
  If \(0\ne x\in A_r(\mathbb F_{\mathscr A})\) or \(0\ne x\in A_r^*(\mathbb F_{\mathscr A})\), then we have
  \[ \norm x,\ \norm x^{-1}\in\mathbb F_{\mathscr A}, \qquad x/\norm x\text{ has coordinates in }\mathbb F_{\mathscr A}. \]
  In each fibre, the coded vectors in the closed unit ball are norm dense in that ball.
\end{lemma}

\begin{proof}
  A field generated by countably many scalars is countable.
  Since each fibre is finite dimensional, we have
  \[
    \begin{aligned}
      |\mathbb F_j|\leq\aleph_0
      \quad\Longrightarrow\quad
      \left|\bigcup_r\bigl(A_r(\mathbb F_j)\cup A_r^*(\mathbb F_j)\bigr)\right|
      \leq\aleph_0  \quad\Longrightarrow\quad |\mathbb F_{j+1}|\leq\aleph_0.
    \end{aligned}
  \]
  By induction, we obtain \(|\mathbb F_{\mathscr A}|=\aleph_0\), and density follows from \(\Q_\K\subseteq \mathbb F_0\).
  In the complex case, \(\mathbb F_0\) is closed under conjugation, and each later field is obtained by adjoining real scalars.
  Thus every \(\mathbb F_j\), and hence \(\mathbb F_{\mathscr A}\), is closed under conjugation.
  The fields are increasing, so the finitely many coordinates of \(x\) belong to one \(\mathbb F_j\).
  Hence we obtain
  \[ \norm x\in \mathbb F_{j+1},\qquad \norm x>0\ \Longrightarrow\ \norm x^{-1}\in \mathbb F_{j+1},\qquad x/\norm x\in \mathbb F_{j+1}^{\dim A_r}. \]
  To prove the density assertion, fix \(x\) in the unit ball and a coded vector \(v\) close to \(x\).
  Then \(w=v/\max\{1,\norm v\}\) is coded, \(\norm w\leq1\), and we have
  \[ \norm{w-x}\leq\norm{v-x}+(\norm v-1)_+ \leq2\norm{v-x}. \]
  Since coordinate convergence is norm convergence in a finite-dimensional fibre, we obtain the assertion from density of \(\mathbb F_{\mathscr A}\).
\end{proof}

Fix an enumeration \((\lambda_i)_{i\in\N}\) of \(\mathbb F_{\mathscr A}\), with \(\lambda_0=0\) and \(\lambda_1=1\).
A scalar is represented by its index, a vector by a finite list of indices and its fibre tag, and a matrix by its source and target tags and its entries.
Thus the syntax consists of finite words of natural numbers and has a fixed G\"odel coding~\cite{Godel1931}.
The syntax is independent of \(\mathscr A\), while the oracle supplies the interpretation of its scalar and fibre tags.

\begin{definition}[Inverse-limit detection oracle]\label{def:detection-system}
  Let \(\mathcal P_{\mathscr A}\) be the following family of finite statements, with all vectors and matrices well typed:
  \begin{enumerate}[label={\textup{(\roman*)}},leftmargin=27pt]
    \item dimensions, scalar-field tags, and coordinate types;
    \item the scalar and structure tables, including
    \[ \lambda_i+\lambda_j=\lambda_k,\quad \lambda_i\lambda_j=\lambda_k,\quad \overline{\lambda_i}=\lambda_k, \]
    and the coded entries of the units, multiplication, bonding maps, and module matrices;
    \item for a coded fibre vector \(x\), a coded finite matrix \(T\), and a coded real scalar \(q\), the comparisons
    \[ \norm x\mathrel{\triangleleft}q,\qquad \norm T\mathrel{\triangleleft}q, \qquad \triangleleft\in\{<,\leq,=,\geq,>\}; \]
    \item finite packet bounds and separation inequalities, including
    \[ \sum_{l=1}^s|c_l|\norm{\psi_l}\leq1, \qquad \operatorname{Re}f(z)>\varepsilon, \qquad f|_{\operatorname{ran}T}=0. \]
  \end{enumerate}
  The detection oracle is the truth set
  \begin{equation}\label{eq:oracle-truth-set}
    \mathcal O_{\mathscr A}
    =\{\ulcorner P\urcorner:P\in\mathcal P_{\mathscr A},\ P\text{ is true}\}.
  \end{equation}
\end{definition}

A query with code \(q\) returns the membership bit \(\mathbf1_{\mathcal O_{\mathscr A}}(q)\in\{0,1\}\).
Figure~\ref{fig:oracle-query} illustrates a query for \(P\in\mathcal P_{\mathscr A}\).
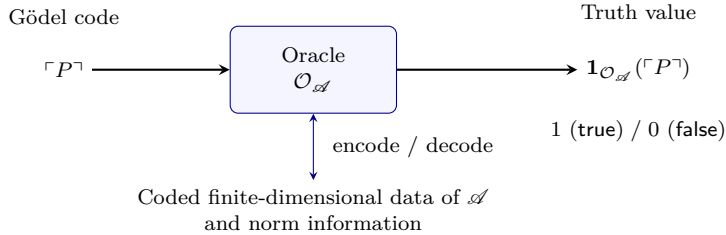
\begin{figure}[!htbp]
  \centering
\begin{tikzpicture}[x=1cm,y=1cm,>=stealth,
  every node/.style={font=\small,align=center}]
  \node (input) at (0.9,0) {\(\ulcorner P\urcorner\)};
  \node[above=8pt] at (input.north) {G\"odel code};
  \node[draw=blue!55!black,fill=blue!4,rounded corners=3pt,
    minimum width=2.2cm,minimum height=1.15cm] (oracle) at (4.2,0)
    {Oracle\\\(\mathcal O_{\mathscr A}\)};
  \node (output) at (8.5,0)
    {\(\mathbf1_{\mathcal O_{\mathscr A}}(\ulcorner P\urcorner)\)};
  \node[above=8pt] at (output.north) {Truth value};
  \node[below=8pt] at (output.south)
    {\(1\ (\mathsf{true})\;/\;0\ (\mathsf{false})\)};
  \draw[->,thick] (input.east) -- (oracle.west);
  \draw[->,thick] (oracle.east) -- (output.west);
  \node (data) at (4.2,-1.85)
    {Coded finite-dimensional data of \(\mathscr A\)\\and norm information};
  \draw[<->,blue!55!black] (data.north) -- node[right=4pt,text=black]
    {encode / decode} (oracle.south);
\end{tikzpicture}
  \caption{Oracle access for the fixed query language \(\mathcal P_{\mathscr A}\).}
  \label{fig:oracle-query}
\end{figure}

\begin{remark}\label{rem:oracle-arithmetic}
  To compute \(\lambda_i+\lambda_j\), we query whether \(\lambda_i+\lambda_j=\lambda_k\) for \(k=0,1,2,\ldots\) until the oracle answers \(1\).
  Since \(\lambda_i+\lambda_j\in\mathbb F_{\mathscr A}\), we find such a \(k\) after finitely many queries.
  We treat multiplication, conjugation, and inversion of nonzero scalars in the same way.
  Alternatively, we could encode scalars by expressions and form the codes of sums and products inside the machine.
\end{remark}

\begin{remark}[Oracle-recursive form]\label{rem:oracle-firewall}
  In Turing's terminology~\cite{Turing1939}, there is one oracle program \(\mathfrak M\), independent of the tower, which performs each transition in Construction~\ref{con:fair-recursion}.
  Write \(\mathcal D_n\) for the code of the finite state through rank \(n\), with \(\mathcal D_0\) the code of the empty state.
  Then we have
  \[ \mathfrak M^{\mathcal O_{\mathscr A}}(n,\mathcal D_{n-1}) =\mathcal D_n,\qquad \mathcal D_{n-1}\preccurlyeq\mathcal D_n, \]
  where \(\preccurlyeq\) means that the coded states extend one another without changing old records.
  Iterating this transition computes \(\mathcal D_n\) from \(n\) relative to \(\mathcal O_{\mathscr A}\).
  The finite transition and its termination are verified in Section~\ref{sec:formal-datum}.
  The oracle \(\mathcal O_{\mathscr A}\) encodes the finite-dimensional tower and may be noncomputable.
\end{remark}


\section{\texorpdfstring{Oracle-recursive BD construction}{Oracle-recursive BD construction}}\label{sec:formal-datum}

The coding of special histories in the Argyros--Haydon construction follows the method of Gowers and Maurey \cite{GowersMaurey1993,ArgyrosHaydon2011}.
The current realized outer stem determines the next inner level and the initial cutoff of the attached inner chain.
After the inner chain is complete, a compatible certificate permits one outer continuation.
We record these steps as separate finite requirements, and prove that every admissible continuation occurs cofinally.

Recall that \(\mathscr A=(A_r,\pi_r^s)\) is the detector tower, \(\pi_r^s:A_s\to A_r\) is its bonding map, and \(\widehat A=\ilim(A_r,\pi_r^s)\).
We apply the block construction of Section~\ref{sec:bd-extension} to the core fibres \(G_r=A_r^*\), with action \(R_b^*\), and the unit-probe fibres \(H_r=A_r\), with action \(L_b\), as in Section~\ref{sec:oracle}.
The finite recursion has four kinds of requirements: generic continuations, seeds, coded inner continuations, and outer continuations with a finite certificate.
Each transition reads only old records and finitely many answers from \(\mathcal O_{\mathscr A}\).
The resulting finite-extension property allows the operator argument to select suitable chains in the completed space.

\subsection{Finite states and parameters}

Fix the numerical parameters recursively by
\[ m_1=2^8,\qquad m_{j+1}=m_j^5,\qquad n_j=2^{j+8}L_jm_j^4, \]
where
\[
\begin{gathered} \mathsf b_1=1,\quad \mathsf b_j=6\max_{i<j}n_i\ (j\geq2),\quad D_j=\lceil2\log_2m_j\rceil,\\
S_j=1+\mathsf b_j+\cdots+\mathsf b_j^{D_j+2},\quad L_j=m_jS_j,\quad \Lambda_j=2^{j+12}S_jn_jm_j^4. \end{gathered}
\]
At step \(j\), first compute \(m_j,\mathsf b_j,D_j,S_j,L_j\), and then \(n_j,\Lambda_j\).
Thus all these integer sequences are computable, with \(\mathsf b_j\) serving as the branching bound used below.
Put
\begin{equation}\label{eq:constants}
  \begin{gathered}
    M=(1-2m_1^{-1})^{-1},\quad \kappa=2M,\quad
    C_{\rm tr}=2+2\kappa,\quad B=4\kappa+4,\\
    C_{\rm off}=8B\kappa,\quad
    K_{\rm dep}=32\kappa(C_{\rm off}+1),\quad C_{\rm ex}=2(3B+M).
  \end{gathered}
\end{equation}
These choices give
\begin{gather}
  \sum_{j=1}^\infty m_j^{-1}<1,\qquad n_j\geq m_j^2,                 \label{eq:numerical-one}\\
  \frac{S_j}{n_j}\leq2^{-j-8}m_j^{-5},\qquad
  \frac{S_j}{\Lambda_j}\leq2^{-j-12}m_j^{-4},                     \label{eq:numerical-two}\\
  m_j^{-1}\leq2^{-j-2},\qquad
  m_1^{-(D_j+2)}\leq m_j^{-2},\qquad
  m_{j+1}^{-1}\leq m_j^{-5},\qquad
  (n_jm_j)^{-1}\leq m_j^{-3}.                                    \label{eq:numerical-three}
\end{gather}
Indeed, we have \(m_j\geq2^{8\cdot5^{j-1}}\), and hence we obtain
\[ \sum_{j\geq1}m_j^{-1} \leq2^{-8}+\sum_{k\geq40}2^{-k}<1, \qquad m_j^{-1}\leq2^{-j-2}. \]
The remaining estimates follow from
\[ n_j\geq2^{j+8}S_jm_j^5,\qquad \frac{S_j}{\Lambda_j}=\frac{1}{2^{j+12}n_jm_j^4},\qquad m_1^{-(D_j+2)}\leq2^{-D_j}\leq m_j^{-2}. \]

Construction ranks belong to \(\N_+\), whereas detector indices belong to \(\N\).
At rank \(n\in\N_+\), there is a finite set \(\Delta_n\) of atoms.
Every atom \(\gamma\) carries a detector index \(r(\gamma)\leq n\), a module \((E_\gamma,\Theta_\gamma)\), and an  level \(\ell(\gamma)\).
Thus \(\rank\gamma=n\) is the construction rank, \(r(\gamma)\) selects the finite algebra, and \(\ell(\gamma)\) is the  level.
For a core record of level \(h\) below, we have \(\ell(\gamma)=h\) and its weight is \(m_h^{-1}\).
An atom records one fibre and its place in the construction. We use three atom tags on two fibre types. Figure~\ref{fig:atom-fibre-types} summarizes their information and uses.
\begin{enumerate}[label={\textup{(\roman*)}}]
  \item a core atom has \(E_\gamma=G_{r(\gamma)}=A_{r(\gamma)}^*\) and \(\Theta_\gamma(a)=R_a^*\).
  \item a unit probe has \(E_\gamma=H_{r(\gamma)}=A_{r(\gamma)}\) and \(\Theta_\gamma(a)=L_a\).
  \item a core probe has the core fibre and the core action, but has zero correction and is never used as a predecessor.
\end{enumerate}
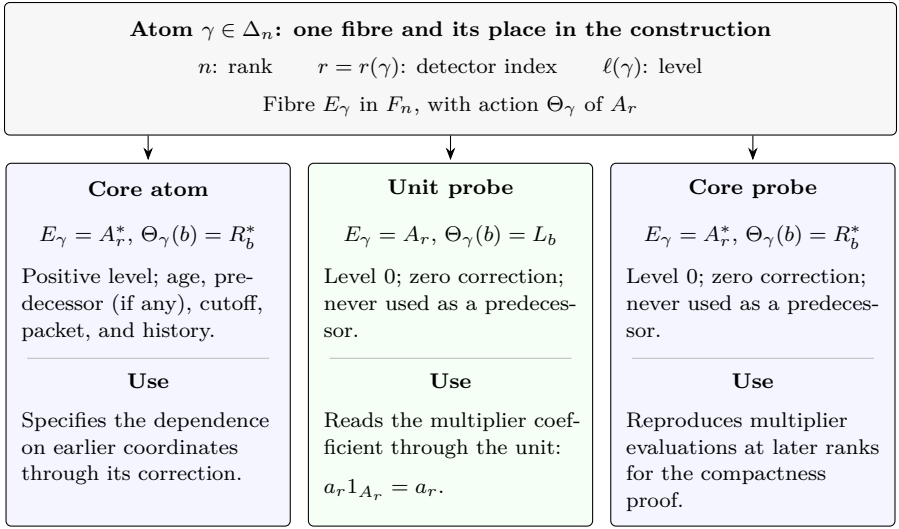
\begin{figure}[!htbp]
  \centering
\begin{tikzpicture}[x=1cm,y=1cm,>=Stealth,
  every node/.style={font=\small},
  heading/.style={align=center,font=\small\bfseries},
  body/.style={align=left,text width=3.35cm,anchor=north}]
  \node[draw,rounded corners=2pt,fill=black!3,align=center,
    inner sep=7pt,minimum width=11.8cm] (record) at (6,1.25)
    {\textbf{Atom \(\gamma\in\Delta_n\): one fibre and its place in the construction}\\[4pt]
     \(n\): rank\qquad \(r=r(\gamma)\): detector index\qquad \(\ell(\gamma)\): level\\[4pt]
     Fibre \(E_\gamma\) in \(F_n\), with action \(\Theta_\gamma\) of \(A_r\)};

  \foreach \x in {2,6,10} {
    \draw[->] (record.south -| \x,0) -- (\x,.04);
  }
  \draw[rounded corners=2pt,fill=blue!4] (.12,0) rectangle (3.88,-4.80);
  \draw[rounded corners=2pt,fill=green!4!white] (4.12,0) rectangle (7.88,-4.80);
  \draw[rounded corners=2pt,fill=blue!4] (8.12,0) rectangle (11.88,-4.80);

  \node[heading] at (2,-.33) {Core atom};
  \node[heading] at (6,-.33) {Unit probe};
  \node[heading] at (10,-.33) {Core probe};

  \node[align=center,anchor=north] at (2,-.70)
    {\(E_\gamma=A_r^*\), \(\Theta_\gamma(b)=R_b^*\)};
  \node[align=center,anchor=north] at (6,-.70)
    {\(E_\gamma=A_r\), \(\Theta_\gamma(b)=L_b\)};
  \node[align=center,anchor=north] at (10,-.70)
    {\(E_\gamma=A_r^*\), \(\Theta_\gamma(b)=R_b^*\)};

  \node[body] at (2,-1.28)
    {Positive level; age, predecessor (if any), cutoff, packet, and history.};
  \node[body] at (6,-1.28)
    {Level \(0\); zero correction; never used as a predecessor.};
  \node[body] at (10,-1.28)
    {Level \(0\); zero correction; never used as a predecessor.};

  \foreach \x in {2,6,10} {
    \draw[black!25] ({\x-1.60},-2.58) -- ({\x+1.60},-2.58);
    \node[heading] at (\x,-2.88) {Use};
  }
  \node[body] at (2,-3.18)
    {Specifies the dependence on earlier coordinates through its correction.};
  \node[body] at (6,-3.18)
    {Reads the multiplier coefficient through the unit:\\[5pt]
     \mbox{\(a_r1_{A_r}=a_r\).}};
  \node[body] at (10,-3.18)
    {Reproduces multiplier evaluations at later ranks for the compactness proof.};
\end{tikzpicture}
  \caption{Information carried by an atom and the use of each tag. Here \(r=r(\gamma)\), \(b\in A_r\), and \(a=(a_r)_r\in\widehat A\).}
  \label{fig:atom-fibre-types}
\end{figure}
The core atoms are the nonprobe atoms.
Both types of probes may occur in represented packets, but neither is used as a predecessor.
Their evaluation and norm estimates are given in Lemma~\ref{lem:formal-probe-isometry} and Proposition~\ref{prop:formal-shielding}.

Put
\[ F_n=\left(\bigoplus_{\gamma\in\Delta_n}E_\gamma\right)_{\ell_\infty}, \qquad \mathcal E_*=\left(\bigoplus_{n\geq1}F_n^*\right)_{\ell_1}. \]
For an atom \(\eta\), let \(\epsilon_\eta^*:E_\eta^*\to\mathcal E_*\) be the ambient predual injection.
A \emph{represented packet} is an ordered finite word
\begin{equation}\label{eq:represented-packet}
  \mathbf b=((c_l,\eta_l,\psi_l))_{l=1}^s,
  \qquad
  b=[\mathbf b]
  =\sum_{l=1}^sc_l\epsilon_{\eta_l}^*(\psi_l),
  \qquad
  \norm{\mathbf b}_{\rm raw}
  =\sum_{l=1}^s|c_l|\norm{\psi_l}\leq1.
\end{equation}
All coefficients, atoms, and fibre vectors in this word are part of the record.
In particular, zero entries and their positions are retained, so two packets can have the same value without being the same represented packet.
Put
\[ \supp_{\rm raw}\mathbf b =\{\rank\eta_l:c_l\psi_l\ne0\}. \]
If \(r\geq r(\eta_l)\) for every entry, the lift of \(\mathbf b\) to detector index \(r\) is the map \(B_{\mathbf b}^*:A_r\to\mathcal E_*\) defined by \eqref{eq:packet-map}--\eqref{eq:packet-lift}.
For the empty word, the value and the lift are zero, and the detector condition is vacuous.
Lemma~\ref{lem:packet} gives
\begin{equation}\label{eq:represented-packet-bounds}
  \norm{B_{\mathbf b}^*}\leq1,
  \qquad B_{\mathbf b}^*1_{A_r}=b.
\end{equation}
A represented packet is \emph{oracle-coded} when its coefficients and the coordinates of all its fibre vectors belong to \(\mathbb F_{\mathscr A}\).
Only oracle-coded packets are admitted to atom records and finite requirements.
An interval restriction is a derived analysis of the ordered entries rather than a new represented packet.
The restriction formula and the records retained by this operation are specified in Lemma~\ref{lem:formal-restriction}.

The coefficient field and oracle are those of Definition~\ref{def:detection-system}.
Every admissibility test uses finitely many coded finite-dimensional statements from that definition.
We use the finite-extension method relative to this oracle \cite[Section~V.2]{Odifreddi1989}.
In its Baire-category formulation, finite conditions determine basic open sets of infinite extensions, and eventual satisfaction of a requirement is expressed by a dense open set above a condition that activates it.
Simultaneous satisfaction of the countably many requirements then corresponds to a comeager set of branches; see \cite[Section~V.3]{Odifreddi1989}.
A finite condition is the finite datum below a rank, ordered by end extension.
Lemma~\ref{lem:formal-legality} proves persistence and least-code fairness, and Proposition~\ref{prop:finite-extension} gives the resulting cofinal realizations.
Proposition~\ref{prop:bd-legality} separately verifies that the prescribed corrections satisfy the block BD estimate.

\subsection{Stage recursion and finite requirements}

\subsubsection{Atom records and triangular coordinates}

The rank, weight, age, and predecessor are the usual BD--AH chain data \cite{ArgyrosHaydon2011,Tarbard2013,Motakis2024}.
The detector index specifies the finite algebra acting on the fibre.
Fix an injective coding \((n,s)\mapsto\operatorname{id}(n,s)\) by \(\N\), where \(s\) is the position of an atom among those of rank \(n\).
These positions are assigned in the fixed order
\[ \mathsf{core}<\mathsf{unit\mbox{-}probe}<\mathsf{core\mbox{-}probe}, \]
followed, for core atoms, by the requirement code and copy index.
Within each probe tag, order the canonical probes by their detector indices.
Identifiers are never reused.
A nonprobe atom \(\gamma\in\Delta_n\) has the core record
\begin{equation}\label{eq:atom-record}
  \mathfrak r_0(\gamma)
  =(n,\mathsf t,h,a,\chi,\omega,\xi,k,\mathbf b,
  r(\gamma),\vartheta,\mathfrak p,\mathfrak c,o),
\end{equation}
whose parameters have the following meaning.
\begin{center}
  {\small
  \renewcommand{\arraystretch}{1.08}
  \begin{tabular}{@{}c p{0.75\textwidth}@{}}
    \toprule
    Parameter & Meaning \\
    \midrule
    \(n\) & the birth rank of \(\gamma\) \\
    \(\mathsf t\) & \(\mathsf{generic}\),
    \(\mathsf{inner}\), or
    \(\mathsf{outer}\) \\
    \(h\) & the  level and root weight \(m_h^{-1}\) \\
    \(a\) & the age, with \(1\leq a\leq n_h\) \\
    \(\chi\) & the initial cutoff \\
    \(\omega\) & the root label, fixed along a predecessor chain \\
    \(\xi\) & the identifier of the predecessor, or \(\varnothing\) at age one \\
    \(k\) & the local cutoff and left endpoint of the cell \((k,n)\) \\
    \(\mathbf b\) & the represented packet in the current correction \\
    \(r(\gamma)\) & the detector index, with \(E_\gamma=G_{r(\gamma)}\) \\
    \(\vartheta,\mathfrak p,\mathfrak c\) & the seed, attached stem, and
    certificate, empty in the generic case \\
    \(o\) & the copy index \\
    \bottomrule
  \end{tabular}}
\end{center}
We write \(\chi(\gamma)\), \(\omega(\gamma)\), and \(\operatorname{age}(\gamma)\) for the corresponding entries of the record.
We write \(\xi\prec\gamma\) when \(\xi\) occurs earlier on the predecessor chain of \(\gamma\), and use \(^{\frown}\) for concatenation of finite words.
The entries satisfy the following compatibility conditions.
\begin{enumerate}[label={\textup{(A\arabic*)}},leftmargin=*]
  \item Every nonprobe atom is a core atom.
  Thus we have \(E_\gamma=G_{r(\gamma)}=A_{r(\gamma)}^*\) and \(\Theta_\gamma(b)=R_b^*\).
  \item If \(a=1\), then we have \(\xi=\varnothing\) and \(k=\chi\).
  If \(a>1\), the identifier \(\xi\) belongs to a unique old core atom, still denoted by \(\xi\), and we have
  \[ \rank\xi=k<n,\qquad \ell(\xi)=h,\qquad \operatorname{age}(\xi)=a-1,\qquad \omega(\xi)=\omega(\gamma),\qquad r(\xi)\leq r(\gamma). \]
  \item The packet uses only old atoms and is supported in the open cell:
  \begin{equation}\label{eq:packet-cell}
    \supp_{\rm raw}\mathbf b\subseteq(k,n),
    \qquad
    r(\eta)\leq r(\gamma)
    \quad(\eta\text{ occurring in }\mathbf b).
  \end{equation}
  \item The root label is
  \[
  \begin{cases} (\mathsf{generic},h,\chi,o_0), &\mathsf t=\mathsf{generic},\\
  (\mathsf{inner},h,\chi,\vartheta,\mathfrak p), &\mathsf t=\mathsf{inner},\\
  (\mathsf{outer},h,\chi,\vartheta), &\mathsf t=\mathsf{outer}, \end{cases}
  \]
  where \(o_0\) is the copy index of the age-one generic atom.
  \item The structural history is the finite word
  \begin{equation}\label{eq:structural-history}
    \mathsf{hist}(\gamma)=
    \begin{cases}
      ((\operatorname{id}\gamma,\mathfrak r_0(\gamma))),&a=1,\\
      \mathsf{hist}(\xi)^\frown
      ((\operatorname{id}\gamma,\mathfrak r_0(\gamma))),&a>1.
    \end{cases}
  \end{equation}
  It is used both for predecessor recursion and common-stem comparison.
\end{enumerate}
The cell of \(\gamma\) is \((k,n)\).
Along a predecessor chain the cells are successive and disjoint.
All core records and histories are fixed at birth and are never subsequently modified.

A probe record has the form
\begin{equation}\label{eq:probe-record}
  (n,\mathsf{probe},0,0,n,\varnothing,\varnothing,n,
  \varnothing,r,\tau,o),
  \qquad \tau\in\{\mathsf u,\mathsf c\}.
\end{equation}
For canonical probes, set \(o=0\).
For \(\tau=\mathsf u\) its fibre is \(H_r=A_r\) with action \(L_b\), whereas for \(\tau=\mathsf c\) its fibre is \(G_r=A_r^*\) with action \(R_b^*\).
Every probe has  level zero and zero correction.
It may occur in a later represented packet, but it is never admitted as a predecessor.

For every atom and \(\varphi\in E_\gamma^*\), let \(c_\gamma^*(\varphi)\) be the correction functional specified below and put
\begin{equation}\label{eq:triangular-coordinate}
  d_{\gamma,\varphi}^*
  =\epsilon_\gamma^*(\varphi)-c_\gamma^*(\varphi).
\end{equation}
We call \(d_{\gamma,\varphi}^*\) the Bourgain--Delbaen (BD) biorthogonal part of the ambient coordinate functional.
Once the ranks below \(n\) have been constructed, the unit lower-triangular change of coordinates identifies their ambient predual sum with the span of the corresponding \(d^*\)-blocks.
We denote by \(P_k^*\) the algebraic projection onto the blocks of rank at most \(k\), and put \(P_{(k,n)}^*=P_{n-1}^*-P_k^*\).
Thus every symbol in a rank-\(n\) correction is already defined before that correction is formed.

At a completed finite stage, invert this finite unit-triangular change of coordinates and take its finite-dimensional adjoint.
We use \(d_\gamma\), \(D_\gamma\), \(U_\gamma\), and \(e_{\gamma,\varphi}^*\) for the resulting block injection, block coefficient, ambient coordinate, and ambient coordinate evaluation.
Equations~\eqref{eq:finite-extension}--\eqref{eq:ambient-coordinate} below prove that these finite objects are compatible under end extension.
More explicitly, on the first \(N\) ambient blocks write the triangular change as \(I-C_N\), where \(C_N\) is strictly block triangular.
Then we have
\[ C_N^N=0,\qquad (I-C_N)^{-1}=\sum_{j=0}^{N-1}C_N^j. \]
Starting with the zero probe corrections, induction shows that every entry of \(C_N\), \((I-C_N)^{-1}\), and their finite adjoints belongs to \(\mathbb F_{\mathscr A}\).
At the inductive step, the only new entries come from the old projections, the coded packet maps, the bonding maps, and \(m_h^{-1}\), through finite sums and products.
Thus all finite coordinate maps used in a certificate have codes obtained by querying the scalar tables in Definition~\ref{def:detection-system}.

Every nonprobe atom \(\gamma\in\Delta_n\) has core fibre \(G_{r(\gamma)}\), so its dual fibre is \(A_{r(\gamma)}\).
If it has age one (type zero) and level \(h\), its correction is
\begin{equation}\label{eq:type-zero}
  c_\gamma^*\varphi
  =\frac{1}{m_h}P_{(k,n)}^*B_{\mathbf b}^*\varphi
  \qquad(\varphi\in A_{r(\gamma)}),
\end{equation}
where \(\mathbf b\) is an old represented packet and \(r(\gamma)\geq r(\eta)\) for every atom \(\eta\) in that packet.
If \(\gamma\) has age greater than one (type one), with old core predecessor \(\xi\), put
\begin{equation}\label{eq:type-one}
  c_\gamma^*\varphi
  =\epsilon_\xi^*(\pi_{r(\xi)}^{r(\gamma)}\varphi)
  +\frac{1}{m_h}P_{(k,n)}^*B_{\mathbf b}^*\varphi.
\end{equation}
All raw packet supports lie in the open cell \((k,n)\).
A probe has zero correction.
Lemma~\ref{lem:packet} makes the unprojected packet map \(B_{\mathbf b}^*\) contractive, and the predecessor map in \eqref{eq:type-one} is contractive.
These observations will be used in Proposition~\ref{prop:bd-legality}.

Fix the recursive partition of the positive levels
\[ \mathcal G=3\N+1,\qquad \mathcal C^{\rm out}=3\N+2,\qquad \mathcal C^{\rm in}=3\N+3, \]
into generic, outer, and coded-inner levels.
Generic levels admit the packet choices in \textup{(R1)}.
An inner level is determined by the current stem through \(\sigma\), and an outer continuation must carry the strong packet and certificate in \textup{(R4)}.
This is the history restriction used in the dependent sequence estimate.
Let \(\mathscr C_{\rm fin}\) be the set of finite well-typed words over the oracle language.
A \emph{formal stem} \(\mathfrak p\) is a word in \(\mathscr C_{\rm fin}\) having the entries displayed below, and its length \(|\mathfrak p|\) is the number of its outer entries.
Before the BD recursion begins, fix an injective map
\[ \sigma:\{(q,\mathfrak p):q\in\mathcal C^{\rm out},\ \mathfrak p\in\mathscr C_{\rm fin}\text{ a formal stem}\} \longrightarrow\mathcal C^{\rm in}. \]
Enumerate the domain by G\"odel code, using the oracle to decide finite typing.
Assign to each word the least unused inner level above every rank, level, and cutoff written in \((q,\mathfrak p)\) which also satisfies
\begin{equation}\label{eq:code-growth}
  m_{\sigma(q,\mathfrak p)}
  \geq2^{q+|\mathfrak p|+10}L_qn_qm_q^4
  \geq\Lambda_q.
\end{equation}
The choice is possible since \(\mathcal C^{\rm in}\) is infinite and \(m_j\to\infty\).
Also, we have
\[ \frac{2^{q+|\mathfrak p|+10}L_qn_qm_q^4}{\Lambda_q} =2^{|\mathfrak p|-2}m_q\geq2^{-2}m_1>1. \]
For any fixed input, only finitely many smaller G\"odel codes are inspected, and only finitely many inner levels have already been assigned.
Thus the search terminates and computes \(\sigma\) relative to \(\mathcal O_{\mathscr A}\).

\begin{definition}[Seeds, realized stems, and certificates]
  \label{def:formal-critical-data}
  A seed is a record
  \[ \vartheta=(q,\chi,o), \qquad q\in\mathcal C^{\rm out}, \]
  born at rank \(\chi\).
  A realized outer stem over \(\vartheta\) is
  \begin{equation}\label{eq:realized-stem}
    \mathfrak p=(\vartheta;
    (\operatorname{id}\zeta_1,\mathfrak r_0(\zeta_1),\mathfrak c_1),
    \ldots,
    (\operatorname{id}\zeta_a,\mathfrak r_0(\zeta_a),\mathfrak c_a)),
    \qquad 0\leq a\leq n_q,
  \end{equation}
  where \(\zeta_1\prec\cdots\prec\zeta_a\) is an outer core chain of level \(q\).
  Put
  \[
  \operatorname{cut}(\mathfrak p)= \begin{cases} \chi,&a=0,\\
  \rank\zeta_a,&a>0. \end{cases}
  \]
  Recursively, if \(\mathfrak p_{s-1}\) is the stem before the \(s\)-th triple, then we have \(p_s=\sigma(q,\mathfrak p_{s-1})\), and \(\eta_s\) is the terminal atom of an attached inner core chain of level \(p_s\) and age \(n_{p_s}\).
  The represented \emph{strong packet} on the outer edge is
  \begin{equation}\label{eq:strong-packet}
    \mathbf s(\eta_s)=((1,\eta_s,1_{A_{r(\eta_s)}})),
    \qquad [\mathbf s(\eta_s)]
    =\epsilon_{\eta_s}^*(1_{A_{r(\eta_s)}}).
  \end{equation}
  If \(r\geq r(\eta_s)\), then, because \(\eta_s\) is a core atom, we obtain
  \begin{equation}\label{eq:strong-packet-scalarization}
    B_{\mathbf s(\eta_s)}^*\varphi
    =\epsilon_{\eta_s}^*(\pi_{r(\eta_s)}^r\varphi)
    \qquad(\varphi\in A_r).
  \end{equation}

  Let \(\eta\) be such an inner terminal and let \(f_\eta=e_{\eta,1_{A_{r(\eta)}}}^*\).
  If a finite-dimensional subspace \(Y\) is given by a stored finite spanning matrix, define
  \[ \supp_{\rm FDD}Y =\{\rank\gamma:D_\gamma v\ne0 \text{ for some }v\in Y\}. \]
  Equivalently, this is the union of the coefficient supports of the stored columns.
  A certificate compatible with \((\mathfrak p,\eta)\) is a finite coded record
  \begin{equation}\label{eq:finite-certificate}
    \mathfrak c=(y,z,I,R,\mathscr M(y),\varepsilon)
  \end{equation}
  satisfying the following conditions.
  \begin{enumerate}[label={\textup{(C\arabic*)}}]
    \item The vectors \(y,z\), the interval \(I\), the integer \(R\), and the finite coefficient matrix defining \(\mathscr M(y)\) belong to the old coded finite-stage structure, and we have
    \[ \supp_{\rm FDD}y\cup\supp_{\rm FDD}z \cup\supp_{\rm FDD}\mathscr M(y) \subseteq I\subseteq(\operatorname{cut}(\mathfrak p),\infty), \qquad \rank\eta\in I. \]
    \item The integer \(R\) dominates \(r(\eta)\) and the detector indices of every atom \(\gamma\) for which \(D_\gamma y\ne 0\), and \(\mathscr M(y)\) is the image of \(A_R\) under the finite coefficient map
    \begin{equation}\label{eq:finite-orbit-space}
      b\longmapsto
      \sum_{D_\gamma y\ne 0}
      d_\gamma\bigl(
      \Theta_\gamma(\pi_{r(\gamma)}^Rb)D_\gamma y\bigr).
    \end{equation}
    The subspace is unchanged when \(R\) is increased, as verified below.
    \item \(U_\eta y=0\), equivalently \(f_\eta|_{\mathscr M(y)}=0\), and \(\operatorname{Re}f_\eta(z)>\varepsilon\) for some \(\varepsilon\in\Q_+\).
  \end{enumerate}
  To verify these assertions, write \(T_{R,y}:A_R\to\mathscr M(y)\) for the map in \eqref{eq:finite-orbit-space}.
  Compatibility and surjectivity give, for \(R'\geq R\),
  \[ T_{R',y}=T_{R,y}\pi_R^{R'},\qquad \operatorname{ran}T_{R',y}=\operatorname{ran}T_{R,y}. \]
  On each nonzero coefficient of \(y\), the definition gives
  \[ D_\gamma T_{R,y}(b) =\Theta_\gamma(\pi_{r(\gamma)}^Rb)D_\gamma y. \]
  Take finite adjoints in \eqref{eq:type-zero}--\eqref{eq:type-one}.
  Every ancestor used to compute \(U_\eta\) has detector index at most \(r(\eta)\leq R\).
  The predecessor maps intertwine by multiplicativity of the bonding maps, and the packet maps intertwine by \eqref{eq:packet-covariance}.
  For an FDD interval \(J\), deleting the coefficient rows outside \(J\) gives
  \[
    \begin{aligned}
      D_\gamma P_JT_{R,y}(b)
      =\mathbf1_{\{\rank\gamma\in J\}}
        \Theta_\gamma(\pi_{r(\gamma)}^Rb)D_\gamma y=D_\gamma T_{R,P_Jy}(b)
    \end{aligned}
  \]
  for each nonzero coefficient of \(y\), with both sides zero at the other atoms.
  Here the same formula \eqref{eq:finite-orbit-space} defines \(T_{R,P_Jy}\).
  Thus \(P_JT_{R,y}(b)=T_{R,P_Jy}(b)\), and the projected packet terms also intertwine.
  Starting with the zero-correction atoms and inducting on these old ranks, we obtain
  \[
    U_\eta T_{R,y}(b)=R_{\pi_{r(\eta)}^Rb}^*U_\eta y.
  \]
  Pairing this identity with \(1_{A_{r(\eta)}}\), we obtain
  \[
    f_\eta(T_{R,y}(b))
      =\left\langle1_{A_{r(\eta)}},
        R_{\pi_{r(\eta)}^Rb}^*U_\eta y\right\rangle
      =\langle\pi_{r(\eta)}^Rb,U_\eta y\rangle.
  \]
  Since \(\pi_{r(\eta)}^R\) is a finite-dimensional metric quotient, it maps \(B_{A_R}\) onto \(B_{A_{r(\eta)}}\).
  Indeed, the minimum norm in each fibre is attained in finite dimensions and equals the quotient norm.
  Hence we obtain
  \begin{equation}\label{eq:finite-certificate-annihilator}
    \sup_{b\in B_{A_R}}|f_\eta(T_{R,y}(b))|
    =\sup_{c\in B_{A_{r(\eta)}}}|\langle c,U_\eta y\rangle|
    =\norm{U_\eta y}.
  \end{equation}
  This proves the equivalence in \textup{(C3)} entirely on the completed finite stage.
  If \(v_1,\ldots,v_d\) are the stored columns spanning \(\mathscr M(y)\), its support and annihilation tests are precisely
  \[
    \begin{aligned}
      \supp_{\rm FDD}\mathscr M(y)
      =\bigcup_{j=1}^d\supp_{\rm FDD}v_j,\qquad
      f_\eta|_{\mathscr M(y)}=0
      \quad\Longleftrightarrow\quad
      f_\eta(v_j)=0\quad(1\leq j\leq d).
    \end{aligned}
  \]
  These are finite algebraic, support, and norm statements, and hence are well-defined oracle queries.
  The later operator argument selects certificates from this coded family.
\end{definition}

The finite requirement language consists of the following four types.
Every requirement also contains a lower rank bound \(N\) and a copy index \(o\).
\begin{enumerate}[label={\textup{(R\arabic*)}},leftmargin=*]
  \item \emph{Generic requirement.} Choose \(h\in\mathcal G\), an initial cutoff \(\chi\), and either an empty generic chain or an old generic core chain of level \(h\), initial cutoff \(\chi\), and age \(a<n_h\).
  In the empty case set \(a=0\), \(\xi=\varnothing\), and \(k=\chi\), and otherwise let \(\xi\) be the last atom and set \(k=\rank\xi\).
  Choose an old oracle-coded represented packet \(\mathbf b\) supported after \(k\), and a detector index at least those of \(\xi\) and all packet atoms.
  Set \(\omega=(\mathsf{generic},h,\chi,o)\) when \(a=0\), and \(\omega=\omega(\xi)\) when \(a>0\).
  The output is the age-one root or the age-\((a+1)\) core continuation.
  Its entries other than the birth rank are
  \[ (\mathsf{generic},h,a+1,\chi,\omega,\xi,k,\mathbf b,r, \varnothing,\varnothing,\varnothing,o). \]
  \item \emph{Seed requirement.} Choose \(q\in\mathcal C^{\rm out}\).
  At a rank \(n>N\), register the new seed \(\vartheta=(q,n,o)\).
  The seed has no root-label entry.
  This adds no nonprobe atom, and the canonical probes at rank \(n\) are still added.
  \item \emph{Inner requirement.} Choose a seed \(\vartheta\), a realized stem \(\mathfrak p\) over it, put \(p=\sigma(q,\mathfrak p)\), and give the attached inner chain initial cutoff \(\chi_{\rm in}=\operatorname{cut}(\mathfrak p)\).
  Choose either the empty inner chain attached to \(\mathfrak p\), in which case \(a=0\), \(\xi=\varnothing\), and \(k=\chi_{\rm in}\), or its old last atom \(\xi\) of age \(a<n_p\), in which case \(k=\rank\xi\).
  Choose an old oracle-coded packet \(\mathbf b\) supported after \(k\) and an output detector index \(r\) above the detector indices already named.
  Set \(\omega=(\mathsf{inner},p,\chi_{\rm in},\vartheta,\mathfrak p)\).
  The output has entries
  \[ (\mathsf{inner},p,a+1,\chi_{\rm in},\omega,\xi,k,\mathbf b,r,\vartheta,\mathfrak p,\varnothing,o). \]
  \item \emph{Outer requirement.} Choose a seed \(\vartheta=(q,\chi,o_0)\), a realized stem \(\mathfrak p\) of length \(a<n_q\), the completed inner chain of level \(p=\sigma(q,\mathfrak p)\), its terminal \(\eta\), and a compatible certificate \(\mathfrak c\).
  Put \(k=\operatorname{cut}(\mathfrak p)\), use the strong packet \(\mathbf s(\eta)\), and choose an output detector index \(r\geq r(\eta)\).
  The output is the age-\(a+1\) outer core atom, whose predecessor is the last atom of \(\mathfrak p\) when \(a>0\), and is empty when \(a=0\).
  Set \(\omega=(\mathsf{outer},q,\chi,\vartheta)\).
  Its entries are
  \[ (\mathsf{outer},q,a+1,\chi,\omega,\xi,k,\mathbf s(\eta),r,\vartheta,\mathfrak p,\mathfrak c,o). \]
  The new realized stem extends \(\mathfrak p\) by the output identifier, its core record, and \(\mathfrak c\).
\end{enumerate}
The requirements use only coded vectors.
The operator argument obtains coded approximations through Lemma~\ref{lem:polar}.

The seed requirement \textup{(R2)} is admissible at rank \(n\) precisely when \(n>N\).
It has no BD weight and requires no inequality between \(n\) and its seed level \(q\).
A nonseed requirement is \emph{admissible at rank \(n\)} when all named atoms, packets, seeds, stems, terminals, and certificates are old, every named packet is oracle-coded, \(n>N\), every required detector index is at most \(n\), the relevant age bound holds, the local cutoff is smaller than \(n\), and the raw packet support is contained in \((k,n)\).
For \textup{(R1)}, \textup{(R3)}, and \textup{(R4)} we also require \(n\geq h\), where \(h\) is the output level, and choose \(r(\gamma)\leq n\) above every named detector index.
In \textup{(R1)}, \textup{(R3)}, and \textup{(R4)}, a continuation with \(a>0\) must satisfy \(\omega=\omega(\xi)\).
For \textup{(R4)}, the whole certificate interval and the inner terminal lie in \((k,n)\).
These are finite tests.
In particular, every atom occurring in a correction has rank strictly below the output rank.

\begin{remark}[Order of the seed, inner, and outer requirements]
  \label{rem:seed-inner-outer}
  Let \textup{(R2)} register \(\vartheta=(q,\chi,o)\), and put \(\mathfrak p_0=(\vartheta;)\).
  For \(1\leq s\leq n_q\), suppose that the realized outer stem \(\mathfrak p_{s-1}\) has length \(s-1\), and set
  \[ k_s=\operatorname{cut}(\mathfrak p_{s-1}),\qquad p_s=\sigma(q,\mathfrak p_{s-1}). \]
  The seed fixes the outer level \(q\).
  The inner level \(p_s\) is fixed by the whole current stem and remains unchanged while \textup{(R3)} builds its attached inner chain.
  Writing \(\xi_{s,1}\prec\cdots\prec\xi_{s,n_{p_s}}=\eta_s\) for that chain, an ensuing \textup{(R4)} realization \(\zeta_s\) satisfies
  \[ k_s<\rank\xi_{s,1}<\cdots<\rank\eta_s<\rank\zeta_s, \qquad \ell(\eta_s)=p_s,\quad \ell(\zeta_s)=q. \]
  The last inner atom \(\eta_s\) supplies the strong packet \(\mathbf s(\eta_s)\), while the outer predecessor is \(\zeta_{s-1}\) when \(s>1\), and is empty when \(s=1\).
  Given a certificate \(\mathfrak c_s\) satisfying \textup{(C1)}--\textup{(C3)}, \textup{(R4)} records
  \[ \mathfrak p_s=\mathfrak p_{s-1}^{\frown} ((\operatorname{id}\zeta_s,\mathfrak r_0(\zeta_s),\mathfrak c_s)), \qquad p_{s+1}=\sigma(q,\mathfrak p_s)\quad(s<n_q). \]
  In the operator argument, Lemma~\ref{lem:root} gives the vanishing condition in \textup{(C3)} by
  \[ y_s=y_s^0-d_{\eta_s}(U_{\eta_s}y_s^0),\qquad U_{\eta_s}y_s=0, \]
  and Lemma~\ref{lem:exactification}, under its stated hypotheses, supplies the remaining certificate data.
  Thus each such outer chain begins with one use of \textup{(R2)}, followed, for each \(s\), by \(n_{p_s}\) uses of \textup{(R3)} and one compatible use of \textup{(R4)}.
  Other requirements may be met between any two of these steps.
\end{remark}

Figure~\ref{fig:seed-inner-outer} records these dependencies.
The dashed arrow indicates the outer predecessor \(\zeta_{s-1}\) when \(s>1\).
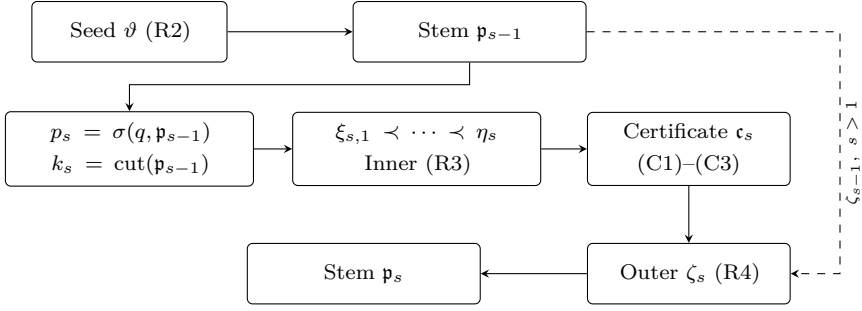
\begin{figure}[!htbp]
  \centering
\begin{tikzpicture}[x=1cm,y=1cm,>=stealth,
  every node/.style={font=\small,align=center},
  record/.style={draw,rounded corners=2pt,minimum height=.80cm,
  inner sep=4pt}]
  \node[record,text width=2.30cm] (seed) at (1.75,3.20)
  {Seed \(\vartheta\) \textup{(R2)}};
  \node[record,text width=2.80cm] (stem) at (6.25,3.20)
  {Stem \(\mathfrak p_{s-1}\)};
  \draw[->] (seed.east) -- (stem.west);

  \node[record,text width=3.00cm] (code) at (1.75,1.65)
  {\(p_s=\sigma(q,\mathfrak p_{s-1})\)\\[2pt]
  \(k_s=\operatorname{cut}(\mathfrak p_{s-1})\)};
  \node[record,text width=3.00cm] (inner) at (5.55,1.65)
  {\(\xi_{s,1}\prec\cdots\prec\eta_s\)\\[2pt]
  Inner \textup{(R3)}};
  \node[record,text width=2.40cm] (certificate) at (9.15,1.65)
  {Certificate \(\mathfrak c_s\)\\[2pt]
  \textup{(C1)--(C3)}};
  \draw[->] (stem.south) -- (6.25,2.50) -- (1.75,2.50)
  -- (code.north);
  \draw[->] (code.east) -- (inner.west);
  \draw[->] (inner.east) -- (certificate.west);

  \node[record,text width=2.40cm] (outer) at (9.15,0)
  {Outer \(\zeta_s\) \textup{(R4)}};
  \node[record,text width=2.80cm] (extended) at (4.85,0)
  {Stem \(\mathfrak p_s\)};
  \draw[->] (certificate.south) -- (outer.north);
  \draw[->] (outer.west) -- (extended.east);
  \draw[->,dashed] (stem.east) -- (11.15,3.20)
  -- (11.15,0) -- (outer.east);
  \node[rotate=90,font=\scriptsize] at (11.33,1.60)
  {\(\zeta_{s-1},\ s>1\)};
\end{tikzpicture}
  \caption{Seed, inner, and outer requirements.}
  \label{fig:seed-inner-outer}
\end{figure}

\begin{construction}[Fair oracle recursion]\label{con:fair-recursion}
  Fix a G\"odel enumeration \((\mathcal R_e)_{e\in\N}\) of the formal candidate requirements of types \textup{(R1)}--\textup{(R4)}.
  A candidate referring to an atom or stem not yet present is inadmissible at that stage.
  Every choice of finite data in \textup{(R1)}--\textup{(R4)}, together with a lower bound \(N\), occurs with infinitely many distinct copy indices, and no copy index is reused.
  An unmet requirement is called active at a rank when it is admissible there.
  Let \(\mathcal D_0\) code the empty state.
  Let \(\mathcal D_{n-1}\) be the code of the completed state below rank \(n\): the atom records, seeds, finite coordinate matrices, and met requirements.
  Write \(\mathcal H_{n-1}\) for its set of met requirement codes, with \(\mathcal H_0=\varnothing\).
  Define
  \begin{equation}\label{eq:least-code-scheduler}
    \begin{aligned}
      \mathcal R(n)
      &=\{e\leq n:e\notin \mathcal H_{n-1},\
      \mathcal R_e\text{ is admissible at rank }n\},\\
      e_n&=\begin{cases}
        \min\mathcal R(n),&\mathcal R(n)\ne\varnothing,\\
        \varnothing,&\mathcal R(n)=\varnothing.
      \end{cases}
    \end{aligned}
  \end{equation}
  The rank-\(n\) transition consists of the following steps.
  \begin{enumerate}[label={\textup{Step \arabic*.}},leftmargin=34pt]
    \item For every \(0\leq r\leq n\), add one unit probe and one core probe of detector index \(r\), with zero correction.
    \item If \(e_n\ne\varnothing\), act on \(\mathcal R_{e_n}\): register its seed, or add its prescribed core atom.
    In either case put
    \[
    \mathcal H_n=\begin{cases} \mathcal H_{n-1}\cup\{e_n\},&e_n\ne\varnothing,\\
    \mathcal H_{n-1},&e_n=\varnothing. \end{cases}
    \]
    \item Number the atoms of rank \(n\) in the fixed order, and record the corrections \eqref{eq:type-zero}--\eqref{eq:type-one}.
    All dependencies have rank below \(n\).
  \end{enumerate}
  Let \(\mathcal D_n\) be the code of the resulting completed state.
  Each rank has \(2(n+1)\) probes and at most one core atom, so we have
  \[ 2(n+1)\leq|\Delta_n|\leq2(n+1)+1, \qquad \mathcal D_{n-1}\preccurlyeq\mathcal D_n, \]
  where \(\preccurlyeq\) denotes end extension.
\end{construction}

The transition is computed by one oracle program \(\mathfrak M\), independent of the detector system.
For \(1\leq j\leq t_n<\infty\), its stage-\(n\) computation has the form
\[
    q_{n,j}=Q(n,\mathcal D_{n-1},a_{n,1},\ldots,a_{n,j-1}),\quad
    a_{n,j}=\mathbf1_{\mathcal O_{\mathscr A}}(q_{n,j}),\quad
    \mathcal D_n=\Phi(n,\mathcal D_{n-1},a_{n,1},\ldots,a_{n,t_n}),
\]
where \(Q\) and \(\Phi\) are computable procedures on codes.
There are only \(n+1\) candidate codes \(e\leq n\).
For each candidate, the stored dependencies can be checked in the finite state.
The scalar-table queries in Definition~\ref{def:detection-system} implement exact arithmetic, and the searches terminate because \(\mathbb F_{\mathscr A}\) is closed under the required field operations.
Computing any required \(\sigma\)-value also terminates by the construction above.
Thus the admissibility tests, the selected action, and the finite triangular inversion all finish after finitely many oracle queries.
In particular, writing \(U_n=\{q_{n,1},\ldots,q_{n,t_n}\}\), the stage output satisfies
\[
  \begin{gathered}
    \mathcal O'\cap U_n=\mathcal O_{\mathscr A}\cap U_n\quad \Longrightarrow\quad
    \mathfrak M^{\mathcal O'}(n,\mathcal D_{n-1})
    =\mathfrak M^{\mathcal O_{\mathscr A}}(n,\mathcal D_{n-1})
    =\mathcal D_n.
  \end{gathered}
\]
Indeed, run the same transition with \(\mathcal O'\), and use primes for its query and answer codes.
Induction on the query number gives
\[ (a'_{n,i})_{i<j}=(a_{n,i})_{i<j} \ \Longrightarrow\ q'_{n,j}=q_{n,j}\in U_n \ \Longrightarrow\ a'_{n,j}=a_{n,j}. \]
The program therefore halts at the same step and returns the same output after the last answer.
This shows the relative recursiveness asserted in Remark~\ref{rem:oracle-firewall}.

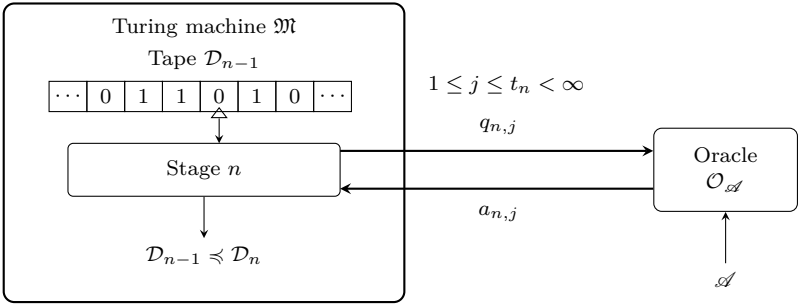
\begin{figure}[!htbp]
  \centering
\begin{tikzpicture}[x=1cm,y=1cm,>=stealth,
  every node/.style={font=\small,align=center}]
  \draw[thick,rounded corners=4pt] (0,-1.85) rectangle (5.30,2.10);
  \node at (2.65,1.78) {Turing machine \(\mathfrak M\)};
  \node at (2.65,1.35) {Tape \(\mathcal D_{n-1}\)};
  \foreach \x in {0.60,1.10,1.60,2.10,2.60,3.10,3.60,4.10}
  {\draw (\x,.67) rectangle +(.50,.40);}
  \node at (.85,.87) {\(\cdots\)};
  \node at (1.35,.87) {\(0\)};
  \node at (1.85,.87) {\(1\)};
  \node at (2.35,.87) {\(1\)};
  \node at (2.85,.87) {\(0\)};
  \node at (3.35,.87) {\(1\)};
  \node at (3.85,.87) {\(0\)};
  \node at (4.35,.87) {\(\cdots\)};
  \draw (2.75,.59) -- (2.95,.59) -- (2.85,.72) -- cycle;
  \node[draw,rounded corners=2pt,minimum width=3.60cm,
  minimum height=.70cm] (transition) at (2.65,-.10)
  {Stage \(n\)};
  \draw[->] (2.85,.58) -- (2.85,.25);
  \node (output) at (2.65,-1.25)
  {\(\mathcal D_{n-1}\preccurlyeq\mathcal D_n\)};
  \draw[->] (transition.south) -- (output.north);
  \node[draw,rounded corners=3pt,
  minimum width=1.90cm,minimum height=1.10cm] (oracle) at (9.55,-.10)
  {Oracle\\\(\mathcal O_{\mathscr A}\)};
  \node at (6.65,1.02) {\(1\leq j\leq t_n<\infty\)};
  \draw[->,thick] (4.45,.15) -- (oracle.west |- 4.45,.15);
  \node[above=3pt] at (6.55,.15) {\(q_{n,j}\)};
  \draw[->,thick] (oracle.west |- 4.45,-.35) -- (4.45,-.35);
  \node[below=3pt] at (6.55,-.35) {\(a_{n,j}\)};
  \node (data) at (9.55,-1.55) {\(\mathscr A\)};
  \draw[->] (data.north) -- (oracle.south);
\end{tikzpicture}
  \caption{One transition of the oracle Turing machine.}
  \label{fig:oracle-machine-interaction}
\end{figure}

Here Figure~\ref{fig:oracle-machine-interaction} shows the finite queries and the records added at one rank.

For the finite-extension and priority-method background, see \cite[Section~V.2]{Odifreddi1989}, \cite[Chapters~6--7]{Soare2016}, and \cite{Cooper2004}.

\begin{lemma}[Persistence and no injury]\label{lem:formal-legality}
  Every action in Construction~\ref{con:fair-recursion} has only lower-rank dependencies.
  Admissibility of a requirement is persistent under end extension, every active requirement is eventually met, and the construction has the no-injury property.
\end{lemma}

\begin{proof}
  Admissibility gives
  \[ \rank\xi=k<n,\qquad \supp_{\rm raw}\mathbf b\subseteq(k,n),\qquad \operatorname{cut}(\mathfrak p)<\rank\eta<n \]
  whenever the corresponding predecessor, packet, or inner terminal is present.
  Probe corrections vanish.
  Thus every correction has only lower-rank dependencies.

  All finite records and all answers to their oracle queries are fixed.
  The remaining rank conditions are lower bounds, so we have
  \[ \mathcal R_e\text{ admissible at }n \quad\Longrightarrow\quad \mathcal R_e\text{ admissible at every }m\geq n. \]
  Suppose \(\mathcal R_e\) is active at \(n_0\geq e\).
  As long as it remains unmet, the scheduler selects an unmet code \(e_n\leq e\).
  If \(e_n<e\), then we have
  \[ |\{d<e:d\notin \mathcal H_n\}| =|\{d<e:d\notin \mathcal H_{n-1}\}|-1. \]
  There are at most \(e\) such codes.
  Hence we obtain
  \[ \min\{n\geq n_0:e\in \mathcal H_n\}\leq n_0+e, \]
  proving fairness.
  Finally, each step assigns fresh identifiers and appends records, so end extension preserves all earlier corrections, packet words, histories, stems, and certificates.
\end{proof}

\begin{proposition}[Finite-extension property]\label{prop:finite-extension}
  Construction~\ref{con:fair-recursion} has the following properties.
  \begin{enumerate}[label={\textup{(\alph*)}},leftmargin=*]
    \item Every admissible oracle-coded generic extension has a realization above every prescribed lower bound.
    Successive admissible choices of packets are realized along one generic predecessor chain, up to its age bound.
    \item Seeds of every outer level occur after every cutoff.
    Given a realized outer stem \(\mathfrak p\), packets may be chosen successively, with later choices allowed to depend on the atoms and ranks already realized.
    Every such finite oracle-coded list is realized along one inner predecessor chain of level \(\sigma(q,\mathfrak p)\), up to age \(n_{\sigma(q,\mathfrak p)}\).
    \item After an inner chain is complete, every compatible coded certificate has an outer continuation after every cutoff.
    Iterating this assertion realizes every finite compatible outer stem, up to age \(n_q\).
    \item The canonical probes give cofinally many unit and core probes of every fixed detector index.
  \end{enumerate}
\end{proposition}

\begin{proof}
  Fix an admissible extension with old dependencies and a desired cutoff \(N\).
  Choose its occurrence \(\mathcal R_e\) with lower bound greater than \(N\).
  If it has already been met, its output has the required rank.
  Otherwise, all rank conditions are lower bounds, so there is \(n_0\geq e\) such that
  \[ \mathcal R_e\text{ is admissible at every }n\geq n_0. \]
  By Lemma~\ref{lem:formal-legality}, this occurrence is met at a rank \(n>N\).
  This proves cofinal realization of each extension.

  For a generic chain, once \(\xi_t\) is realized, select the next admissible packet with
  \[ \supp_{\rm raw}\mathbf b_{t+1}\subseteq(\rank\xi_t,\infty), \qquad N_{t+1}>\max\bigl(\{\rank\xi_t\}\cup \supp_{\rm raw}\mathbf b_{t+1}\bigr). \]
  After its entries are old, the preceding argument produces \(\xi_{t+1}\) with predecessor exactly \(\xi_t\) and \(\rank\xi_{t+1}>N_{t+1}\).
  Induction proves \textup{(a)}.
  A seed is produced by the same argument.
  For a fixed realized stem \(\mathfrak p\), every inner continuation names that same stem, so its level remains \(\sigma(q,\mathfrak p)\).
  After each realization, the next packet may be chosen using the newly known identifier and rank.
  The enumeration contains the resulting finite requirement, and the preceding argument applies again.
  This proves \textup{(b)}.

  For \textup{(c)}, let \(\eta\) be the completed inner terminal and \(I\) the interval of a compatible old certificate.
  An occurrence of \textup{(R4)} with
  \[ N'>\max\{N,\rank\eta,\max I\} \]
  produces the required outer continuation.
  Iteration extends the stem up to age \(n_q\).
  Finally, every rank \(n\geq\max\{1,r\}\) contains both canonical probes of detector index \(r\), proving \textup{(d)}.
\end{proof}

\subsection{Bourgain--Delbaen realization of the finite records}
\label{sec:bd-realization}

For \(\varphi=(\varphi_\gamma)_{\gamma\in\Delta_n}\in
F_n^*=(\bigoplus_{\gamma\in\Delta_n}E_\gamma^*)_{\ell_1}\), let
\(\iota_n\varphi=\sum_{\gamma\in\Delta_n}\epsilon_\gamma^*(\varphi_\gamma)\) be the ambient injection and put
\[
  c_n^*\varphi=\sum_{\gamma\in\Delta_n}c_\gamma^*(\varphi_\gamma),
  \qquad
  d_n^*\varphi=\sum_{\gamma\in\Delta_n}d_{\gamma,\varphi_\gamma}^*
  =\iota_n\varphi-c_n^*\varphi.
\]
Thus \(c_n^*:F_n^*\to\bigoplus_{k<n}F_k^*\) and \(d_n^*:F_n^*\to\mathcal E_*\)
are the rank-\(n\) maps used in Lemma~\ref{lem:block-bd}.

We apply Lemma~\ref{lem:block-bd} with \(\theta=m_1^{-1}\).
The constants \(M\) and \(\kappa\) are therefore those in \eqref{eq:constants}.

\begin{proposition}[Bourgain--Delbaen correction forms]
  \label{prop:bd-legality}
  Every correction recorded by Construction~\ref{con:fair-recursion} satisfies the hypotheses of Lemma~\ref{lem:block-bd}.
\end{proposition}

\begin{proof}
  Lemma~\ref{lem:formal-legality} gives strict lower-rank dependence.
  For a represented packet \(\mathbf b\), Lemma~\ref{lem:packet} gives
  \[ \norm{B_{\mathbf b}^*\varphi} \leq\norm\varphi \qquad(\varphi\in A_{r(\gamma)}). \]
  In a type-one correction, the predecessor map
  \(\pi_{r(\xi)}^{r(\gamma)}:A_{r(\gamma)}\to A_{r(\xi)}\)
  is contractive by the metric-quotient hypothesis \eqref{eq:metric-quotient}, and its injection into the \(\xi\)-summand of \(F_k^*\) is isometric.
  Thus the map \(A^*\) in \eqref{eq:block-type-one} has norm at most 1.

  Since the raw support of the packet lies in \((k,n)\), it belongs to the first \(n-1\) ambient blocks.
  Unit triangularity therefore gives \(P_{n-1}^*B_{\mathbf b}^*=B_{\mathbf b}^*\), and hence we obtain
  \[
    P_{(k,n)}^*B_{\mathbf b}^*
    =(P_{n-1}^*-P_k^*)B_{\mathbf b}^*
    =(I-P_k^*)B_{\mathbf b}^*.
  \]
  Thus type-zero and type-one records have respectively the forms \eqref{eq:block-type-zero} and \eqref{eq:block-type-one}, with \(\beta=m_h^{-1}\leq m_1^{-1}<\frac{1}{2}\).
  A zero probe correction is the type-zero form with \(\beta=0\).
\end{proof}

For \(N\geq1\), put
\[ \mathcal E_{[1,N]} =\left(\bigoplus_{n=1}^NF_n\right)_{\ell_\infty}, \qquad \mathcal E_{[1,N],*} =\left(\bigoplus_{n=1}^NF_n^*\right)_{\ell_1}. \]
Unit triangularity gives \(\mathcal E_{[1,N],*}=\bigoplus_{n=1}^N d_n^*(F_n^*)\) algebraically.
For \(s\leq N\), let
\[ Q_{s,N}^*:\mathcal E_{[1,N],*}\longrightarrow \mathcal E_{[1,s],*}, \qquad Q_{s,N}^*f=P_s^*f. \]
Eliminating the top ambient block first shows that, if \(L\geq N\geq s\), we have
\begin{equation}\label{eq:finite-dual-consistency}
  Q_{s,L}^*|_{\mathcal E_{[1,N],*}}=Q_{s,N}^*.
\end{equation}
Lemma~\ref{lem:block-bd} and Proposition~\ref{prop:bd-legality} give \(\norm{Q_{s,N}^*}\leq M\).
Define
\begin{equation}\label{eq:finite-extension}
  i_{s,N}=(Q_{s,N}^*)^*:
  \mathcal E_{[1,s]}\longrightarrow\mathcal E_{[1,N]}.
\end{equation}
By finite-dimensional duality, we have \(\norm{i_{s,N}}=\norm{Q_{s,N}^*}\leq M\).
To verify compatibility of the maps \(i_{s,N}\), take \(f\in\mathcal E_{[1,N],*}\) and \(u\in\mathcal E_{[1,s]}\).
Then we have
\[ \langle f,(i_{s,L}u)|_{[1,N]}\rangle =\langle Q_{s,L}^*f,u\rangle =\langle Q_{s,N}^*f,u\rangle =\langle f,i_{s,N}u\rangle. \]
Thus \eqref{eq:finite-dual-consistency} implies
\[ (i_{s,L}u)|_{[1,N]}=i_{s,N}u \qquad(L\geq N\geq s). \]
The compatible vectors \(i_{s,N}u\) define \(i_su\) in the ambient \(\ell_\infty\)-sum, and we have
\[ \norm{i_su} =\sup_{N\geq s}\norm{i_{s,N}u} \leq M\norm u. \]
Since \(Q_{s,N}^*\) fixes the first \(s\) ambient blocks, \(i_s\) preserves the prescribed initial coordinates.
For \(\rank\gamma>s\), we have \(P_s^*d_{\gamma,\varphi}^*=0\), and finite-stage adjointness gives
\[
  \begin{aligned}
    \langle\varphi,(i_su)_\gamma\rangle
    &=\langle P_s^*\epsilon_\gamma^*(\varphi),u\rangle=\langle P_s^*c_\gamma^*(\varphi),u\rangle
     =c_\gamma^*(\varphi)(i_su).
  \end{aligned}
\]
The last functional uses only ranks below \(\rank\gamma\).
Thus the extensions satisfy
\begin{align}
  (i_su)|_{[1,s]}&=u,                                             \label{eq:extension-head}\\
  \langle\varphi,(i_su)_\gamma\rangle
  &=c_\gamma^*(\varphi)
  ((i_su)|_{[1,\rank\gamma)})
  \quad(\rank\gamma>s).                                    \label{eq:extension-recurrence}
\end{align}
Conversely, these equations determine all later coordinates by induction on the rank, so the adjoint and recursive definitions agree.
For \(s\leq t\), the vectors \(i_t((i_su)|_{[1,t]})\) and \(i_su\) have the same first \(t\) coordinates and satisfy the same recurrence thereafter.
By uniqueness, we obtain
\begin{equation}\label{eq:nested-extensions}
  i_t((i_su)|_{[1,t]})=i_su\qquad(s\leq t).
\end{equation}

Put
\[ X^{\rm alg}=\bigcup_s i_s\mathcal E_{[1,s]}, \qquad X=\overline{X^{\rm alg}}, \qquad Z_n=i_nF_n, \]
where in the last formula the vector in \(F_n\) is placed in the last block and all earlier blocks are zero.
If \(\gamma\in\Delta_n\), define \(d_\gamma u=i_n(0,\ldots,0,u,0,\ldots)\) for \(u\in E_\gamma\).
For atoms \(\gamma,\eta\), finite-stage adjointness gives the atomic biorthogonality relation
\begin{equation}\label{eq:atomic-biorthogonality}
  d_{\eta,\psi}^*(d_\gamma u)
  =\begin{cases}
    \langle\psi,u\rangle,&\eta=\gamma,\\
    0,&\eta\ne\gamma.
  \end{cases}
\end{equation}
Indeed, project \(d_{\eta,\psi}^*\) to the birth rank of \(\gamma\).
If \(\rank\eta\ne\rank\gamma\), its \(d^*\)-block is either retained away from the one nonzero input coordinate or deleted, whereas at the same rank the correction of \(\eta\) is old, leaving only the ambient coordinate.

Equation~\eqref{eq:atomic-biorthogonality} defines the coefficient maps
\begin{equation}\label{eq:coefficient-map}
  D_\gamma:X^{\rm alg}\to E_\gamma,
  \qquad
  \langle\psi,D_\gamma x\rangle=d_{\gamma,\psi}^*(x).
\end{equation}
If \(x=i_su\), then \(P_s^*d_{\gamma,\psi}^*=0\) for \(\rank\gamma>s\), so finite-stage adjointness gives \(D_\gamma x=0\) at those ranks.
The invertible triangular change of coordinates gives the same first \(s\) ambient coordinates for \(x\) and \(\sum_{\rank\gamma\leq s}d_\gamma D_\gamma x\).
Both vectors belong to \(i_s\mathcal E_{[1,s]}\), so uniqueness of the extension gives
\[
  x=\sum_{\rank\gamma\leq s}d_\gamma D_\gamma x.
\]
Consequently, \(X^{\rm alg}\) is exactly the algebraic span of the blocks \(Z_n\), and the initial FDD projection is
\[
  P_{[1,t]}x=i_t(x|_{[1,t]}).
\]
Indeed, pairing either side with \(f\in\mathcal E_*\) gives \(\langle P_t^*f,x\rangle\).
For \(x\in X^{\rm alg}\), let \(\supp_{\rm FDD}x\) be the set of ranks at which one of these coefficients is nonzero, and let \(\ran x\) be its integer interval hull.
For nonzero finitely supported vectors \(x,y\) and an atom \(\gamma\), write \(x<y\), \(x<\gamma\), or \(\gamma<x\) when, respectively,
\[ \max\supp_{\rm FDD}x<\min\supp_{\rm FDD}y,\qquad \max\supp_{\rm FDD}x<\rank\gamma,\qquad \rank\gamma<\min\supp_{\rm FDD}x. \]
This is the FDD/rank order used below.
The ambient coordinate vector \(U_\gamma x\in E_\gamma\) is defined by
\begin{equation}\label{eq:ambient-coordinate}
  \langle\psi,U_\gamma x\rangle=e_{\gamma,\psi}^*(x),
  \qquad
  e_{\gamma,\psi}^*=\epsilon_\gamma^*(\psi)|_X.
\end{equation}
When ambient functionals act on \(X\), we use the same notation for their restrictions to \(X\), and in particular \(P_I^*f=f\circ P_I\) on \(X\).
For \(\psi\in E_\gamma^*\), we have \(\norm{e_{\gamma,\psi}^*}\leq\norm\psi\), and the evaluations with \(\psi\in B_{E_\gamma^*}\) form a norming family for \(X\).
The block maps and projections satisfy
\begin{equation}\label{eq:formal-block-bounds}
  \norm{i_s}\leq M,\qquad
  \norm{P_{[1,s]}}\leq M,\qquad
  \norm{P_I}\leq\kappa=2M,\qquad
  \norm{d_\gamma}\leq M,\qquad
  \norm{D_\gamma}\leq\kappa.
\end{equation}
Indeed, since \(X\) carries the norm inherited from the ambient \(\ell_\infty\)-sum, we have
\[ \norm x =\sup_{\eta}\norm{U_\eta x} =\sup_{\eta}\sup_{\psi\in B_{E_\eta^*}} \abs{e_{\eta,\psi}^*(x)} \qquad(x\in X). \]
If \(\gamma\in\Delta_n\) and \(v\in\mathcal E_{[1,n]}\) has only the \(\gamma\)-coordinate \(u\), then we have \(\norm v=\norm u\) and
\[ \norm{d_\gamma u} =\norm{i_nv} \leq M\norm v =M\norm u. \]
Moreover, \(D_\gamma x\) is the \(\gamma\)-ambient coordinate of the singleton FDD projection \(P_{\{n\}}x\).
Contractivity of the ambient coordinate map and the interval estimate give
\[ \norm{D_\gamma x} =\norm{U_\gamma P_{\{n\}}x} \leq\norm{P_{\{n\}}x} \leq\kappa\norm x. \]
This proves the last two bounds in \eqref{eq:formal-block-bounds}, whereas the first three are the extension, initial-projection, and interval-projection estimates above.
In particular, every \(D_\gamma\) extends uniquely to \(X\), and the same symbol denotes this extension.

The maps \(D_\gamma\) and \(U_\gamma\) give different coordinates of the same vector, as illustrated in Figure~\ref{fig:bd-ambient-coordinates}.
For \(x=i_su\), we have \(D_\gamma x=0\) whenever \(\rank\gamma>s\), while the later ambient coordinates are determined by the correction equations \eqref{eq:extension-recurrence}.
Thus a vector with finite FDD support may have nonzero ambient coordinates at later ranks.
For a probe atom, the correction is zero, so we have \(D_\gamma x=U_\gamma x\).

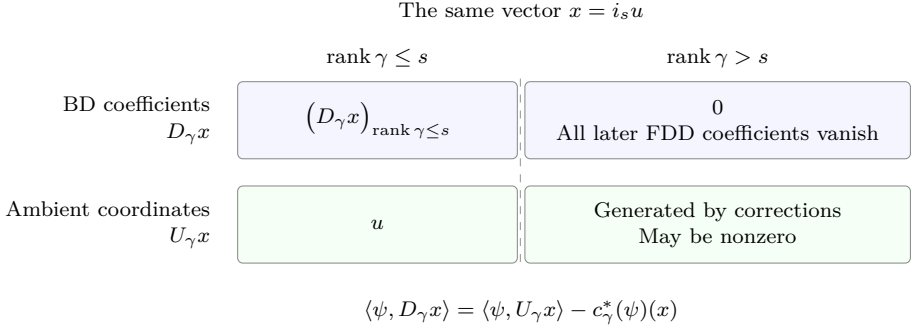
\begin{figure}[!htbp]
  \centering
\begin{tikzpicture}[x=1cm,y=1cm,
  every node/.style={font=\small,inner sep=3pt},
  cell/.style={draw=black!45,rounded corners=2pt,minimum height=1.02cm},
  head/.style={cell,minimum width=3.7cm},
  tail/.style={cell,minimum width=5.1cm}]
  \node at (6.05,1.1) {The same vector \(x=i_su\)};
  \node at (4.15,0.48) {\(\rank\gamma\leq s\)};
  \node at (8.65,0.48) {\(\rank\gamma>s\)};

  \node[anchor=east,align=right] at (2.05,-0.35)
    {BD coefficients\\\(D_\gamma x\)};
  \node[head,fill=blue!4] at (4.15,-0.35)
    {\(\bigl(D_\gamma x\bigr)_{\rank\gamma\leq s}\)};
  \node[tail,fill=blue!4,align=center] at (8.65,-0.35)
    {\(0\)\\All later FDD coefficients vanish};

  \node[anchor=east,align=right] at (2.05,-1.73)
    {Ambient coordinates\\\(U_\gamma x\)};
  \node[head,fill=green!4] at (4.15,-1.73) {\(u\)};
  \node[tail,fill=green!4,align=center] at (8.65,-1.73)
    {Generated by corrections\\May be nonzero};

  \draw[dashed,black!40] (6.04,0.2) -- (6.04,-2.33);
  \node at (6.05,-2.89)
    {\(\langle\psi,D_\gamma x\rangle
       =\langle\psi,U_\gamma x\rangle-c_\gamma^*(\psi)(x)\)};
\end{tikzpicture}
  \caption{BD coefficients and ambient coordinates of \(x=i_su\). Each column groups all atoms in the indicated range of ranks.}
  \label{fig:bd-ambient-coordinates}
\end{figure}

\section{\texorpdfstring{Evaluation analyses and multipliers}{Evaluation analyses and multipliers}}

Recall that \(D_\gamma x\) is the BD block coefficient, whereas \(U_\gamma x\) is the ambient coordinate and \(e_{\gamma,\varphi}^*(x)=\langle\varphi,U_\gamma x\rangle\).
We write \(P_I\) for the FDD projection onto the ranks in \(I\), and \(B_{\mathbf b}^*\) for the covariant packet lift in \eqref{eq:packet-map}--\eqref{eq:packet-lift}.
The constants \(M\) and \(\kappa=2M\) are the extension and interval-projection bounds in \eqref{eq:formal-block-bounds}, and we continue to denote the  level of an atom by \(\ell(\gamma)\).
We derive the evaluation analysis and prove the probe isometries, scalarization identity, and finite shielding estimates used in Part~II.

\subsection{Evaluation analysis}

The following expansion is the block version of the scalar evaluation analysis in \cite[Proposition~4.5]{ArgyrosHaydon2011}.

Let \(\gamma=\xi_a\) be a nonprobe atom.
Its stored history gives the predecessor chain
\[ \xi_1\prec\cdots\prec\xi_a=\gamma, \qquad p_t=\rank\xi_t, \qquad p_0=\chi(\gamma). \]
For \(\varphi\in B_{E_\gamma^*}=B_{A_{r(\gamma)}}\), put
\begin{equation}\label{eq:chain-scalarizations}
  R_t=\pi_{r(\xi_t)}^{r(\gamma)},
  \qquad
  \varphi_t=R_t\varphi,
  \qquad
  b_{t,\varphi}^*=B_{\mathbf b_{\xi_t}}^*\varphi_t.
\end{equation}
Both maps in \eqref{eq:chain-scalarizations} are contractions.

\begin{lemma}[Evaluation analysis]\label{lem:formal-evaluation-analysis}
  Every scalarized evaluation of  level \(h>0\) has the analysis
  \begin{equation}\label{eq:evaluation-grammar}
    e_{\gamma,\varphi}^*
    =\sum_{t=1}^ad_{\xi_t,\varphi_t}^*
    +\frac{1}{m_h}\sum_{t=1}^a
    b_{t,\varphi}^*P_{(p_{t-1},\infty)},
    \qquad a\leq n_h.
  \end{equation}
  The cells \((p_{t-1},p_t)\) are successive, and each \(b_{t,\varphi}^*\) lies in the absolute convex hull of lower-rank scalarized ambient coordinates occurring in that cell.
  A probe evaluation has  level zero and equals its own \(d^*\)-coordinate.
\end{lemma}

\begin{proof}
  The packet \(b_{t,\varphi}^*\) is supported on ambient ranks below \(p_t\), so unit triangularity gives \(P_{p_t-1}^*b_{t,\varphi}^*=b_{t,\varphi}^*\).
  Hence the finite interval projection in the correction can be replaced by a tail projection:
  \[
    b_{t,\varphi}^*P_{(p_{t-1},p_t)}
    =b_{t,\varphi}^*P_{(p_{t-1},\infty)}.
  \]
  For \(1\leq t\leq a\), the triangular coordinate identity now gives
  \[
    \begin{aligned}
      e_{\xi_t,\varphi_t}^*
      =d_{\xi_t,\varphi_t}^*
        +\mathbf{1}_{\{t>1\}}e_{\xi_{t-1},\varphi_{t-1}}^*+\frac{1}{m_h}b_{t,\varphi}^*P_{(p_{t-1},\infty)}.
    \end{aligned}
  \]
  The predecessor term is omitted when \(t=1\).
  When \(t>1\), compatibility of the bonding maps gives
  \[
    \pi_{r(\xi_{t-1})}^{r(\xi_t)}\varphi_t
    =\pi_{r(\xi_{t-1})}^{r(\gamma)}\varphi
    =\varphi_{t-1}.
  \]
  Summing the one-edge identities cancels every intermediate evaluation and proves \eqref{eq:evaluation-grammar}.
  The record conditions give \(a\leq n_h\) and \(\supp_{\rm raw}\mathbf b_{\xi_t}\subseteq(p_{t-1},p_t)\), so the packet cells are successive.

  If \(\mathbf b_{\xi_t}=((c_l,\eta_l,\psi_l))_l\), then we have
  \[ b_{t,\varphi}^* =\sum_lc_l\epsilon_{\eta_l}^* \bigl(\Theta_{\eta_l} (\pi_{r(\eta_l)}^{r(\xi_t)}\varphi_t)^*\psi_l\bigr). \]
  Put
  \[
  v_l= \Theta_{\eta_l} (\pi_{r(\eta_l)}^{r(\xi_t)}\varphi_t)^*\psi_l, \qquad \widehat\psi_l= \begin{cases} v_l/\norm{v_l},&v_l\ne0,\\
  0,&v_l=0. \end{cases}
  \]
  The bonding maps and module actions are contractive.
  Since \(\norm{\varphi_t}\leq\norm\varphi\leq1\), we have
  \( \norm{v_l} \leq\norm{\psi_l} \norm{\pi_{r(\eta_l)}^{r(\xi_t)}\varphi_t} \leq\norm{\psi_l}. \)
  Consequently, we obtain
  \[
    b_{t,\varphi}^*=\sum_lc_l\norm{v_l}\epsilon_{\eta_l}^*(\widehat\psi_l).
  \]
  For the coefficients, we have
  \[
    \sum_l\abs{c_l}\norm{v_l}\leq\sum_l\abs{c_l}\norm{\psi_l}\leq1.
  \]
  This proves both contractivity and the absolute-convex-hull assertion.
  The maps \(\varphi\mapsto\varphi_t\) are contractive because they are bonding maps, and \(\varphi\mapsto b_{t,\varphi}^*\) is contractive by Lemma~\ref{lem:packet}.
  For a probe, we have \(c_\gamma^*=0\), so we obtain \(e_{\gamma,\psi}^*=d_{\gamma,\psi}^*\).
\end{proof}

\begin{lemma}[Isometric probes]\label{lem:formal-probe-isometry}
  If \(\gamma\) is a unit or core probe and \(j_\gamma=d_\gamma:E_\gamma\to X\), then we have
  \begin{equation}\label{eq:formal-probe-bounds}
    D_\gamma=U_\gamma,
    \qquad \norm{D_\gamma}\leq1,
    \qquad \norm{j_\gamma u}=\norm u
    \quad(u\in E_\gamma).
  \end{equation}
  In particular, we have \(\norm{j_\gamma}=1\).
\end{lemma}

\begin{proof}
  Since \(c_\gamma^*=0\), we have \(D_\gamma=U_\gamma\), \(\norm{D_\gamma}\leq1\), and \(D_\gamma j_\gamma=I\).
  Fix \(0\ne u\in E_\gamma\), and put \(x=j_\gamma u\) and \(q=\rank\gamma\).
  By \eqref{eq:extension-head} and atomic biorthogonality, we have
  \[
    \supp_{\rm FDD}x=\{q\},
    \qquad U_\eta x=0\quad(\rank\eta<q).
  \]
  For every probe \(\alpha\), we have
  \[
    e_{\alpha,\varphi}^*(x)=d_{\alpha,\varphi}^*(d_\gamma u)
    =\begin{cases}
      \langle\varphi,u\rangle,&\alpha=\gamma,\\
      0,&\alpha\ne\gamma.
    \end{cases}
  \]
  If \(\ell(\alpha)=h>0\) and \(\varphi\in B_{E_\alpha^*}\), all the \(d^*\)-terms in its evaluation analysis vanish on \(x\), because a probe is never a predecessor.
  Using the raw support of \(b_{t,\varphi}^*\) below \(p_t\), the vanishing of the ambient coordinates of \(x\) below \(q\), and \(\supp_{\rm FDD}x=\{q\}\), we obtain
  \[
    p_t\leq q\ \Longrightarrow\ b_{t,\varphi}^*(x)=0,
    \qquad
    p_{t-1}\geq q\ \Longrightarrow\ P_{(p_{t-1},\infty)}x=0.
  \]
  The cells are successive, so there is at most one index satisfying \(p_{t-1}<q<p_t\).
  On that cell the tail projection fixes \(x\).
  Hence we obtain
  \[ \abs{e_{\alpha,\varphi}^*(x)} \leq\frac{1}{m_h} \sum_{\{t:p_{t-1}<q<p_t\}}\abs{b_{t,\varphi}^*(x)} \leq\frac{1}{m_h}\norm x. \]
  Taking the supremum over all atoms \(\alpha\) and \(\varphi\in B_{E_\alpha^*}\) gives
  \[
    \norm u\leq\norm x
    \leq\max\{\norm u,m_1^{-1}\norm x\}.
  \]
  If \(\norm x>\norm u\), then the preceding inequality would give \(\norm x\leq m_1^{-1}\norm x\), contrary to \(m_1^{-1}<1\) and \(x\ne0\).
  Thus we obtain \(\norm x=\norm u\).
\end{proof}

For \(a\in\widehat A\), define on \(X^{\rm alg}\)
\begin{equation}\label{eq:multiplier-definition}
  M_ax=\sum_\gamma
  d_\gamma\bigl(\Theta_\gamma(a_{r(\gamma)})D_\gamma x\bigr).
\end{equation}
The sum is finite because \(x\) has finite FDD support, and it introduces no new nonzero BD coefficients.

\begin{lemma}[Scalarization identity]\label{lem:scalarization}
  For  \(\gamma\), \(\psi\in E_\gamma^*\), \(a,b\in\widehat A\), and \(x\in X^{\rm alg}\), we have
  \begin{align}
    U_\gamma M_ax
    &=\Theta_\gamma(a_{r(\gamma)})U_\gamma x,                      \label{eq:full-chain}\\
    e_{\gamma,\psi}^*(M_ax)
    &=\left\langle\psi,
    \Theta_\gamma(a_{r(\gamma)})U_\gamma x\right\rangle.       \label{eq:scalar-chain}
  \end{align}
  The multiplier extends to \(X\), and we have
  \begin{equation}\label{eq:exact-algebra}
    \norm{M_a}=\norm a,\qquad M_aM_b=M_{ab},\qquad M_1=I_X.
  \end{equation}
\end{lemma}

\begin{proof}
  We first show that the action on BD coefficients agrees with the action on ambient coordinates.
  This will give the upper norm bound, while unit probes will give the reverse bound.
  In this proof, abbreviate \(a_\eta=a_{r(\eta)}\) for each atom \(\eta\).
  Suppose first that \(\gamma\) is a nonprobe atom of level \(h\), local cutoff \(k\), packet \(\mathbf b=\mathbf b_\gamma=((c_l,\eta_l,\psi_l))_{l=1}^s\), and optional predecessor \(\xi\).
  Taking adjoints in the correction formula gives the recurrence for \(U_\gamma x\):
  \begin{equation}\label{eq:formal-coordinate-recurrence}
    U_\gamma x
    =D_\gamma x
    +\mathbf{1}_{\{\xi\ne\varnothing\}}
    JU_\xi x
    +\frac{1}{m_h}\sum_{l=1}^sc_lS_l^{\mathbf b}
    U_{\eta_l}P_{(k,\rank\gamma)}x .
  \end{equation}
  When \(\xi\ne\varnothing\), put \(J=(\pi_{r(\xi)}^{r(\gamma)})^*\), and when \(\xi=\varnothing\), omit the whole predecessor term here and in the induction below.
  The map \(S_l^{\mathbf b}=(R_l^{\mathbf b})^*\) is from Lemma~\ref{lem:packet}.
  For a probe, the recurrence is \(U_\gamma x=D_\gamma x\).

  We induct on \(\rank\gamma\).
  The definition of \(M_a\) gives, for every atom \(\eta\), \(D_\eta M_ax=\Theta_\eta(a_\eta)D_\eta x\).
  It also shows that \(M_a\) commutes with every FDD interval projection on \(X^{\rm alg}\).
  Indeed, we have
  \[
    \begin{aligned}
      D_\eta(P_IM_ax)
      =\mathbf1_{\{\rank\eta\in I\}}\Theta_\eta(a_\eta)D_\eta x=\Theta_\eta(a_\eta)D_\eta(P_Ix)
       =D_\eta(M_aP_Ix).
    \end{aligned}
  \]
  And hence we have \(P_IM_a=M_aP_I\) on \(X^{\rm alg}\).

  If \(\gamma\) is a probe, then we have \(U_\gamma=D_\gamma\), and we obtain
  \[ U_\gamma M_ax=D_\gamma M_ax =\Theta_\gamma(a_\gamma)D_\gamma x =\Theta_\gamma(a_\gamma)U_\gamma x. \]
  This proves the probe case.
  For a nonprobe predecessor, compatibility and multiplicativity of \(\pi_{r(\xi)}^{r(\gamma)}\) give, for \(u\in A_{r(\xi)}^*\) and \(c\in A_{r(\gamma)}\),
  \[
    \begin{aligned}
      \langle c,JR_{a_\xi}^*u\rangle
      =\langle\pi_{r(\xi)}^{r(\gamma)}c\,a_\xi,u\rangle=\langle\pi_{r(\xi)}^{r(\gamma)}(ca_\gamma),u\rangle
       =\langle c,R_{a_\gamma}^*Ju\rangle.
    \end{aligned}
  \]
  Therefore we obtain
  \[ JR_{a_\xi}^* =R_{a_\gamma}^*J. \]
  Every packet term intertwines by \eqref{eq:packet-covariance}.
  Since the predecessor and all packet atoms have smaller rank, the inductive hypothesis, the commutation with interval projections, and \eqref{eq:formal-coordinate-recurrence} give
  \begin{align*}
    U_\gamma M_ax
    &=R_{a_\gamma}^*D_\gamma x
    +\mathbf1_{\{\xi\ne\varnothing\}}JR_{a_\xi}^*U_\xi x
    +\frac{1}{m_h}\sum_{l=1}^sc_lS_l^{\mathbf b}
    \Theta_{\eta_l}(a_{\eta_l})U_{\eta_l}P_{(k,\rank\gamma)}x\\
    &=R_{a_\gamma}^*\left(D_\gamma x
    +\mathbf1_{\{\xi\ne\varnothing\}}JU_\xi x
    +\frac{1}{m_h}\sum_{l=1}^sc_lS_l^{\mathbf b}
    U_{\eta_l}P_{(k,\rank\gamma)}x\right)
    =\Theta_\gamma(a_\gamma)U_\gamma x.
  \end{align*}
  This proves \eqref{eq:full-chain}.
  Pairing with \(\psi\) proves \eqref{eq:scalar-chain}.

  On the full ambient \(\ell_\infty\)-sum, let \(\widetilde\Theta_a\) act on the \(\gamma\)-coordinate as \(\Theta_\gamma(a_\gamma)\).
  Contractivity of both module actions gives
  \[ \norm{\widetilde\Theta_ax} =\sup_\gamma\norm{\Theta_\gamma(a_\gamma)U_\gamma x} \leq\sup_\gamma\norm{a_\gamma}\norm{U_\gamma x} \leq\norm a\norm x. \]
  Thus we have \(\norm{\widetilde\Theta_a}\leq\norm a\), uniformly over both fibre types.
  Equations \eqref{eq:extension-recurrence} and \eqref{eq:full-chain} show that each range \(i_s\mathcal E_{[1,s]}\) is invariant under \(\widetilde\Theta_a\), whose restriction to that range equals \(M_a\).
  Hence we have \(\norm{M_ax}=\norm{\widetilde\Theta_ax}\leq\norm a\norm x\) for \(x\in X^{\rm alg}\).
  Therefore \(M_a\) extends to \(X\) and \(\norm{M_a}\leq\norm a\).
  For each detector index \(r\), choose a unit probe \(\gamma\) of index \(r\) and put \(z_\gamma=j_\gamma1_{A_r}\).
  By Lemma~\ref{lem:formal-probe-isometry}, we have \(\norm{z_\gamma}=1\), while \eqref{eq:multiplier-definition} gives \(M_az_\gamma=j_\gamma a_r\) and \(\norm{M_az_\gamma}=\norm{a_r}\).
  Since \(\norm{z_\gamma}=1\), we obtain \(\norm{M_a}\geq\norm{M_az_\gamma}=\norm{a_r}\).
  Taking the supremum over \(r\) gives \(\norm{M_a}\geq\sup_r\norm{a_r}=\norm a\), which proves the norm identity in \eqref{eq:exact-algebra}.
  Finally, using \(R_a^*R_b^*=R_{ab}^*\) on core fibres and \(L_aL_b=L_{ab}\) on unit-probe fibres, we obtain
  \[ D_\gamma M_aM_b =\Theta_\gamma(a_\gamma) \Theta_\gamma(b_{r(\gamma)})D_\gamma =\Theta_\gamma((ab)_{r(\gamma)})D_\gamma. \]
  Thus every coefficient of \(M_aM_b-M_{ab}\) is zero.
  Since the FDD is total, we obtain \(M_aM_b=M_{ab}\).
  The same coefficient calculation and unitality of every module give \(M_1=I_X\), proving \eqref{eq:exact-algebra}.
  The definition also gives \(M_{\lambda a+\mu b}=\lambda M_a+\mu M_b\) on \(X^{\rm alg}\).
  Therefore boundedness and density extend this identity to \(X\).
\end{proof}

\subsection{Interval restrictions, common stems, and shielding}

Let \(I\) be an FDD interval.
Restrict a labelled scalarized coordinate evaluation recursively by applying \(P_I^*\) to its analysis in Lemma~\ref{lem:formal-evaluation-analysis}.
At a probe, retain \(d_{\eta,\psi}^*\) when \(\rank\eta\in I\) and replace it by zero otherwise.
For a scalarized packet \(B_{\mathbf b}^*\varphi=\sum_lc_le_{\eta_l,\psi_l'}^*\), restrict each lower-rank coordinate recursively in its original position and retain zero positions.
This operation changes only the analysis and leaves the represented packet stored in the atom record unchanged.
The recursion terminates because ranks decrease strictly.

For an evaluation as in Lemma~\ref{lem:formal-evaluation-analysis}, put
\[ e_t=e_{\xi_t,\varphi_t}^*,\qquad e_0=0, \qquad t_t=\frac{1}{m_h}b_{t,\varphi}^*P_{(p_{t-1},\infty)}, \qquad t_{a+1}=0. \]
If \(w\in\N\), let \(r_w\) be determined by \(p_{r_w}\leq w<p_{r_w+1}\), with \(r_w=0\) before \(p_1\) and \(r_w=a\) from \(p_a\) onwards.
Since \(P_{p_t-1}^*b_{t,\varphi}^*=b_{t,\varphi}^*\), we have
\[
  t_t=\frac{1}{m_h}(P_{p_t-1}^*-P_{p_{t-1}}^*)b_{t,\varphi}^*
      =P_{(p_{t-1},p_t)}^*t_t.
\]
Thus each \(t_t\) has BD support inside \((p_{t-1},p_t)\), and we obtain
\[
  P_{(w,\infty)}^*t_t=
  \begin{cases}
    0,&t\leq r_w,\\
    t_t,&t\geq r_w+2.
  \end{cases}
  \qquad
  P_{(w,\infty)}^*e_{r_w}=0.
\]
Only \(t_{r_w+1}\) can be cut by the boundary \(w\), so we obtain
\begin{equation}\label{eq:tail-restriction}
  \begin{aligned}
    P_{(w,\infty)}^*e_a
    &=\sum_{t=r_w+1}^ad_{\xi_t,\varphi_t}^*
      +\sum_{t=r_w+2}^at_t
      +P_{(w,\infty)}^*t_{r_w+1}\\
    &=e_a-e_{r_w}+(P_{(w,\infty)}^*-I)t_{r_w+1}.
  \end{aligned}
\end{equation}
Here empty sums are zero.
The conventions \(e_0=0\) and \(t_{a+1}=0\) cover both \(w<p_1\) and \(w\geq p_a\).

The restriction formula below is the scalarized form of the tail calculation in \cite[Lemma~5.3]{ArgyrosHaydon2011}.

\begin{lemma}[Interval restriction formula]\label{lem:formal-restriction}
  If \(u,v\in\N\) and \(u<v\), then we have
  \begin{equation}\label{eq:interval-restriction}
    \begin{aligned}
      P_{(u,v]}^*e_a
      =e_{r_v}-e_{r_u}
        +(P_{(u,\infty)}^*-I)t_{r_u+1}-(P_{(v,\infty)}^*-I)t_{r_v+1}.
    \end{aligned}
  \end{equation}
  The restriction preserves the stored labelled analysis and may partially cut only the two boundary packet cells.
\end{lemma}

\begin{proof}
  Since \(P_{(u,v]}^*=P_{(u,\infty)}^*-P_{(v,\infty)}^*\), subtracting \eqref{eq:tail-restriction} at \(v\) from the identity at \(u\) gives \eqref{eq:interval-restriction}.
  Since the packet cells are successive, only the first and last cells meeting \((u,v]\) can be partially cut.
  Recursive restriction of packet entries terminates because ranks strictly decrease. By definition, the stored labels and histories are unchanged.
\end{proof}

The next comparison uses the injective-coding argument of \cite[Lemma~4.6]{ArgyrosHaydon2011}.

\begin{lemma}[Common-stem comparison]\label{lem:formal-common-stem}
  Let two outer histories have seed level \(q\).
  If their seed records differ, every coded inner level in one history differs from every coded inner level in the other.
  If the seed records agree, let \(\mathfrak p\) be their longest common realized stem.
  Their labelled strong packets and certificates agree along \(\mathfrak p\).
  When both strong packets immediately following \(\mathfrak p\) exist, their inner terminals have the common level \(\sigma(q,\mathfrak p)\), while apart from this possible pair, every later inner level in one history differs from every later inner level in the other.
\end{lemma}

\begin{proof}
  Let \(\mathfrak p_{i-1}\) and \(\tau_{j-1}\) be the stems preceding the \(i\)-th and \(j\)-th strong packets in the two histories.
  Injectivity gives
  \[ \sigma(q,\mathfrak p_{i-1})=\sigma(q,\tau_{j-1}) \quad\Longleftrightarrow\quad \mathfrak p_{i-1}=\tau_{j-1}. \]
  Each stem contains its seed, so different seeds make this equality impossible.
  Suppose the seeds agree and their longest common stem \(\mathfrak p\) has length \(a\).
  Equality of finite words forces equality of their lengths and every earlier entry, so whenever both packets exist, we have
  \[ \mathfrak p_{i-1}=\tau_{j-1} \quad\Longleftrightarrow\quad i=j\leq a+1. \]
  For \(i=j\leq a\), the common word fixes the terminal identifier, strong packet, and certificate.
  At \(i=j=a+1\), it fixes only the common input \((q,\mathfrak p)\), and thus the common inner level.
  There are no further equal inner levels across the two histories.
\end{proof}

For a finite-stage vector \(x=i_nu\) and a set \(L\subseteq\N\), let \(Q_Lu\) retain the atom summands of the finite \(\ell_\infty\)-sum whose levels lie in \(L\).
Both complementary coordinate projections are contractive: \(\norm{Q_Lu}\leq\norm u\) and \(\norm{(I-Q_L)u}\leq\norm u\).
The two vectors \(i_nQ_Lu\) and \(i_n(I-Q_L)u\) sum to \(x\), and we have
\begin{equation}\label{eq:level-deletion}
  \norm{i_nQ_Lu},\ \norm{i_n(I-Q_L)u}\leq M\norm x.
\end{equation}
Indeed, \(i_n\) fixes its first \(n\) ambient blocks, so we have \(\norm u\leq\norm x\).
Hence we obtain
\[ \norm{i_nTu}\leq M\norm{Tu}\leq M\norm u\leq M\norm x \qquad(T=Q_L\text{ or }I-Q_L). \]
This is the finite-stage local-level decomposition used in Part~II.

Let \(\mathcal P_r^{\rm u}\) and \(\mathcal P_r^{\rm c}\) be the unit and core probes of detector index \(r\), and put \(j_\gamma=d_\gamma:E_\gamma\to X\) for every probe.
Their norm properties are given by Lemma~\ref{lem:formal-probe-isometry}.

\begin{proposition}[Uniform probe shielding]\label{prop:formal-shielding}
  If \(\gamma\) is a probe and \(h>0\), then we have
  \begin{equation}\label{eq:shielding}
    \sup_{\substack{\ell(\alpha)=h,\ \psi\in B_{E_\alpha^*}\\
    u\in B_{E_\gamma}}}
    \abs{e_{\alpha,\psi}^*(j_\gamma u)}
    \leq\frac{1}{m_h}.
  \end{equation}
  The estimate is uniform in the detector indices, fibre dimensions, and probe tags.
\end{proposition}

This follows from the evaluation estimate in the proof of Lemma~\ref{lem:formal-probe-isometry} and the identity \(\norm{j_\gamma u}=\norm u\).

\subsection{Properties of the constructed space}

\begin{theorem}[Construction and properties of the space]\label{thm:datum}
  Let \(\mathscr A=(A_r,\pi_r^s)_{0\leq r\leq s}\) be any inverse system satisfying \eqref{eq:inverse-system}--\eqref{eq:metric-quotient}.
  Construction~\ref{con:fair-recursion} is a well-defined oracle-recursive monotone finite-extension construction.
  It produces a separable block BD space \(X=X_{\mathscr A}\) with an FDD satisfying the following conditions.
  \begin{enumerate}[label={\textup{(G\arabic*)}},leftmargin=*]
    \item the initial projection bound is \(M\) and the interval projection bound is \(\kappa=2M\).
    \item every positive-level scalarized evaluation has the analysis \eqref{eq:evaluation-grammar}, with contractive predecessor and packet maps, while every level-zero probe evaluation equals its own \(d^*\)-coordinate and every probe injection is isometric as in \eqref{eq:formal-probe-bounds}.
    \item canonical restriction, common-stem comparison, and the decomposition by local levels have the properties stated above.
    \item the cofinal extension properties of Proposition~\ref{prop:finite-extension} hold, and canonical unit and core probes of every detector index occur cofinally.
    \item the shielding estimate \eqref{eq:shielding} holds uniformly in the detector indices and fibre dimensions.
    \item the operators in \eqref{eq:multiplier-definition} satisfy the scalarization, norm, and multiplicative identities \eqref{eq:full-chain}--\eqref{eq:exact-algebra}.
  \end{enumerate}
\end{theorem}

\begin{proof}
  Lemma~\ref{lem:formal-legality} shows that each correction uses only atoms of smaller rank, and Proposition~\ref{prop:bd-legality} verifies the two correction forms.
  Lemma~\ref{lem:block-bd} therefore gives the uniform dual projection bound \(M\).
  The finite adjoint construction \eqref{eq:finite-extension} and its consistency \eqref{eq:finite-dual-consistency} produce the nested extension ranges, and their closure is a block BD space.
  Every \(F_n\) is a finite sum of finite-dimensional fibres and there are countably many ranks, so \(X\) is separable.
  This proves \textup{(G1)}.

  Lemma~\ref{lem:formal-evaluation-analysis} and Lemma~\ref{lem:formal-probe-isometry} prove \textup{(G2)} from the one-edge recurrence, including contractivity, the absolute-convex-hull statement, and the probe norm identity.
  Lemma~\ref{lem:formal-restriction}, Lemma~\ref{lem:formal-common-stem}, and \eqref{eq:level-deletion} prove the three assertions in \textup{(G3)}.
  Restriction preserves every stored packet and stem, as required by the no-injury property.

  Persistence and fairness follow from Lemma~\ref{lem:formal-legality}, and Proposition~\ref{prop:finite-extension} allows packet choices after earlier atoms have been realized.
  This gives \textup{(G4)}.
  Proposition~\ref{prop:formal-shielding} gives \textup{(G5)} with the same constants for both fibre types and all dimensions.
  Finally, Lemma~\ref{lem:scalarization}, proved using \eqref{eq:formal-coordinate-recurrence}, gives \textup{(G6)}.
\end{proof}

  \part{Analytic estimates and operator classification}


\section{Rapidly increasing sequences and the basic inequality}
\label{sec:ris-method}

We use the RIS and basic-inequality method of the Argyros--Haydon construction and its subsequent Bourgain--Delbaen variants \cite{ArgyrosHaydon2011,Tarbard2012,Tarbard2013,Zisimopoulou2014,MotakisPuglisiZisimopoulou2016,ManoussakisPelczarSwietek2017,ArgyrosMotakis2019,MotakisPuglisiTolias2020,Motakis2024,MotakisPelczar2025}.
For the background in recursive norming, weighted averages, and mixed Tsirelson spaces, see \cite{Schlumprecht1991,GowersMaurey1993,GowersMaurey1997,ArgyrosDeliyanni1997,ArgyrosDeliyanniKutzarovaManoussakis1998}.
A general treatment of HI methods is given in \cite{ArgyrosTolias2004}.
For vector-valued BD sums and later AH implementations, see also \cite{Tarbard2013,Zisimopoulou2014,ManoussakisPelczarSwietek2017,MotakisPuglisiTolias2020,MotakisPelczar2025}.

We retain the BD notation from Subsection~\ref{sec:bd-realization} and the evaluation analysis in \eqref{eq:evaluation-grammar}.
Thus \(E_\gamma\) is the fibre at \(\gamma\), \(\ell(\gamma)\) and \(\rank\gamma\) are its level and rank, and \(e_{\gamma,\varphi}^*\) and \(d_{\gamma,\varphi}^*\) are the scalarized ambient and BD coordinates.
The FDD projection onto an interval \(I\) is denoted by \(P_I\), with \(\norm{P_I}\leq\kappa\).
For a scalarized evaluation \(F=e_{\gamma,\varphi}^*\), we write \(\ell(F)=\ell(\gamma)\).
The sequences \((m_j),(n_j)\) and the constants below are those fixed in \eqref{eq:constants}--\eqref{eq:numerical-three}.
Unless otherwise stated, scalarized evaluations in norm estimates have \(\varphi\in B_{E_\gamma^*}\), and all unqualified suprema over \((\gamma,\varphi)\) use this range.

\subsection{Auxiliary norming set}

Let \(c_{00}(\N_+)\) be the space of finitely supported scalar sequences, with unit vectors \((e_k)\) and coordinate functionals \(e_k^*\).
We use the mixed Tsirelson norming set generated by successive \((3n_h,m_h^{-1})\)-averages, as in \cite{ArgyrosDeliyanni1997} and \cite[Section~2.4]{ArgyrosHaydon2011}.
Since a functional may have more than one tree analysis, we retain the root level as part of the data.
Define \(\mathscr W\) recursively as follows.
\begin{enumerate}[label={\textup{(\roman*)}},leftmargin=27pt]
  \item For \(k\in\N_+\) and \(\abs{\zeta}=1\), the tagged leaf \([\zeta e_k^*]\) belongs to \(\mathscr W\), has value \(\operatorname{val}[\zeta e_k^*]=\zeta e_k^*\), and has root level \(0\).
  \item If \(\mathfrak g_1,\ldots,\mathfrak g_d\in\mathscr W\) have successive supports of their values \(\operatorname{val}(\mathfrak g_1),\ldots,\operatorname{val}(\mathfrak g_d)\) and \(d\leq3n_h\), then the new tagged tree \(\mathfrak g=[h;\mathfrak g_1,\ldots,\mathfrak g_d]\) belongs to \(\mathscr W\), has \(\operatorname{root}(\mathfrak g)=h\), and has value
  \[ \operatorname{val}(\mathfrak g) =\frac{1}{m_h}\sum_{t=1}^d \operatorname{val}(\mathfrak g_t). \]
\end{enumerate}
Adjoin a zero symbol and set \(\operatorname{val}(0)=0\).
For \(x\in c_{00}(\N_+;\K)\), put
\[ \norm{x}_{T_{\mathrm{aux}}} =\sup_{\mathfrak g\in\mathscr W} \abs{\operatorname{val}(\mathfrak g)(x)}. \]
Let \(T_{\mathrm{aux}}\) be the completion of \(c_{00}(\N_+\) under this norm.

The following definition is the fibre version of \cite[Definition~5.1]{ArgyrosHaydon2011}.

\begin{definition}\label{def:ris}
  A block sequence \((x_k)\) is a \(C\)-RIS if \(\norm{x_k}\leq C\) and there are strictly increasing positive integers \((j_k)\) such that
  \[ j_{k+1}>\max\ran x_k,\qquad \abs{e_{\gamma,\varphi}^*(x_k)}\leq\frac{C}{m_h} \]
  whenever \(0<\ell(\gamma)=h<j_k\) and \(\varphi\in B_{E_\gamma^*}\).
\end{definition}

\begin{corollary}[Shielded probe RISs]\label{cor:probe-ris}
  The shielding estimate in Proposition~\ref{prop:formal-shielding} has the following consequences.
  \begin{enumerate}[label={\textup{(\roman*)}}]
    \item Every sufficiently separated sequence \((j_{\gamma_k}u_k)\), with \(u_k\in B_{E_{\gamma_k}}\), is a \(1\)-RIS.
    \item If probes $\gamma_k,\delta_k$ satisfy \(\max\{\rank\gamma_k,\rank\delta_k\}<\min\{\rank\gamma_{k+1},\rank\delta_{k+1}\},\)
    then uniformly for \(u_k\in B_{E_{\gamma_k}}\) and \(v_k\in B_{E_{\delta_k}}\), we have both \((j_{\gamma_k}u_k)\) and \((j_{\delta_k}v_k)\) are \(1\)-RISs, and \( (j_{\gamma_k}u_k+j_{\delta_k}v_k) \)is a \(2\)-RIS.
    \item If \((x_k)\) is a \(C\)-RIS, then, after passing to a subsequence, probes \(\gamma_k\) can be chosen with \(x_k<\gamma_k<x_{k+1}\) so that \((x_k+j_{\gamma_k}u_k)\) is a \((C+1)\)-RIS uniformly for \(u_k\in B_{E_{\gamma_k}}\).
  \end{enumerate}
\end{corollary}

\begin{proof}
  For \textup{(i)}, choose the RIS index \(j_k\) strictly after the preceding probe block and strictly before the next one.
  If \(0<h<j_k\), we obtain from Proposition~\ref{prop:formal-shielding} \(\abs{e_{\alpha,\psi}^*(j_{\gamma_k}u_k)}\leq m_h^{-1}\).
  Lemma~\ref{lem:formal-probe-isometry} gives \(\norm{j_{\gamma_k}u_k}\leq1\), so Definition~\ref{def:ris} applies.

  For \textup{(ii)}, put \(v_k'=j_{\gamma_k}u_k+j_{\delta_k}v_k\).
  Both component estimates hold with a common RIS index before the two probe ranks, and we have
  \[
    \supp_{\rm FDD}v_k'
      \subseteq[\min\{\rank\gamma_k,\rank\delta_k\},
                 \max\{\rank\gamma_k,\rank\delta_k\}],
    \qquad\norm{v_k'}\leq2.
  \]
  The separation of these intervals gives the RIS range condition.
  And we obtain
  \[ \abs{e_{\alpha,\psi}^* (j_{\gamma_k}u_k+j_{\delta_k}v_k)} \leq\frac{1}{m_h}+\frac{1}{m_h} =\frac{2}{m_h}. \]
  Hence the sums form a \(2\)-RIS.
  For \textup{(iii)}, let \((j_l)\) be the original RIS indices.
  After \(x_{l_k}\), choose \(\gamma_k\), then choose \(l_{k+1}>l_k\) such that
  \[ j_{l_{k+1}}>\rank\gamma_k,\qquad \min\ran x_{l_{k+1}}>\rank\gamma_k. \]
  This is possible because the original indices and block supports tend to infinity.
  Relabel the selected vectors and their retained indices.
  Then we have \(x_k<\gamma_k<x_{k+1}\) and \(j_{k+1}>\rank\gamma_k\geq\max\ran(x_k+j_{\gamma_k}u_k)\).
  For \(0<h<j_k\), we obtain from the original RIS estimate and shielding
  \[ \abs{e_{\alpha,\psi}^*(x_k+j_{\gamma_k}u_k)} \leq\frac{C}{m_h}+\frac{1}{m_h}=\frac{C+1}{m_h}. \]
  Together with \(\norm{x_k+j_{\gamma_k}u_k}\leq C+1\), this proves the assertion.
\end{proof}

\subsection{Truncation estimates}

We adapt the truncation estimates of \cite[Lemmas~5.2 and~5.3]{ArgyrosHaydon2011} to scalarized evaluations.

\begin{lemma}[Two truncation estimates]\label{lem:truncation}
  Let \((x_k)\) be a \(C\)-RIS with RIS indices \((j_k)\), let \(s\in\N\), and let \(f=e_{\gamma,\varphi}^*\) have level \(h>0\).
  Then we have
  \begin{equation}\label{eq:truncation}
    \abs{fP_{(s,\infty)}x_k}
    \leq
    \begin{cases}
      C_{\rm tr}C/m_h,&h<j_k,\\
      \kappa C/m_h,&h\geq j_{k+1}.
    \end{cases}
  \end{equation}
\end{lemma}

\begin{proof}
  Let \(\xi_1\prec\cdots\prec\xi_a=\gamma\), where \(p_t=\rank\xi_t\), \(p_0=\chi(\gamma)\) is the initial cutoff, and \(a\leq n_h\), be the stored level-\(h\) predecessor chain.
  Put \(R_t=\pi_{r(\xi_t)}^{r(\gamma)}\).
  Write \(e_t=e_{\xi_t,R_t\varphi}^*\), put \(e_0=0\), and put \(t_t=m_h^{-1}b_{t,\varphi}^*P_{(p_{t-1},\infty)}\) for \(1\leq t\leq a\).
  For \(w\in\N\), let \(r_w=\max\{t:p_t\leq w\}\), with \(r_w=0\) if the set is empty, and put \(t_{a+1}=0\).

  Put \(\nu=\max\ran x_k\).
  If \(s\geq\nu\), then we have \(P_{(s,\infty)}x_k=0\).
  We may therefore assume \(s<\nu\), in which case \(P_{(s,\infty)}x_k=P_{(s,\nu]}x_k\).
  If \(h<j_k\), subtract \eqref{eq:tail-restriction} for \(w=s\) and \(w=\nu\).
  The two inherited evaluations satisfy
  \[ \abs{e_{r_s}(x_k)},\ \abs{e_{r_\nu}(x_k)} \leq\frac{C}{m_h}. \]
  For \(w\in\{s,\nu\}\) with \(r_w<a\), the boundary term satisfies
  \[
    \begin{aligned}
      \abs{((P_{(w,\infty)}^*-I)t_{r_w+1})(x_k)}
      =\frac{1}{m_h}\abs{b_{r_w+1,\varphi}^*
      P_{(p_{r_w},p_{r_w+1})\cap(0,w]}x_k}\leq\frac{\norm{b_{r_w+1,\varphi}^*}\kappa\norm{x_k}}{m_h}
      \leq\frac{\kappa C}{m_h}.
    \end{aligned}
  \]
  If \(r_w=a\), this boundary term is zero.
  From the two inherited evaluations and the two boundary terms, we therefore obtain
  \[ \abs{fP_{(s,\infty)}x_k} \leq\frac{2C+2\kappa C}{m_h} =\frac{C_{\rm tr}C}{m_h}. \]

  Suppose that \(h\geq j_{k+1}\).
  Rank admissibility of a level-\(h\) predecessor chain and the RIS range condition give \(p_1\geq h\geq j_{k+1}>\nu\).
  Every predecessor term and every packet after the first is supported beyond \(x_k\), and only the first packet can act.
  Hence we obtain
  \[ \abs{fP_{(s,\infty)}x_k} =\frac{1}{m_h}\abs{b_{1,\varphi}^* P_{(p_0,p_1)\cap(s,\nu]}x_k} \leq\frac{\kappa C}{m_h}. \]
\end{proof}

\subsection{The basic inequality}

We next estimate scalarized evaluations on finite RIS sums, using the basic-inequality argument of \cite[Proposition~5.4]{ArgyrosHaydon2011}.
Recall from \eqref{eq:constants} that \(B=4\kappa+4\geq\max\{C_{\rm tr},2\kappa,1\}\).

\begin{lemma}[Scalarized basic inequality]\label{lem:basic}
  Let \(I\) be a nonempty finite integer interval, let \((x_k)_{k\in I}\) be a finite part of a \(C\)-RIS, let \((\lambda_k)_{k\in I}\subseteq\K\), let \(s\in\N\), and let \(f=e_{\gamma,\varphi}^*\) be a scalarized evaluation.
  There are \(k_0\in I\) and \(\mathfrak g\in\mathscr W\cup\{0\}\) such that \(\operatorname{val}(\mathfrak g)\) has nonnegative coefficients and, if \(\mathfrak g\ne0\), then \(\operatorname{root}(\mathfrak g)=\ell(\gamma)\) and its value is supported after \(k_0\), and
  \begin{equation}\label{eq:basic}
    \abs{fP_{(s,\infty)}\sum_{k\in I}\lambda_kx_k}
    \leq BC\left(
    \abs{\lambda_{k_0}}
    +\operatorname{val}(\mathfrak g)
    \left(\sum_{k\in I}\abs{\lambda_k}e_k\right)\right).
  \end{equation}
\end{lemma}

\begin{proof}
  Put \(v=\sum_{k\in I}\abs{\lambda_k}e_k\).
  We induct on the rank of the ambient coordinate defining \(f\).
  We prove the stronger assertion that \(\supp\operatorname{val}(\mathfrak g)\subseteq I\).
  A level-zero evaluation meets at most one block, so the assertion holds with \(\mathfrak g=0\).

  Let \(f\) have level \(h>0\), and put \(l=\max(\{\min I\}\cup\{k\in I:j_k\leq h\})\).
  Then we have \(h<j_k\) for \(k>l\), while \(h\geq j_{k+1}\) whenever \(k<l\).
  By Lemma~\ref{lem:truncation} and monotonicity of \((m_j)\), we obtain
  \[
    \begin{aligned}
      \sum_{k<l}\abs{\lambda_k}\abs{fP_{(s,\infty)}x_k}
      &\leq\frac{\kappa C}{m_h}\sum_{k<l}\abs{\lambda_k}
      \leq\kappa C\sum_{k<l}\frac{\abs{\lambda_k}}{m_{j_k}}\\
      &\leq\kappa C\left(\sum_{r\geq1}m_r^{-1}\right)
        \max_{k\leq l}\abs{\lambda_k}
      \leq\kappa C\max_{k\leq l}\abs{\lambda_k}.
    \end{aligned}
  \]
  Since \(\norm f\leq1\) and \(\norm{P_{(s,\infty)}}\leq\kappa\), we have
  \( \abs{\lambda_l}\abs{fP_{(s,\infty)}x_l} \leq\kappa C\abs{\lambda_l}. \) 
  Choose \(k_0\leq l\) so that \(\abs{\lambda_{k_0}}=\max_{k\leq l}\abs{\lambda_k}\).
  Therefore we obtain
  \begin{equation}\label{eq:basic-initial}
    \sum_{k\leq l}\abs{\lambda_k}\,
    \abs{fP_{(s,\infty)}x_k}
    \leq2\kappa C\abs{\lambda_{k_0}}
    \leq BC\abs{\lambda_{k_0}}.
  \end{equation}

  Put \(I'=\{k\in I:k>l\}\).
  In \eqref{eq:evaluation-grammar}, put
  \[
    \begin{aligned}
      I'_0&=\{k\in I':p_t\in\ran x_k\text{ for some }1\leq t\leq a\},\\
      I'_t&=\{k\in I'\setminus I'_0:
        \ran x_k\cap(s,\infty)\cap(p_{t-1},p_t)\ne\varnothing\}
        \quad(1\leq t\leq a).
    \end{aligned}
  \]
  Since the block ranges are successive, we have \(\abs{I'_0}\leq a\), and the nonempty \(I'_t\)'s are successive integer intervals.
  A block outside \(I'_0\) cannot meet two packet cells without containing their common boundary rank.
  The evaluation analysis therefore gives
  \[
    fP_{(s,\infty)}x_k=
    \begin{cases}
      m_h^{-1}b_{t,\varphi}^*P_{(s\vee p_{t-1},\infty)}x_k,
        &k\in I'_t,\\
      0,&k\in I'\setminus\bigl(I'_0\cup\bigcup_{t=1}^a I'_t\bigr).
    \end{cases}
  \]
  In the second case the truncated block misses every predecessor rank and packet cell.
  For each nonempty \(I'_t\), put
  \(z_t=P_{(s\vee p_{t-1},\infty)}\sum_{k\in I'_t}\lambda_kx_k\).

  Fix \(t\) with \(I'_t\ne\varnothing\).
  By the scalarization in the proof of Lemma~\ref{lem:formal-evaluation-analysis}, we have a finite expansion, with \(\omega_\ell\in E_{\eta_\ell}^*\),
  \[ b_{t,\varphi}^*=\sum_{\ell\in L_t}a_\ell e_{\eta_\ell,\omega_\ell}^*,\qquad \sum_{\ell\in L_t}\abs{a_\ell}\norm{\omega_\ell}\leq1. \]
  Omit zero terms in this expansion and put \(c_\ell=a_\ell\norm{\omega_\ell}\) and \(\psi_\ell=\omega_\ell/\norm{\omega_\ell}\).
  Then we have
  \[ b_{t,\varphi}^*=\sum_{\ell\in L_t}c_\ell e_{\eta_\ell,\psi_\ell}^*,\qquad \norm{\psi_\ell}=1,\qquad \sum_{\ell\in L_t}\abs{c_\ell}\leq1. \]
  If \(b_{t,\varphi}^*=0\), put \(f_t=0\), \(k_t=\min I'_t\), and \(\mathfrak g_t=0\).
  Otherwise, choose \(\ell_t\in L_t\) such that
  \[ \abs{e_{\eta_{\ell_t},\psi_{\ell_t}}^*(z_t)} =\max_{\ell\in L_t}\abs{e_{\eta_\ell,\psi_\ell}^*(z_t)}, \qquad f_t=e_{\eta_{\ell_t},\psi_{\ell_t}}^*. \]
  The packet support condition gives \(p_{t-1}<\rank\eta_{\ell_t}<p_t\leq\rank\gamma\).
  In this case, we obtain
  \[
    \begin{aligned}
      \abs{b_{t,\varphi}^*(z_t)}
      \leq\sum_{\ell\in L_t}\abs{c_\ell}
      \abs{e_{\eta_\ell,\psi_\ell}^*(z_t)}
      \leq\abs{f_t(z_t)}=\abs{f_tP_{(s\vee p_{t-1},\infty)}
      \sum_{k\in I'_t}\lambda_kx_k}.
    \end{aligned}
  \]
  For a zero packet, we have \(\abs{b_{t,\varphi}^*(z_t)}=\abs{f_t(z_t)}=0\).
  Thus we obtain
  \begin{align}
    \abs{fP_{(s,\infty)}\sum_{k\in I'}\lambda_kx_k}
    &\leq\frac{C_{\rm tr}C}{m_h}
    \sum_{k\in I'_0}\abs{\lambda_k}
    +\frac{1}{m_h}\sum_{I'_t\ne\varnothing}
    \abs{b_{t,\varphi}^*(z_t)}\notag\\
    &\leq\frac{C_{\rm tr}C}{m_h}
    \sum_{k\in I'_0}\abs{\lambda_k}
    +\frac{1}{m_h}\sum_{I'_t\ne\varnothing}
    \abs{f_t(z_t)}.                         \label{eq:basic-cell}
  \end{align}
  For every nonempty \(I'_t\) with \(f_t\ne0\), the coordinate defining \(f_t\) has strictly smaller rank.
  The inductive hypothesis with cutoff \(s\vee p_{t-1}\) gives \(k_t\in I'_t\) and \(\mathfrak g_t\in\mathscr W\cup\{0\}\) satisfying the following inequality, and the choices above give the same inequality when \(f_t=0\):
  \begin{equation}\label{eq:basic-one-cell}
    \abs{f_t(z_t)}
    \leq BC\bigl(\abs{\lambda_{k_t}}
    +\operatorname{val}(\mathfrak g_t)(v)\bigr).
  \end{equation}
  The stronger inductive assertion gives
  \(\supp\operatorname{val}(\mathfrak g_t)\subseteq I'_t\cap(k_t,\infty)\).
  Thus the leaf at \(k_t\) precedes \(\mathfrak g_t\), and both stay inside the same cell-index interval.
  In support order, the coordinate leaves \([e_k^*]\), \(k\in I'_0\), and the pairs \([e_{k_t}^*],\mathfrak g_t\), \(I'_t\ne\varnothing\), satisfy
  \[ \abs{I'_0}+2\abs{\{t:I'_t\ne\varnothing\}} \leq a+2a\leq3n_h. \]
  Omit zero children, arrange the remaining terms in support order, and denote the resulting list by \(\mathcal C_h\).
  If the list is empty, put \(\mathfrak g=0\).
  Otherwise put \(\mathfrak g=[h;\mathcal C_h]\in\mathscr W\).
  In the nonzero case its root level is \(h\), and its value is supported after \(k_0\).
  By this definition, we have
  \[
    \operatorname{val}(\mathfrak g)(v)
    =\frac{1}{m_h}\left[
      \sum_{k\in I'_0}\abs{\lambda_k}
      +\sum_{I'_t\ne\varnothing}
        \bigl(\abs{\lambda_{k_t}}+\operatorname{val}(\mathfrak g_t)(v)\bigr)
      \right].
  \]
  Since \(B\geq\max\{C_{\rm tr},1\}\), we obtain from \eqref{eq:basic-cell} and \eqref{eq:basic-one-cell}
  \[ \abs{fP_{(s,\infty)}\sum_{k\in I'}\lambda_kx_k} \leq BC\operatorname{val}(\mathfrak g)(v). \]
  Adding this estimate to \eqref{eq:basic-initial} gives \eqref{eq:basic}.
\end{proof}

\subsection{Auxiliary and special averages}

The stopping-event formulation below was prompted by a suggestion from a generative-AI system.
Subsequent literature review identified the underlying counting and tree-truncation techniques in \cite[Lemma~2.4 and Proposition~2.5]{ArgyrosHaydon2011}, \cite[Sections~2 and~3.2]{ManoussakisPelczarSwietek2017}, and \cite[Proposition~11.7]{MotakisPelczar2025}.
We adapt these techniques to the parameters fixed in \eqref{eq:numerical-one}--\eqref{eq:numerical-three} and record the first stopping event on each branch.
For further details, see the \hyperref[sec:ai-assistance]{AI assistance statement} at the end of the paper.

\begin{lemma}[Auxiliary special averages]
  \label{lem:quantitative-auxiliary}
  For \(\bar e_j=n_j^{-1}\sum_{k=1}^{n_j}e_k\), we have
  \begin{equation}\label{eq:quantitative-auxiliary}
    \norm{\bar e_j}_{T_{\mathrm{aux}}}\leq\frac{2}{m_j}.
  \end{equation}
\end{lemma}

\begin{proof}
  Fix \(j\geq2\) and \(\mathfrak g\in\mathscr W\).
  Stop its tagged tree at the first node of level at least \(j\), at the \(D_j\)-th node of level less than \(j\), or at a coordinate leaf.
  Let \(\mathcal H\), \(\mathcal D\), and \(\mathcal L\) be the corresponding classes of coordinate indices in \(\{1,\ldots,n_j\}\).
  Successive child supports place each coordinate index on at most one branch, so these classes are disjoint and cover the coordinates which contribute to \(\operatorname{val}(\mathfrak g)(\bar e_j)\).
  The coefficient of each such coordinate is the product of the weights on its branch, and weights below the stopping node can only decrease its absolute value.
  A node of level less than \(j\) has at most \(\mathsf b_j\) successors.
  Thus we obtain
  \[ \abs{\mathcal L}\leq\sum_{r=0}^{D_j-1}\mathsf b_j^r\leq S_j, \qquad \abs{\mathcal H},\ \abs{\mathcal D}\leq n_j, \]
  and the branch coefficients on \(\mathcal H\), \(\mathcal D\), and \(\mathcal L\) are at most \(m_j^{-1}\), \(m_1^{-D_j}\), and one, respectively.
  Therefore we obtain
  \[
    \begin{aligned}
      \abs{\operatorname{val}(\mathfrak g)(\bar e_j)}
      \leq\frac{\abs{\mathcal L}}{n_j}
        +\frac{\abs{\mathcal H}}{n_jm_j}
        +\frac{\abs{\mathcal D}}{n_jm_1^{D_j}}\leq\frac{S_j}{n_j}+\frac{1}{m_j}+\frac{1}{m_1^{D_j}}.
    \end{aligned}
  \]
  Consequently, we have
  \[
    \abs{\operatorname{val}(\mathfrak g)(\bar e_j)}
    \leq\frac{1}{m_j}+\frac{1}{m_j^2}+\frac{1}{2^{j+8}m_j^5}
    <\frac{2}{m_j}.
  \]
  Taking the supremum over \(\mathfrak g\) proves the assertion for \(j\geq2\).

  For \(j=1\), we obtain from a coordinate root and a weighted root, respectively,
  \[
  \abs{\operatorname{val}(\mathfrak g)(\bar e_1)} \leq \begin{cases} n_1^{-1},&\mathfrak g\text{ is a coordinate leaf},\\
  m_1^{-1},&\mathfrak g\text{ has a weighted root}. \end{cases}
  \]
  Since \(n_1\geq m_1^2\geq m_1\), taking the supremum proves \eqref{eq:quantitative-auxiliary} also for \(j=1\).
\end{proof}

\begin{corollary}[Bounds for RIS special averages]
  \label{lem:ris-consequences}
  If \((x_k)_{k=1}^{n_j}\) is a \(C\)-RIS, then we have
  \begin{align}
    \norm{\frac{m_j}{n_j}\sum_{k=1}^{n_j}x_k}
    &\leq3BC,                                                     \label{eq:ris-average}\\
    \sup_{\gamma,\varphi}
    \abs{d_{\gamma,\varphi}^*
    \left(\frac{m_j}{n_j}\sum_{k=1}^{n_j}x_k\right)}
    &\leq\frac{\kappa C}{m_j}.                                  \label{eq:ris-coordinate}
  \end{align}
\end{corollary}

\begin{proof}
  Put \(x=(m_j/n_j)\sum_{k=1}^{n_j}x_k\), and apply Lemma~\ref{lem:basic} with \(s=0\) and \(\lambda_k=m_j/n_j\).
  For every scalarized evaluation \(f\), using \(n_j\geq m_j^2\) and \eqref{eq:quantitative-auxiliary}, we obtain
  \[ \abs{f(x)} \leq BC\left(\frac{m_j}{n_j}+m_j\norm{\bar e_j}_{T_{\mathrm{aux}}}\right) \leq BC(m_j^{-1}+2)\leq3BC. \]
  Since scalarized ambient coordinates norm \(X\), we prove \eqref{eq:ris-average}.
  A fixed \(d^*\)-coordinate meets at most one block.
  Moreover, it is the restriction of an ambient evaluation to one FDD coordinate, so its norm is at most \(\kappa\).
  Hence we obtain
  \[ \abs{d_{\gamma,\varphi}^*(x)} \leq\frac{m_j}{n_j} \sum_{k=1}^{n_j}\abs{d_{\gamma,\varphi}^*(x_k)} \leq\frac{\kappa Cm_j}{n_j}\leq\frac{\kappa C}{m_j}, \]
  which is \eqref{eq:ris-coordinate}.
\end{proof}

\begin{corollary}[RIS are weakly null]\label{cor:ris-weakly-null}
  Every RIS is weakly null.
\end{corollary}

\begin{proof}
  Let \((\lambda_k)\) be finitely supported and put \(v=\sum_k\abs{\lambda_k}e_k\).
  For the \(k_0\) and \(\mathfrak g\) supplied by Lemma~\ref{lem:basic}, the coordinate leaves belong to \(\mathscr W\), so we have \(\abs{\lambda_{k_0}},\operatorname{val}(\mathfrak g)(v)\leq\norm{v}_{T_{\mathrm{aux}}}\).
  Thus every scalarized ambient evaluation \(f\) satisfies \(\abs{f(\sum_k\lambda_kx_k)}\leq2BC\norm{v}_{T_{\mathrm{aux}}}\).
  Taking the supremum over the norming scalarized evaluations, we obtain
  \begin{equation}\label{eq:ris-dominated}
    \norm{\sum_k\lambda_kx_k}
    \leq2BC\norm{\sum_k\abs{\lambda_k}e_k}_{T_{\mathrm{aux}}}.
  \end{equation}
  Suppose a \(C\)-RIS has a subsequence \((x_{k_r})\), a functional \(x^*\), and \(\varepsilon>0\) with \(\abs{x^*(x_{k_r})}\geq\varepsilon\).
  Choose \(\theta_r\) with \(\abs{\theta_r}=1\) so that \(x^*(\theta_rx_{k_r})=\abs{x^*(x_{k_r})}\).
  Reindexing this subsequence gives a \(C\)-RIS \((\theta_rx_{k_r})_r\) with indices \((j_{k_r})_r\), since
  \[
    j_{k_{r+1}}\geq j_{k_r+1}>\max\ran x_{k_r}.
  \]
  For each \(j\), we obtain from \eqref{eq:ris-dominated} and Lemma~\ref{lem:quantitative-auxiliary}
  \[ \norm{\frac{1}{n_j}\sum_{r=1}^{n_j}\theta_rx_{k_r}} \leq2BC\norm{\frac{1}{n_j}\sum_{r=1}^{n_j}e_r}_{T_{\mathrm{aux}}} \leq\frac{4BC}{m_j}\longrightarrow0. \]
  On the other hand, we have
  \[ \varepsilon \leq\frac{1}{n_j}\sum_{r=1}^{n_j}\abs{x^*(x_{k_r})} =\abs{x^*\left(\frac{1}{n_j}\sum_{r=1}^{n_j} \theta_rx_{k_r}\right)} \leq\norm{x^*} \norm{\frac{1}{n_j}\sum_{r=1}^{n_j}\theta_rx_{k_r}} \longrightarrow0, \]
  a contradiction.
\end{proof}

\subsection{Shrinking and the compactness test}

We use the local-weight splitting method of \cite[Propositions~5.10--5.12]{ArgyrosHaydon2011}, separating a bounded block sequence according to its local levels.
Each fixed low-level part is a RIS, and a diagonal selection makes the remaining high-level parts a RIS as well.
This reduces both shrinking and compactness to the estimates already proved.

For \(0\ne x\in X^{\rm alg}\), put \(n=\max\ran x\) and write \(x=i_nu\) with \(u\in\mathcal E_{[1,n]}\).
Define \[\supp_{\rm loc}x=\{\eta:u_\eta\ne0\},\] and define \(\supp_{\rm loc}0=\varnothing\).
The next estimate is the block version of \cite[Lemma~5.8]{ArgyrosHaydon2011}.

\begin{lemma}[Local-level estimate]\label{lem:local-level}
  If \(f=e_{\gamma,\varphi}^*\) has level \(h>0\) and \(\ell(\eta)\ne h\) for every \(\eta\in\supp_{\rm loc}x\), then we have
  \begin{equation}\label{eq:local-level}
    \abs{f(x)}\leq\frac{\kappa}{m_h}\norm{x}.
  \end{equation}
\end{lemma}

\begin{proof}
  If \(x=0\), the conclusion is immediate.
  Assume that \(x\ne0\).
  Let \(n=\max\ran x\) and use the predecessor chain in the proof of Lemma~\ref{lem:truncation}.
  Put \(r=\max\{t:p_t\leq n\}\), with \(r=0\) if the set is empty.
  Every predecessor in the finite stage has level \(h\), so the local-support hypothesis gives \(e_r(x)=0\).
  Since \(P_{(n,\infty)}x=0\), we obtain from \eqref{eq:tail-restriction}
  \[
    f(x)=f(x)-e_r(x)=-((P_{(n,\infty)}^*-I)t_{r+1})(x).
  \]
  If \(r=a\), then we have \(t_{a+1}=0\) and \(f(x)=0\).
  Otherwise, only one packet term in the level-\(h\) evaluation remains, and we obtain from the boundary formula
  \[
    \begin{aligned}
      \abs{f(x)}
      =\frac{1}{m_h}\abs{b_{r+1,\varphi}^*
        P_{(p_r,p_{r+1})\cap(0,n]}x}\leq\frac{\norm{b_{r+1,\varphi}^*}}{m_h}
        \norm{P_{(p_r,p_{r+1})\cap(0,n]}}\norm{x}
      \leq\frac{\kappa}{m_h}\norm x.
    \end{aligned}
  \]
\end{proof}

\begin{lemma}[RIS reduction and compactness]\label{lem:compactness-test}
  The FDD of \(X\) is shrinking.
  Let \(Y\) be a Banach space.
  If \(R:X\to Y\) is bounded and \(\norm{Rx_k}\to0\) for every RIS, then \(R\) is compact.
\end{lemma}

\begin{proof}
  Let \((x_k)\) be a bounded block sequence, \(\norm{x_k}\leq C_x\), put \(s_k=\max\ran x_k\), and write \(x_k=i_{s_k}u_k\).
  For \(N\geq1\), split \(u_k=v_k^N+w_k^N\), where \(v_k^N\) contains the atoms of level at most \(N\) and \(w_k^N\) the atoms of level greater than \(N\).
  Put \(y_k^N=i_{s_k}v_k^N\) and \(z_k^N=i_{s_k}w_k^N\).
  These vectors remain in the block range of \(x_k\).
  Indeed, put \(a_k=\min\ran x_k-1\).
  Since \(P_{[1,a_k]}x_k=0\), the ambient coordinates of \(u_k\) up to rank \(a_k\) vanish.
  From coordinate deletion and the definition of the BD extension, we obtain
  \[ v_k^N|_{[1,a_k]}=w_k^N|_{[1,a_k]}=0,\qquad P_{[1,a_k]}y_k^N=P_{[1,a_k]}z_k^N=0. \]
  Both vectors belong to \(i_{s_k}\mathcal E_{[1,s_k]}\), so we have \(\ran y_k^N,\ran z_k^N\subseteq\ran x_k\), with zero vectors omitted when forming block sequences.
  If either vector has a smaller last FDD rank, its coordinates at that stage are the restriction of \(v_k^N\) or \(w_k^N\).
  Therefore its local support still contains only levels at most \(N\), or only levels greater than \(N\), respectively.
  The coordinate restriction of \(x_k=i_{s_k}u_k\) gives \(\norm{u_k}\leq\norm{x_k}\).
  Deleting coordinates according to their levels is contractive on the ambient finite sum, and \(\norm{i_{s_k}}\leq M\).
  Hence we obtain
  \begin{equation}\label{eq:level-splitting-bound}
    \max\{\norm{y_k^N},\norm{z_k^N}\}
    \leq M\norm{u_k}\leq MC_x.
  \end{equation}

  For fixed \(N\), choose RIS indices \(j_k>N\), increasing beyond the preceding block ranges.
  If \(N<h<j_k\), we obtain from Lemma~\ref{lem:local-level} \(\abs{e_{\gamma,\varphi}^*(y_k^N)}\leq\kappa\norm{y_k^N}/m_h\).
  If \(h\leq N\), we obtain from monotonicity and contractivity of the scalarized evaluation \(\abs{e_{\gamma,\varphi}^*(y_k^N)}\leq\norm{y_k^N} \leq m_N\norm{y_k^N}/m_h\).
  Therefore \((y_k^N)_k\) is a \(C_N\)-RIS, where \(C_N=MC_x\max\{1,\kappa,m_N\}\) is independent of \(k\).
  Hence we have \(\norm{Ry_k^N}\to0\) for each fixed \(N\).

  Choose recursively \(N_t\geq1\) and \(k_t\) so that
  \[ N_{t+1}>\max\{N_t,s_{k_t}\},\qquad k_{t+1}>k_t,\qquad \norm{Ry_{k_t}^{N_t}}<2^{-t}. \]
  This is possible because, after the next level cutoff has been fixed, \(\norm{Ry_k^N}\to0\) permits a sufficiently late choice of the next block index.
  Put \(j_t=N_t+1\).
  The local support of \(z_{k_t}^{N_t}\) has levels at least \(j_t\), and we have \(j_{t+1}=N_{t+1}+1>s_{k_t}\geq\max\ran z_{k_t}^{N_t}\).
  If \(0<h<j_t\), the local support contains no atom of level \(h\).
  From Lemma~\ref{lem:local-level} and \eqref{eq:level-splitting-bound}, we therefore obtain
  \[ \abs{e_{\gamma,\varphi}^*(z_{k_t}^{N_t})} \leq\frac{\kappa\norm{z_{k_t}^{N_t}}}{m_h} \leq\frac{\kappa MC_x}{m_h}. \]
  Since \(\norm{z_{k_t}^{N_t}}\leq MC_x\leq\kappa MC_x\), the sequence \((z_{k_t}^{N_t})_t\) is a \(\kappa MC_x\)-RIS.
  Hence we have \(\norm{Rz_{k_t}^{N_t}}\to0\), and hence we obtain \(\norm{Rx_{k_t}}\leq\norm{Ry_{k_t}^{N_t}}+\norm{Rz_{k_t}^{N_t}}\to0\).
  The same construction applies inside every subsequence.
  If \(\norm{Rx_k}\not\to0\), a subsequence bounded below away from zero would therefore have a further subsequence converging to zero, a contradiction.
  Thus \(R\) sends the original bounded block sequence to a sequence converging to zero in norm.

  Finally, suppose \(\norm{R(I-P_{[1,N]})}\not\to0\).
  Since \(\norm{I-P_{[1,N]}}\leq1+M\leq\kappa\), there are \(\eta>0\), arbitrarily large \(N\), and \(v_N\in B_X\) such that \(\norm{R(I-P_{[1,N]})v_N}>\eta\).
  Put \(u_N=(I-P_{[1,N]})v_N/ \norm{(I-P_{[1,N]})v_N}\).
  Then \(u_N\) is a unit vector in the tail after \(N\), and we have
  \[ \norm{Ru_N} =\frac{\norm{R(I-P_{[1,N]})v_N}} {\norm{(I-P_{[1,N]})v_N}} >\frac{\eta}{\kappa}. \]
  Put \(q_0=0\), and choose \(0<2\delta<\eta/\kappa\) and, successively, such unit tail vectors \(u_k\) after cutoffs \(N_k>q_{k-1}\).
  Since the FDD projections converge strongly to the identity and \(P_{[1,q_{k-1}]}u_k=0\), we may choose \(q_k>N_k\) so that
  \[
    \widetilde x_k=P_{(q_{k-1},q_k]}u_k=P_{[1,q_k]}u_k,
    \qquad
    \norm{R(u_k-\widetilde x_k)}
      \leq\norm R\,\norm{(I-P_{[1,q_k]})u_k}<\delta.
  \]
  The interval-projection bound gives
  \( \norm{\widetilde x_k}\leq\kappa. \)
  For its image under \(R\), we have
  \[ \norm{R\widetilde x_k}\geq\norm{Ru_k}-\norm{R(u_k-\widetilde x_k)}>\frac{\eta}{\kappa}-\delta>\delta. \]
  This contradicts the block-sequence conclusion.
  Therefore we have \(\norm{R(I-P_{[1,N]})}\to0\), and \(R\) is the norm limit of the finite-rank operators \(RP_{[1,N]}\), hence compact.
  Apply the same tail conclusion to \(R=x^*\in X^*\), using Corollary~\ref{cor:ris-weakly-null}.
  Then we have
  \[ \norm{x^*-P_{[1,N]}^*x^*} =\norm{x^*(I-P_{[1,N]})}\longrightarrow0,\qquad P_{[1,N]}^*x^*\in\sum_{n=1}^N P_{\{n\}}^*X^*. \]
  Thus we have \(X^*=\overline{\operatorname{span}}\{P_{\{n\}}^*X^*:n\geq1\}\), which proves that the FDD is shrinking.
\end{proof}

\section{\texorpdfstring{Inner and outer chains: collisions}{Inner and outer chains: collisions}}
\label{sec:inner-outer}

Recall that \(\widehat A=\widehat A_*^*\), where \(J_r^\infty:A_r^*\to\widehat A_*\) is the canonical isometric map from Lemma~\ref{lem:bounded-limit}.
The action on a fibre \(E_\eta\) is denoted by \(\Theta_\eta\), and \(U_\eta x\) is its ambient coordinate, as in \eqref{eq:ambient-coordinate}.
We use the multipliers \(M_a\) from \eqref{eq:multiplier-definition} and the inner-level coding map \(\sigma(q,\mathfrak p)\) from \eqref{eq:code-growth}.
For a represented packet, \(\norm{\cdot}_{\rm raw}\) is the sum of the norms of its represented entries, as in \eqref{eq:represented-packet}.

The exact-pair and dependent-sequence argument follows \cite[Sections~6 and~7]{ArgyrosHaydon2011}.
For rooted auxiliary averages, see \cite[Proposition~2.5]{ArgyrosHaydon2011} and compare \cite[Proposition~11.7]{MotakisPelczar2025}.
For off-weight estimates, see \cite[Definition~6.1 and the following remark]{ArgyrosHaydon2011} and \cite[Definition~6.7]{Motakis2024}.
The broader exact-pair methods and their subsequent developments are treated in \cite{ArgyrosTolias2004,Tarbard2013,Zisimopoulou2014,ManoussakisPelczarSwietek2017,MotakisPuglisiTolias2020}.
The proofs below give the versions needed for our parameters, interval restrictions, and root perturbations.
Here the orbit is \(\mathscr M(x)=\{M_ax:a\in\widehat A\}\).
A persistent separation from this orbit first gives a coded packet with small annihilation error.
The inner chain and its root perturbation make that error zero while retaining the operator separation (Lemmas~\ref{lem:polar} and~\ref{lem:exactification}).
Successive outer continuations then give the lower estimate in Lemma~\ref{lem:critical-extraction}, whereas Proposition~\ref{prop:dependent-mean} gives the upper estimate that proves Theorem~\ref{thm:local-orbit}.

\subsection{\texorpdfstring{The annihilation error}{The annihilation error}}

For a finite scalarized packet \(h=\sum_{l=1}^sc_le_{\eta_l,\psi_l}^*\) and \(x\in X^{\rm alg}\), we retain the fixed finite packet representation when writing \(\norm h_{\rm raw}\).
Combining entries at the same atom, we obtain
\[ \norm h_{\mathcal E_*} =\sum_\eta\left\|\sum_{\eta_l=\eta}c_l\psi_l\right\| \leq\sum_l\abs{c_l}\norm{\psi_l}=\norm h_{\rm raw}. \]
Thus the sum of the entry norms can exceed the norm of the represented element in \(\mathcal E_*\).
In a difference of packets we use a common finite presentation, adding zero entries if necessary.

Define the \emph{annihilation error} by
\begin{equation}\label{eq:error}
  \Delta(h,x)
  =\sum_{l=1}^sc_lJ_{r(\eta_l)}^\infty
  \bigl(\Theta_{\eta_l,*}(\psi_l)U_{\eta_l}x\bigr)
  \in\widehat A_*.
\end{equation}
Here, for \(\psi\in E_\eta^*\) and \(u\in E_\eta\), the coefficient functional \(\Theta_{\eta,*}(\psi)u\in A_{r(\eta)}^*\) is determined by
\begin{equation}\label{eq:preadjoint-product}
  \langle b,\Theta_{\eta,*}(\psi)u\rangle
  =\langle\psi,\Theta_\eta(b)u\rangle
  \qquad(b\in A_{r(\eta)}).
\end{equation}
It satisfies \(\norm{\Theta_{\eta,*}(\psi)u}\leq\norm{\psi}\norm{u}\).

\begin{lemma}[Error identity and continuity]\label{lem:error}
  For every \(a\in\widehat A\), we have
  \begin{equation}\label{eq:error-identity}
    h(M_ax)=\langle a,\Delta(h,x)\rangle.
  \end{equation}
  Consequently, we have
  \begin{equation}\label{eq:error-annihilator}
    h|_{\{M_ax:a\in\widehat A\}}=0
    \quad\Longleftrightarrow\quad
    \Delta(h,x)=0.
  \end{equation}
  If \(h,k\) lie in one finite sum of ambient predual blocks, then we have
  \begin{equation}\label{eq:error-continuity}
    \norm{\Delta(h,x)-\Delta(k,y)}
    \leq\norm{h-k}_{\rm raw}\norm{x}
    +\norm{k}_{\rm raw}\norm{x-y}.
  \end{equation}
\end{lemma}

\begin{proof}
  By \eqref{eq:scalar-chain} and \eqref{eq:preadjoint-product}, we obtain
  \[ h(M_ax) = \sum_{l=1}^s c_l\langle \psi_l,\Theta_{\eta_l}(a_{r(\eta_l)})U_{\eta_l} x\rangle=\sum_{l=1}^sc_l\left\langle a_{r(\eta_l)}, \Theta_{\eta_l,*}(\psi_l)U_{\eta_l}x\right\rangle =\langle a,\Delta(h,x)\rangle. \]
  Lemma~\ref{lem:bounded-limit} gives \(\widehat A=\widehat A_*^*\), so \eqref{eq:error-annihilator} follows.
  Moreover, we have
  \begin{equation}\label{eq:error-norm-formula}
    \norm{\Delta(h,x)}
    =\sup_{a\in B_{\widehat A}}\abs{h(M_ax)}
    \leq\norm{h}_{\rm raw}\norm{x}.
  \end{equation}
  In particular, \(\Delta(h,x)\) is independent of the representation of the ambient functional \(h\).
  By bilinearity, we have \(\Delta(h,x)-\Delta(k,y)=\Delta(h-k,x)+\Delta(k,x-y)\), and \eqref{eq:error-norm-formula} gives \eqref{eq:error-continuity}.
\end{proof}

For a finite vector \(x\), put \(R_x=\max(\{0\}\cup\{r(\gamma):D_\gamma x\ne0\})\).
Then we have
\begin{equation}\label{eq:finite-orbit}
  \mathscr M(x)=\{M_bx:b\in A_{R_x}\},
\end{equation}
where any compatible lift of \(b\) is used on the right.
This is well defined because all lower coefficients are determined by \(b\). The map \(b\mapsto M_bx\) is linear, so \(\mathscr M(x)\) is a finite-dimensional, hence closed, linear subspace of \(X\).
Lemma~\ref{lem:bounded-limit} gives \(q_{R_x}(B_{\widehat A})=B_{A_{R_x}}\), so \(\{M_ax:a\in B_{\widehat A}\}\) is compact.

\begin{lemma}[Coded approximate annihilator]\label{lem:polar}
  Let \(S\in\Bcal(X)\), let \((x_k)\) be a RIS, and assume that
  \begin{equation}\label{eq:persistent-separation}
    \dist(Sx_k,\mathscr M(x_k))\geq\varepsilon
    \qquad(k\in\N).
  \end{equation}
  Given a cutoff \(a\), \(\delta>0\), and positive \((\delta_k)\), there are a later \(x_k\), a coded block vector \(u\) with the same FDD range, and a coded packet \(b\), such that
  \begin{equation}\label{eq:polar-output}
    \begin{gathered}
      \min\ran u>a,\qquad \supp_{\rm raw}b\subseteq(a,\infty),
      \qquad \norm{u-x_k}<\delta_k,\\
      \norm{b}_{\rm raw}\leq1,\qquad
      \operatorname{Re}b(Su)>\frac{\varepsilon}{2},
      \qquad \norm{\Delta(b,u)}<\delta.
    \end{gathered}
  \end{equation}
  Moreover, \(u\) may be chosen so that \(Au\) and \((\abs{\lambda_l(u)})_l\)  are arbitrarily small, for any finite-rank \(A\) and any finitely many functionals \((\lambda_l)_l\).
\end{lemma}

\begin{proof}
  By Corollary~\ref{cor:ris-weakly-null}, we know that \((x_k)\) is weakly null, and so is \((Sx_k)\).
  Choose \(k\) and \(r>\max\ran x_k\) such that
  \[ \min\ran x_k>a,\qquad \norm{P_{[1,a]}Sx_k}+\norm{P_{(r,\infty)}Sx_k} <\frac{\varepsilon}{8}. \]
  Put \(z=P_{(a,r]}Sx_k\).
  Since multipliers preserve FDD supports, we have \(\mathscr M(x_k)\subseteq P_{(a,r]}X\).
  The truncation error satisfies
  \[ \norm{Sx_k-z}<\frac{\varepsilon}{8}. \]
  For the distance to the multiplier orbit, we then obtain
  \( \dist(z,\mathscr M(x_k))\geq\varepsilon-\norm{Sx_k-z}>7\varepsilon/8. \)
  By the Hahn--Banach theorem, we can choose \(x^*\in B_{X^*}\) such that
  \[ x^*|_{\mathscr M(x_k)}=0,\qquad \operatorname{Re}x^*(z)>\frac{3\varepsilon}{4}. \]

  Scalarized ambient coordinates norm \(X\).
  By the bipolar theorem, we obtain a bounded net of finite scalarized packets \(b_\alpha\), with sum of entry norms at most one, converging weak-star to \(x^*\).
  The set \(K_k=\{M_ax_k:a\in B_{\widehat A}\}\) is compact by \eqref{eq:finite-orbit}.
  To prove uniform convergence on \(K_k\), fix \(\eta>0\) and choose a finite \(\eta/4\)-net \(z_1,\ldots,z_d\) in \(K_k\).
  Since \(\norm{b_\alpha}\leq\norm{b_\alpha}_{\rm raw}\leq1\) and \(\norm{x^*}\leq1\), we have \(\norm{b_\alpha-x^*}\leq2\).
  Weak-star convergence gives an index after which \(\abs{(b_\alpha-x^*)(z_l)}<\eta/2\) for every \(l\).
  For \(z\in K_k\), choose \(l\) with \(\norm{z-z_l}<\eta/4\).
  Then we have
  \[ \abs{(b_\alpha-x^*)(z)} \leq\abs{(b_\alpha-x^*)(z_l)} +2\norm{z-z_l}<\frac{\eta}{2}+2\frac{\eta}{4}=\eta. \]
  Thus we have \(b_\alpha\to x^*\) uniformly on \(K_k\).
  Since \(x^*\) annihilates \(K_k\), we obtain from \eqref{eq:error-norm-formula} \(\norm{\Delta(b_\alpha,x_k)}\to0\).

  If an ambient coordinate has rank at most \(a\), then it belongs to the span of the first \(a\) dual FDD blocks and hence vanishes on \(P_{(a,r]}X\).
  Since multipliers preserve FDD supports, deleting all such entries from a packet changes neither its value on \(z\) nor any of its values on \(K_k\).
  In view of \eqref{eq:error-norm-formula}, we may therefore choose a finite packet \(b^0\), after this deletion, such that
  \[ \supp_{\rm raw}b^0\subseteq(a,\infty),\qquad \norm{b^0}_{\rm raw}\leq1,\qquad \operatorname{Re}b^0(z)>\frac{5\varepsilon}{8},\qquad \norm{\Delta(b^0,x_k)}<\frac{\delta}{4}. \]
  Therefore we obtain
  \( \operatorname{Re}b^0(Sx_k) \geq\operatorname{Re}b^0(z)-\norm{Sx_k-z} >5\varepsilon/8-\varepsilon/8 =\varepsilon/2.\)
  Choose \(0<\theta<1\) sufficiently close to one and put \(b^1=\theta b^0\).
  The strict inequalities above allow the choice so that
  \[ \norm{b^1}_{\rm raw}<1, \qquad \operatorname{Re}b^1(Sx_k)>\frac{\varepsilon}{2}, \qquad \norm{\Delta(b^1,x_k)}<\frac{\delta}{2}. \]
  Put \(H=\norm{x_k}+1\) and \(\mu=\operatorname{Re}b^1(Sx_k)-\varepsilon/2>0\).
  Choose
  \[
    0<\rho<\min\left\{
      1,\delta_k,1-\norm{b^1}_{\rm raw},
      \frac{\mu}{(1+\norm S)(H+1)},
      \frac{\delta}{2(H+1)}\right\}.
  \]
  Density in the fixed finite coded spaces gives \(u\) with the same FDD range as \(x_k\), and \(b\) supported in the finite packet block of \(b^1\), such that
  \[
    \norm{u-x_k}<\rho,\qquad
    \norm{b-b^1}_{\rm raw}<\rho,\qquad \norm u\leq H.
  \]
  For the change in \(b(Su)\), we obtain
  \begin{align*}
    \abs{b(Su)-b^1(Sx_k)}
    \leq\norm{b-b^1}_{\rm raw}\norm S\norm u
      +\norm{b^1}_{\rm raw}\norm S\norm{u-x_k}<\rho\norm S(H+1)<\mu.
  \end{align*}
  For the change in \(\Delta(b,u)\), we obtain from \eqref{eq:error-continuity} that
  \begin{align*}
    \norm{\Delta(b,u)-\Delta(b^1,x_k)}
    \leq\norm{b-b^1}_{\rm raw}\norm u
      +\norm{b^1}_{\rm raw}\norm{u-x_k}<\rho(H+1)<\frac{\delta}{2}.
  \end{align*}
  Consequently, we have
  \( \norm b_{\rm raw}\leq\norm{b^1}_{\rm raw}+\rho<1. \)
  For the value of \(b\) at \(Su\), we obtain
  \( \operatorname{Re}b(Su)>\operatorname{Re}b^1(Sx_k)-\mu=\varepsilon/2. \)
  The annihilation error satisfies
  \( \norm{\Delta(b,u)}<\norm{\Delta(b^1,x_k)}+\delta/2<\delta. \)
  This gives \eqref{eq:polar-output}.

  If \(A:X\to F\) has finite-dimensional range and \(\lambda_1,\ldots,\lambda_t\in X^*\), fix positive target bounds \(\beta_A,\beta_1,\ldots,\beta_t\).
  First pass to a tail on which \(\norm{Ax_k}<\beta_A/2\) and \(\abs{\lambda_l(x_k)}<\beta_l/2\) for \(1\leq l\leq t\).
  Replace \(\delta_k\) by
  \[ \delta'_k=\min\left\{\delta_k, \frac{\beta_A}{2\max\{1,\norm A\}}, \min_{1\leq l\leq t} \frac{\beta_l}{2\max\{1,\norm{\lambda_l}\}}\right\}. \]
  Then we obtain \(\norm{Au}\leq\norm{Ax_k}+\norm A\,\delta'_k<\beta_A\), and similarly \(\abs{\lambda_l(u)}<\beta_l\) for \(1\leq l\leq t\).
\end{proof}

\subsection{\texorpdfstring{Root perturbation}{Root perturbation}}

Let \(b_1,\ldots,b_N\), \(N=n_p\), be successive coded packets represented along a complete level-\(p\) inner chain.
Let \(\xi_1\prec\cdots\prec\xi_N=\eta\) be its core atoms, put \(\nu_t=\rank\xi_t\), and put \(\nu_0=\chi(\eta)\).
Suppose that the coded vectors \(u_1<\cdots<u_N\) satisfy
\[
  \nu_{t-1}<\min\ran u_t,\qquad
  \max\bigl(\ran u_t\cup\supp_{\rm raw}b_t\bigr)<\nu_t
  \quad(1\leq t\leq N).
\]
Put \(f=e_{\eta,1_{A_{r(\eta)}}}^*\) and \(y^0=(m_p/N)\sum_{t=1}^Nu_t\).

\begin{lemma}[Root perturbation]\label{lem:root}
  With the preceding notation, we have
  \begin{equation}\label{eq:error-average}
    J_{r(\eta)}^\infty U_\eta y^0
    =\Delta(f,y^0)
    =\frac{1}{N}\sum_{t=1}^N\Delta(b_t,u_t).
  \end{equation}
  If \(y=y^0-d_\eta(U_\eta y^0)\), then we have
  \[
    U_\eta y=0,\qquad \norm{y-y^0}\leq M\norm{\Delta(f,y^0)}.
  \]
  Every complete scalarization of the matching outer strong packet annihilates \(y\).
\end{lemma}

\begin{proof}
  Since \(\eta\) is a core atom, we have \(E_\eta=A_{r(\eta)}^*\), and for \(b\in A_{r(\eta)}\) and \(u\in A_{r(\eta)}^*\), we obtain
  \[ \left\langle b,\Theta_{\eta,*}(1_{A_{r(\eta)}})u\right\rangle =\left\langle1_{A_{r(\eta)}},R_b^*u\right\rangle =\left\langle R_b1_{A_{r(\eta)}},u\right\rangle =\langle b,u\rangle. \]
  Hence we have \(\Theta_{\eta,*}(1_{A_{r(\eta)}})u=u\).
  Thus the first equality in \eqref{eq:error-average} follows from \eqref{eq:error}.
  By \eqref{eq:represented-packet-bounds}, we obtain the terminal analysis
  \[
    f=\sum_{t=1}^Nd_{\xi_t,1}^*
    +\frac{1}{m_p}\sum_{t=1}^N
    b_tP_{(\nu_{t-1},\infty)}.
  \]
  Fix \(a\in\widehat A\).
  Multipliers preserve FDD supports.
  Since \(\ran M_au_t\subseteq\ran u_t\subseteq(\nu_{t-1},\nu_t)\), we obtain from the placement conditions
  \[
    d_{\xi_s,1}^*(M_au_t)=0,\qquad
    b_sP_{(\nu_{s-1},\infty)}M_au_t
      =\begin{cases}b_t(M_au_t),&s=t,\\0,&s\ne t.\end{cases}
  \]
  For \(s>t\), the tail projection vanishes, whereas for \(s<t\), every ambient coordinate in \(b_s\) has rank below \(\min\ran M_au_t\) whenever \(M_au_t\ne0\).
  When we apply the terminal analysis to \(M_ay^0\), the resulting double sum reduces to its diagonal, and we obtain
  \[
    \begin{aligned}
      \langle a,J_{r(\eta)}^\infty U_\eta y^0\rangle
      =f(M_ay^0)
      &=\frac{1}{N}\sum_{s,t=1}^N b_sP_{(\nu_{s-1},\infty)}M_au_t\\
      &=\frac{1}{N}\sum_{t=1}^N b_t(M_au_t)
      =\left\langle a,\frac{1}{N}\sum_{t=1}^N\Delta(b_t,u_t)\right\rangle.
    \end{aligned}
  \]
  Since \(\widehat A=\widehat A_*^*\), this proves the second equality in \eqref{eq:error-average}.

  By biorthogonality, we have \(U_\eta d_\eta=I_{E_\eta}\), and hence \(U_\eta y=0\).
  The map \(J_{r(\eta)}^\infty\) is an isometry, and the block-injection norm is at most \(M\).
  Therefore we obtain
  \[ \norm{y-y^0} \leq M\norm{U_\eta y^0} =M\norm{\Delta(f,y^0)}. \]

  Let the complete matching outer strong packet be scalarized by \(\varphi\in A_r\), \(r\geq r(\eta)\), and let \(k<\min\ran y\) be the local cutoff of its outer atom.
  For the core action, the packet formula gives the following identity.
  If \(c=\pi_{r(\eta)}^r\varphi\) and \(u\in A_{r(\eta)}^*\), then we have
  \[ \left\langle u,(R_c^*)^*1_{A_{r(\eta)}}\right\rangle =\left\langle R_c^*u,1_{A_{r(\eta)}}\right\rangle =\left\langle u,R_c1_{A_{r(\eta)}}\right\rangle =\langle u,c\rangle. \]
  Therefore we have \(\Theta_\eta(\pi_{r(\eta)}^r\varphi)^*1_{A_{r(\eta)}} =\pi_{r(\eta)}^r\varphi\).
  Hence the scalarization is \(e_{\eta,\pi_{r(\eta)}^r\varphi}^*P_{(k,\infty)}\), and its value on \(y\) is \(\langle\pi_{r(\eta)}^r\varphi,U_\eta y\rangle=0\).
\end{proof}
\begin{figure}[!htbp]
  \centering
\begin{tikzpicture}[x=1cm,y=1cm,>=Stealth,
  every node/.style={font=\small,align=center,inner sep=2pt},
  oldblock/.style={draw=blue!55!black,fill=blue!7,line width=.45pt},
  correction/.style={draw=green!40!black,fill=green!10,line width=.55pt},
  ambient/.style={draw=blue!45!black,fill=blue!4,
    rounded corners=2pt,minimum width=2.75cm,minimum height=.72cm,
    inner sep=3pt}]
  \path[use as bounding box] (0,-3.45) rectangle (11.8,1.02);

  \node[font=\small\bfseries] at (3.70,.82) {FDD blocks};
  \node[font=\small\bfseries] at (10.05,.82)
    {Ambient coordinate at \(\eta\)};

  \draw[densely dashed,black!35] (6.05,.50) -- (6.05,-1.64);
  \foreach \yy in {0,-1.30} {
    \draw[->,black!60] (1.02,\yy) -- (7.05,\yy);
    \foreach \a/\b in {1.24/2.10,2.55/3.40,4.18/4.92} {
      \draw[oldblock] (\a,{\yy-.20}) rectangle (\b,{\yy+.20});
    }
    \node[fill=white,inner sep=1pt] at (3.78,\yy) {\(\cdots\)};
  }
  \node[anchor=east] at (.89,0) {\(y^0\)};
  \node[anchor=east] at (.89,-1.30) {\(y\)};
  \node[fill=white,inner sep=3pt] at (6.05,0) {\(0\)};
  \node[correction,minimum width=1.35cm,minimum height=.46cm,
    inner sep=2pt] at (6.05,-1.30) {\(-d_\eta q\)};
  \node[font=\scriptsize] at (6.05,-1.84) {\(\rank\eta\)};
  \node[font=\scriptsize,anchor=north east] at (7.10,-1.43) {rank};
  \node[font=\scriptsize,text width=4.50cm] at (3.08,-1.90)
    {Earlier FDD blocks unchanged};

  \node[ambient] (before) at (10.05,0)
    {\(q=U_\eta y^0\)\\[-1pt]
     {\scriptsize possibly nonzero}};
  \node[ambient,draw=green!40!black,fill=green!10] (after)
    at (10.05,-1.30) {\(U_\eta y=q-q=0\)};
  \draw[->] (7.35,0) -- node[above,font=\scriptsize] {\(U_\eta\)}
    (before.west);
  \draw[->] (7.35,-1.30) --
    node[above,font=\scriptsize] {\(U_\eta\)} (after.west);
  \draw[->] (before.south) --
    node[right,font=\scriptsize,inner sep=3pt] {subtract \(q\)}
    (after.north);

  \node at (5.90,-2.38)
    {\(y=y^0-d_\eta q,\qquad U_\eta d_\eta q=q\)};
  \node at (5.90,-2.84)
    {\(\norm{y-y^0}=\norm{d_\eta q}
       \leq M\norm{q}=M\norm{\Delta(f,y^0)}\)};
  \node[font=\scriptsize] at (5.90,-3.28)
    {Complete matching strong packet:
     \(\langle\psi,U_\eta y\rangle=0\) for every scalarization.};
\end{tikzpicture}
  \caption{A correction in the terminal FDD block cancels the ambient coordinate at \(\eta\).}
  \label{fig:root-coordinate-cancellation}
\end{figure}
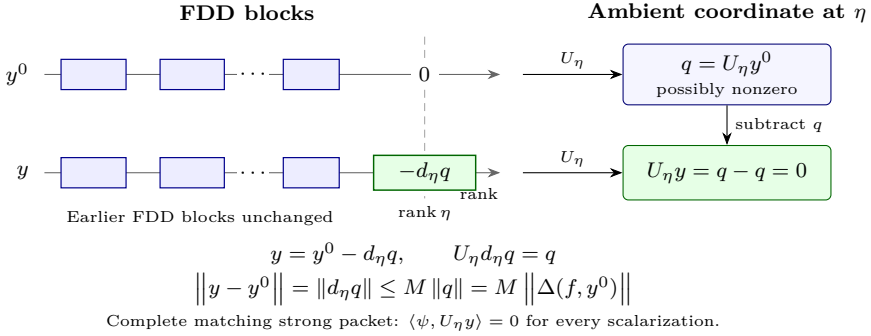
Although \(y^0\) is supported before \(\eta\) in the FDD, its ambient coordinate \(q=U_\eta y^0\) may be nonzero.
As Figure~\ref{fig:root-coordinate-cancellation} shows, subtracting \(d_\eta q\) leaves the earlier FDD blocks unchanged and gives \(U_\eta y=q-q=0\).
We therefore obtain exact cancellation for every scalarization of the complete matching outer strong packet.
By \eqref{eq:error-average}, the correction is small whenever the packet errors are small.

Recall from \eqref{eq:constants} that \(C_{\rm ex}=2(3B+M)\).
The next selection adapts the operator-dependent exact-pair argument in \cite[Lemma~7.2]{ArgyrosHaydon2011} to multiplier orbits.

\begin{lemma}[Operator-dependent pairs]
  \label{lem:exactification}
  Let \(S\in\Bcal(X)\), let \((x_k)\) be a \(C_0\)-RIS, and assume
  \[ \dist(Sx_k,\mathscr M(x_k))\geq\varepsilon \qquad(k\in\N). \]
  Let \(q\in\mathcal C^{\rm out}\), let \(\mathfrak p\) be a realized outer stem over a seed of level \(q\), with age below \(n_q\), and put \(p=\sigma(q,\mathfrak p)\).
  After every cutoff there are a coded root-perturbed vector \(y\) and the terminal evaluation \(f=e_{\eta,1}^*\) of a completed inner chain such that
  \begin{align}
    U_\eta y=0,&\qquad 
    \operatorname{Re}f(Sy)\geq\frac{\varepsilon}{4},              \label{eq:exact-separation}\\
    \norm{y}\leq C_{\rm ex}C_0,&\qquad
    \sup_{\gamma,\varphi}\abs{d_{\gamma,\varphi}^*(y)}
    \leq\frac{C_{\rm ex}C_0}{m_p}.                             \label{eq:exact-coordinate}
  \end{align}
  The pair gives a coded certificate compatible with \((\mathfrak p,\eta)\) for an outer continuation.
  For a fixed stem \(\mathfrak p\), these pairs may be chosen successively so that the sequence of vectors \((y^{(j)})\) is bounded and weakly null.
\end{lemma}

\begin{proof}
  Pass to a tail with RIS indices \(j_k\geq2\), and put \(C=2C_0\).
  Perturbations \(u_k\) with the same FDD range and \(\norm{u_k-x_k}\leq C_0/m_{j_k-1}\) form a \(C\)-RIS.
  Indeed, their norms satisfy
  \[ \norm{u_k}\leq C_0+\frac{C_0}{m_{j_k-1}}\leq C. \]
  For the evaluations of level \(0<h<j_k\), we have
  \[ \abs{e_{\gamma,\varphi}^*(u_k)}\leq\frac{C_0}{m_h}+\frac{C_0}{m_{j_k-1}}\leq\frac{2C_0}{m_h}=\frac{C}{m_h}. \]
  Equality of the FDD ranges preserves the block order and the RIS range condition.

  Choose
  \begin{equation}\label{eq:exactification-tau}
    0<\tau<\min\left\{
    \frac{C}{m_p},
    \frac{\varepsilon}{16M(1+\norm{S})}\right\}.
  \end{equation}
  Write \(N=n_p\), \(m=m_p\), and put
  \[
    \alpha_d=\frac{\varepsilon}{32mN},\qquad
    \alpha_\times=\frac{\varepsilon}{16N},\qquad
    \alpha_0=\frac{\varepsilon}{32},\qquad
    \delta_{\rm tail}=\kappa^{-1}
    \min\{\alpha_d,\alpha_\times\}=\frac{\varepsilon}{32\kappa mN}.
  \]
  Put \(\nu_0=\operatorname{cut}(\mathfrak p)\).

  Suppose \(u_s,b_s,\xi_s\), \(s<t\), have been chosen, and put \(\nu_s=\rank\xi_s\).
  Apply Lemma~\ref{lem:polar} to the finite-dimensional operator \(P_{[1,\nu_{t-1}]}S\) and the functionals \(d_{\xi_s,1}^*S\) and \(b_sP_{(\nu_{s-1},\infty)}S\), \(s<t\).
  With perturbation radius at most \(C_0/m_{j_k-1}\), it gives a coded \(u_t\) and a coded packet \(b_t\) such that
  \begin{align}
    \abs{d_{\xi_s,1}^*(Su_t)}&<\alpha_d\quad (s<t),                    \label{eq:inner-old-d}\\
    \abs{b_sP_{(\nu_{s-1},\infty)}Su_t}&<\alpha_\times\quad (s<t),    \label{eq:inner-old-packet}\\
    \norm{P_{[1,\nu_{t-1}]}Su_t}&<\alpha_0,                       \label{eq:inner-head}\\
    \norm{\Delta(b_t,u_t)}&<\tau,\qquad 
    \operatorname{Re}b_t(Su_t)>\frac{\varepsilon}{2}.           \label{eq:inner-diagonal}
  \end{align}
  Choose \(r_t\) beyond \(\supp_{\rm raw}b_t\cup\ran u_t\) so that
  \[
    \norm{P_{(r_t,\infty)}Su_s}<\delta_{\rm tail}
    \qquad(1\leq s\leq t).
  \]
  Proposition~\ref{prop:finite-extension}\textup{(b)} then selects, after \(r_t\), the next occurrence \(\xi_t\) in the same level-\(p\) inner chain, representing exactly \(b_t\).
  Only after \(\xi_t\) is fixed is \(u_{t+1}\) chosen.

  The terminal evaluation has the analysis
  \( f=\sum_{s=1}^Nd_{\xi_s,1}^* +m^{-1}\sum_{s=1}^N b_sP_{(\nu_{s-1},\infty)}. \)
  For \(y^0=(m/N)\sum_tu_t\), we obtain from Lemma~\ref{lem:root}
  \begin{equation}\label{eq:inner-error-average}
    \norm{\Delta(f,y^0)}
    =\norm{\frac{1}{N}\sum_{t=1}^N\Delta(b_t,u_t)}
    \leq\frac{1}{N}\sum_{t=1}^N\norm{\Delta(b_t,u_t)}
    <\tau.
  \end{equation}
  Expanding the terminal evaluation at \(Sy^0\), we obtain
  \begin{align}
    f(Sy^0)-\frac{1}{N}\sum_tb_t(Su_t)
    ={}&\frac{m}{N}\sum_{s,t}d_{\xi_s,1}^*(Su_t)
    +\frac{1}{N}\sum_{s<t}b_sP_{(\nu_{s-1},\infty)}Su_t
    \notag\\
    &+\frac{1}{N}\sum_{s>t}b_sP_{(\nu_{s-1},\infty)}Su_t
    +\frac{1}{N}\sum_t b_t(P_{(\nu_{t-1},\infty)}-I)Su_t.
    \label{eq:inner-four-errors}
  \end{align}

  For \(s<t\), use \eqref{eq:inner-old-d}.
  For \(s\geq t\), we have \(\nu_s>r_t\), and hence
  \[
    \abs{d_{\xi_s,1}^*(Su_t)}
      =\abs{d_{\xi_s,1}^*P_{(r_t,\infty)}Su_t}
      \leq\kappa\norm{P_{(r_t,\infty)}Su_t}
      <\kappa\delta_{\rm tail}\leq\alpha_d.
  \]
  There are \(N(N-1)/2\) pairs with \(s<t\) and \(N(N+1)/2\) with \(s\geq t\).
  Therefore we obtain
  \begin{equation}\label{eq:inner-d-error}
    \abs{\frac{m}{N}\sum_{s,t}d_{\xi_s,1}^*(Su_t)}
    <\frac{m}{N}\left(
    \frac{N(N-1)}{2}\alpha_d+
    \frac{N(N+1)}{2}\kappa\delta_{\rm tail}\right)
    \leq mN\alpha_d=\frac{\varepsilon}{32}.
  \end{equation}
  Similarly, we obtain from \eqref{eq:inner-old-packet}
  \begin{equation}\label{eq:inner-old-packet-error}
    \abs{\frac{1}{N}\sum_{s<t}b_sP_{(\nu_{s-1},\infty)}Su_t}
    <\frac{N-1}{2}\alpha_\times<\frac{\varepsilon}{32}.
  \end{equation}
  For \(s>t\), we have \(\nu_{s-1}>r_t\), so nested tail projections give
  \[
    \begin{aligned}
      \abs{b_sP_{(\nu_{s-1},\infty)}Su_t}
      &=\abs{b_sP_{(\nu_{s-1},\infty)}P_{(r_t,\infty)}Su_t}\\
      &\leq\kappa\norm{P_{(r_t,\infty)}Su_t}
      <\kappa\delta_{\rm tail}\leq\alpha_\times.
    \end{aligned}
  \]
  Summing over these \(N(N-1)/2\) pairs, we obtain
  \begin{equation}\label{eq:inner-new-packet-error}
    \abs{\frac{1}{N}\sum_{s>t}b_sP_{(\nu_{s-1},\infty)}Su_t}
    <\frac{N-1}{2}\kappa\delta_{\rm tail}
    \leq\frac{N-1}{2}\alpha_\times<\frac{\varepsilon}{32}.
  \end{equation}
  Finally, we have \(P_{(\nu_{t-1},\infty)}-I=-P_{[1,\nu_{t-1}]}\), so we obtain
  \begin{equation}\label{eq:inner-head-error}
    \abs{\frac{1}{N}\sum_t b_t(P_{(\nu_{t-1},\infty)}-I)Su_t}
    \leq\frac{1}{N}\sum_{t=1}^N\norm{P_{[1,\nu_{t-1}]}Su_t}
    <\alpha_0=\frac{\varepsilon}{32}.
  \end{equation}
  From \eqref{eq:inner-d-error}--\eqref{eq:inner-head-error}, we obtain
  \begin{equation}\label{eq:inner-output}
    \abs{f(Sy^0)-\frac{1}{N}\sum_tb_t(Su_t)}
    \leq\frac{\varepsilon}{32}+\frac{\varepsilon}{32}
    +\frac{\varepsilon}{32}+\frac{\varepsilon}{32}
    =\frac{\varepsilon}{8}.
  \end{equation}
  By \eqref{eq:inner-diagonal}, we obtain
  \[
    \begin{aligned}
      \operatorname{Re}f(Sy^0)
      \geq\frac{1}{N}\sum_{t=1}^N\operatorname{Re}b_t(Su_t)
        -\abs{f(Sy^0)-\frac{1}{N}\sum_tb_t(Su_t)}>\frac{\varepsilon}{2}-\frac{\varepsilon}{8}
        =\frac{3\varepsilon}{8}.
    \end{aligned}
  \]

  Put \(y=y^0-d_\eta(U_\eta y^0)\).
  Lemma~\ref{lem:root}, \eqref{eq:inner-error-average}, and \eqref{eq:exactification-tau} give \(U_\eta y=0\).
  Moreover, we have
  \[ \abs{f(S(y-y^0))} \leq M\norm S\norm{\Delta(f,y^0)} <M\norm S\tau<\frac{\varepsilon}{16}. \]
  Hence we obtain
  \[ \operatorname{Re}f(Sy) \geq\operatorname{Re}f(Sy^0)-\abs{f(S(y-y^0))} >\frac{3\varepsilon}{8}-\frac{\varepsilon}{16} =\frac{5\varepsilon}{16} >\frac{\varepsilon}{4}. \]

  From Corollary~\ref{lem:ris-consequences} and the correction estimate, we obtain
  \[
    \norm y\leq\norm{y^0}+\norm{y-y^0}
      \leq3BC+M\tau\leq C_{\rm ex}C_0.
  \]
  The correction has only the BD coefficient at \(\eta\).
  Thus, using biorthogonality, \eqref{eq:inner-error-average}, and \eqref{eq:exactification-tau}, we obtain
  \[
    \sup_{\gamma,\varphi}\abs{d_{\gamma,\varphi}^*(y)}
    \leq\frac{\kappa C}{m_p}+\norm{U_\eta y^0}
    <\frac{\kappa C}{m_p}+\tau
    \leq\frac{C_{\rm ex}C_0}{m_p}.
  \]
  This proves \eqref{eq:exact-separation}--\eqref{eq:exact-coordinate}.

  The vectors and functionals constructed above lie in one finite coded structure.
  The finite-dimensional space \eqref{eq:finite-orbit} has a coded spanning matrix, and Lemma~\ref{lem:error} gives \(f|_{\mathscr M(y)}=0\).
  The proof gives \(\operatorname{Re}f(Sy)>5\varepsilon/16\).
  The inner analysis of \(f\) is supported after \(\nu_0=\operatorname{cut}(\mathfrak p)\), so \(f(Sy)=f(P_{(\nu_0,\infty)}Sy)\).
  Choose a coded finite vector \(z\), supported after \(\nu_0\), such that
  \[ \norm{z-P_{(\nu_0,\infty)}Sy}<\frac{\varepsilon}{16}. \]
  Since \(\norm f\leq1\), we have
  \[ \operatorname{Re}f(z) \geq\operatorname{Re}f(Sy) -\norm{z-P_{(\nu_0,\infty)}Sy} >\frac{\varepsilon}{4}. \]
  Choose \(\varepsilon'\in\Q_+\) with \(0<\varepsilon'<\operatorname{Re}f(z)\), a finite integer interval \(I\subseteq(\nu_0,\infty)\) containing \(\supp_{\rm FDD}y\cup\supp_{\rm FDD}z\cup \supp_{\rm FDD}\mathscr M(y)\cup\{\rank\eta\}\), and an integer \(R\) satisfying \textup{(C2)}.
  Since \(f=f_\eta\), the record \((y,z,I,R,\mathscr M(y),\varepsilon')\) is a coded certificate compatible with \((\mathfrak p,\eta)\).
  Proposition~\ref{prop:finite-extension}\textup{(c)} realizes its outer continuation.
  Fix the stem \(\mathfrak p\), and hence the inner level \(p\).
  Repeating the construction after successive cutoffs gives pairs \((y^{(j)},f^{(j)})\) with
  \[ \max\ran y^{(j)}<\min\ran y^{(j+1)},\qquad \sup_j\norm{y^{(j)}}\leq C_{\rm ex}C_0. \]
  From the shrinking conclusion of Lemma~\ref{lem:compactness-test}, we now obtain, for every \(x^*\in X^*\),
  \[ x^*(y^{(j)})\longrightarrow0,\qquad x^*(Sy^{(j)})=(S^*x^*)(y^{(j)})\longrightarrow0. \]
  In particular, every finite-rank operator sends \((y^{(j)})\) to a norm-null sequence.
\end{proof}

\subsection{\texorpdfstring{Collisions and dependent means}{Collisions and dependent means}}

The following definition adapts \cite[Definition~6.3]{ArgyrosHaydon2011} to the stored outer histories.

\begin{definition}[Dependent sequence]\label{def:zero-dependent}
  Let \(q\) be an outer seed level and let \(C\geq1\).
  A \(C\)-bounded \(q\)-dependent sequence is a successive family \((y_i,f_i)_{i=1}^{n_q}\) with realized outer stems \((\mathfrak p_i)_{i=0}^{n_q}\) over one fixed seed.
  Here \(\mathfrak p_0\) is the age-zero stem, and \(\mathfrak p_i\) is the outer-age \(i\) continuation of \(\mathfrak p_{i-1}\) whose edge uses the completed inner terminal \(\eta_i\) and whose certificate contains \(y_i\).
  Put \(\chi_i=\operatorname{cut}(\mathfrak p_i)\) for \(0\leq i\leq n_q\), and put \(p_i=\sigma(q,\mathfrak p_{i-1})\) for \(1\leq i\leq n_q\).
  The following conditions are required.
  \begin{enumerate}[label={\textup{(\roman*)}}]
    \item \(\chi_{i-1}<\min\ran y_i\leq\max\ran y_i<\chi_i\), and \(f_i=e_{\eta_i,1}^*\) is the terminal evaluation of the completed matching level-\(p_i\) inner chain.
    \item
    \[ y_i=\frac{m_{p_i}}{n_{p_i}} \sum_{t=1}^{n_{p_i}}u_{i,t}+r_i,\qquad \norm{r_i}\leq\frac{C}{m_{p_i}}, \]
    where \((u_{i,t})_{t=1}^{n_{p_i}}\) is a \(C\)-RIS.
    \item
    \[ \norm{y_i}\leq C,\qquad \sup_{\gamma,\varphi} \abs{d_{\gamma,\varphi}^*(y_i)} \leq\frac{C}{m_{p_i}}. \]
    \item \(U_{\eta_i}y_i=0\).
  \end{enumerate}
\end{definition}

The stems have distinct lengths.
Injectivity of \(\sigma\) makes the levels \(p_i\) pairwise distinct, and we obtain from \eqref{eq:code-growth}
\begin{equation}\label{eq:coded-level-lower}
  m_{p_i}\geq\Lambda_q\qquad(1\leq i\leq n_q).
\end{equation}
Recall from \eqref{eq:constants} that \(C_{\rm off}=8B\kappa\).

The next lemma gives off-weight estimates for interval restrictions of the root-perturbed vectors.
For the exact-pair estimates underlying this argument, see \cite[Definition~6.1 and the following remark]{ArgyrosHaydon2011}.

\begin{lemma}[Rooted auxiliary averages and off-weight estimates]
  \label{lem:off-weight}
  Let \(h\ne k\), let \(\mathfrak g\) be a tagged auxiliary analysis with weighted root level \(h\), and put \(\bar e_k=n_k^{-1}\sum_{s=1}^{n_k}e_s\).
  Then we have
  \begin{equation}\label{eq:rooted-auxiliary}
    \abs{\operatorname{val}(\mathfrak g)(\bar e_k)}
    \leq
    \begin{cases}
      2/(m_hm_k),&h<k,\\
      1/m_h,&h>k.
    \end{cases}
  \end{equation}
  Consequently, if
  \[ y=\frac{m_k}{n_k}\sum_{s=1}^{n_k}u_s+r,\qquad \norm{r}\leq\frac{C}{m_k}, \]
  where \((u_s)\) is a \(C\)-RIS, then every level-\(h\) scalarized evaluation \(f\), with \(h>0\) and \(h\ne k\), satisfies
  \begin{equation}\label{eq:off-weight}
    \abs{f(P_Iy)}
    \leq\frac{C_{\rm off}C}{m_{\min\{h,k\}}}
  \end{equation}
  for every FDD interval \(I\).
\end{lemma}

\begin{proof}
  Suppose first that \(h<k\).
  Write \(\mathfrak g=[h;\mathfrak g_1,\ldots,\mathfrak g_d]\) and \(g_r=\operatorname{val}(\mathfrak g_r)\), so that
  \[
    \operatorname{val}(\mathfrak g)=\frac{1}{m_h}\sum_{r=1}^dg_r,
    \qquad g_1<\cdots<g_d,\qquad d\leq3n_h\leq\mathsf b_k.
  \]
  Apply the stopping argument from the proof of Lemma~\ref{lem:quantitative-auxiliary} to each child tree, counting the \(D_k\) low-level nodes within that tree.
  Let \(\mathcal H,\mathcal D,\mathcal L\) be the unions of the corresponding classes of coordinate indices in \(\{1,\ldots,n_k\}\).
  Since the supports of the \(g_r\)'s are disjoint and every node of level less than \(k\) has at most \(\mathsf b_k\) successors, we obtain
  \[
    \max\{\abs{\mathcal H},\abs{\mathcal D}\}\leq n_k,\qquad
    \abs{\mathcal L}\leq d\sum_{s=0}^{D_k-1}\mathsf b_k^s
      \leq\sum_{s=1}^{D_k}\mathsf b_k^s
      \leq\frac{L_k}{m_k}.
  \]
  Retaining the root factor \(m_h^{-1}\), we obtain
  \[
    \begin{aligned}
      \abs{\operatorname{val}(\mathfrak g)(\bar e_k)}
      &\leq\frac{1}{m_h}\left(
        \frac{\abs{\mathcal H}}{n_km_k}
        +\frac{\abs{\mathcal D}}{n_km_1^{D_k}}
        +\frac{\abs{\mathcal L}}{n_k}\right)\\
      &\leq\frac{1}{m_h}\left(
        \frac{1}{m_k}+\frac{1}{m_1^{D_k}}
        +\frac{L_k}{m_kn_k}\right)\\
      &\leq\frac{1}{m_h}\left(
        \frac{1}{m_k}+\frac{1}{m_k^2}
        +\frac{1}{2^{k+8}m_k^5}\right)\\
      &=\frac{1}{m_hm_k}\left(
        1+\frac{1}{m_k}+\frac{1}{2^{k+8}m_k^4}\right)
        <\frac{2}{m_hm_k}.
    \end{aligned}
  \]

  If \(h>k\), write \(\mathfrak g=[h;\mathfrak g_1,\ldots,\mathfrak g_d]\) and \(g_r=\operatorname{val}(\mathfrak g_r)\).
  Then we have \(\operatorname{val}(\mathfrak g)=m_h^{-1}\sum_{r=1}^dg_r\), where \(g_1<\cdots<g_d\).
  The supports are disjoint and every coordinate coefficient of every \(g_r\) has modulus at most one.
  Hence we obtain
  \[ \abs{\operatorname{val}(\mathfrak g)(\bar e_k)} \leq\frac{1}{m_hn_k}\sum_{r=1}^d \abs{\supp g_r\cap\{1,\ldots,n_k\}} \leq\frac{1}{m_h}. \]
  This proves \eqref{eq:rooted-auxiliary}.

  Write \(I=(a,b]\) and apply Lemma~\ref{lem:basic} to the two tails \(P_{(a,\infty)}\) and \(P_{(b,\infty)}\), with every coefficient equal to \(m_k/n_k\).
  For either auxiliary tree produced by the basic inequality, we have
  \[ \operatorname{val}(\mathfrak g) \left(\frac{m_k}{n_k}\sum_{s=1}^{n_k}e_s\right) =m_k\operatorname{val}(\mathfrak g)(\bar e_k). \]
  From the two exceptional coefficients and the two auxiliary-tree terms, we therefore obtain
  \[
    \abs{f(P_I(y-r))}
    \leq2BC\left(
    \frac{m_k}{n_k}
    +m_k\sup_{\operatorname{root}(\mathfrak g)=h}
    \abs{\operatorname{val}(\mathfrak g)(\bar e_k)}\right).
  \]
  By \eqref{eq:rooted-auxiliary}, we obtain
  \[
    \abs{f(P_I(y-r))}
    \leq\begin{cases}
      \displaystyle 2BC\left(\frac{1}{m_k}+\frac{2}{m_h}\right)
        \leq\frac{6BC}{m_h},&h<k,\\[4pt]
      \displaystyle 2BC\left(\frac{1}{m_k}+\frac{1}{m_k^4}\right)
        \leq\frac{4BC}{m_k},&h>k.
    \end{cases}
  \]
  Since \(\abs{f(P_Ir)}\leq\norm f\norm{P_I}\norm r\leq\kappa C/m_k\), we get
  \[
    \abs{f(P_Iy)}
    \leq\frac{6BC}{m_{\min\{h,k\}}}+\frac{\kappa C}{m_k}
    \leq\frac{(6B+\kappa)C}{m_{\min\{h,k\}}}
    \leq\frac{C_{\rm off}C}{m_{\min\{h,k\}}}.
  \]
\end{proof}

Every \(C\)-bounded dependent sequence is a \(C_{\rm off}C\)-RIS with RIS indices \((p_i)\).
Indeed, from the placement conditions and \eqref{eq:code-growth}, we obtain
\[ p_i\leq\rank\eta_i<\chi_i \quad(1\leq i\leq n_q), \qquad \max\ran y_i<\chi_i<p_{i+1} \quad(1\leq i<n_q). \]
Moreover, we have \(\norm{y_i}\leq C\leq C_{\rm off}C\).
If \(0<h<p_i\), we obtain from Lemma~\ref{lem:off-weight}, for every level-\(h\) scalarized evaluation \(f\) and every FDD interval \(I\),
\[ \abs{f(P_Iy_i)} \leq\frac{C_{\rm off}C}{m_h}. \]
Thus Definition~\ref{def:ris} applies.

We use the common-initial-segment comparison from \cite[Lemma~6.5]{ArgyrosHaydon2011}.
The interval restriction leaves at most two boundary cells in addition to the first cell where the histories differ.

\begin{lemma}[Stem collision]\label{lem:collision}
  Let \((y_i,f_i)_{i=1}^{n_q}\) be a \(C\)-bounded \(q\)-dependent sequence, let \(Q\) be a scalarized outer evaluation of level \(q\), and let \(J\) be an FDD interval.
  Put \(I=\{i:\ran y_i\subseteq J\}\) and assume \(I\ne\varnothing\).
  Then we have
  \begin{equation}\label{eq:collision}
    \abs{QP_J\left(\frac{1}{n_q}\sum_{i\in I}y_i\right)}
    \leq\frac{3\kappa C}{n_qm_q}
    +\frac{\kappa C}{\Lambda_q}
    +\frac{3C_{\rm off}C}{m_q\Lambda_q}.
  \end{equation}
\end{lemma}

\begin{proof}
  Let \(d\leq n_q\) be the age of \(Q\), and write \(Q=e_{\alpha_d,\varphi}^*\).
  Let \(\alpha_1\prec\cdots\prec\alpha_d\) be its stored outer predecessor chain, and put \(\varphi_s=\pi_{r(\alpha_s)}^{r(\alpha_d)}\varphi\) for \(1\leq s\leq d\).
  Thus \(\varphi_s\in B_{E_{\alpha_s}^*}\).
  For \(1\leq s\leq d\), let \(K_s\) be its packet cell, let \(\tau_{s-1}\) be the stored stem before this cell, and let \(\eta_s'\) be the inner terminal carried by its strong packet.
  Put \(h_s=\sigma(q,\tau_{s-1})\).
  Packet covariance writes the scalarized strong packet as \(F_s=e_{\eta_s',\theta_s}^*\), where \(\ell(\eta_s')=h_s\), \(\theta_s\in B_{E_{\eta_s'}^*}\), and \(\norm{F_s}\leq1\).
  The restricted evaluation analysis is
  \begin{equation}\label{eq:collision-analysis}
    QP_J=\sum_{s=1}^dd_{\alpha_s,\varphi_s}^*P_J
    +\frac{1}{m_q}\sum_{s=1}^dF_sP_{K_s\cap J}.
  \end{equation}
  Put \(L_{s,i}=K_s\cap\ran y_i\), and define subsets of \(\{1,\ldots,d\}\times I\) by
  \[
    \begin{aligned}
      \mathcal B=\{(s,i):d_{\alpha_s,\varphi_s}^*(y_i)\ne0\},\quad
      \mathcal A=\{(s,i):F_s(P_{L_{s,i}}y_i)\ne0\},\quad
      \mathcal A_{=}=\{(s,i)\in\mathcal A:h_s=p_i\}.
    \end{aligned}
  \]
  Since \(P_Jy_i=y_i\) for \(i\in I\), we obtain from \eqref{eq:collision-analysis}
  \begin{align}\label{eq:collision-three-sums}
    &\abs{QP_J\left(\frac{1}{n_q}\sum_{i\in I}y_i\right)}
    \\
    \leq& \frac{1}{n_q}\sum_{(s,i)\in\mathcal B}
    \abs{d_{\alpha_s,\varphi_s}^*(y_i)}
    +\frac{1}{n_qm_q}\sum_{(s,i)\in\mathcal A_=}
    \abs{F_s(P_{L_{s,i}}y_i)}\notag+\frac{1}{n_qm_q}\sum_{(s,i)\in\mathcal A\setminus\mathcal A_=}
    \abs{F_s(P_{L_{s,i}}y_i)}.
  \end{align}

  The block ranges are successive.
  Consequently, we have
  \[ \abs{\mathcal B} \leq\sum_{s=1}^d \#\{i\in I:\rank\alpha_s\in\ran y_i\} \leq d\leq n_q. \]
  From Definition~\ref{def:zero-dependent} and \eqref{eq:coded-level-lower}, we therefore obtain
  \begin{equation}\label{eq:collision-bd-sum}
    \frac{1}{n_q}\sum_{(s,i)\in\mathcal B}
    \abs{d_{\alpha_s,\varphi_s}^*(y_i)}
    \leq\frac{1}{n_q}\sum_{(s,i)\in\mathcal B}\frac{C}{m_{p_i}}
    \leq\frac{dC}{n_q\Lambda_q}
    \leq\frac{\kappa C}{\Lambda_q}.
  \end{equation}

  We next count \(\mathcal A_=\).
  Compare the history of \(Q\) with the history of the dependent sequence.
  If the seeds agree, let \(\tau\) be their longest common realized stem, of age \(c\).
  By injectivity of the coding map, we have
  \[ h_s=p_i\ \Longrightarrow\ \tau_{s-1}=\mathfrak p_{i-1} \ \Longrightarrow\ s=i\leq c+1. \]
  Indeed, equal stems have equal ages, and if \(s=i>c+1\), their common prefix would be longer than \(\tau\).
  By Lemma~\ref{lem:formal-common-stem}, packets on \(\tau\) agree.
  For a matching pair on this stem, we obtain from placement and root annihilation
  \[ P_{L_{s,i}}y_i=y_i,\qquad F_s(P_{L_{s,i}}y_i) =\langle\theta_s,U_{\eta_i}y_i\rangle=0. \]
  Among complete cells, the only possible nonzero equal-level incidence is therefore \((c+1,c+1)\).
  If one history ends at \(\tau\), this pair is absent.
  Distinct seeds give no equal-level pair.
  Thus we obtain
  \[ \#\{(s,i)\in\mathcal A_= : K_s\cap J=K_s\}\leq1. \]
  Since the cells \(K_s\) are successive and the \(p_i\)'s are distinct, we have
  \[ \#\{s:\varnothing\ne K_s\cap J\ne K_s\}\leq2, \qquad \#\{i\in I:h_s=p_i\}\leq1. \]
  Splitting the matching incidences into complete and partial cells, we obtain
  \begin{align*}
    \abs{\mathcal A_=}
    \leq\#\{(s,i)\in\mathcal A_= : K_s\cap J=K_s\}
    +\sum_{\substack{1\leq s\leq d\\
    \varnothing\ne K_s\cap J\ne K_s}}
    \#\{i\in I:(s,i)\in\mathcal A_=\}\leq1+2=3.
  \end{align*}
  For each such incidence, we have \(\abs{F_s(P_{L_{s,i}}y_i)} \leq\norm{F_s}\norm{P_{L_{s,i}}}\norm{y_i}\leq\kappa C\).
  Hence we obtain
  \begin{equation}\label{eq:collision-matching-sum}
    \frac{1}{n_qm_q}\sum_{(s,i)\in\mathcal A_=}
    \abs{F_s(P_{L_{s,i}}y_i)}
    \leq\frac{\abs{\mathcal A_=}\kappa C}{n_qm_q}
    \leq\frac{3\kappa C}{n_qm_q}.
  \end{equation}

  On \(\mathcal A\setminus\mathcal A_=\), we obtain from the code-growth condition and \eqref{eq:coded-level-lower} that
  \( m_{\min\{h_s,p_i\}}\geq\Lambda_q. \)
  Thus we get from Lemma~\ref{lem:off-weight}
  \[ \abs{F_s(P_{L_{s,i}}y_i)} \leq\frac{C_{\rm off}C}{m_{\min\{h_s,p_i\}}} \leq\frac{C_{\rm off}C}{\Lambda_q}. \]
  And we get the incidence bound
  \(
    \abs{\mathcal A\setminus\mathcal A_=}\leq d+\abs I-1\leq2n_q-1<3n_q.
  \)
  Therefore we have
  \begin{equation}\label{eq:collision-off-level-sum}
    \frac{1}{n_qm_q}\sum_{(s,i)\in\mathcal A\setminus\mathcal A_=}
    \abs{F_s(P_{L_{s,i}}y_i)}
    \leq\frac{\abs{\mathcal A\setminus\mathcal A_=}}{n_qm_q}
    \frac{C_{\rm off}C}{\Lambda_q}
    \leq\frac{3C_{\rm off}C}{m_q\Lambda_q}.
  \end{equation}
  Adding \eqref{eq:collision-bd-sum}--\eqref{eq:collision-off-level-sum} proves \eqref{eq:collision}.
\end{proof}

For evaluations of level below \(q\), we retain the carrier intervals while stopping the evaluation tree.
We identified the remaining gap in the proof as the dependent-sequence estimate for evaluations of level \(0<h<q\).
GPT-5.6 Sol (OpenAI) then supplied the carrier-labelled stopping-tree construction and its counting proof used in Lemma~\ref{lem:frontier} and Proposition~\ref{prop:dependent-mean} to close this gap.
We rewrote the proof in our notation, verified it in the setting of our construction, and checked its technical antecedents in the literature.
This review confirmed earlier uses of the underlying tree and counting techniques, whose sources are cited below.
For further details, see the \hyperref[sec:ai-assistance]{AI assistance statement} at the end of the paper.

For trees, branches, and asymptotic games, see Odell and Schlumprecht \cite{OdellSchlumprecht2002}.
Interval-labelled tree analyses and stopping by weight and depth are used in \cite[Sections~2 and~3.2]{ManoussakisPelczarSwietek2017}.
For a related decomposition of auxiliary averages by node weights and branch lengths, see \cite[Proposition~11.7]{MotakisPelczar2025}.
Here one scalarized constituent is selected from each packet for the fixed vector \(w\), and incomparable nodes retain disjoint carriers.
For related estimates on repeated averages using tree analyses and a refined basic inequality, compare \cite[Propositions~11.8 and~13.6, Corollary~13.7]{MotakisPelczar2025}.

\begin{lemma}[Carrier-labelled tree]\label{lem:frontier}
  Let \(F\) be a scalarized evaluation, let \(J_\varnothing\) be a nonempty FDD interval, let \(w\) have finite FDD support, and fix \(q\geq2\).
  There are a finite rooted tree \(\mathcal T\), with terminal set \(\mathcal T_{\rm term}\), pairwise disjoint nonempty intervals \(J_t\subseteq J_\varnothing\), coefficients \(0\leq a_t\leq1\), and functionals \(\Phi_t\), \(t\in\mathcal T_{\rm term}\), such that
  \begin{equation}\label{eq:frontier-domination}
    \abs{F(P_{J_\varnothing}w)}
    \leq\sum_{t\in\mathcal T_{\rm term}}
      a_t\abs{\Phi_t(P_{J_t}w)}.
  \end{equation}
  Moreover, we have
  \begin{equation}\label{eq:frontier-size}
    \abs{\mathcal T}\leq
    1+\mathsf b_q+\cdots+\mathsf b_q^{D_q+2}=S_q.
  \end{equation}
  After zero terms are omitted, each \(\Phi_t\) is either \(d_{\gamma,\psi}^*\), with \(\psi\in B_{E_\gamma^*}\) and \(J_t=\{\rank\gamma\}\), or a scalarized evaluation of positive level.
  In the latter case, we have
  \begin{equation}\label{eq:frontier-deep-coefficient}
    a_t\leq m_1^{-(D_q+2)}
    \qquad\bigl(0<\ell(\Phi_t)<q\bigr).
  \end{equation}
\end{lemma}

\begin{proof}
  Since \(w\) has finite support, we may replace \(J_\varnothing\) by a nonempty finite subinterval without changing \(P_{J_\varnothing}w\).
  Give the root the data \((\Phi_\varnothing,J_\varnothing,a_\varnothing)=(F,J_\varnothing,1)\), and write \(|t|\) for the depth of a node.
  At every node, stop with coefficient zero if \(a_t\abs{\Phi_t(P_{J_t}w)}=0\).
  Otherwise, stop at a BD coordinate, a level-zero evaluation, an evaluation of level at least \(q\), or depth \(D_q+2\), whichever occurs first.
  For a level-zero evaluation \(\Phi_t=e_{\gamma,\varphi}^*=d_{\gamma,\varphi}^*\), nonzero contribution implies \(\rank\gamma\in J_t\), and we replace \(J_t\) by \(\{\rank\gamma\}\).

  Suppose that a node \(t\) has not stopped, and put \(h_t=\ell(\Phi_t)\), so that \(0<h_t<q\).
  Write its same-level chain as \(\xi_{t,1}\prec\cdots\prec\xi_{t,d_t}\), with \(d_t\leq n_{h_t}\), ranks \(p_{t,r}=\rank\xi_{t,r}\), and initial cutoff \(p_{t,0}\).
  Let \(\psi_{t,r}\in B_{E_{\xi_{t,r}}^*}\) and \(b_{t,r}\) be the scalarized fibre functionals and packets in its evaluation analysis.
  Define
  \[
    J_{t,r}^{\rm BD}=J_t\cap\{p_{t,r}\},
    \qquad J_{t,r}^{\rm pkt}=J_t\cap(p_{t,r-1},p_{t,r}).
  \]
  Since every raw constituent of \(b_{t,r}\) has rank below \(p_{t,r}\), it annihilates FDD blocks of rank at least \(p_{t,r}\).
  Thus \eqref{eq:evaluation-grammar} gives the identity
  \[
    \Phi_tP_{J_t}
    =\sum_{r=1}^{d_t}d_{\xi_{t,r},\psi_{t,r}}^*P_{J_{t,r}^{\rm BD}}
      +\frac{1}{m_{h_t}}\sum_{r=1}^{d_t}b_{t,r}P_{J_{t,r}^{\rm pkt}},
  \]
  where an empty carrier contributes zero.
  The nonempty carriers in this identity are successive and disjoint.
  Discard zero packet contributions on \(w\).
  For every remaining packet, normalized scalarization gives
  \[
    b_{t,r}=\sum_{l=1}^{s_{t,r}}c_{t,r,l}F_{t,r,l},
    \qquad \sum_{l=1}^{s_{t,r}}\abs{c_{t,r,l}}\leq1,
  \]
  with \(s_{t,r}\geq1\) and scalarized evaluations \(F_{t,r,l}\) of smaller rank than \(\Phi_t\).
  Choose \(l(t,r)\) maximizing \(\abs{F_{t,r,l}(P_{J_{t,r}^{\rm pkt}}w)}\).
  Then we obtain
  \[
    \begin{aligned}
      \abs{b_{t,r}(P_{J_{t,r}^{\rm pkt}}w)}
      \leq\sum_{l=1}^{s_{t,r}}\abs{c_{t,r,l}}
        \abs{F_{t,r,l}(P_{J_{t,r}^{\rm pkt}}w)}\leq&\left(\sum_{l=1}^{s_{t,r}}\abs{c_{t,r,l}}\right)
        \abs{F_{t,r,l(t,r)}(P_{J_{t,r}^{\rm pkt}}w)}\\
        \leq&\abs{F_{t,r,l(t,r)}(P_{J_{t,r}^{\rm pkt}}w)}.
    \end{aligned}
  \]
  For each nonempty BD carrier and each retained packet carrier, give the corresponding child \(u\) the data
  \[
    (\Phi_u,J_u,a_u)=
    \begin{cases}
      (d_{\xi_{t,r},\psi_{t,r}}^*,J_{t,r}^{\rm BD},a_t),
        &\text{for a BD child},\\
      (F_{t,r,l(t,r)},J_{t,r}^{\rm pkt},a_t/m_{h_t}),
        &\text{for a packet child}.
    \end{cases}
  \]
  Applying the stopping rules to these children, we obtain
  \[
    a_t\abs{\Phi_t(P_{J_t}w)}
    \leq\sum_{u\in\operatorname{succ}(t)}a_u\abs{\Phi_u(P_{J_u}w)}.
  \]
  In particular, a nonzero node which has not stopped has at least one child.
  For \(0\leq r\leq D_q+2\), let
  \[
    \mathcal F_r=\{t\in\mathcal T_{\rm term}:|t|<r\}
      \cup\{t\in\mathcal T:|t|=r\}.
  \]
  For \(0\leq r<D_q+2\), replacing only the nonterminal terms in \(\mathcal F_r\) by their children gives
  \[
    \sum_{t\in\mathcal F_r}a_t\abs{\Phi_t(P_{J_t}w)}
    \leq\sum_{t\in\mathcal F_{r+1}}a_t\abs{\Phi_t(P_{J_t}w)}.
  \]
  Since \(\mathcal F_0=\{\varnothing\}\) and \(\mathcal F_{D_q+2}=\mathcal T_{\rm term}\), iteration proves \eqref{eq:frontier-domination}.

  Obviously, incomparable nodes have successive disjoint carriers.
  For the number of children of a node, we have
  \[
    \abs{\operatorname{succ}(t)}\leq2d_t\leq2n_{h_t}\leq\mathsf b_q.
  \]
  Consequently, the number of nodes at depth \(r\) satisfies
  \[
    \#\{t\in\mathcal T:|t|=r\}\leq\mathsf b_q^r.
  \]
  Summing over \(0\leq r\leq D_q+2\) proves \eqref{eq:frontier-size}.
  For a node with nonzero coefficient, let \(\mathcal P(t)\) be the set of ancestors whose edge on the path to \(t\) is a packet edge.
  The recursive definition of the coefficients gives
  \[
    a_t=\prod_{u\in\mathcal P(t)}\frac{1}{m_{h_u}}
    \leq m_1^{-\abs{\mathcal P(t)}}.
  \]
  A BD edge ends its branch.
  Hence a terminal evaluation with \(0<\ell(\Phi_t)<q\) has only packet edges on its path and stops at depth \(D_q+2\).
  Thus \(\abs{\mathcal P(t)}=|t|=D_q+2\), which proves \eqref{eq:frontier-deep-coefficient}.

\end{proof}

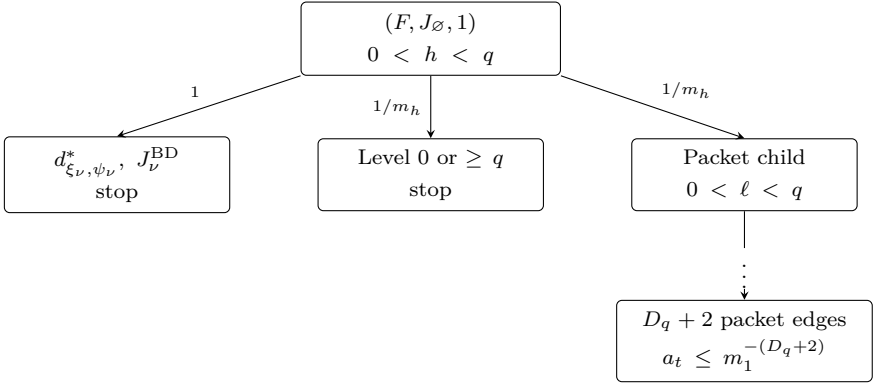
\begin{figure}[!htbp]
  \centering
\begin{tikzpicture}[x=1cm,y=1cm,>=stealth,
  every node/.style={font=\small,align=center},
  treebox/.style={draw,rounded corners=2pt,minimum height=.85cm,
  inner sep=4pt}]
  \node[treebox,text width=3.15cm] (root) at (5.80,3.45)
  {\((F,J_\varnothing,1)\)\\[2pt]
  \(0<h<q\)};
  \node[treebox,text width=2.70cm] (bd) at (1.65,1.65)
  {\(d_{\xi_\nu,\psi_\nu}^*,\ J_\nu^{\rm BD}\)\\[2pt]
  stop};
  \node[treebox,text width=2.70cm] (level) at (5.80,1.65)
  {Level \(0\) or \(\geq q\)\\[2pt]
  stop};
  \node[treebox,text width=2.70cm] (low) at (9.95,1.65)
  {Packet child\\[2pt]
  \(0<\ell<q\)};
  \draw[->] (root.south west) --
  node[above left,font=\scriptsize] {\(1\)} (bd.north);
  \draw[->] (root.south) --
  node[left,font=\scriptsize] {\(1/m_h\)} (level.north);
  \draw[->] (root.south east) --
  node[above right,font=\scriptsize] {\(1/m_h\)} (low.north);

  \node[treebox,text width=3.10cm] (deep) at (9.95,-.55)
  {\(D_q+2\) packet edges\\[2pt]
  \(a_t\leq m_1^{-(D_q+2)}\)};
  \draw (low.south) -- (9.95,.68);
  \node at (9.95,.41) {\(\vdots\)};
  \draw[->] (9.95,.14) -- (deep.north);
\end{tikzpicture}
  \caption{Stopping rules.}
  \label{fig:stopping-tree}
\end{figure}
Figure~\ref{fig:stopping-tree} shows representative branches for a root of level \(0<h<q\).
The edge labels give the coefficient factors: \(1\) for a BD edge and \(1/m_\ell\) for a packet edge from level \(\ell\).
The bottom terminal is reached after \(D_q+2\) packet edges.
Each branch stops at its first stopping event.

Recall from \eqref{eq:constants} that \(K_{\rm dep}=32\kappa(C_{\rm off}+1)\).

The next estimate separates boundary incidences from blocks contained in terminal carriers, as in the interval decompositions of \cite[Section~3]{GowersMaurey1997}.
We also use disjointness when summing the depth remainder (compare \cite[Section~3.2]{ManoussakisPelczarSwietek2017}).
With \(I_t\) denoting the contained-block index sets in the proof, the boundary count involves \(S_q\), while the depth remainder uses \(\sum_t\abs{I_t}\leq n_q\).

\begin{proposition}[Dependent mean estimate]\label{prop:dependent-mean}
  Every \(C\)-bounded \(q\)-dependent sequence satisfies
  \begin{equation}\label{eq:dependent-mean}
    \norm{\frac{1}{n_q}\sum_{i=1}^{n_q}y_i}
    \leq\frac{K_{\rm dep}C}{m_q^2}.
  \end{equation}
\end{proposition}

\begin{proof}
  Put \(\bar y=n_q^{-1}\sum_{i=1}^{n_q}y_i\) and fix a scalarized evaluation \(F\) of level \(h\).
  If \(h=0\), \(F\) meets at most one block, say \(y_{i_0}\).
  Hence we obtain
  \begin{equation}\label{eq:dependent-level-zero}
    \abs{F(\bar y)}
    \leq\frac{1}{n_q}\norm{F}
    \norm{P_{\ran y_{i_0}}}\norm{y_{i_0}}
    \leq\frac{1}{n_q}\cdot1\cdot\kappa\cdot C
    =\frac{\kappa C}{n_q}
    \leq\frac{\kappa C}{m_q^2}.
  \end{equation}
  Suppose \(h>q\), and put \(I_h=\{i:p_i=h\}\).
  The levels \(p_i\) are distinct, so we have \(\abs{I_h}\leq1\).
  Since \(p_i>q\), we obtain from the norm bound on \(I_h\) and Lemma~\ref{lem:off-weight} on its complement
  \begin{equation}\label{eq:dependent-above}
    \abs{F(\bar y)}
    \leq\frac{1}{n_q}
    \left(\sum_{i\in I_h}\abs{F(y_i)}
    +\sum_{i\notin I_h}\abs{F(y_i)}\right)
    \leq\frac{\kappa C}{n_q}+\frac{C_{\rm off}C}{m_{q+1}}.
  \end{equation}
  If \(h=q\), the disjointness of the sets of inner and outer levels makes \(F\) an outer evaluation, and we obtain from Lemma~\ref{lem:collision}
  \begin{equation}\label{eq:dependent-equal}
    \abs{F(\bar y)}
    \leq\frac{3\kappa C}{n_qm_q}
    +\frac{\kappa C}{\Lambda_q}
    +\frac{3C_{\rm off}C}{m_q\Lambda_q}.
  \end{equation}

  Suppose \(0<h<q\).
  Apply Lemma~\ref{lem:frontier} to \(F\), a carrier containing \(\bar y\), and \(w=\bar y\).
  Put
  \[
    \begin{aligned}
      \mathcal E_\partial
      =\{(t,i):t\in\mathcal T_{\rm term},\
        \ran y_i\cap J_t\ne\varnothing,\ \ran y_i\nsubseteq J_t\},\qquad
      I_t=\{i:\ran y_i\subseteq J_t\}
        \quad(t\in\mathcal T_{\rm term}).
    \end{aligned}
  \]
  A partial intersection contains an endpoint of \(J_t\).
  Hence we obtain
  \[
    \begin{aligned}
      \abs{\mathcal E_\partial}
      \leq\sum_{t\in\mathcal T_{\rm term}}
        \sum_{v\in\{\min J_t,\max J_t\}}\#\{i:v\in\ran y_i\}\leq2\abs{\mathcal T_{\rm term}}\leq2S_q.
    \end{aligned}
  \]
  A block may contain endpoints of several terminal carriers, so the summation must count pairs \((t,i)\).
  Moreover, the terminal carriers are disjoint, and therefore we have
  \[ I_t\cap I_u=\varnothing\quad(t\ne u),\qquad \sum_t\abs{I_t}\leq n_q. \]
  For every terminal carrier, the support decomposition is
  \[
    P_{J_t}\bar y
      =\frac{1}{n_q}\sum_{i\in I_t}y_i
       +\frac{1}{n_q}\sum_{i:(t,i)\in\mathcal E_\partial}P_{J_t}y_i.
  \]
  Inserting this identity into \eqref{eq:frontier-domination}, we obtain
  \[
    \begin{aligned}
      \abs{F(\bar y)}
      \leq\frac{1}{n_q}\sum_{(t,i)\in\mathcal E_\partial}
        a_t\abs{\Phi_t(P_{J_t}y_i)}+\sum_{t\in\mathcal T_{\rm term}}a_t
        \abs{\Phi_t\left(\frac{1}{n_q}\sum_{i\in I_t}y_i\right)}.
    \end{aligned}
  \]
  If \(\Phi_t\) is a positive-level evaluation, then we have
  \[ a_t\abs{\Phi_t(P_{J_t}y_i)} \leq a_t\norm{\Phi_t}\norm{P_{J_t}}\norm{y_i} \leq1\cdot1\cdot\kappa\cdot C =\kappa C. \]
  If \(\Phi_t\) is a BD or level-zero coordinate, its singleton carrier is retained by \(P_{J_t}\), and we obtain from the coordinate estimate
  \[ a_t\abs{\Phi_t(P_{J_t}y_i)} \leq\abs{\Phi_t(y_i)} \leq\frac{C}{m_{p_i}} \leq C\leq\kappa C. \]
  Thus the boundary terms satisfy
  \[
    \frac{\kappa C}{n_q}\abs{\mathcal E_\partial}
    \leq\frac{2\kappa CS_q}{n_q}.
  \]

  After omitting terms with \(a_t=0\), a BD terminal or a level-zero coordinate terminal has a singleton carrier and meets at most one block.
  From the coordinate estimate and \eqref{eq:coded-level-lower}, we obtain, for that block,
  \[ a_t\abs{\Phi_t(P_{J_t}y_i)} \leq\abs{\Phi_t(y_i)} \leq\frac{C}{m_{p_i}} \leq\frac{C}{\Lambda_q}. \]
  There are at most \(\abs{\mathcal T_{\rm term}}\leq S_q\) such terminals.
  Therefore we obtain
  \[
    \sum_{\substack{t\in\mathcal T_{\rm term}\\
    \Phi_t\ {\rm BD\ or\ level\ zero}}}
    a_t\abs{\Phi_tP_{J_t}
    \left(\frac{1}{n_q}\sum_{i\in I_t}y_i\right)}
    \leq\frac{1}{n_q}S_q\frac{C}{\Lambda_q}
    =\frac{CS_q}{n_q\Lambda_q}.
  \]
  Omit zero terms and empty \(I_t\)'s from the remaining sums.
  For the level-\(q\) terminals, we have \(\#\{t:\ell(\Phi_t)=q\}\leq S_q\) and \(a_t\leq1\), so we obtain from Lemma~\ref{lem:collision}
  \[
    \sum_{\ell(\Phi_t)=q}
    a_t\abs{\Phi_t\left(\frac{1}{n_q}\sum_{i\in I_t}y_i\right)}
    \leq S_q\left(\frac{3\kappa C}{n_qm_q}
    +\frac{\kappa C}{\Lambda_q}
    +\frac{3C_{\rm off}C}{m_q\Lambda_q}\right).
  \]

  For the terminals above level \(q\), define
  \[ \mathcal E_> =\{(t,i):\ell(\Phi_t)>q,\ i\in I_t\}, \qquad \mathcal E_= =\{(t,i)\in\mathcal E_>:\ell(\Phi_t)=p_i\}. \]
  Since the \(p_i\)'s are distinct, we have
  \[
    \abs{\mathcal E_=}
      =\sum_{\ell(\Phi_t)>q}\#\{i\in I_t:p_i=\ell(\Phi_t)\}
      \leq\#\{t:\ell(\Phi_t)>q\}\leq S_q.
  \]
  By the disjointness of the \(I_t\)'s, we also have
  \[
    \abs{\mathcal E_>\setminus\mathcal E_=}\leq\sum_{\ell(\Phi_t)>q}\abs{I_t}\leq n_q.
  \]
  For \(i\in I_t\), we have \(P_{J_t}y_i=y_i\).
  From the norm bound and Lemma~\ref{lem:off-weight}, we therefore obtain
  \[
  a_t\abs{\Phi_t(y_i)} \leq \begin{cases} \kappa C,&(t,i)\in\mathcal E_=,\\
  C_{\rm off}C/m_{q+1}, &(t,i)\in\mathcal E_>\setminus\mathcal E_=. \end{cases}
  \]
  Splitting the sum over these two sets, we obtain
  \[
    \begin{aligned}
      \sum_{\ell(\Phi_t)>q}a_t
      \abs{\Phi_t\left(\frac{1}{n_q}\sum_{i\in I_t}y_i\right)}
      \leq\frac{\abs{\mathcal E_=}\kappa C}{n_q}
        +\frac{\abs{\mathcal E_>\setminus\mathcal E_=}
          C_{\rm off}C}{n_qm_{q+1}}\leq\frac{S_q\kappa C}{n_q}
        +\frac{C_{\rm off}C}{m_{q+1}}.
    \end{aligned}
  \]

  Finally, after the zero terms have been omitted, a terminal of level strictly between \(0\) and \(q\) can occur only at the depth cutoff.
  By \eqref{eq:frontier-deep-coefficient}, we have
  \[
    \begin{aligned}
      \sum_{0<\ell(\Phi_t)<q}a_t
      \abs{\Phi_tP_{J_t}\left(\frac{1}{n_q}\sum_{i\in I_t}y_i\right)}
      &\leq\frac{1}{n_q}\sum_{0<\ell(\Phi_t)<q}
        a_t\sum_{i\in I_t}\abs{\Phi_t(y_i)}\\
      &\leq\frac{\kappa C m_1^{-(D_q+2)}}{n_q}
        \sum_{0<\ell(\Phi_t)<q}\abs{I_t}\\
      &\leq\frac{\kappa C m_1^{-(D_q+2)}}{n_q}\,n_q
        =\kappa C m_1^{-(D_q+2)}.
    \end{aligned}
  \]

  Adding the five terminal and boundary estimates gives, when \(0<h<q\),
  \[
    \begin{aligned}
      \abs{F(\bar y)}
      \leq C\Bigl[&\frac{2\kappa S_q}{n_q}
        +\frac{S_q}{n_q\Lambda_q}
        +\frac{3\kappa S_q}{n_qm_q}
        +\frac{\kappa S_q}{\Lambda_q}
        +\frac{3C_{\rm off}S_q}{m_q\Lambda_q}+\frac{\kappa S_q}{n_q}
        +\frac{C_{\rm off}}{m_{q+1}}
        +\kappa m_1^{-(D_q+2)}\Bigr]\\
      \leq C\Bigl[&(6\kappa+1)\frac{S_q}{n_q}
        +(\kappa+3C_{\rm off})\frac{S_q}{\Lambda_q}
        +\frac{C_{\rm off}}{m_{q+1}}
        +\kappa m_1^{-(D_q+2)}\Bigr].
    \end{aligned}
  \]
  Together with \eqref{eq:dependent-level-zero}--\eqref{eq:dependent-equal}, we obtain, for every level \(h\),
  \[
    \begin{aligned}
      \abs{F(\bar y)}
      &\leq8\kappa(C_{\rm off}+1)C
        \left[\frac{S_q}{n_q}+\frac{S_q}{\Lambda_q}
          +m_1^{-(D_q+2)}+\frac{1}{m_{q+1}}+\frac{1}{n_qm_q}\right]\\
      &\leq32\kappa(C_{\rm off}+1)\frac{C}{m_q^2}
        =\frac{K_{\rm dep}C}{m_q^2}.
    \end{aligned}
  \]
  Since scalarized ambient evaluations norm \(X\), we prove \eqref{eq:dependent-mean}.
\end{proof}
In Figure~\ref{fig:counting-intervals}, \(y_3\) contributes both \((t_1,3)\) and \((t_2,3)\), and the dashed boxes are contained block ranges.
  Here, we have
  \[ \mathcal E_\partial=\{(t_1,1),(t_1,3),(t_2,3),(t_2,5)\},\qquad I_{t_1}=\{2\},\quad I_{t_2}=\{4\}. \]
  \begin{figure}[!htbp]
    \centering
\begin{tikzpicture}[x=1cm,y=1cm,>=stealth,
  every node/.style={font=\small,inner sep=2pt}]
  \draw[thick] (1.5,1.05) rectangle (4.0,1.28);
  \draw[thick] (4.8,1.05) rectangle (8.0,1.28);
  \node[above] at (2.75,1.30) {\(J_{t_1}\)};
  \node[above] at (6.40,1.30) {\(J_{t_2}\)};
  \foreach \x in {1.5,4.0,4.8,8.0}
  \draw[densely dashed] (\x,-0.13) -- (\x,1.04);
  \foreach \a/\b/\i in {0.2/1.9/1,3.7/5.4/3,7.5/9.0/5}
  {
  \draw[thick] (\a,-0.13) rectangle (\b,0.13);
  \node[below] at ({(\a+\b)/2},-0.20) {\(\ran y_{\i}\)};
  }
  \foreach \a/\b/\i in {2.4/3.1/2,6.0/6.7/4}
  {
  \draw[thick,densely dashed] (\a,-0.13) rectangle (\b,0.13);
  \node[below] at ({(\a+\b)/2},-0.20) {\(\ran y_{\i}\)};
  }
  \draw[->] (0.1,-0.60) -- (9.15,-0.60)
  node[right,font=\scriptsize] {rank};
\end{tikzpicture}
    \caption{Boundary incidences and contained blocks.}
    \label{fig:counting-intervals}
  \end{figure}
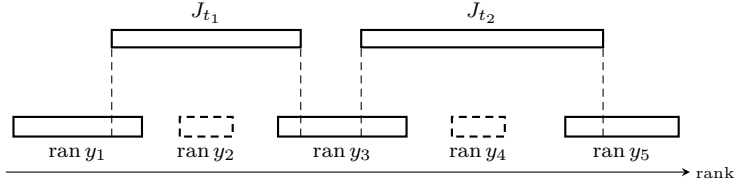

\subsection{The local-orbit theorem}

\begin{lemma}[Operator-dependent extraction]
  \label{lem:critical-extraction}
  Let \(S\in\Bcal(X)\), let \((x_k)\) be a \(C_0\)-RIS with \(C_0\geq1\), and assume
  \[ \dist(Sx_k,\mathscr M(x_k))\geq\varepsilon \qquad(k\in\N). \]
  For each sufficiently large seed level \(q\), there is a \(C_{\rm dep}\)-bounded \(q\)-dependent sequence \((y_s,f_s)_{s=1}^{n_q}\), where \(C_{\rm dep}=C_{\rm ex}C_0\), such that
  \[ \operatorname{Re}f_s(Sy_s)\geq\frac{\varepsilon}{4} \qquad(1\leq s\leq n_q). \]
  If \(\bar y=n_q^{-1}\sum_sy_s\), then we have
  \begin{equation}\label{eq:outer-lower}
    \norm{S\bar y}\geq\frac{\varepsilon}{8m_q}.
  \end{equation}
\end{lemma}

\begin{proof}
  Put \(C=2C_0\), \(N=n_q\), and \(m=m_q\), and fix a seed with initial cutoff \(r_0\).
  Put
  \[
    \beta_d=\frac{\varepsilon}{64mN},\qquad
    \beta_\times=\frac{\varepsilon}{64N},\qquad
    \beta_0=\frac{\varepsilon}{64},\qquad
    \theta=\kappa^{-1}\min\{\beta_d,\beta_\times\}=\frac{\varepsilon}{64\kappa mN}.
  \]
  Suppose that the first \(s-1\) outer cells have been selected.
  Their realized stem \(\mathfrak p_{s-1}\) determines \(p_s=\sigma(q,\mathfrak p_{s-1})\).
  Lemma~\ref{lem:exactification} gives a successive sequence of pairs \((y^{(j)},f^{(j)})\) over this fixed stem, with \((y^{(j)})\) weakly null.
  Put \(\widetilde f_l=f_lP_{(r_{l-1},\infty)}\) for \(l<s\).
  Every previous functional is now fixed, and the head projection has finite-dimensional range.
  Hence, with empty sums interpreted as zero, we have
  \[ \norm{P_{[1,r_{s-1}]}Sy^{(j)}}+ \sum_{l<s}\bigl(\abs{d_{\zeta_l,1}^*(Sy^{(j)})} +\abs{\widetilde f_l(Sy^{(j)})}\bigr)\longrightarrow0. \]
  Choose \(j\) and put \((y_s,f_s)=(y^{(j)},f^{(j)})\) so that the following three inequalities hold:
  \begin{align}
    \abs{d_{\zeta_l,1}^*(Sy_s)}&<\beta_d\quad(l<s),                  \label{eq:outer-old-d}\\
    \abs{\widetilde f_l(Sy_s)}&<\beta_\times\quad(l<s),             \label{eq:outer-old-packet}\\
    \norm{P_{[1,r_{s-1}]}Sy_s}&<\beta_0,                         \label{eq:outer-head}
  \end{align}
  After this pair and its coded certificate are fixed, the finite set \(\{Sy_t:1\leq t\leq s\}\) satisfies
  \[ \lim_{R\to\infty}\max_{1\leq t\leq s} \norm{P_{(R,\infty)}Sy_t}=0. \]
  Choose \(R_s\) beyond the supports of the certificate so that
  \begin{equation}\label{eq:outer-tail}
    \norm{P_{(R_s,\infty)}Sy_t}<\theta
    \qquad(1\leq t\leq s).
  \end{equation}
  Proposition~\ref{prop:finite-extension}\textup{(c)} realizes after \(R_s\) an outer successor \(\zeta_s\) carrying \(f_s\).
  Put \(r_s=\rank\zeta_s\), and choose the next pair only after this successor is fixed.

  After \(N\) steps, put \(\widetilde f_s=f_sP_{(r_{s-1},\infty)}\) for \(1\leq s\leq N\), and let \(G_q=e_{\zeta_N,1_{A_{r(\zeta_N)}}}^*\).
  This terminal outer evaluation has the analysis
  \begin{equation}\label{eq:outer-analysis}
    G_q=\sum_{s=1}^Nd_{\zeta_s,1}^*
    +\frac{1}{m}\sum_{s=1}^N\widetilde f_s.
  \end{equation}
  In the \(s\)-th application of Lemma~\ref{lem:exactification}, write \(y_s=y_s^0+\widetilde r_s\), where \(y_s^0\) is the inner special average, and denote the parameter in \eqref{eq:exactification-tau} by \(\tau_s\).
  Since \(C=2C_0\), we obtain from Lemma~\ref{lem:root} and the construction
  \[ \norm{\widetilde r_s} \leq M\norm{\Delta(f_s,y_s^0)} <M\tau_s<\frac{MC}{m_{p_s}} \leq\frac{C_{\rm dep}}{m_{p_s}}, \qquad C\leq C_{\rm dep}. \]
  Thus the inner \(C\)-RIS is also a \(C_{\rm dep}\)-RIS, and clause \textup{(ii)} of Definition~\ref{def:zero-dependent} holds.
  Lemma~\ref{lem:exactification} and the root perturbation give the other three clauses, so the resulting sequence is a \(C_{\rm dep}\)-bounded dependent sequence, where \(C_{\rm dep}=C_{\rm ex}C_0\).

  For predecessor terms with \(l<s\), we use \eqref{eq:outer-old-d}, while for \(l\geq s\), the rank \(r_l\) lies after \(R_s\), so \eqref{eq:outer-tail} and \(\norm{d_{\zeta_l,1}^*}\leq\kappa\) apply.
  For packet terms the same split uses \eqref{eq:outer-old-packet} when \(l<s\), while \(r_{l-1}>R_s\) and \(\norm{\widetilde f_l}\leq\kappa\) apply when \(l>s\).
  Finally, we have \(\widetilde f_s-f_s=-f_sP_{[1,r_{s-1}]}\).
  From the definitions of \(\beta_d\), \(\beta_\times\), \(\beta_0\), and \(\theta\), we obtain
  \begin{align}
    \abs{\sum_{l,s=1}^Nd_{\zeta_l,1}^*(Sy_s)}
    &<\frac{N(N-1)}{2}\beta_d
    +\frac{N(N+1)}{2}\kappa\theta
    \leq N^2\beta_d=\frac{\varepsilon N}{64m},                 \label{eq:outer-d-error}\\*
    \sum_{l<s}\abs{\widetilde f_l(Sy_s)}
    &<\frac{N(N-1)}{2}\beta_\times
    <\frac{\varepsilon N}{128},                                \label{eq:outer-old-packet-error}\\*
    \sum_{l>s}\abs{\widetilde f_l(Sy_s)}
    &<\frac{N(N-1)}{2}\kappa\theta
    \leq\frac{N(N-1)}{2}\beta_\times
    <\frac{\varepsilon N}{128},                                \label{eq:outer-new-packet-error}\\*
    \abs{\sum_{s=1}^N(\widetilde f_s-f_s)(Sy_s)}
    &\leq\sum_{s=1}^N\norm{P_{[1,r_{s-1}]}Sy_s}
    <N\beta_0=\frac{\varepsilon N}{64}.                        \label{eq:outer-head-error}
  \end{align}

  Apply \eqref{eq:outer-analysis} to \(\bar y=N^{-1}\sum_sy_s\).
  Subtracting the diagonal packet terms, we obtain the exact identity
  \[
    \begin{aligned}
      &G_q(S\bar y)-\frac{1}{mN}\sum_{s=1}^Nf_s(Sy_s)\\
      =&\frac{1}{N}\sum_{l,s=1}^Nd_{\zeta_l,1}^*(Sy_s)+\frac{1}{mN}\sum_{l<s}\widetilde f_l(Sy_s)
      +\frac{1}{mN}\sum_{l>s}\widetilde f_l(Sy_s)+\frac{1}{mN}\sum_{s=1}^N(\widetilde f_s-f_s)(Sy_s).
    \end{aligned}
  \]
  The predecessor sum has coefficient \(1/N\), whereas each of the other three sums has coefficient \(1/(mN)\).
  From \eqref{eq:outer-d-error}--\eqref{eq:outer-head-error}, we therefore obtain
  \begin{equation}\label{eq:outer-table}
    \begin{aligned}
      \abs{G_q(S\bar y)-\frac{1}{mN}\sum_{s=1}^Nf_s(Sy_s)}
      &<\frac{1}{N}\frac{\varepsilon N}{64m}
        +\frac{1}{mN}\left(
          \frac{\varepsilon N}{128}+\frac{\varepsilon N}{128}
          +\frac{\varepsilon N}{64}\right)\\
      &=\frac{3\varepsilon}{64m}<\frac{\varepsilon}{8m}.
    \end{aligned}
  \end{equation}
  Since \((mN)^{-1}\sum_{s=1}^N\operatorname{Re}f_s(Sy_s) \geq\varepsilon/(4m)\), we obtain from \eqref{eq:outer-table}
  \[
    \begin{aligned}
      \operatorname{Re}G_q(S\bar y)
      \geq\frac{1}{mN}\sum_{s=1}^N\operatorname{Re}f_s(Sy_s)
        -\abs{G_q(S\bar y)-\frac{1}{mN}\sum_{s=1}^Nf_s(Sy_s)}>\frac{\varepsilon}{4m}-\frac{\varepsilon}{8m}
        =\frac{\varepsilon}{8m}.
    \end{aligned}
  \]
  As \(\norm{G_q}\leq1\), we have
  \[ \norm{S\bar y} \geq\abs{G_q(S\bar y)} \geq\operatorname{Re}G_q(S\bar y) >\frac{\varepsilon}{8m}, \]
  which proves \eqref{eq:outer-lower}.
\end{proof}

For the corresponding scalar-orbit result, see \cite[Proposition~7.3]{ArgyrosHaydon2011}.

\begin{theorem}[Local multiplier orbit]\label{thm:local-orbit}
  For every \(S\in\Bcal(X)\) and every  RIS \((x_k)\), we have
  \begin{equation}\label{eq:local-orbit}
    \dist\bigl(Sx_k,\{M_ax_k:a\in\widehat A\}\bigr)
    \longrightarrow0.
  \end{equation}
\end{theorem}

\begin{proof}
  Suppose \eqref{eq:local-orbit} fails.
  Pass to a \(C_0\)-RIS subsequence with \(C_0\geq1\) on which the distances are at least \(\varepsilon>0\).
  For every sufficiently large seed level \(q\), Lemma~\ref{lem:critical-extraction} gives a \(C_{\rm dep}\)-bounded dependent sequence, and we have
  \[ \bar y=\frac{1}{n_q}\sum_{s=1}^{n_q}y_s,\qquad \norm{S\bar y}\geq\frac{\varepsilon}{8m_q}. \]
  Proposition~\ref{prop:dependent-mean} gives \(\norm{\bar y}\leq K_{\rm dep}C_{\rm dep}/m_q^2\).
  Therefore we obtain
  \[ \frac{\varepsilon}{8m_q} \leq\norm{S\bar y} \leq\norm{S}\norm{\bar y} \leq\norm{S}\frac{K_{\rm dep}C_{\rm dep}}{m_q^2}. \]
  The constant \(C_{\rm dep}=C_{\rm ex}C_0\) is independent of \(q\), while the seed levels are unbounded.
  Multiplying by \(8m_q^2/\varepsilon\), we obtain
  \[
    m_q\leq\frac{8\norm S K_{\rm dep}C_{\rm dep}}{\varepsilon}.
  \]
  Choosing a seed level \(q\) for which \(m_q\) exceeds this fixed bound gives a contradiction.
\end{proof}


\section{RIS synchronization and global operator classification}
\label{sec:globalization}

    Throughout this section, \(X\) is the block BD space of Theorem~\ref{thm:datum} and \(\widehat A=\ilim(A_r,\pi_r^s)\).
    For \(a=(a_r)_r\in\widehat A\), let \(M_a\) be the multiplier defined in \eqref{eq:multiplier-definition}.
    We shall find one such coefficient for each \(S\in\Bcal(X)\) and prove that the remainder \(R=S-M_a\) is compact.
    Unit probes determine the coordinates \(a_r\), and paired probes give compatibility.
    The convergence is uniform over detector indices.
    After subtracting \(M_a\), a core probe reproduces the annihilation error of a finite block, giving the compactness argument in Lemma~\ref{lem:no-ghost}.
    We retain the normalization \(\psi\in B_{E_\eta^*}\) for scalarized evaluations \(e_{\eta,\psi}^*\) used in the estimates below.
    
    For each probe \(\gamma\), recall that \(j_\gamma:E_\gamma\to X\), \(D_\gamma:X\to E_\gamma\), and \(U_\gamma:X\to E_\gamma\) are the block injection, block coefficient, and ambient coordinate maps.
    The fibre carries the action \(\Theta_\gamma\).
    Let \(\mathcal P_r^{\rm u}\) and \(\mathcal P_r^{\rm c}\) denote the cofinal unit and core probes of detector index \(r\), with fibres \(A_r\) and \(A_r^*\), respectively.
    By Lemma~\ref{lem:formal-probe-isometry}, we have
    \(D_\gamma=U_\gamma\), \(\norm{D_\gamma}\leq1\), and
    \(\norm{j_\gamma u}=\norm u\) for \(u\in E_\gamma\).
    For a unit probe \(\gamma\in\mathcal P_r^{\rm u}\), write \(z_\gamma=j_\gamma1_{A_r}\).

\subsection{Compatibility of finite-dimensional coefficients}

We adapt the paired-RIS argument in \cite[proof of Theorem~7.4]{ArgyrosHaydon2011} to coefficients in the detector algebras.

\begin{lemma}[Local detector coefficients]\label{lem:local-coefficients}
  Let \(S\in\Bcal(X)\), and put
  \begin{equation}\label{eq:unit-probe-distance}
    \alpha_N(S)
    =\sup_{\substack{r\in\N,\ \gamma\in\mathcal P_r^{\rm u}\\
    \rank\gamma\geq N}}
    \dist(Sz_\gamma,j_\gamma A_r).
  \end{equation}
  Then \(\alpha_N(S)\to0\).
  For every \(\gamma\in\mathcal P_r^{\rm u}\), we may choose \(b_\gamma\in A_r\) such that, with \(e_\gamma=Sz_\gamma-j_\gamma b_\gamma\), we have
  \begin{equation}\label{eq:unit-probe-approximation}
    \norm{e_\gamma}
    \leq\alpha_{\rank\gamma}(S)+2^{-\rank\gamma}.
  \end{equation}
  In particular, we have
  \[
  \sup_{\substack{r\in\N,\ \gamma\in\mathcal P_r^{\rm u}\\
  \rank\gamma\geq N}} \norm{e_\gamma} \leq\alpha_N(S)+2^{-N}\longrightarrow0.
  \]
  The coefficients satisfy \(\norm{b_\gamma}\leq\norm S+\norm{e_\gamma}\).
  Finally, we have
  \begin{equation}\label{eq:bonding-cauchy}
    \omega_N(S)
    :=\sup_{\substack{r\leq s,\ \gamma\in\mathcal P_r^{\rm u},
    \delta\in\mathcal P_s^{\rm u}\\
    \rank\gamma,\rank\delta\geq N}}
    \norm{b_\gamma-\pi_r^sb_\delta}
    \longrightarrow0.
  \end{equation}
\end{lemma}

\begin{proof}
  For a unit probe of index \(r\), Lemma~\ref{lem:bounded-limit} gives \(\{a_r:a\in\widehat A\}=A_r\).
  Since \(\Theta_\gamma(b)1_{A_r}=b\), we have
  \[
    \{M_az_\gamma:a\in\widehat A\}
    =\{j_\gamma a_r:a\in\widehat A\}
    =j_\gamma A_r.
  \]
  Suppose that \(\alpha_N(S)\) does not tend to zero.
  Since the supremum in \eqref{eq:unit-probe-distance} is taken over every detector index, we can choose \(\varepsilon>0\) and unit probes \(\gamma_k\in\mathcal P_{r_k}^{\rm u}\) recursively so that
  \[
    \rank\gamma_k>\rank\gamma_{k-1}+2,
    \qquad
    \dist(Sz_{\gamma_k},j_{\gamma_k}A_{r_k})\geq\varepsilon.
  \]
  By Corollary~\ref{cor:probe-ris}, \((z_{\gamma_k})\) is a RIS.
  Theorem~\ref{thm:local-orbit} now gives
  \[
    0<\varepsilon
    \leq\dist(Sz_{\gamma_k},j_{\gamma_k}A_{r_k})
    =\dist\bigl(Sz_{\gamma_k},
        \{M_az_{\gamma_k}:a\in\widehat A\}\bigr)
    \longrightarrow0,
  \]
  a contradiction.
  Thus we have \(\alpha_N(S)\to0\), uniformly in the detector index.
  Choose \(b_\gamma\) so that
  \[
    \begin{aligned}
      \norm{Sz_\gamma-j_\gamma b_\gamma}
      \leq\dist(Sz_\gamma,j_\gamma A_r)+2^{-\rank\gamma}\leq\alpha_{\rank\gamma}(S)+2^{-\rank\gamma}.
    \end{aligned}
  \]
  This is \eqref{eq:unit-probe-approximation}.

  By the probe isometry, we have
  \( \norm{b_\gamma} =\norm{j_\gamma b_\gamma} \leq\norm{Sz_\gamma}+\norm{e_\gamma} \leq\norm S+\norm{e_\gamma}. \)
  Suppose that \eqref{eq:bonding-cauchy} fails.
  Since \(\omega_N(S)\) decreases with \(N\), we can choose \(\tau>0\) such that every tail contains a pair with compatibility error at least \(\tau\).
  Choose these pairs recursively so that
  \[ \max\{\rank\gamma_{k-1},\rank\delta_{k-1}\}+2 <\min\{\rank\gamma_k,\rank\delta_k\}. \]
  The two atoms in each pair are distinct, since an identical pair has zero compatibility error.
  The choices satisfy
  \[ \gamma_k\in\mathcal P_{r_k}^{\rm u}, \qquad \delta_k\in\mathcal P_{s_k}^{\rm u}, \qquad r_k\leq s_k, \qquad \norm{b_{\gamma_k}-\pi_{r_k}^{s_k}b_{\delta_k}}\geq\tau. \]
  Corollary~\ref{cor:probe-ris} makes the two component sequences and \((z_{\gamma_k}+z_{\delta_k})\) RISs.
  The local-orbit theorem supplies \(c^{(k)}\in\widehat A\) and \(w_k\in X\) such that
  \begin{equation}\label{eq:paired-unit-probe-approximation}
    S(z_{\gamma_k}+z_{\delta_k})
    =j_{\gamma_k}c_{r_k}^{(k)}+j_{\delta_k}c_{s_k}^{(k)}+w_k,
    \qquad \norm{w_k}\longrightarrow0.
  \end{equation}
  The individual unit-probe approximations are \(Sz_{\gamma_k}=j_{\gamma_k}b_{\gamma_k}+e_{\gamma_k}\) and \(Sz_{\delta_k}=j_{\delta_k}b_{\delta_k}+e_{\delta_k}\), where \(\norm{e_{\gamma_k}}+\norm{e_{\delta_k}}\to0\).

  By biorthogonality, we have
  \[ D_{\gamma_k}j_{\delta_k}=D_{\delta_k}j_{\gamma_k}=0, \qquad D_{\gamma_k}j_{\gamma_k}=I_{A_{r_k}}, \qquad D_{\delta_k}j_{\delta_k}=I_{A_{s_k}}. \]
  Subtract the individual approximations from \eqref{eq:paired-unit-probe-approximation} and apply the two contractive coefficient maps.
  We obtain
  \[
    \begin{aligned}
      b_{\gamma_k}-c_{r_k}^{(k)}
      &=D_{\gamma_k}(w_k-e_{\gamma_k}-e_{\delta_k}),\\
      b_{\delta_k}-c_{s_k}^{(k)}
      &=D_{\delta_k}(w_k-e_{\gamma_k}-e_{\delta_k}).
    \end{aligned}
  \]
  Therefore we obtain
  \[ \max\bigl\{\norm{b_{\gamma_k}-c_{r_k}^{(k)}}, \norm{b_{\delta_k}-c_{s_k}^{(k)}}\bigr\} \leq\norm{w_k}+\norm{e_{\gamma_k}}+\norm{e_{\delta_k}} \longrightarrow0. \]
  Since \(c_{r_k}^{(k)}=\pi_{r_k}^{s_k}c_{s_k}^{(k)}\) and the bonding map is contractive, we obtain
  \[
    \begin{aligned}
      \norm{b_{\gamma_k}-\pi_{r_k}^{s_k}b_{\delta_k}}
      \leq\norm{b_{\gamma_k}-c_{r_k}^{(k)}}
        +\norm{\pi_{r_k}^{s_k}(c_{s_k}^{(k)}-b_{\delta_k})}\leq2\bigl(\norm{w_k}
        +\norm{e_{\gamma_k}}+\norm{e_{\delta_k}}\bigr)
        \longrightarrow0.
    \end{aligned}
  \]
  This contradicts the choice of the pairs and proves \eqref{eq:bonding-cauchy}.
\end{proof}

Figure~\ref{fig:paired-probe-compatibility} summarizes the argument for the separated sequence of probe pairs used above.
Applying the local-orbit theorem to \(z_{\gamma_k}+z_{\delta_k}\), we obtain one \(c^{(k)}\in\widehat A\) for both probes.
Its coordinates satisfy the bonding relation exactly, while the coefficient maps show that \(b_{\gamma_k}\) and \(b_{\delta_k}\) approach those coordinates.
The detector indices \(r_k\leq s_k\) may vary with \(k\), as required for the uniform estimate in \eqref{eq:bonding-cauchy}.

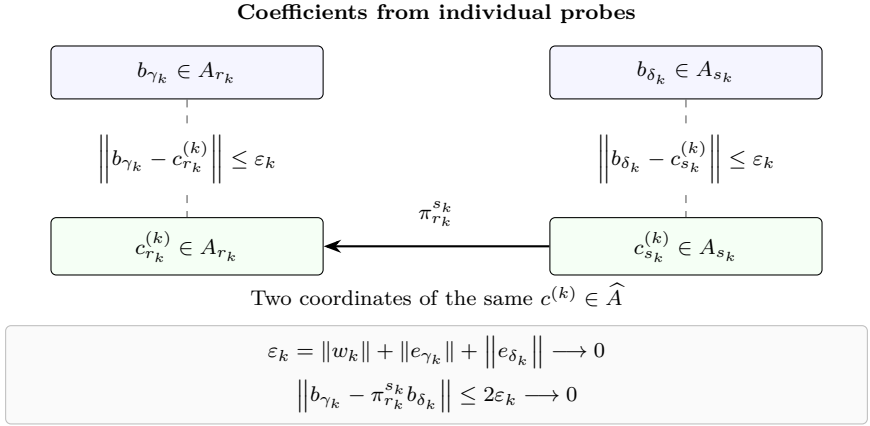
\begin{figure}[!htbp]
  \centering
\begin{tikzpicture}[x=1cm,y=1cm,>=Stealth,
  every node/.style={font=\small,align=center},
  coefficient/.style={draw,rounded corners=2pt,
    minimum width=3.60cm,minimum height=.70cm,inner sep=5pt},
  error/.style={fill=white,inner sep=3pt}]
  \node[font=\small\bfseries] at (5.80,.78)
    {Coefficients from individual probes};
  \node[coefficient,fill=blue!4] (bg) at (2.50,0)
    {\(b_{\gamma_k}\in A_{r_k}\)};
  \node[coefficient,fill=blue!4] (bd) at (9.10,0)
    {\(b_{\delta_k}\in A_{s_k}\)};

  \node[coefficient,fill=green!4!white] (cr) at (2.50,-2.30)
    {\(c_{r_k}^{(k)}\in A_{r_k}\)};
  \node[coefficient,fill=green!4!white] (cs) at (9.10,-2.30)
    {\(c_{s_k}^{(k)}\in A_{s_k}\)};
  \draw[dashed,black!65] (bg.south) -- (cr.north);
  \draw[dashed,black!65] (bd.south) -- (cs.north);
  \node[error] at (2.50,-1.15)
    {\(\norm{b_{\gamma_k}-c_{r_k}^{(k)}}\leq\varepsilon_k\)};
  \node[error] at (9.10,-1.15)
    {\(\norm{b_{\delta_k}-c_{s_k}^{(k)}}\leq\varepsilon_k\)};
  \draw[->,thick] (cs.west) --
    node[above=4pt] {\(\pi_{r_k}^{s_k}\)} (cr.east);
  \node at (5.80,-2.94)
    {Two coordinates of the same \(c^{(k)}\in\widehat A\)};

  \draw[rounded corners=2pt,fill=black!2,draw=black!25]
    (.10,-3.35) rectangle (11.50,-4.65);
  \node at (5.80,-3.70)
    {\(\varepsilon_k=\norm{w_k}+\norm{e_{\gamma_k}}
       +\norm{e_{\delta_k}}\longrightarrow0\)};
  \node at (5.80,-4.28)
    {\(\norm{b_{\gamma_k}-\pi_{r_k}^{s_k}b_{\delta_k}}
       \leq2\varepsilon_k\longrightarrow0\)};
\end{tikzpicture}
  \caption{Compatibility obtained from paired probes. The solid arrow records the exact bonding relation for one \(c^{(k)}\), while the dashed links record norm errors tending to zero.}
  \label{fig:paired-probe-compatibility}
\end{figure}

\begin{proposition}[Common multiplier coefficient]\label{prop:globalization}
  For every \(S\in\Bcal(X)\), there is a unique bounded compatible family \(a=(a_r)_r\in\widehat A\) such that, for \(R=S-M_a\),
  \begin{equation}\label{eq:probe-zero}
    \sup_{\substack{\rank\gamma\geq N\\
    \gamma\in\mathcal P_r^{\rm u}
    \cup\mathcal P_r^{\rm c},\ r\in\N}}
    \norm{Rj_\gamma}
    \longrightarrow0.
  \end{equation}
  Moreover, \(\norm a\leq\norm S\).
\end{proposition}

\begin{proof}
  Fix \(r\).
  Equation~\eqref{eq:bonding-cauchy}, with \(s=r\), shows that, whenever \(\gamma,\delta\in\mathcal P_r^{\rm u}\) have ranks at least \(N\), \(\norm{b_\gamma-b_\delta}\leq\omega_N(S)\).
  By completeness of \(A_r\), we therefore obtain
  \[
  a_r=\lim_{\substack{\gamma\in\mathcal P_r^{\rm u}\\
  \rank\gamma\to\infty}}b_\gamma.
  \]

  Now fix \(r\leq s\), and choose unit probes \(\gamma_N\in\mathcal P_r^{\rm u}\) and \(\delta_N\in\mathcal P_s^{\rm u}\) of ranks at least \(N\).
  Then \(\norm{b_{\gamma_N}-\pi_r^sb_{\delta_N}}\leq\omega_N(S)\to0\).
  By contractivity of \(\pi_r^s\), we have
  \[ \norm{a_r-\pi_r^sa_s} \leq\norm{a_r-b_{\gamma_N}}+\omega_N(S) +\norm{b_{\delta_N}-a_s}\longrightarrow0. \]
  Thus we obtain \(\pi_r^sa_s=a_r\) for \(r\leq s\).
  Along the unit probes of index \(r\), we have
  \[ \norm{a_r}=\lim_\gamma\norm{b_\gamma} \leq\limsup_\gamma\bigl(\norm S+\norm{e_\gamma}\bigr)=\norm S. \]
  Thus we have \(a=(a_r)_r\in\widehat A\) and \(\norm a=\sup_r\norm{a_r}\leq\norm S\).
  If \(\gamma\in\mathcal P_r^{\rm u}\) and \(\rank\gamma\geq N\), let a unit probe \(\delta\) of the same index tend to infinity in \eqref{eq:bonding-cauchy}.
  Since \(b_\delta\to a_r\), we obtain
  \begin{equation}\label{eq:uniform-unit-probe-limit}
    \sup_{\substack{r\in\N,\ \gamma\in\mathcal P_r^{\rm u}\\
    \rank\gamma\geq N}}
    \norm{b_\gamma-a_r}
    \leq\omega_N(S)\longrightarrow0.
  \end{equation}

  Combining \eqref{eq:unit-probe-approximation} and \eqref{eq:uniform-unit-probe-limit}, we obtain, for every unit probe of rank at least \(N\),
  \[
    \begin{aligned}
      \norm{(S-M_a)z_\gamma}
      =\norm{e_\gamma+j_\gamma(b_\gamma-a_r)}\leq\alpha_N(S)+2^{-N}+\omega_N(S)
        \longrightarrow0.
    \end{aligned}
  \]
  The bound is independent of \(r\) and \(\gamma\).

  It remains to pass from these unit vectors to whole probe fibres.
  We shall apply the local-orbit theorem first to a fibre vector and then to its sum with a unit probe of the same detector index.
  The first application makes the coefficient in the other probe fibre small, and the second compares the local multiplier coefficient with \(a_r\).
  Suppose that \eqref{eq:probe-zero} fails.
  There are separated probes \(\delta_k\), with detector indices \(r_k\), and \(u_k\in B_{E_{\delta_k}}\) such that
  \[
    \norm{S j_{\delta_k}u_k
    -j_{\delta_k}\Theta_{\delta_k}(a_{r_k})u_k}
    \geq\varepsilon.
  \]
  After choosing \(\delta_k\), use cofinality to choose \(\gamma_k\in\mathcal P_{r_k}^{\rm u}\) of larger rank.
  Failure of \eqref{eq:probe-zero} persists on every tail, so the next probe satisfying this inequality may be chosen to satisfy
  \[ \rank\delta_k<\rank\gamma_k, \qquad \rank\gamma_k+2<\rank\delta_{k+1}. \]
  Corollary~\ref{cor:probe-ris} makes both \((j_{\delta_k}u_k)\) and \((z_{\gamma_k}+j_{\delta_k}u_k)\) RISs.

  Apply the local-orbit theorem to \((j_{\delta_k}u_k)\).
  There are \(d^{(k)}\in\widehat A\) and \(v_k\in X\), with \(\norm{v_k}\to0\), such that
  \[ S j_{\delta_k}u_k =j_{\delta_k} \Theta_{\delta_k}(d_{r_k}^{(k)})u_k+v_k. \]
  Since \(\gamma_k\ne\delta_k\), we obtain by biorthogonality \(D_{\gamma_k}j_{\delta_k}=0\).
  Hence we obtain
  \begin{equation}\label{eq:cross-probe-zero}
    D_{\gamma_k}S j_{\delta_k}u_k=D_{\gamma_k}v_k,
    \qquad
    \norm{D_{\gamma_k}S j_{\delta_k}u_k}
    \leq\norm{v_k}\longrightarrow0.
  \end{equation}

  Apply the local-orbit theorem to \((z_{\gamma_k}+j_{\delta_k}u_k)\).
  There are \(c^{(k)}\in\widehat A\) and \(w_k\in X\), with \(\norm{w_k}\to0\), such that
  \begin{equation}\label{eq:whole-fibre-pair}
    S(z_{\gamma_k}+j_{\delta_k}u_k)
    =j_{\gamma_k}c_{r_k}^{(k)}
    +j_{\delta_k}
    \Theta_{\delta_k}(c_{r_k}^{(k)})u_k+w_k.
  \end{equation}
  Write \(Sz_{\gamma_k}=j_{\gamma_k}b_{\gamma_k}+e_{\gamma_k}\), where \(\norm{e_{\gamma_k}}\to0\).
  Applying \(D_{\gamma_k}\) to \eqref{eq:whole-fibre-pair} and using \eqref{eq:cross-probe-zero}, we obtain the identity
  \[
    c_{r_k}^{(k)}-b_{\gamma_k}
    =D_{\gamma_k}\bigl(e_{\gamma_k}+v_k-w_k\bigr).
  \]
  By contractivity and \eqref{eq:cross-probe-zero}, we obtain
  \[ \norm{c_{r_k}^{(k)}-b_{\gamma_k}} \leq\norm{e_{\gamma_k}}+\norm{v_k}+\norm{w_k} \longrightarrow0. \]
  Combining this with \eqref{eq:uniform-unit-probe-limit}, we obtain
  \[
    \begin{aligned}
      \norm{c_{r_k}^{(k)}-a_{r_k}}
      &\leq\norm{c_{r_k}^{(k)}-b_{\gamma_k}}
        +\norm{b_{\gamma_k}-a_{r_k}}\\
      &\leq\norm{e_{\gamma_k}}+\norm{v_k}+\norm{w_k}
        +\omega_{\rank\gamma_k}(S)
        \longrightarrow0.
    \end{aligned}
  \]

  Subtracting the unit-probe identity from \eqref{eq:whole-fibre-pair}, we obtain
  \[
    \begin{aligned}
      S j_{\delta_k}u_k
      -j_{\delta_k}\Theta_{\delta_k}(a_{r_k})u_k
      ={}&j_{\gamma_k}(c_{r_k}^{(k)}-b_{\gamma_k})+j_{\delta_k}
          \Theta_{\delta_k}(c_{r_k}^{(k)}-a_{r_k})u_k
          +w_k-e_{\gamma_k}.
    \end{aligned}
  \]
  Since \(\norm{u_k}\leq1\), we obtain from the probe isometries and contractive module actions
  \[
    \begin{aligned}
      \norm{S j_{\delta_k}u_k
        -j_{\delta_k}\Theta_{\delta_k}(a_{r_k})u_k}
      &\leq\norm{c_{r_k}^{(k)}-b_{\gamma_k}}
        +\norm{c_{r_k}^{(k)}-a_{r_k}}
        +\norm{w_k}+\norm{e_{\gamma_k}}\\
      &\leq3\norm{e_{\gamma_k}}+2\norm{v_k}
        +3\norm{w_k}+\omega_{\rank\gamma_k}(S)
        \longrightarrow0.
    \end{aligned}
  \]
  This contradicts the choice of \(\delta_k\) and \(u_k\) and proves \eqref{eq:probe-zero}.

  Finally, suppose that \(a,a'\in\widehat A\) both satisfy \eqref{eq:probe-zero}.
  Fix \(r\) and let \(\gamma\in\mathcal P_r^{\rm u}\) tend to infinity.
  By the probe isometry, we have
  \[
    \begin{aligned}
      \norm{a_r-a_r'}
      &=\norm{(M_a-M_{a'})z_\gamma}\\
      &\leq\norm{(S-M_a)j_\gamma}
        +\norm{(S-M_{a'})j_\gamma}\longrightarrow0.
    \end{aligned}
  \]
  Thus we have \(a_r=a_r'\) for every \(r\), and \(a=a'\).
\end{proof}

\subsection{The compact remainder}

The reduction to RISs in Lemma~\ref{lem:compactness-test} follows \cite[Proposition~5.11]{ArgyrosHaydon2011}.
We combine it below with the core probes.

\begin{lemma}[Compactness of the remainder]\label{lem:no-ghost}
  Let \(R\in\Bcal(X)\) satisfy \eqref{eq:probe-zero}.
  Then \(R\) is compact.
\end{lemma}

\begin{proof}
  By Lemma~\ref{lem:compactness-test}, it is enough to prove that \(R\) sends every RIS to zero in norm.
  Suppose, after passing to a subsequence, that a \(C\)-RIS \((x_k)\), with \(C\geq1\), satisfies \(\norm{Rx_k}\geq\varepsilon>0\).
  We shall add a core-probe vector \(x_k^{\rm pr}\) to \(x_k\) and construct \(g_k^*\in2B_{X^*}\) such that
  \[
    g_k^*\bigl(M_a(x_k+x_k^{\rm pr})\bigr)=0
    \quad(a\in\widehat A),
    \qquad
    \operatorname{Re}g_k^*
      \bigl(R(x_k+x_k^{\rm pr})\bigr)\geq\varepsilon/3
  \]
  for all large \(k\).
  These two estimates will contradict the local-orbit theorem.
  Recall that \(e_{\eta,\psi}^*(x)=\langle\psi,U_\eta x\rangle\).
  These scalarized ambient coordinates norm \(X\), so we may choose \(e_{\eta_k,\psi_k}^*\) with \(\abs{e_{\eta_k,\psi_k}^*(Rx_k)}\geq\varepsilon/2\).
  Changing its scalar phase, we obtain
  \begin{equation}\label{eq:large-remainder}
    \operatorname{Re}e_{\eta_k,\psi_k}^*(Rx_k)
    \geq\frac{\varepsilon}{2}.
  \end{equation}
  Define \(q_k\in A_{r(\eta_k)}^*\) by
  \[ q_k(b) =\left\langle\psi_k, \Theta_{\eta_k}(b)U_{\eta_k}x_k\right\rangle \qquad(b\in A_{r(\eta_k)}). \]
  By \eqref{eq:preadjoint-product} and \eqref{eq:error}, this is the finite-dimensional representative of the annihilation error:
  \[ q_k=\Theta_{\eta_k,*}(\psi_k)U_{\eta_k}x_k,\qquad J_{r(\eta_k)}^\infty q_k=\Delta(e_{\eta_k,\psi_k}^*,x_k). \]
  By contractivity, we have
  \[
    \begin{aligned}
      \abs{q_k(b)}
      \leq\norm{\psi_k}\,
        \norm{\Theta_{\eta_k}(b)}\,
        \norm{U_{\eta_k}x_k}\leq\norm b\,\norm{x_k}
        \leq C\norm b
        \quad(b\in A_{r(\eta_k)}).
    \end{aligned}
  \]
  Taking the supremum over \(b\in B_{A_{r(\eta_k)}}\), we obtain \(\norm{q_k}\leq C\).

  Pass recursively to a subsequence.
  After \(x_k\) and \(\eta_k\) have been chosen, the convergence of the FDD decomposition gives \(P_{(N,\infty)}Rx_k\to0\) as \(N\to\infty\).
  Hence we may take \(N_k\) so that
  \[ N_k>\max\{\max\ran x_k,\rank\eta_k\}, \qquad \norm{P_{(N_k,\infty)}Rx_k}<2^{-k}. \]
  Choose a core probe \(\gamma_k\) after \(N_k\), of detector index \(s_k\geq r(\eta_k)\), and then choose the next original RIS member so that its support and retained RIS index both exceed \(\rank\gamma_k\).
  After relabelling the selected vectors and their indices \((j_k)\), we obtain
  \[ x_k<\gamma_k<x_{k+1},\qquad j_{k+1}>\rank\gamma_k. \]
  Put \(x_k^{\rm pr}=j_{\gamma_k}J_{r(\eta_k)}^{s_k}q_k\)
  and \(y_k=x_k+x_k^{\rm pr}\), and define
  \[
    g_k^*=e_{\eta_k,\psi_k}^*
      -e_{\gamma_k,1_{A_{s_k}}}^*.
  \]
  The predual connecting maps \(J_r^s=(\pi_r^s)^*\) and the probe injections are isometries, so we have
  \[
    \norm{x_k^{\rm pr}}=\norm{J_{r(\eta_k)}^{s_k}q_k}=\norm{q_k}\leq C.
  \]
  Since \(y_k=x_k+x_k^{\rm pr}\) and \(\norm{x_k}\leq C\), we obtain
  \( \norm{y_k}\leq2C. \)
  Apply Corollary~\ref{cor:probe-ris} to \(x_k/C\) and the probe vector \(x_k^{\rm pr}/C\).
  For every scalarized evaluation \(e_{\alpha,\psi}^*\) of level \(0<h<j_k\), we obtain
  \[
    \begin{aligned}
      \abs{e_{\alpha,\psi}^*(y_k)}
      \leq\abs{e_{\alpha,\psi}^*(x_k)}
        +\abs{e_{\alpha,\psi}^*(x_k^{\rm pr})}\leq\frac{C}{m_h}+\frac{C}{m_h}
        =\frac{2C}{m_h}.
    \end{aligned}
  \]
  Since \(\max\ran y_k\leq\rank\gamma_k<j_{k+1}\), we conclude that \((y_k)\) is a \(2C\)-RIS.
  In Figure~\ref{fig:core-mirror}, the two evaluation arrows give the same scalar \(q_k(a_{r(\eta_k)})\) on the corresponding multiplier images, as verified below.

  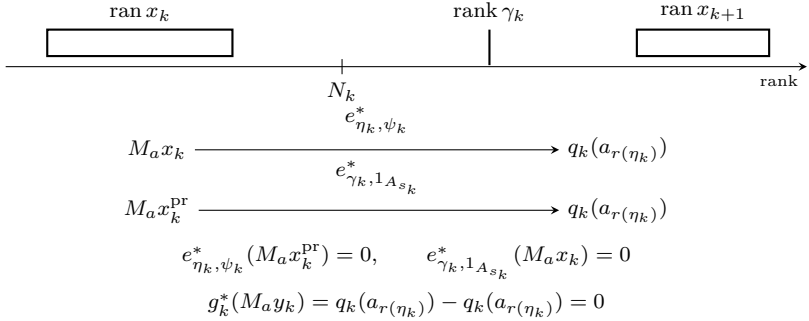
\begin{figure}[!htbp]
    \centering
\begin{tikzpicture}[x=1cm,y=1cm,>=stealth,
  every node/.style={font=\small,align=center}]
  \draw[->] (0,0) -- (10.60,0)
    node[below left,font=\scriptsize] {rank};
  \draw[thick] (.55,.14) rectangle (3.00,.49);
  \node[above] at (1.78,.49) {\(\operatorname{ran}x_k\)};
  \draw (4.45,.10) -- (4.45,-.10);
  \node[below] at (4.45,-.10) {\(N_k\)};
  \draw[thick] (6.40,.02) -- (6.40,.49);
  \node[above] at (6.40,.49) {\(\rank\gamma_k\)};
  \draw[thick] (8.35,.14) rectangle (10.10,.49);
  \node[above] at (9.23,.49) {\(\operatorname{ran}x_{k+1}\)};

  \node (original) at (2.00,-1.10) {\(M_ax_k\)};
  \node (probe) at (2.00,-1.90) {\(M_ax_k^{\rm pr}\)};
  \node (originalvalue) at (8.10,-1.10) {\(q_k(a_{r(\eta_k)})\)};
  \node (probevalue) at (8.10,-1.90) {\(q_k(a_{r(\eta_k)})\)};
  \draw[->] (original.east) -- node[above=3pt]
    {\(e_{\eta_k,\psi_k}^*\)} (originalvalue.west);
  \draw[->] (probe.east) -- node[above=3pt]
    {\(e_{\gamma_k,1_{A_{s_k}}}^*\)} (probevalue.west);
  \node at (5.30,-2.58)
    {\(e_{\eta_k,\psi_k}^*(M_ax_k^{\rm pr})=0,
       \qquad e_{\gamma_k,1_{A_{s_k}}}^*(M_ax_k)=0\)};
  \node at (5.30,-3.15)
    {\(g_k^*(M_ay_k)
       =q_k(a_{r(\eta_k)})-q_k(a_{r(\eta_k)})=0\)};
\end{tikzpicture}
    \caption{A core probe beyond \(N_k\) gives equal scalar values and vanishing cross terms.}
    \label{fig:core-mirror}
  \end{figure}

  For every \(a\in\widehat A\), we obtain from the scalarization identity
  \[
    \begin{aligned}
      e_{\eta_k,\psi_k}^*(M_ax_k)
      &=q_k(a_{r(\eta_k)}),\\
      e_{\gamma_k,1}^*(M_ax_k^{\rm pr})
      &=(J_{r(\eta_k)}^{s_k}q_k)(a_{s_k})=q_k(\pi_{r(\eta_k)}^{s_k}a_{s_k})
        =q_k(a_{r(\eta_k)}).
    \end{aligned}
  \]
  Since \(\gamma_k\) is a core probe, we have \(e_{\gamma_k,1}^*=d_{\gamma_k,1}^*\).
  The multiplier commutes with the FDD projections, so we have
  \[
    P_{[1,N_k]}M_ax_k=M_ax_k,\qquad P_{[1,N_k]}M_ax_k^{\rm pr}=0.
  \]
  Since \(\rank\eta_k<N_k\), we obtain from the evaluation analysis that
  \[
    e_{\eta_k,\psi_k}^*=e_{\eta_k,\psi_k}^*P_{[1,N_k]}.
  \]
  As \(\rank\gamma_k>N_k\), we also have \(e_{\gamma_k,1}^*P_{[1,N_k]}=0\).
  It follows that
  \[ e_{\eta_k,\psi_k}^*(M_ax_k^{\rm pr})=0, \qquad e_{\gamma_k,1}^*(M_ax_k)=0. \]
  Therefore, we obtain
  \begin{equation}\label{eq:mirror-annihilation}
    g_k^*(M_ay_k)
    =q_k(a_{r(\eta_k)})
    -q_k(\pi_{r(\eta_k)}^{s_k}a_{s_k})=0
    \qquad(a\in\widehat A).
  \end{equation}

  By the choice of \(N_k\), we have
  \[ \abs{e_{\gamma_k,1}^*(Rx_k)} =\abs{e_{\gamma_k,1}^* (P_{(N_k,\infty)}Rx_k)} \leq\norm{P_{(N_k,\infty)}Rx_k}<2^{-k}. \]
  By \eqref{eq:probe-zero}, we have
  \[ \norm{Rx_k^{\rm pr}}\leq C\norm{Rj_{\gamma_k}}\longrightarrow0. \]
  Expanding \(g_k^*(Ry_k)\) and using contractivity of the two scalarized evaluations and \eqref{eq:large-remainder}, we obtain, for all large \(k\),
  \[
    \begin{aligned}
      \operatorname{Re}g_k^*(Ry_k)
      &=\operatorname{Re}e_{\eta_k,\psi_k}^*(Rx_k)
        -\operatorname{Re}e_{\gamma_k,1}^*(Rx_k)
        +\operatorname{Re}g_k^*(Rx_k^{\rm pr})\\
      &\geq\frac{\varepsilon}{2}
        -\abs{e_{\gamma_k,1}^*(Rx_k)}
        -\norm{g_k^*}\,\norm{Rx_k^{\rm pr}}\\
      &\geq\frac{\varepsilon}{2}
        -2^{-k}-2C\norm{Rj_{\gamma_k}}
        \geq\frac{\varepsilon}{3}.
    \end{aligned}
  \]
  Since \(\norm{g_k^*}\leq2\), we obtain from \eqref{eq:mirror-annihilation}, for every \(a\in\widehat A\),
  \[
    \begin{aligned}
      2\norm{Ry_k-M_ay_k}
      \geq\abs{g_k^*(Ry_k-M_ay_k)}=\abs{g_k^*(Ry_k)}
        \geq\operatorname{Re}g_k^*(Ry_k)
        \geq\frac{\varepsilon}{3}.
    \end{aligned}
  \]
  Therefore, we know
  \[ \dist\bigl(Ry_k,\{M_ay_k:a\in\widehat A\}\bigr) \geq\frac{\varepsilon}{6}. \]
  This contradicts Theorem~\ref{thm:local-orbit}.
  Hence \(R\) is compact.
\end{proof}

\begin{theorem}[Global multiplier classification]\label{thm:global-classification}
  For every \(S\in\Bcal(X)\), there is a unique \(a\in\widehat A\) such that \(S-M_a\) is compact.
  Moreover, we have
  \[ \norm{[M_a]}_{\Cal(X)}=\norm a. \]
\end{theorem}

\begin{proof}
  Proposition~\ref{prop:globalization} gives \(a\in\widehat A\) such that \(R=S-M_a\) satisfies \eqref{eq:probe-zero}.
  Lemma~\ref{lem:no-ghost} makes \(R\) compact.

  For the lower bound on the essential norm, fix \(r\) and choose a sufficiently separated cofinal sequence \((\gamma_k)\subseteq\mathcal P_r^{\rm u}\).
  Corollary~\ref{cor:probe-ris} makes \((z_{\gamma_k})\) a \(1\)-RIS, and Corollary~\ref{cor:ris-weakly-null} makes it weakly null.
  Compact operators send weakly null sequences to norm-null sequences, so we have \(Kz_{\gamma_k}\to0\) for every compact \(K\), whereas \(M_az_{\gamma_k}=j_{\gamma_k}a_r\).
  Since \(\norm{z_{\gamma_k}}=1\), we obtain from the probe isometry
  \[
    \begin{aligned}
      \norm{M_a-K}
      &\geq\limsup_k\norm{(M_a-K)z_{\gamma_k}}\\
      &\geq\limsup_k
        \bigl(\norm{j_{\gamma_k}a_r}-\norm{Kz_{\gamma_k}}\bigr)
        =\norm{a_r}.
    \end{aligned}
  \]
  Take the supremum over \(r\) and the infimum over \(K\in\Kcal(X)\).
  Combining this with \(\norm{M_a}=\norm a\), we obtain
  \[ \norm a=\sup_r\norm{a_r} \leq\norm{[M_a]}\leq\norm{M_a}=\norm a. \]
  Thus we have \(\norm{[M_a]}=\norm a\).
  If both \(S-M_a\) and \(S-M_b\) are compact, then \(M_{a-b}=(S-M_b)-(S-M_a)\) is compact.
  The essential-norm identity gives \(\norm{a-b}=\norm{[M_{a-b}]}=0\), and hence \(a=b\).
\end{proof}

\begin{proof}[Proof of Theorem~\ref{thm:main}]
  By Theorem~\ref{thm:datum}, we have a separable block BD space \(X\) and an isometric unital representation \(a\mapsto M_a\).
  Compose this representation with the quotient map to define
  \(\Phi:\widehat A\to\Cal(X)\) by \(\Phi(a)=[M_a]\).
  For \(a,b\in\widehat A\), we have
  \(\Phi(ab)=[M_aM_b]=\Phi(a)\Phi(b)\) and \(\Phi(1)=[I_X]\).
  Theorem~\ref{thm:global-classification} gives
  \[
    \norm{\Phi(a)}=\norm a,
    \qquad
    [S]=[M_a]=\Phi(a)
    \quad\text{for some }a\in\widehat A
    \quad(S\in\Bcal(X)).
  \]
  Thus \(\Phi\) is an isometric unital algebra isomorphism onto \(\Cal(X)\).
  The same theorem gives the unique decomposition in \eqref{eq:classification} and the essential-norm identity \eqref{eq:essential-norm}.
\end{proof}

Corollary~\ref{cor:filtered-realization} follows by applying Theorem~\ref{thm:main} to the quotient tower from Lemma~\ref{lem:detector-quotient-tower}.
Proposition~\ref{prop:bounded-detector-density} gives coordinatewise approximation by images of bounded sequences.
Its injectivity and isometry follow from Proposition~\ref{prop:canonical-detector-map}, while its surjectivity when closed norm balls are compact in \(\tau_{(I_r)}\) follows from Theorem~\ref{thm:detector-compact-surjectivity}.

  \part{Examples and applications}


\section{Examples}
\label{sec:examples}

We apply Theorem~\ref{thm:main} to finite-dimensional quotients, dual operator algebras, and convolution algebras.
Recall that \(\ilim(A_r,\pi_r^s)\) consists of bounded compatible sequences, with norm \(\sup_r\norm{a_r}\).
For a decreasing sequence \((I_r)\) of finite-codimensional ideals in \(B\), the detection envelope \(\Env^b_{(I_r)}(B)\) is the bounded inverse limit of the quotients \(B/I_r\).
Theorem~\ref{thm:detector-compact-surjectivity} identifies this envelope isometrically with \(B\) when the ideals separate points and every closed ball is compact for the quotient seminorms.
In each example below we identify the whole limit, including its norm and multiplication.

Many of the concrete examples below, both within and beyond our initially planned classes, were suggested by GPT-5.6 Sol (OpenAI) and GPT-6 Astra (OpenAI).
These include the incidence-algebra applications.
For details, see the \hyperref[sec:ai-assistance]{AI assistance statement} at the end of the paper.

We shall repeatedly use the following observation.
If \(q:Y\to Z\) is contractive and every \(z\in Z\) has a lift \(y_z\) of the same norm, then we have
\[ \norm z\leq\inf_{qy=z}\norm y\leq\norm{y_z}=\norm z. \]
Thus \(q\) is a metric quotient.

\subsection{Finite detection and first consequences}

\begin{corollary}[Finite-dimensional algebras]
  \label{cor:finite-dimensional-algebras}
  Let \(F\) be a finite-dimensional unital Banach algebra with \(\norm{1_F}=1\).
  Then there is a separable Banach space \(X_F\) such that \(\Cal(X_F)\simeq F\) isometrically as unital Banach algebras.
\end{corollary}

This follows by applying Theorem~\ref{thm:main} to the constant system \((F,I_F)\).

\begin{corollary}[Finite incidence algebras]
  \label{cor:finite-incidence}
  Let \(P\) be a nonempty finite partially ordered set.
  Equip its incidence algebra (see \cite[Section~3]{Rota1964})
  \[ \operatorname{IA}(P) =\left\{a=(a_{xy})_{x,y\in P}: a_{xy}=0\text{ whenever }x\nleq y\right\} \]
  with incidence multiplication and the row norm
  \(\norm a_{\rm row}=\max_{x\in P}\sum_{y\geq x}\abs{a_{xy}}\).
  Then \(\operatorname{IA}(P)\) is isometrically a Calkin algebra.
\end{corollary}

\begin{proof}
  Under its canonical matrix representation on \(\ell_\infty(P)\), the algebra \(\operatorname{IA}(P)\) is a finite-dimensional unital subalgebra of \(\Bcal(\ell_\infty(P))\).
  For \(z\in\ell_\infty(P)\), we have
  \[ \abs{(az)_x} \leq\sum_{y\geq x}\abs{a_{xy}}\abs{z_y} \leq\norm a_{\rm row}\norm z_\infty, \]
  so we obtain \(\norm a_{\Bcal(\ell_\infty(P))}\leq\norm a_{\rm row}\).
  For a fixed \(x\), choose \(z_y\) of modulus one so that \(a_{xy}z_y=\abs{a_{xy}}\).
  Then we have \(\norm{az}_\infty\geq\abs{(az)_x}=\sum_{y\geq x}\abs{a_{xy}}\).
  Taking the maximum over \(x\), we obtain the reverse inequality.
  Thus the operator norm is this row norm.
  In particular, we have \(\norm{1_{\operatorname{IA}(P)}}_{\rm row}=1\).
  Apply Corollary~\ref{cor:finite-dimensional-algebras}.
\end{proof}

Corollary~\ref{cor:finite-incidence} answers \cite[Question~5.7]{Acuaviva2025Posets} affirmatively.

\begin{proposition}[Separation by finite-dimensional representations]
  \label{prop:weak-star-rfd}
  Let \(B\) be a unital Banach algebra with \(\norm{1_B}=1\), and suppose that \(B=E^*\) isometrically.
  Suppose that there are finite-dimensional algebras \(F_n\) and weak-star-to-norm continuous unital homomorphisms
  \[ \rho_n:B\longrightarrow F_n\quad(n\in\N), \qquad \bigcap_{n=0}^\infty\ker\rho_n=\{0\}. \]
  Then \(B\) is isometrically its bounded detection envelope and is isometrically a Calkin algebra.
\end{proposition}

\begin{proof}
  Put \(I_r=\bigcap_{n=0}^r\ker\rho_n\), and give \(B/I_r\) its quotient norm.
  Each \(I_r\) is a weak-star closed finite-codimensional two-sided ideal.
  Since the maps \(\rho_n\) are unital and jointly faithful, these ideals are proper and satisfy \(\bigcap_r I_r=\{0\}\).
  Assertion~\textup{(ii)} of Theorem~\ref{thm:finite-detection-duality} gives \(B\simeq\Env^b_{(I_r)}(B)\) isometrically.
  Apply Theorem~\ref{thm:main}.
  Here \(B/I_r\) carries its quotient norm, which may differ from the norm inherited by a representation range.
\end{proof}

\begin{lemma}[Finite-dimensional compressions]
  \label{lem:finite-triangular-corners}
  Let \(H\) be a nonzero separable Hilbert space, and let \(\mathcal A\subseteq\Bcal(H)\) be a weak-star closed unital algebra.
  Suppose that there are increasing finite-rank orthogonal projections \(P_n\) converging strongly to \(I_H\) such that, for every \(n\),
  \[ P_n\mathcal A(I_H-P_n)=0 \qquad\text{or}\qquad (I_H-P_n)\mathcal A P_n=0. \]
  Then \(\mathcal A\) admits a countable separating family of weak-star continuous finite-dimensional unital representations and is isometrically a Calkin algebra.
\end{lemma}

\begin{proof}
  Since \(\mathcal A\) is weak-star closed in \(\Bcal(H)\), it is isometrically a dual space, with predual the quotient of the trace class by its preannihilator.
  Discarding finitely many initial zero projections and reindexing, we may assume that \(P_nH\ne\{0\}\) for every \(n\).
  Define
  \[ \rho_n:\mathcal A\longrightarrow\Bcal(P_nH), \qquad \rho_n(T)=P_nT|_{P_nH}. \]
  This map is unital, has finite-dimensional range, and is weak-star continuous, since each of its matrix entries is a matrix coefficient of \(T\).
  Under either assumption, we have
  \[ \rho_n(ST)-\rho_n(S)\rho_n(T) =P_nS(I_H-P_n)TP_n|_{P_nH}=0 \qquad(S,T\in\mathcal A). \]
  Thus \(\rho_n\) is a homomorphism.

  For \(T\in\mathcal A\) and \(\xi\in H\), we obtain from strong convergence that
  \[ \norm{P_nTP_n\xi-T\xi} \leq\norm T\norm{P_n\xi-\xi}+\norm{(P_n-I_H)T\xi} \longrightarrow0. \]
  Since \(\norm{\rho_n(T)}\leq\norm T\), we deduce that
  \[ \norm T=\sup_n\norm{\rho_n(T)}, \qquad \bigcap_n\ker\rho_n=\{0\}. \]
  Proposition~\ref{prop:weak-star-rfd} applies.
  The finite quotients carry their quotient norms, and the displayed equality alone does not identify an individual quotient norm with its compression norm.
\end{proof}

In the next proposition and the convolution examples below, \(\Delta:C\to C\otimes_{\min}C\) denotes a coproduct. Here \(\otimes_{\min}\) denotes the minimal \(C^*\)-tensor product \cite[Chapter~12]{Paulsen2003}.

\begin{proposition}[Realization from finite-dimensional subcoalgebras]\label{prop:coalgebra}
  Let \(C\) be a unital \(C^*\)-algebra with a unital completely positive map \(\Delta:C\to C\otimes_{\min}C\) and a character \(\varepsilon:C\to\mathbb C\) satisfying
  \[ (\Delta\otimes I)\Delta=(I\otimes\Delta)\Delta, \qquad (\varepsilon\otimes I)\Delta=I=(I\otimes\varepsilon)\Delta. \]
  Suppose that there are increasing finite-dimensional subspaces
  \[ C_0\subseteq C_1\subseteq\cdots\subseteq C, \qquad \overline{\bigcup_rC_r}=C, \]
  such that \(1_C\in C_r\) and \(\Delta(C_r)\subseteq C_r\otimes C_r\).
  Then the convolution algebra \(C^*\), with
  \[ (\mu\star\nu)(x) =(\mu\otimes\nu)(\Delta x), \]
  is isometrically its bounded detection envelope and is a Calkin algebra.
\end{proposition}

\begin{proof}
  A unital completely positive map is contractive \cite[Proposition~3.2]{Paulsen2003}, and the norm of a product functional is the product of the norms \cite[Chapter~12]{Paulsen2003}.
  Hence, for \(x\in C\), we have
  \[ \abs{(\mu\star\nu)(x)} =\abs{(\mu\otimes\nu)(\Delta x)} \leq\norm{\mu\otimes\nu}\norm{\Delta x} \leq\norm\mu\norm\nu\norm x, \]
  and therefore we obtain \(\norm{\mu\star\nu}\leq\norm\mu\norm\nu\).
  Coassociativity gives associativity of \(\star\), and the two counit identities make \(\varepsilon\) its unit.
  As \(\varepsilon\) is a character on a unital \(C^*\)-algebra, we have \(\norm\varepsilon=1\).
  Thus \(C^*\) is a unital Banach algebra.

  Put \(I_r=C_r^\perp\).
  If \(\mu\in I_r\), \(\nu\in C^*\), and \(x\in C_r\), write \(\Delta x=\sum_{j=1}^m y_j\otimes z_j\) with \(y_j,z_j\in C_r\).
  Then we have
  \[ (\mu\star\nu)(x)=\sum_j\mu(y_j)\nu(z_j)=0, \qquad (\nu\star\mu)(x)=\sum_j\nu(y_j)\mu(z_j)=0. \]
  Hence \(I_r\) is a weak-star closed finite-codimensional two-sided ideal.
  By the Hahn--Banach theorem, every element of \(C_r^*\) has a norm-preserving extension to \(C^*\).
  Thus restriction \(C^*\to C_r^*\) is a metric quotient with kernel \(I_r\).
  Density of \(\bigcup_rC_r\) gives \(\bigcap_rI_r=\{0\}\).
  Assertion~\textup{(iii)} of Theorem~\ref{thm:finite-detection-duality} therefore applies with \(E=C\), and Theorem~\ref{thm:main} completes the proof.
\end{proof}

\begin{remark}
  Assertion~\textup{(iii)} of Theorem~\ref{thm:finite-detection-duality} is stated in terms of finite-dimensional predual subspaces and their annihilator ideals, without assuming a coproduct or positivity.
  Here complete positivity gives contractivity and the tensor ampliations on the minimal tensor product \cite[Proposition~3.2 and Theorem~12.3]{Paulsen2003}.
  A merely contractive linear map \(C\to C\otimes_{\min}C\) does not provide the minimal-tensor ampliations.
  Every unital \(^*\)-homomorphism is unital completely positive \cite[p.~28]{Paulsen2003}.
\end{remark}

The subspaces \(C_r\) in Proposition~\ref{prop:coalgebra} need not be subalgebras.
The finite coalgebra condition ensures that each annihilator \(C_r^\perp\) is a two-sided ideal.

\subsection{\texorpdfstring{Sequence, convolution, and incidence algebras}{Sequence, convolution, and incidence algebras}}

\begin{corollary}[\(\ell_\infty\)-sums of finite algebras]
  \label{cor:finite-products}
  Let \((F_k)_{k\geq0}\) be finite-dimensional unital Banach algebras with \(\norm{1_{F_k}}=1\).
  Then
  \( \left(\bigoplus_{k=0}^\infty F_k\right)_{\ell_\infty} \)
  is isometrically a Calkin algebra.
\end{corollary}

\begin{proof}
  Put \(A_r=F_0\oplus_\infty\cdots\oplus_\infty F_r\), with coordinatewise multiplication.
  Then we have \(\norm{ab}\leq\norm a\norm b\).
  Coordinate restriction \(\sigma_r^s:A_s\to A_r\) is a contractive unital homomorphism.
  Zero extension preserves the norm, so \(\sigma_r^s\) is a metric quotient.
  A compatible sequence is determined by \((a_k)_{k\geq0}\), and we have
  \[ a^{(r)}=(a_0,\ldots,a_r),\qquad \sup_r\norm{a^{(r)}}=\sup_r\max_{k\leq r}\norm{a_k} =\sup_k\norm{a_k}. \]
  Its product has \(k\)-th coordinate \(a_kb_k\).
  Thus the bounded limit is the displayed \(\ell_\infty\)-sum as a normed algebra, and Theorem~\ref{thm:main} applies.
\end{proof}

\begin{corollary}[The algebra \(\ell_\infty\)]\label{ex:ell-infinity}
  There is a separable Banach space \(X_\infty\) with \(\Cal(X_\infty)\simeq C(\beta\N)\simeq\ell_\infty\) isometrically as unital Banach algebras.
\end{corollary}

This follows directly from Corollary~\ref{cor:finite-products} with \(F_k=\K\) and the isometric algebra identification \(C(\beta\N)\simeq\ell_\infty\)  \cite[Section~V.6]{Conway1990}.
Corollary~\ref{ex:ell-infinity} answers Motakis's question on the existence of a nonseparable \(C(K)\) Calkin algebra \cite[Section~9, Problem~2]{Motakis2024}.

For the standard notions of unconditional and boundedly complete bases, see \cite[Sections~3.1 and~3.2]{AlbiacKalton2016}.

\begin{corollary}[Unconditional sums of algebras]
  \label{cor:unconditional-algebra-sums}
  Let \(U\) have a normalized \(1\)-unconditional boundedly complete Schauder basis \((u_n)_{n\geq1}\).
  For each \(n\), let \(B_n=\ilim(B_{n,r},\pi_{n,r}^s)\) isometrically as Banach algebras, where the \(B_{n,r}\) are finite-dimensional and the bonding maps are metric-quotient algebra homomorphisms.
  Write \(J=(\bigoplus_{n\geq1}B_n)_U\) for the coordinate sum, defined by
  \begin{equation}\label{eq:unconditional-algebra-sum}
    b=(b_n)\in J \ \Longleftrightarrow\ \sum_n\norm{b_n}u_n\text{ converges in }U,\qquad \norm b_U =\left\lVert\sum_n\norm{b_n}u_n\right\rVert_U.
  \end{equation}
  Give \(J\) coordinatewise multiplication and \(J\oplus\K I\) the standard unitization product and norm
  \begin{equation}\label{eq:unconditional-sum-unitization}
    (b+\lambda I)(c+\mu I)=bc+\lambda c+\mu b+\lambda\mu I,\qquad \norm{b+\lambda I}=\norm b_U+\abs\lambda.
  \end{equation}
  Then there is a separable Banach space \(X\) such that 
  \( \Cal(X)\simeq\left(\bigoplus_{n\geq1}B_n\right)_U\oplus\K I \) 
  isometrically as unital Banach algebras.
\end{corollary}

\begin{proof}
  By normalization and unconditionality, we have
  \(
  \sup_n\norm{b_n}\leq\norm b_U.
  \)
  And we then obtain
  \[
    \left\lVert\sum_{n=p}^q\norm{b_nc_n}u_n\right\rVert \leq\norm b_U\left\lVert\sum_{n=p}^q\norm{c_n}u_n\right\rVert.
  \]
  Thus we have \(bc\in J\) and \(\norm{bc}_U\leq\norm b_U\norm c_U\).
  The identification below also proves completeness.
  For \(r\geq1\), take the finite-dimensional unital Banach algebra
  \begin{equation}\label{eq:unconditional-sum-tower}
    A_r=\left(\bigoplus_{n=1}^rB_{n,r}\right)_U\oplus\K I,\qquad \pi_r^s\bigl((b_n)_{n\leq s}+\lambda I\bigr) =(\pi_{n,r}^sb_n)_{n\leq r}+\lambda I.
  \end{equation}
  Its norm and product are those of \eqref{eq:unconditional-sum-unitization}, with finite sums.
  Hence we have \(\norm{1_{A_r}}=1\), and \(\pi_r^s\) is a contractive unital homomorphism.
  Given \((b_n)_{n\leq r}+\lambda I\in A_r\), choose \(c_n\in B_{n,s}\) with \(\pi_{n,r}^sc_n=b_n\) and \(\norm{c_n}=\norm{b_n}\), and put \(c_n=0\) for \(r<n\leq s\).
  These minimum-norm lifts exist because \(B_{n,s}\) is finite-dimensional.
  The resulting lift has norm
  \[ \abs\lambda+\left\lVert\sum_{n=1}^s\norm{c_n}u_n\right\rVert =\abs\lambda+\left\lVert\sum_{n=1}^r\norm{b_n}u_n\right\rVert. \]
  Thus every \(\pi_r^s\) is a metric quotient.

  Let \((x_r)\) be a compatible sequence with \(C=\sup_r\norm{x_r}<\infty\).
  The bonding maps preserve the scalar coordinate, so we may write \(x_r=(b_{n,r})_{n\leq r}+\lambda I\) with one fixed \(\lambda\).
  For fixed \(n\), the coordinates \((b_{n,r})_{r\geq n}\) form a compatible sequence on a cofinal set of stages.
  Applying the bonding maps defines the missing earlier coordinates, so this sequence determines \(b_n\in B_n\), with \(\norm{b_n}=\sup_{r\geq n}\norm{b_{n,r}}\).
  These coordinate norms increase with \(r\).
  For every \(N\), we have
  \[
  \left\lVert\sum_{n=1}^N\norm{b_n}u_n\right\rVert =\lim_{r\to\infty}\left\lVert\sum_{n=1}^N\norm{b_{n,r}}u_n\right\rVert \leq C-\abs\lambda.
  \]
  By bounded completeness, we obtain \(b=(b_n)\in J\).
  Conversely, the coordinate images of every \(b\in J\) give a bounded compatible sequence.
  Combining the preceding estimate with contractivity, we obtain
  \[
    \sup_r\left(\abs\lambda+
    \left\lVert\sum_{n=1}^r\norm{b_{n,r}}u_n\right\rVert\right)
    =\abs\lambda+\norm b_U.
  \]
  Hence we obtain \(J\oplus\K I\simeq\ilim(A_r,\pi_r^s)\) isometrically and multiplicatively.
  Theorem~\ref{thm:main} applies.
\end{proof}

\begin{corollary}[Unconditional finite-dimensional decompositions]
  \label{cor:ufdd-unitization}
  Let \(J\) be a Banach algebra with a \(1\)-unconditional boundedly complete finite-dimensional decomposition \((F_n)_{n\geq1}\) into two-sided ideals.
  Then \(J\oplus\K I\), with the unitization product in \eqref{eq:unconditional-sum-unitization} and norm \(\norm{x+\lambda I}=\norm x+\abs\lambda\), is isometrically a Calkin algebra of a separable Banach space.
\end{corollary}

\begin{proof}
  Since \(F_nF_m\subseteq F_n\cap F_m=\{0\}\) for \(n\ne m\), we know that each partial-sum projection \(P_r\) is a contractive algebra homomorphism.
  Put
  \[
    A_r=P_rJ\oplus\K I,\qquad \pi_r^s(x+\lambda I)=P_rx+\lambda I\quad(r\leq s).
  \]
  Each \(A_r\) is finite-dimensional and has unit norm one.
  The inclusion \(A_r\hookrightarrow A_s\) is an isometric right inverse of \(\pi_r^s\), so these maps are unital metric-quotient homomorphisms.
  A bounded compatible sequence has the form \((x_r+\lambda I)_r\), with \(P_rx_s=x_r\) for \(r\leq s\).
  By bounded completeness, we obtain \(x\in J\) with \(P_rx=x_r\) and \(\sup_r\norm{x_r+\lambda I}=\norm x+\abs\lambda\).
  Conversely, every \(x+\lambda I\in J\oplus\K I\) gives such a sequence.
  Hence \(J\oplus\K I\simeq\ilim(A_r,\pi_r^s)\) isometrically as unital Banach algebras, and Theorem~\ref{thm:main} applies.
\end{proof}

The same argument applies to any boundedly complete FDD into two-sided ideals whose initial projections are contractive.

\begin{proposition}[Unital summands]
  \label{prop:unconditional-sum-converse}
  Let \((u_n)\) be a normalized \(1\)-unconditional Schauder basis of \(U\), and let each \(B_n\) be a nonzero unital Banach algebra with \(\norm{1_{B_n}}=1\).
  Give \(A=(\bigoplus_n B_n)_U\oplus\K I\) the norm and product in \eqref{eq:unconditional-sum-unitization}.
  Then \(A\) is isometrically isomorphic, as a unital Banach algebra, to the bounded inverse limit of a countable system satisfying \eqref{eq:inverse-system}--\eqref{eq:metric-quotient} if and only if \((u_n)\) is boundedly complete and each \(B_n\) has such a representation.
\end{proposition}

\begin{proof}
  For sufficiency, use the tower \eqref{eq:unconditional-sum-tower}.
  For necessity, write \(A=\ilim(A_r,\pi_r^s)\), with coordinate maps \(q_r:A\to A_r\).
  Each \(q_r\) is onto: from \(a_r\in A_r\), successively take minimum-norm lifts to obtain
  \[ \pi_j^{j+1}a_{j+1}=a_j,\qquad \norm{a_{j+1}}=\norm{a_j}\quad(j\geq r). \]
  The lower coordinates are fixed by the bonding maps, so this is a bounded compatible sequence of norm \(\norm{a_r}\).
  Let \(e_n\in A\) have \(n\)-th coordinate \(1_{B_n}\) and all other coordinates, including the scalar one, zero.
  It is a central idempotent of norm one, and \(e_nA\simeq B_n\) isometrically.
  Fix \(n\) and put \(p_r=q_r(e_n)\).
  Surjectivity of \(q_r\) makes \(p_r\) central in \(A_r\).
  Since the \(q_r\)'s separate points, we have \(p_r\ne0\) for some \(r\) and then for all later \(r\).
  On this tail, we have \(p_r^2=p_r\) and \(1\leq\norm{p_r}\leq1\).
  The finite-dimensional corner \(p_rA_r\) therefore has unit norm one.

  The restrictions of \(\pi_r^s\) to these corners are unital metric quotients.
  Indeed, for \(y\in p_rA_r\), choose a minimum-norm lift \(z\in A_s\).
  Then we have
  \[ \pi_r^s(p_sz)=p_ry=y,\qquad \norm y\leq\norm{p_sz}\leq\norm z=\norm y. \]
  A bounded compatible sequence \((y_r)_r\), with \(y_r\in p_rA_r\), determines \(y\in A\), and \(q_r(e_ny)=p_ry_r=y_r=q_r(y)\) on the cofinal tail.
  Thus we have \(y=e_ny\), and we obtain
  \[ B_n\simeq e_nA\simeq\ilim(p_rA_r,\pi_r^s|_{p_sA_s}) \]
  isometrically as unital Banach algebras.

  Suppose now that \((u_n)\) is not boundedly complete.
  By unconditionality, there are nonnegative scalars \((a_n)\) whose partial sums are bounded and not Cauchy.
  Choose \(\delta>0\) and successive finite intervals \(I_j\), separated by at least one index, such that \(\norm{\sum_{n\in I_j}a_nu_n}\geq\delta\).
  Set \(c_n=a_n\) on \(\bigcup_jI_j\), and \(c_n=0\) elsewhere.
  By suppression unconditionality, we have
  \[ \sup_N\left\lVert\sum_{n=1}^Nc_nu_n\right\rVert<\infty, \qquad \left\lVert\sum_{n\in I_j}c_nu_n\right\rVert\geq\delta. \]
  Thus \(\sum_nc_nu_n\) does not converge and infinitely many \(c_n\) vanish.
  Put \(s_N=\sum_{n\leq N}c_ne_n\).
  This sequence is bounded in \(A\).
  Finite-dimensional compactness and diagonal extraction give \(N_j\to\infty\) such that \(q_r(s_{N_j})\to x_r\) for every \(r\).
  The limit sequence is bounded and compatible, so it determines \(x=b+\lambda I\in A\).
  For every \(m,r\), we have
  \[ q_r(e_mx)=\lim_jq_r(e_ms_{N_j})=c_mq_r(e_m). \]
  By separation, we have \(e_mx=c_me_m\), and hence we obtain \(b_m=(c_m-\lambda)1_{B_m}\).
  Since \(b\in(\bigoplus_nB_n)_U\), we have \(\norm{b_m}\to0\).
  Along the infinite set where \(c_m=0\), this forces \(\lambda=0\).
  Consequently, we have \(\norm{b_m}=c_m\), contrary to the divergence of \(\sum_mc_mu_m\).
\end{proof}

These conditions characterize the finite-dimensional quotient representation used in Theorem~\ref{thm:main}, but need not hold for Calkin realizations obtained by other methods.
The converse uses the units of all the summands.
The sufficient condition of Corollary~\ref{cor:unconditional-algebra-sums} also applies to nonunital summands.

\begin{remark}[Sequence spaces with the Fatou property]
  \label{rem:fatou-algebra-sums}
  The outer sequence norm can be generalized.
  Let \(E\subseteq\K^{\N_+}\) be a solid Banach sequence space whose standard unit vectors \((e_n)\) have norm one.
  Solidity means that coordinatewise domination of absolute values implies membership and domination of norms.
  Write \(P_Nc=(c_1,\ldots,c_N,0,\ldots)\).
  Assume that, for every scalar sequence \(c\),
  \[ c\in E\ \Longleftrightarrow\ \sup_N\norm{P_Nc}_E<\infty, \qquad \norm c_E=\sup_N\norm{P_Nc}_E\quad(c\in E). \]
  For the algebras \(B_n\) of Corollary~\ref{cor:unconditional-algebra-sums}, put
  \[ J_E=\{b\in\textstyle\prod_nB_n:(\norm{b_n})_n\in E\}, \qquad \norm b_{J_E}=\norm{(\norm{b_n})_n}_E. \]
  Then \(J_E\oplus\K I\), with coordinatewise multiplication and \(\norm{b+\lambda I}=\norm b_{J_E}+\abs\lambda\), is isometrically a Calkin algebra.
  To see this, replace the finite \(U\)-norms in \eqref{eq:unconditional-sum-tower} by the corresponding \(E\)-norms.
  Solidity gives contractivity and the same norm-preserving lifts.
  Let \((x_r)_r\), where \(x_r=(b_{n,r})_{n\leq r}+\lambda I\), be a bounded compatible sequence and put \(C=\sup_r\norm{x_r}\).
  For each \(n\), let \(b_n\in B_n\) be the element determined by the compatible coordinate sequence \((b_{n,r})_{r\geq n}\).
  By finite-dimensional continuity, we have
  \[
  \norm{P_N(\norm{b_n})_n}_E =\lim_{r\to\infty} \norm{(\norm{b_{1,r}},\ldots,\norm{b_{N,r}},0,\ldots)}_E \leq C-\abs\lambda.
  \]
  By the Fatou property, we obtain \(b=(b_n)\in J_E\).
  Combining this with contractivity, we obtain
  \[
  \abs\lambda+\norm b_{J_E} =\abs\lambda+\sup_N\norm{P_N(\norm{b_n})_n}_E \leq C\leq\abs\lambda+\norm b_{J_E}.
  \]
  Conversely, the coordinate images of each \(b\in J_E\) form a bounded compatible sequence.
  The coordinate maps preserve multiplication, so this identifies \(J_E\oplus\K I\) with the bounded inverse limit isometrically as unital Banach algebras.
  Apply Theorem~\ref{thm:main}.

  For a normalized \(1\)-unconditional basis, its scalar sequence space has this property precisely when the basis is boundedly complete.
  Finite sequences need not be dense in a general \(E\).
  In particular, \(E=\ell_\infty\) gives the \(\ell_\infty\)-sum of the \(B_n\)'s.
  If all these algebras are unital with unit norm one, the \(\ell_\infty\)-sum itself is unital with unit norm one.
  The recursive lifting argument in Proposition~\ref{prop:unconditional-sum-converse} makes every coordinate map \(B_n\to B_{n,r}\) onto.
  Its image of \(1_{B_n}\) is therefore the unit of \(B_{n,r}\), of norm one whenever this stage is nonzero.
  Discarding initial zero stages and reindexing each component tower gives unital bonding maps and normalized units.
  Omitting the extra scalar coordinate from the finite tower then realizes the \(\ell_\infty\)-sum directly.
\end{remark}

Taking constant component towers shows that the summands in Corollary~\ref{cor:unconditional-algebra-sums} may be arbitrary finite-dimensional Banach algebras, including matrix and finite incidence algebras.
Infinite-dimensional summands are also allowed whenever they have the specified representations, for example \(H^\infty\) and the multiplier algebras in Corollaries~\ref{cor:hardy-oracles} and~\ref{cor:rkhs-multipliers}.
The next two corollaries give the block \(\ell_p\) and scalar cases.

\begin{corollary}[Unitized block \(\ell_p\)-algebras]
  \label{cor:block-lp-unitizations}
  Let \(1\leq p<\infty\), and let \((F_k)_{k\geq0}\) be finite-dimensional Banach algebras.
  Give \(J_p=(\bigoplus_{k=0}^\infty F_k)_{\ell_p}\) coordinatewise multiplication.
  Then
  \[ J_p\oplus\K I,\qquad \norm{a+\lambda I}=\norm a_p+\abs\lambda, \]
  with the unitization product in \eqref{eq:unconditional-sum-unitization}, is isometrically a Calkin algebra.
\end{corollary}

The unit vector basis of \(\ell_p\) is normalized, \(1\)-unconditional, and boundedly complete.
Taking \(U=\ell_p\) and the constant component towers \(B_{n,r}=F_{n-1}\) in Corollary~\ref{cor:unconditional-algebra-sums}, we obtain the conclusion with the stated norm and product.

The dimensions and multiplication tables of the \(F_k\)'s are arbitrary.
For example, one may take \(F_k=M_{d_k}(\mathbb C)\) with unbounded matrix sizes, or fixed truncated-polynomial algebras \(\K[t]/(t^d)\) with their coefficient \(\ell_1\)-norms.

\begin{corollary}[Boundedly complete unconditional bases]
  \label{cor:boundedly-complete-unitization}
  Let \(U\) have a normalized \(1\)-unconditional boundedly complete basis \((u_n)_{n\geq1}\), and give \(U\) coordinatewise multiplication.
  The algebra \(U\oplus\K I\), with \(\norm{u+\lambda I}=\norm u+\abs\lambda\) and the product in \eqref{eq:unconditional-sum-unitization}, is isometrically isomorphic, as a unital Banach algebra, to \(\Cal(X)\) for a separable Banach space \(X\) with a Schauder basis.
  There are infinite-rank projections \((Q_n)_{n\geq1}\) on \(X\) such that \(Q_nQ_m=0\) for \(n\ne m\) and \(\norm{Q_n}=\norm{[Q_n]}=1\).
  For every finitely supported scalar sequence \((a_n)\) and every \(\lambda\in\K\), we have
  \begin{equation}\label{eq:unconditional-lifting-norm}
    \left\lVert\lambda I_X+\sum_na_nQ_n\right\rVert =\left\lVert\lambda[I_X]+\sum_na_n[Q_n]\right\rVert_{\Cal(X)} =\abs\lambda+\left\lVert\sum_na_nu_n\right\rVert_U.
  \end{equation}
  In particular, taking \(\lambda=0\) gives an isometric lifting of the prescribed basis.
  Moreover, we have the topological direct sum
  \begin{equation}\label{eq:unconditional-operator-decomposition}
    \Bcal(X)=\Kcal(X)\oplus
    \overline{\operatorname{span}}\{Q_n:n\geq1\}\oplus\K I_X.
  \end{equation}
  Equivalently, we have the isometric identification
  \[ \Cal(X)\simeq \left(\bigoplus_{n\geq1}\K[Q_n]\right)_U\oplus\K[I_X], \]
  with the norm and unitization product of \eqref{eq:unconditional-sum-unitization}.
\end{corollary}

\begin{proof}
  Apply Corollary~\ref{cor:unconditional-algebra-sums} with \(B_n=\K\) and constant component towers.
  In this case, we have
  \[ U_r=\operatorname{span}\{u_1,\ldots,u_r\},\qquad A_r=U_r\oplus\K I,\qquad \pi_r^s(u+\lambda I)=P_ru+\lambda I, \]
  where \(P_r\) is the basis projection onto \(U_r\).
  Use the multiplier representation of Theorem~\ref{thm:main} and put \(Q_n=M_{u_n}\), viewing \(u_n\) in \(U\oplus\K I\).
  Since \(u_nu_m=\delta_{nm}u_n\), we obtain from the exact representation and essential-norm identities that
  \[ Q_n^2=Q_n,\qquad Q_nQ_m=0\quad(n\ne m),\qquad \norm{Q_n}=\norm{[Q_n]}=1. \]
  In particular, \(Q_n\) is not compact and therefore has infinite rank.
  The same identities give \eqref{eq:unconditional-lifting-norm}.

  The map \(u\mapsto M_u\) is an isometry onto \(\overline{\operatorname{span}}\{Q_n:n\geq1\}\).
  Every operator has a unique expression \(S=K+M_u+\lambda I_X\), and we have
  \[
  \abs\lambda+\norm u=\norm{[S]}\leq\norm S.
  \]
  For the compact term \(K\), we obtain
  \[
    \norm K \leq\norm S+\norm{M_u+\lambda I_X} =\norm S+\norm{[S]} \leq2\norm S.
  \]
  These estimates prove that the three summands in \eqref{eq:unconditional-operator-decomposition} are closed and their coordinate projections are bounded.

  It remains to verify the Schauder basis assertion.
  Return to the BD notation \((Z_m)\), \(P_I\), \(i_m\), \(M\), and \(\kappa\) of Subsection~\ref{sec:bd-realization}, where \(P_I\) denotes an FDD projection on \(X\).
  Both \(A_r=U_r\oplus_1\K I\) and \(A_r^*=U_r^*\oplus_\infty\K\) have their canonical normalized \(1\)-unconditional bases.
  These are precisely the two fibre types used in the construction.
  At construction rank \(m\), their finite \(\ell_\infty\)-sum \(F_m\) therefore has a concatenated basis whose partial-sum projections \(C_{m,k}\) have norm at most one.
  Let \(T_m=i_m|_{F_m}:F_m\to Z_m\), with the earlier coordinates set equal to zero.
  Since \(i_m\) fixes its first \(m\) ambient blocks, we obtain from \eqref{eq:formal-block-bounds} that
  \[ \norm{T_m}\leq M,\qquad \norm{T_m^{-1}}\leq1,\qquad \norm{P_{\{m\}}}\leq\kappa. \]
  Concatenate the transported bases of \(Z_m\) in rank order.
  For the partial-sum projections, we have
  \[
  \begin{aligned}
    \left\lVert P_{[1,m-1]}+T_mC_{m,k}T_m^{-1}P_{\{m\}}\right\rVert\leq\norm{P_{[1,m-1]}}
    +\norm{T_m}\norm{C_{m,k}}\norm{T_m^{-1}}\norm{P_{\{m\}}}\leq M+M\kappa.
  \end{aligned}
  \]
  Their span is dense because \((Z_m)\) is an FDD.
  The uniform bound therefore proves that the concatenated sequence is a Schauder basis.
\end{proof}

\begin{remark}[Comparison with the reflexive construction]
  \label{rem:mpb-bounded-completeness}
  Corollary~\ref{cor:boundedly-complete-unitization} applies to every space in \cite[Theorem~1.2]{MotakisPelczar2025}.
  Indeed, under the hypothesis of that theorem, \cite[Theorem~7.1]{MotakisPelczar2025} gives a blocking \(\Upsilon\), constants \(\theta>0\) and \(1\leq q<\infty\), and the lower estimate
  \begin{equation}\label{eq:mpb-lower-q}
    \theta\left(\sum_{j=1}^m\norm{x_j}^q\right)^{1/q}
    \leq\left\lVert\sum_{j=1}^mx_j\right\rVert
  \end{equation}
  for every finite successive \(\Upsilon\)-block sequence.
  Suppose that \(C=\sup_N\norm{\sum_{n=1}^Na_nu_n}<\infty\), but the series does not converge.
  Contractivity of interval projections implies that the partial sums at the blocking endpoints are not Cauchy.
  Thus there are successive intervals \(I_j\), each a union of complete \(\Upsilon\)-blocks, and \(\delta>0\) such that
  \[ x_j=\sum_{n\in I_j}a_nu_n,\qquad \norm{x_j}\geq\delta. \]
  By suppression unconditionality and \eqref{eq:mpb-lower-q}, we have
  \[ \theta\delta m^{1/q} \leq\left\lVert\sum_{j=1}^mx_j\right\rVert\leq C \qquad(m\geq1), \]
  a contradiction.
  Hence the basis is boundedly complete.
  In \cite[Theorem~1.2]{MotakisPelczar2025}, the realizing projections and their Calkin images are \(128\)-equivalent to the prescribed basis. Here \eqref{eq:unconditional-lifting-norm} gives an isometric unital algebra lift for the standard unitization norm.
\end{remark}

In particular, Corollary~\ref{cor:boundedly-complete-unitization} applies whenever \(U\) is reflexive and has a normalized \(1\)-unconditional basis, since every basis of a reflexive space is boundedly complete \cite[Theorem~3.2.19]{AlbiacKalton2016}.
Indeed, bounded partial sums have a weak cluster point whose basis coefficients are the prescribed scalars, so the basis expansion of this point gives norm convergence.
We fix the notation for the next example. For nonempty finite sets \(E,F\subseteq\N_+\), write \(E<F\) if \(\max E<\min F\), and put \(Ex=\sum_{n\in E}x_ne_n\) for \(x=(x_n)\in c_{00}(\N_+)\).
Following the Figiel--Johnson construction \cite[Section~2]{FigielJohnson1974}, define iteratively
\[
  \norm{x}_0=\norm{x}_\infty,\qquad \norm{x}_{m+1}=\max\left\{\norm{x}_m, \frac12\sup_{\substack{k\leq\min E_1\\E_1<\cdots<E_k}} \sum_{j=1}^k\norm{E_jx}_m\right\}.
\]
The norms increase and are bounded above by \(\norm{x}_{\ell_1}\).
The Figiel--Johnson Tsirelson space \(T\) is the completion of \(c_{00}(\N_+)\) for \(\norm{x}_T=\lim_{m\to\infty}\norm{x}_m\).
We write \(T^*\) for its Banach dual, which is Tsirelson's original space \cite{Tsirelson1974}, and \((e_n^*)\) for the basis biorthogonal to \((e_n)\) (see \cite[Section~2.4]{BaudierLancienSchlumprecht2018}).
The coordinatewise multiplication used below is determined by \(e_i^*e_j^*=\delta_{ij}e_i^*\).
The next example shows that our admissible class of bases is strictly larger than the class in \cite[Theorem~1.2]{MotakisPelczar2025}.

\begin{corollary}[Tsirelson's original space]
  \label{cor:tsirelson-original}
  Let \(T^*\) be Tsirelson's original space with its usual normalized \(1\)-unconditional basis and coordinatewise multiplication.
  Then \(T^*\oplus\K I\), with the norm and product in \eqref{eq:unconditional-sum-unitization}, is isometrically a Calkin algebra of a Banach space with a Schauder basis, with the projections and decomposition in Corollary~\ref{cor:boundedly-complete-unitization}.
  With this basis, \(T^*\) does not satisfy the hypothesis of \cite[Theorem~1.2]{MotakisPelczar2025}.
\end{corollary}

\begin{proof}
  The space \(T^*\) is reflexive and has a normalized \(1\)-unconditional basis \((e_n^*)\) \cite[Section~2.4]{BaudierLancienSchlumprecht2018}, so the basis is boundedly complete \cite[Theorem~3.2.19]{AlbiacKalton2016}.
  The first assertion follows from Corollary~\ref{cor:boundedly-complete-unitization}.
  For successive blocks \(x_1^*<\cdots<x_N^*\) in the unit ball with \(N\leq\min\operatorname{supp}x_1^*\), we obtain from \cite[Section~2.4, (2.13)]{BaudierLancienSchlumprecht2018} that
  \[ \left\lVert\sum_{j=1}^Nx_j^*\right\rVert\leq2. \]
  Here \(<\) means that the basis supports are successive.
  For any blocking \(\Upsilon\) and any \(N\), choose basis vectors \(e_{k_1}^*,\ldots,e_{k_N}^*\) in distinct successive \(\Upsilon\)-blocks with \(k_1\geq N\).
  If this blocking satisfied \eqref{eq:mpb-lower-q}, then we would have
  \[ \theta N^{1/q} \leq\left\lVert\sum_{j=1}^Ne_{k_j}^*\right\rVert\leq2 \qquad(N\geq1), \]
  which is impossible.
  By \cite[Theorem~7.1]{MotakisPelczar2025}, \(T^*\) with its canonical basis is therefore outside that class.
\end{proof}

\begin{remark}
  Proposition~\ref{prop:unconditional-sum-converse} explains the role of bounded completeness for the usual coordinatewise sum.
  Without bounded completeness, the finite coordinate tower can give a larger inverse limit.
  For the unit vector basis \((e_n)\) of \(c_0\), we obtain from the same construction that
  \[ \ilim\bigl(\operatorname{span}\{e_1,\ldots,e_r\}\oplus\K I, \pi_r^s\bigr)=\ell_\infty\oplus\K I. \]
  Indeed, the norm of a compatible sequence is
  \[ \sup_{r\geq1}\left(\abs\lambda+\max_{1\leq n\leq r}\abs{a_n}\right) =\abs\lambda+\sup_n\abs{a_n}. \]
  The constant coefficient sequence \(a_n=1\) defines a bounded compatible sequence although it does not belong to \(c_0\).
  Thus this tower does not recover \(c_0\oplus\K I\).
  Both sums carry the norm and product in \eqref{eq:unconditional-sum-unitization}.
\end{remark}

The following construction by finite quotients covers several convolution examples (see \cite[Chapters~3 and~4]{DalesLauStrauss2010} for background on semigroup quotients and convolution algebras).

\begin{corollary}[Weighted monoid algebras]
  \label{cor:rees-filtered-monoids}
  Let \(S\) be a countable monoid with identity \(e\).
  Suppose that there are finite subsets
  \( e\in S_0\subseteq S_1\subseteq\cdots\subseteq S,  S=\bigcup_{r=0}^\infty S_r, \)
  such that
  \begin{equation}\label{eq:rees-factor-closed}
    uv\in S_r
    \quad\Longrightarrow\quad
    u,v\in S_r.
  \end{equation}
  Let \(\omega:S\to(0,\infty)\) satisfy \(\omega(e)=1\) and \(\omega(uv)\leq\omega(u)\omega(v)\) for \(u,v\in S\).
  Then the weighted convolution algebra
  \[ \ell_1(S,\omega) =\left\{a=(a_u)_{u\in S}: \sum_{u\in S}\abs{a_u}\omega(u)<\infty\right\} \]
  is isometrically a Calkin algebra.
\end{corollary}

\begin{proof}
  Condition~\eqref{eq:rees-factor-closed} says that \(S\setminus S_r\) is a two-sided semigroup ideal.
  Write \(e_u\) for the coefficient family that is \(1\) at \(u\) and \(0\) elsewhere.
  Let \(A_r\) be the span of \(\{e_u:u\in S_r\}\), with the product
  \[
  e_u\cdot_re_v= \begin{cases} e_{uv},&uv\in S_r,\\
  0,&uv\notin S_r. \end{cases}
  \]
  This is the quotient of the algebraic semigroup algebra by the span of \(\{e_u:u\notin S_r\}\), and is therefore associative.
  Give it the norm \(\norm a_{r,\omega}=\sum_{u\in S_r}\abs{a_u}\omega(u)\).
  Its unit is \(e_e\), of norm one.
  For \(a,b\in A_r\), we have
  \begin{align*}
    \norm{a\cdot_rb}_{r,\omega} &=\sum_{w\in S_r} \abs{\sum_{\substack{uv=w\\u,v\in S_r}}a_ub_v}\omega(w)\\
    &\leq\sum_{\substack{u,v\in S_r\\uv\in S_r}} \abs{a_u}\abs{b_v}\omega(uv) \leq\sum_{u,v\in S_r}\abs{a_u}\omega(u)\abs{b_v}\omega(v) =\norm a_{r,\omega}\norm b_{r,\omega}.
  \end{align*}
  Hence \(A_r\) is a finite-dimensional unital Banach algebra.

  For \(r\leq s\), let \(\sigma_r^s:A_s\to A_r\) be coefficient restriction.
  Since \(S\setminus S_r\) is a two-sided ideal, \(\sigma_r^s\) is a unital homomorphism.
  It is contractive, and zero extension preserves the weighted norm.
  Hence every bonding map is a metric quotient.

  A compatible sequence has one coefficient \(a_u\) for every \(u\in S\), and by monotone convergence we have
  \[ \sup_r\sum_{u\in S_r}\abs{a_u}\omega(u) =\sum_{u\in S}\abs{a_u}\omega(u). \]
  If \(w\in S_r\) and \(uv=w\), then we obtain \(u,v\in S_r\) from \eqref{eq:rees-factor-closed}.
  In particular, \(w\) has only finitely many factorizations \(uv=w\), and the \(w\)-coordinate of the stage product is the full convolution coefficient.
  Hence the inverse-limit multiplication is convolution, and the bounded inverse limit is \(\ell_1(S,\omega)\) isometrically.
  Apply Theorem~\ref{thm:main}.
\end{proof}

\begin{corollary}[Free-monoid convolution algebras]
  \label{cor:free-semigroup}
  If \(1\leq d<\infty\), then \(\ell_1(\mathbb F_d^+)\), with convolution on the free monoid on \(d\) generators, is isometrically a Calkin algebra.
\end{corollary}

\begin{proof}
  Take \(S_r\) to be the words of length at most \(r\) and \(\omega=1\).
  Since there are finitely many generators, each \(S_r\) is finite.
  Additivity of word length gives \eqref{eq:rees-factor-closed}, and every word belongs to some \(S_r\).
  Corollary~\ref{cor:rees-filtered-monoids} applies.
\end{proof}

For \(d=1\), this recovers \(\ell_1(\N)\), previously obtained by Tarbard \cite{Tarbard2013}, while for \(d\geq2\) the algebra is noncommutative.

\begin{corollary}[Dirichlet convolution algebras]
  \label{cor:dirichlet-convolution}
  Let \(\N_\times=\{1,2,\ldots\}\), with multiplication, and let \(\omega:\N_\times\to(0,\infty)\) satisfy
  \( \omega(1)=1\) and  \(\omega(mn)\leq\omega(m)\omega(n). \)
  Then \(\ell_1(\N_\times,\omega)\), with Dirichlet convolution (see \cite[Section~3, Example~1]{Rota1964}), is isometrically a Calkin algebra.
\end{corollary}

\begin{proof}
  Let \(p_1,p_2,\ldots\) be the prime numbers, and write \(v_p(m)\) for the exponent of the prime \(p\) in \(m\).
  Put \(S_0=\{1\}\), and, for \(r\geq1\), put
  \[ S_r= \left\{ p_1^{\alpha_1}\cdots p_r^{\alpha_r}: \alpha_1,\ldots,\alpha_r\in\N,\quad \sum_{j=1}^r\alpha_j\leq r \right\}. \]
  Every \(S_r\) is finite, the sets increase, and unique prime factorization gives \(\N_\times=\bigcup_{r=0}^\infty S_r\).
  The factor-closure condition is immediate for \(S_0=\{1\}\).
  If \(r\geq1\) and \(mn\in S_r\), neither \(m\) nor \(n\) contains a prime factor larger than \(p_r\), and we have
  \[ \sum_{j=1}^r v_{p_j}(m) \leq\sum_{j=1}^r v_{p_j}(mn) \leq r, \]
  with the same estimate for \(n\).
  Hence we have \(m,n\in S_r\), so \eqref{eq:rees-factor-closed} holds.
  Apply Corollary~\ref{cor:rees-filtered-monoids}.
\end{proof}

Under prime factorization, the unweighted algebra in Corollary~\ref{cor:dirichlet-convolution} is the absolutely summable formal power-series algebra in countably many commuting variables:
\[ \left\{ \sum_{\alpha\in\N^{(\N_+)}}a_\alpha z^\alpha: \sum_\alpha\abs{a_\alpha}<\infty \right\}. \]
Here \(\N^{(\N_+)}\) denotes the set of finitely supported sequences of nonnegative integers, and \(z^\alpha=\prod_{j\geq1}z_j^{\alpha_j}\) \cite[Section~1.1]{CarandoEtAl2015}.
This algebra is not a scalar \(C(K)\)-algebra.
Indeed, if \(a,b\ne0\), let
\[ m=\min\{k:a_k\ne0\}, \qquad n=\min\{k:b_k\ne0\}. \]
If a summand \(a_db_{mn/d}\) in \((a*b)_{mn}\) is nonzero, then we have \(d\geq m\) and \(mn/d\geq n\).
Since their product is \(mn\), both inequalities are equalities.
Therefore we obtain
\[ (a*b)_{mn}=a_mb_n\ne0. \]
Thus the algebra is an infinite-dimensional integral domain, whereas \(C(K)\) has nonzero zero divisors whenever \(K\) contains more than one point.

\begin{corollary}[Finite-quiver path algebras]
  \label{cor:finite-quiver-path}
  Let \(Q=(Q_0,Q_1)\) be a nonempty finite directed graph, and let \(\mathcal P(Q)\) be its set of finite paths, including the paths of length zero (see \cite[Section~II.1]{AssemSimsonSkowronski2006} for the path-algebra construction).
  Write \(s(p)\), \(t(p)\), and \(|p|\) for the initial vertex, terminal vertex, and length of a path \(p\).
  The concatenation \(pq\) first traverses \(p\) and then \(q\), and is defined when \(t(p)=s(q)\).
  Write \(e_p\) for the coefficient family that is \(1\) at \(p\) and \(0\) elsewhere.
  Suppose that \(\omega:\mathcal P(Q)\to(0,\infty)\) satisfies
  \[ \omega(v)=1\quad(v\in Q_0), \qquad \omega(pq)\leq\omega(p)\omega(q) \]
  whenever \(p\) and \(q\) are composable.
  Let \(\mathcal A_\omega(Q)\) consist of the coefficient families \(a=(a_p)_{p\in\mathcal P(Q)}\) for which the following norm is finite, with path convolution as multiplication:
  \[
  \norm a_{\rm row,\omega} =\max_{v\in Q_0} \sum_{\substack{p\in\mathcal P(Q)\\
  s(p)=v}} \abs{a_p}\omega(p).
  \]
  Then \(\mathcal A_\omega(Q)\) is isometrically a Calkin algebra.
\end{corollary}

\begin{proof}
  For paths \(w\), define
  \[ (a*b)_w=\sum_{pq=w}a_pb_q. \]
  Each sum has \(|w|+1\) terms, corresponding to the possible places at which to split \(w\).
  Associativity follows by expanding either iterated product over the finite set of factorizations of \(w\) into three paths.
  Fix \(v\in Q_0\).
  By absolute convergence and the weight inequality, we have
  \begin{align*}
    \sum_{s(w)=v}\abs{(a*b)_w}\omega(w) &\leq\sum_{s(w)=v}\sum_{pq=w}\abs{a_p}\abs{b_q}\omega(pq)\\
    &\leq\sum_{s(p)=v}\abs{a_p}\omega(p) \sum_{s(q)=t(p)}\abs{b_q}\omega(q)\\
    &\leq\norm b_{\rm row,\omega} \sum_{s(p)=v}\abs{a_p}\omega(p) \leq\norm a_{\rm row,\omega}\norm b_{\rm row,\omega}.
  \end{align*}
  Taking the maximum over \(v\) proves that the displayed norm is an algebra norm.
  Since \(Q_0\) is finite, the algebra is a finite \(\ell_\infty\)-sum of weighted \(\ell_1\)-spaces and is complete.
  Its unit is \(1_{\mathcal A_\omega(Q)}=\sum_{v\in Q_0}e_v\).
  Every source row contains exactly one nonzero coefficient of this element, so we have \(\norm{1_{\mathcal A_\omega(Q)}}_{\rm row,\omega}=1\).

  Let \(A_r\) be the span of the paths of length at most \(r\), with truncated path concatenation and the corresponding row norm.
  The paths of length greater than \(r\) span a two-sided ideal in the algebraic path algebra, so truncation gives an associative quotient product.
  Since \(Q\) is finite, \(A_r\) is finite dimensional.
  Length restriction \(\sigma_r^s:A_s\to A_r\) is a unital homomorphism because a path of length at most \(r\) can only be factored into paths of length at most \(r\).
  It is contractive, and zero extension preserves the row norm.
  Hence the bonding maps are metric quotients.

  A compatible sequence has one coefficient \(a_p\) for every finite path.
  Since \(Q_0\) is finite, we have
  \[
  \sup_r\max_{v\in Q_0} \sum_{\substack{s(p)=v\\|p|\leq r}}\abs{a_p}\omega(p) =\max_{v\in Q_0}\sup_r \sum_{\substack{s(p)=v\\|p|\leq r}}\abs{a_p}\omega(p) =\max_{v\in Q_0}\sum_{s(p)=v}\abs{a_p}\omega(p).
  \]
  Thus the coefficient identification gives \(\ilim(A_r,\sigma_r^s)\simeq\mathcal A_\omega(Q)\) isometrically as Banach algebras.
  Apply Theorem~\ref{thm:main}.
\end{proof}

When \(Q\) has one vertex and \(d\) loops and \(\omega=1\), Corollary~\ref{cor:finite-quiver-path} reduces to Corollary~\ref{cor:free-semigroup}.
The row norm gives \(\norm{1}=1\), whereas the global coefficient \(\ell_1\)-norm gives \(\norm{1}=\abs{Q_0}\).

\begin{corollary}[Tensor power-series algebras]
  \label{cor:tensor-power-series}
  Let \(E\) be a finite-dimensional Banach space, and let \((\omega_n)_{n\geq0}\subseteq(0,\infty)\) satisfy 
  \( \omega_0=1\) and \( \omega_{i+j}\leq\omega_i\omega_j~(i,j\in\N). \)
  Put \(E^{\widehat\otimes_\pi0}=\K\), and let
  \[ \mathcal T_\omega^1(E) =\left\{x=(x_n)_{n\geq0}: x_n\in E^{\widehat\otimes_\pi n},\quad \sum_{n=0}^\infty \omega_n\norm{x_n}_\pi<\infty\right\}. \]
  Equip this space with multiplication and norm
  \[ (xy)_n=\sum_{i+j=n}x_i\otimes y_j,\qquad \norm x_{\mathcal T_\omega^1}=\sum_{n=0}^\infty\omega_n\norm{x_n}_\pi. \]
  Then \(\mathcal T_\omega^1(E)\) is isometrically a Calkin algebra.
\end{corollary}

\begin{proof}
  We use the natural associative identifications of the projective tensor powers.
  Expanding the degree-\(n\) coefficient of either iterated product over \(i+j+k=n\) proves associativity.
  Since \(\norm{x_i\otimes y_j}_\pi=\norm{x_i}_\pi\norm{y_j}_\pi\) \cite[Proposition~2.1]{Ryan2002}, we obtain from the weight inequality that
  \begin{align*}
    \norm{xy}_{\mathcal T_\omega^1} &=\sum_{n=0}^\infty\omega_n\norm{\sum_{i+j=n}x_i\otimes y_j}_\pi\\
    &\leq\sum_{i,j\geq0}\omega_{i+j}\norm{x_i}_\pi\norm{y_j}_\pi \leq\sum_{i,j\geq0}\omega_i\norm{x_i}_\pi\omega_j\norm{y_j}_\pi =\norm x_{\mathcal T_\omega^1}\norm y_{\mathcal T_\omega^1}.
  \end{align*}
  The weighted \(\ell_1\)-sum is complete, and the scalar in degree zero is a unit of norm one.

  Let
  \( A_r=\bigoplus_{n=0}^rE^{\widehat\otimes_\pi n}, \)
  with the weighted \(\ell_1\)-norm and multiplication truncated above degree \(r\).
  The terms of degree greater than \(r\) form a two-sided ideal, so this product is associative.
  Every \(A_r\) is finite dimensional.
  Degree restriction \(\sigma_r^s:A_s\to A_r\) is a unital homomorphism.
  It is contractive, and zero extension gives a lift of every \(x\in A_r\) with the same norm.
  Thus the bonding maps are metric quotients.

  Finally, a compatible sequence of finite truncations is determined by coefficients \(x_n\in E^{\widehat\otimes_\pi n}\), and we have
  \[ \sup_r\sum_{n=0}^r\omega_n\norm{x_n}_\pi =\sum_{n=0}^\infty\omega_n\norm{x_n}_\pi. \]
  For each \(n\), every stage \(r\geq n\) has product coefficient \(\sum_{i+j=n}x_i\otimes y_j\).
  Thus the bounded inverse limit is \(\mathcal T_\omega^1(E)\), with its prescribed norm and multiplication.
  Apply Theorem~\ref{thm:main}.
\end{proof}

For \(E=\ell_1^d\), with its canonical basis and \(\omega_n=1\), the projective tensor-product identification for \(\ell_1\)-spaces \cite[Section~2.1]{Ryan2002} shows that this recovers \(\ell_1(\mathbb F_d^+)\).
Other finite-dimensional norms on \(E\) give further isometric tensor-series examples.

The next example uses the incidence multiplication of \cite[Section~3]{Rota1964} on a lower-finite poset, with a row norm. For the \(c_0\)-incidence algebra notation, see \cite[Definition~5.9]{Acuaviva2025Posets}.

\begin{corollary}[Lower-finite \(c_0\)-incidence algebras]
  \label{cor:lower-finite-incidence}
  Let \(P\) be a nonempty countable lower-finite partially ordered set, meaning that \(\{x\in P:x\leq y\}\) is finite for every \(y\in P\).
  Let \(\operatorname{IA}(P,c_0)\) be the algebra of matrices \(a=(a_{xy})_{x,y\in P}\) satisfying
  \[ a_{xy}=0\quad(x\nleq y), \qquad \norm a_{\rm row} =\sup_{x\in P}\sum_{y\geq x}\abs{a_{xy}}<\infty. \]
  Then these matrices define bounded operators on \(c_0(P)\), the displayed norm is their operator norm, and \(\operatorname{IA}(P,c_0)\) is isometrically a Calkin algebra.
\end{corollary}

\begin{proof}
  For each \(y\in P\), the \(y\)-th column of an incidence matrix is supported by the finite set \(\{x\in P:x\leq y\}\).
  If \(z\in c_{00}(P)\), only finitely many columns contribute to \(az\), and hence we have \(az\in c_{00}(P)\).
  Moreover, for \(x\in P\), we have
  \[ \abs{(az)_x} \leq\sum_{y\geq x}\abs{a_{xy}}\abs{z_y} \leq\norm a_{\rm row}\norm z_\infty. \]
  Thus \(a\) extends to an operator on \(c_0(P)\) with \(\norm a_{\Bcal(c_0(P))}\leq\norm a_{\rm row}\).

  Conversely, fix \(x\in P\) and a finite set \(F\subseteq\{y:y\geq x\}\).
  Choose \(z\in c_{00}(P)\), supported by \(F\), with \(\norm z_\infty\leq1\) and \(a_{xy}z_y=\abs{a_{xy}}\) for \(y\in F\).
  Then we have
  \[ \norm a_{\Bcal(c_0(P))} \geq\abs{(az)_x} =\sum_{y\in F}\abs{a_{xy}}. \]
  Taking the supremum first over finite \(F\), and then over \(x\), we obtain
  \[ \norm a_{\Bcal(c_0(P))} \geq\sup_{x\in P}\sum_{y\geq x}\abs{a_{xy}} =\norm a_{\rm row}. \]
  Hence the two norms coincide.

  Incidence multiplication is well defined and submultiplicative.
  In fact, for every \(x\in P\), we have
  \begin{align*}
    \sum_{z\geq x}\abs{(ab)_{xz}} &\leq\sum_{z\geq x}\sum_{x\leq y\leq z}\abs{a_{xy}}\abs{b_{yz}}\\
    &=\sum_{y\geq x}\abs{a_{xy}}\sum_{z\geq y}\abs{b_{yz}}\\
    &\leq\norm b_{\rm row}\sum_{y\geq x}\abs{a_{xy}} \leq\norm a_{\rm row}\norm b_{\rm row}.
  \end{align*}
  The diagonal matrix is a unit of row norm one.
  The zero pattern is closed under operator-norm limits, so \(\operatorname{IA}(P,c_0)\) is a unital Banach algebra.

  Choose a sequence \((p_k)\) whose range is \(P\), repeating elements if \(P\) is finite, and put
  \( P_r=\bigcup_{k=0}^r\{x\in P:x\leq p_k\}. \)
  Lower finiteness makes \(P_r\) finite.
  These lower subsets increase and satisfy \(P=\bigcup_{r=0}^\infty P_r\).
  Let \(A_r=\operatorname{IA}(P_r)\), with the finite row norm, and let \(\sigma_r^s:A_s\to A_r\) be matrix restriction.

  The lower-set property makes \(\sigma_r^s\) multiplicative.
  Indeed, if \(x,z\in P_r\) and \(x\leq y\leq z\), then we have \(y\leq z\in P_r\), so we obtain \(y\in P_r\).
  Therefore we have
  \[
  \begin{aligned}
    \bigl(\sigma_r^s(ab)\bigr)_{xz} =(ab)_{xz}=\sum_{x\leq y\leq z}a_{xy}b_{yz}
    =\sum_{\substack{x\leq y\leq z\\y\in P_r}} \bigl(\sigma_r^sa\bigr)_{xy}\bigl(\sigma_r^sb\bigr)_{yz} =\bigl((\sigma_r^sa)(\sigma_r^sb)\bigr)_{xz}.
  \end{aligned}
  \]
  The restriction map is unital and contractive.
  Extending a matrix by zero on pairs involving \(P_s\setminus P_r\) preserves its row norm.
  Hence every bonding map is a metric quotient.

  A compatible sequence determines one coefficient \(a_{xy}\) for every pair \(x\leq y\).
  Since the sets \(P_r\) increase to \(P\), we obtain its bounded-limit norm by monotone convergence:
  \[
  \sup_r\max_{x\in P_r}\sum_{\substack{y\in P_r\\
  y\geq x}}\abs{a_{xy}} =\sup_{x\in P}\sum_{y\geq x}\abs{a_{xy}} =\norm a_{\rm row}.
  \]
  Thus the coefficient identification gives \(\ilim(A_r,\sigma_r^s)\simeq\operatorname{IA}(P,c_0)\) isometrically as Banach algebras.
  Theorem~\ref{thm:main} applies.
\end{proof}

\begin{corollary}[Infinite upper-triangular matrices]
  \label{cor:upper-triangular}
  Let \(\operatorname{UT}_\infty\) be the algebra of scalar upper-triangular matrices \(a=(a_{ij})_{0\leq i\leq j<\infty}\) with \(\norm a_{\rm row}=\sup_i\sum_{j=i}^\infty\abs{a_{ij}}<\infty\).
  Then \(\operatorname{UT}_\infty\), with the row norm, is isometrically a Calkin algebra.
\end{corollary}

This follows directly from Corollary~\ref{cor:lower-finite-incidence} with \(P=\N\) in its usual order.

Further examples are given by countable rooted trees ordered by ancestry, \(\N^d\) with the coordinate order, and the Young lattice \cite[Example~3.4.4(b)]{Stanley2012}.
Whenever \(P\) contains distinct comparable elements, its incidence algebra is noncommutative.

\subsection{Compactness for quotient seminorms and dual operator algebras}

We next apply Theorem~\ref{thm:detector-compact-surjectivity} to function algebras and Proposition~\ref{prop:weak-star-rfd} to operator algebras.
For the scalar Lipschitz spaces and their algebra structure, see \cite[Chapters~2 and~7]{Weaver2018}.

\begin{corollary}[Algebra-valued H\"older classes]
  \label{cor:algebra-valued-holder}
  Let \((M,d)\) be a nonempty compact metric space, let \(0<\theta\leq1\), and let \(\widehat A=\ilim(A_r,\pi_r^s)\) be the bounded limit of a countable inverse system satisfying \eqref{eq:inverse-system}--\eqref{eq:metric-quotient}.
  For \(f:M\to\widehat A\), put
  \[ [f]_\theta =\sup_{x\ne y} \frac{\norm{f(x)-f(y)}_{\widehat A}}{d(x,y)^\theta}. \]
  If \(M\) is a singleton, this seminorm is defined to be zero.
  Equip the algebra \(\operatorname{Lip}_\theta(M,\widehat A) =\{f:M\to\widehat A:[f]_\theta<\infty\}\) with pointwise multiplication and the norm \(\norm f_{\operatorname{Lip}_\theta}=\norm f_\infty+[f]_\theta\).
  Then \(\operatorname{Lip}_\theta(M,\widehat A)\) is isometrically a Calkin algebra.
\end{corollary}

\begin{proof}
  Fix \(x_0\in M\).
  For \(f\) with \([f]_\theta<\infty\), we have \(\norm{f(x)}_{\widehat A}\leq\norm{f(x_0)}_{\widehat A}+[f]_\theta\operatorname{diam}(M)^\theta\) for every \(x\in M\).
  Thus every such \(f\) is bounded, and its stated norm is finite.
  From the identity \(f(x)g(x)-f(y)g(y)=f(x)\bigl(g(x)-g(y)\bigr)+\bigl(f(x)-f(y)\bigr)g(y)\), we obtain \([fg]_\theta\leq\norm f_\infty[g]_\theta+[f]_\theta\norm g_\infty\).
  For the full norm, we then get
  \[
    \norm{fg}_{\operatorname{Lip}_\theta}\leq\norm f_\infty\norm g_\infty+\norm f_\infty[g]_\theta+[f]_\theta\norm g_\infty
    \leq\norm f_{\operatorname{Lip}_\theta}\norm g_{\operatorname{Lip}_\theta}.
 \]
 The constant function with value \(1_{\widehat A}\) is a unit of norm one.

  If \((f_n)\) is Cauchy in this norm, let \(f\) be its uniform limit.
  By lower semicontinuity of the H\"older seminorm, we have
  \([f_n-f]_\theta\leq\liminf_{m\to\infty}[f_n-f_m]_\theta\to0\).
  Hence \(f\in\operatorname{Lip}_\theta(M,\widehat A)\) and \(\norm{f_n-f}_{\operatorname{Lip}_\theta}\to0\), which proves completeness.

  Let \(q_r:\widehat A\to A_r\) be the coordinate homomorphism, choose a sequence \((x_j)\) with dense range in \(M\), and put
  \[
    I_r=\{f:q_r(f(x_j))=0\text{ for }0\leq j\leq r\}.
  \]
  Each \(I_r\) is a closed two-sided ideal and \(I_{r+1}\subseteq I_r\).
  The map \(f+I_r\mapsto\bigl(q_r(f(x_0)),\ldots,q_r(f(x_r))\bigr)\) embeds the quotient into \(A_r^{r+1}\), so \(I_r\) has finite codimension.
  It is proper because it does not contain the constant unit.
  If \(f\in\bigcap_rI_r\), then, for fixed \(j,r\) and \(s\geq\max\{j,r\}\), we have \(q_r(f(x_j))=\pi_r^sq_s(f(x_j))=0\).
  Thus \(f(x_j)=0\) for every \(j\), and density and continuity give \(f=0\).
  Hence the filtration is separating.

  Fix \(R>0\), and let \((f_n)\subseteq R B_{\operatorname{Lip}_\theta(M,\widehat A)}\).
  Passing to a subsequence, suppose that \(\norm{f_n}_\infty\to\alpha\) and \([f_n]_\theta\to\lambda\), where \(\alpha+\lambda\leq R\).
  For each \(r\), we have \(\norm{q_rf_n}_\infty\leq R\) and \(\norm{q_rf_n(x)-q_rf_n(y)}\leq R d(x,y)^\theta\).
  Since \(A_r\) is finite dimensional, we obtain from the Arzel\`a--Ascoli theorem and a diagonal selection functions \(g_r\in C(M,A_r)\) such that, along a further subsequence,
  \[
    \sup_{x\in M}\norm{q_rf_n(x)-g_r(x)}\longrightarrow0
    \qquad(r\in\N).
  \]
  Passing to the limit in \(\pi_r^sq_sf_n(x)=q_rf_n(x)\), we obtain \(\pi_r^sg_s(x)=g_r(x)\) whenever \(r\leq s\).
  Moreover, \(\norm{g_r(x)}\leq\alpha\) and \(\norm{g_r(x)-g_r(y)}\leq\lambda d(x,y)^\theta\) for every \(r\).
  Thus \(f(x)=(g_r(x))_r\) belongs to \(\widehat A\), with \(\norm f_\infty\leq\alpha\) and \([f]_\theta\leq\lambda\), and we obtain
  \[
    \norm f_{\operatorname{Lip}_\theta}\leq\alpha+\lambda\leq R.
  \]
  The finite-dimensional embedding of the quotient into \(A_r^{r+1}\) has a bounded inverse on its range, so there is \(C_r<\infty\) such that \(\norm{h+I_r}\leq C_r\max_{0\leq j\leq r}\norm{q_r(h(x_j))}\) for every \(h\in\operatorname{Lip}_\theta(M,\widehat A)\).
  Applying this to \(h=f_n-f\), we obtain
  \[
    \norm{(f_n-f)+I_r}\leq C_r\max_{0\leq j\leq r}\norm{q_rf_n(x_j)-g_r(x_j)}
    \leq C_r\sup_{x\in M}\norm{q_rf_n(x)-g_r(x)}\longrightarrow0.
  \]
  
  Thus every sequence in the closed \(R\)-ball has a subsequence converging in every quotient seminorm to a member of that ball.
  The topology \(\tau_{(I_r)}\) is metrizable because the filtration is countable and separating.
  Hence the ball is compact in \(\tau_{(I_r)}\).
  Apply Theorem~\ref{thm:detector-compact-surjectivity} and then Theorem~\ref{thm:main}.
\end{proof}

\begin{corollary}[Scalar H\"older algebras]\label{cor:scalar-holder}\label{cor:lipschitz}
  Let \((M,d)\) be a nonempty compact metric space and let \(0<\theta\leq1\).
  The scalar H\"older algebra
  \[
    C^{0,\theta}(M)=\{f\in C(M):[f]_\theta<\infty\},\qquad \norm f_{C^{0,\theta}}=\norm f_\infty+[f]_\theta,
  \]
  with pointwise multiplication, is isometrically a Calkin algebra.
\end{corollary}

This follows directly from Corollary~\ref{cor:algebra-valued-holder} for the constant inverse system \(A_r=\K\) with identity bonding maps.
In this case, we have \(\widehat A=\K\) and \(C^{0,\theta}(M)=\operatorname{Lip}_\theta(M,\K)\), with the same norm and pointwise multiplication.

Other choices give, for example,
\( \operatorname{Lip}_\theta(M,\ell_\infty) \) and \( \operatorname{Lip}_\theta\left(M, \left(\bigoplus_{k\in\N}F_k\right)_{\ell_\infty}\right), \)
where the \(F_k\)'s are finite-dimensional unital Banach algebras with \(\norm{1_{F_k}}=1\).

\begin{corollary}[Atomic nest algebras]
  \label{cor:omega-atomic-nest}
  Let \(H=\bigoplus_{j=0}^\infty H_j\), where \(1\leq\dim H_j<\infty\), and put
  \[
    N_r=\bigoplus_{j=0}^rH_j,
    \qquad
    \mathcal N=\{0,N_0,N_1,\ldots,H\}.
  \]
  Then the nest algebra \(\operatorname{Alg}\mathcal N\), with its operator norm, is isometrically a Calkin algebra.
\end{corollary}

\begin{proof}
  Let \(P_r\) be the orthogonal projection onto \(N_r\).
  Recall that \(\operatorname{Alg}\mathcal N\) consists of the operators leaving every member of \(\mathcal N\) invariant \cite{Davidson1988}.
  This algebra is weak-operator closed and unital, and hence weak-star closed.
  Since \(N_r\) is invariant under every \(T\in\operatorname{Alg}\mathcal N\), we have
  \[ (I_H-P_r)TP_r=0. \]
  Moreover, \(P_rH\) is finite dimensional and \(P_r\to I_H\) strongly.
  Apply Lemma~\ref{lem:finite-triangular-corners}.
\end{proof}

In the reproducing-kernel examples below, \(\mathcal H\) consists of complex-valued functions on a set \(\Omega\), with reproducing kernel \(k\), and we write \(k_x=k(\,\cdot\,,x)\) for its kernel vectors \cite[Part~I, Sections~1--2]{Aronszajn1950}.

\begin{corollary}[Matrix-valued RKHS multipliers]
  \label{cor:rkhs-multipliers}
  \label{cor:matrix-rkhs-multipliers}
  Let \(\mathcal H\) be a nonzero separable complex scalar reproducing-kernel Hilbert space and let \(q\in\N_+\).
  Then
  \[ \operatorname{Mult}(\mathcal H\otimes\mathbb C^q) =M_q\bigl(\operatorname{Mult}(\mathcal H)\bigr), \]
  with the multiplier norm, is isometrically a Calkin algebra.
  Multiplier symbols inducing the same operator are identified throughout.
\end{corollary}

\begin{proof}
  For \(\Phi:\Omega\to M_q(\mathbb C)\), write \((M_\Phi f)(x)=\Phi(x)f(x)\).
  The multiplier algebra consists of the symbols for which \(M_\Phi\) is bounded on \(\mathcal H\otimes\mathbb C^q\), with norm \(\norm{M_\Phi}\). See \cite[Section~2]{AglerMcCarthy2000} for matrix-valued multipliers.
  Taking matrix entries relative to the standard basis of \(\mathbb C^q\) shows that every entry of \(M_\Phi\) is a bounded scalar multiplier.
  Conversely, a finite matrix of bounded scalar multiplier operators defines a bounded \(M_\Phi\).
  Thus the two multiplier algebras in the statement agree.
  Put \(\mathcal M=M_q(\operatorname{Mult}(\mathcal H))\).
  For each nonzero kernel vector \(k_x\), let \(E_x\) be the orthogonal projection onto \(k_x\otimes\mathbb C^q\).
  Then we have
  \[ T\in\mathcal M \quad\Longleftrightarrow\quad (I-E_x)T^*E_x=0\quad\text{for every such }x. \]
  Indeed, invariance defines a matrix \(\Phi(x)\) by \(T^*(k_x\otimes v)=k_x\otimes\Phi(x)^*v\), and we have
  \[ \langle(Tf)(x),v\rangle =\langle f,T^*(k_x\otimes v)\rangle =\langle\Phi(x)f(x),v\rangle. \]
  Thus we obtain \(T=M_\Phi\).
  At points with \(k_x=0\), every function vanishes, so no additional condition is needed.
  Each displayed corner condition is weak-operator closed.
  Hence \(\mathcal M\) is weak-star closed and has the quotient of the trace class by its preannihilator as a predual.

  Separability gives a sequence \((x_n)\) of points with \(k_{x_n}\ne0\) such that \(\overline{\operatorname{span}}\{k_{x_n}:n\in\N\}=\mathcal H\).
  Define \(\rho_n:\mathcal M\to M_q(\mathbb C)\) by \(\rho_n(M_\Phi)=\Phi(x_n)\).
  These are unital homomorphisms, and they are weak-star-to-norm continuous because we have
  \[ \langle\Phi(x_n)u,v\rangle =\frac{\langle M_\Phi(k_{x_n}\otimes u), k_{x_n}\otimes v\rangle}{\norm{k_{x_n}}^2}. \]
  Their joint faithfulness follows from
  \[ \rho_n(M_\Phi)=0\ (n\in\N) \ \Longrightarrow\ M_\Phi^*(k_{x_n}\otimes v)=0\ (n,v) \ \Longrightarrow\ M_\Phi=0. \]
  Proposition~\ref{prop:weak-star-rfd} applies.

  More explicitly, put \(I_r=\bigcap_{n=0}^r\ker\rho_n\) and identify \(\mathcal M/I_r\) with the finite algebra of attainable values \(W=(\Phi(x_0),\ldots,\Phi(x_r))\).
  Its norm is
  \begin{equation}\label{eq:rkhs-quotient-minimum}
    \norm W_{\mathcal M/I_r}
    =\min\{\norm{M_\Phi}:\Phi(x_i)=W_i,\ 0\leq i\leq r\}.
  \end{equation}
  To see attainment, take a bounded minimizing sequence of multipliers and a weak-star convergent subnet.
  The evaluation identities preserve its values, while weak-star lower semicontinuity gives the minimum.
  Coordinate deletion is the bonding map, and multiplication is \((W_i)_i(V_i)_i=(W_iV_i)_i\).
  To check the metric-quotient property, take a norm-minimizing multiplier \(M_\Phi\) for \(W\) at stage \(r\).
  For \(s\geq r\), its values \(W'=(\Phi(x_0),\ldots,\Phi(x_s))\) give a lift, and we have
  \[ \norm W_{\mathcal M/I_r}
     \leq\norm{W'}_{\mathcal M/I_s}
     \leq\norm{M_\Phi}
     =\norm W_{\mathcal M/I_r}. \]
  Theorem~\ref{thm:finite-detection-duality} therefore identifies the bounded limit with \(\mathcal M\), including its multiplier norm and product.
\end{proof}

\begin{corollary}[Complete Pick multiplier algebras]
  \label{cor:complete-pick}
  Every multiplier algebra of a nonzero separable scalar complete Pick space is isometrically a Calkin algebra.
  Fix distinct points \(x_0,\ldots,x_r\) with \(k(x_i,x_i)>0\), and let \((w_i)_{i=0}^r\) be attainable multiplier values at these points.
  Write \(\norm{(w_i)_{i=0}^r}\) for the quotient norm in \eqref{eq:rkhs-quotient-minimum}, with \(q=1\).
  Then, for every \(C\geq0\), we have
  \begin{equation}\label{eq:pick-certificate}
    \norm{(w_i)_{i=0}^r}\leq C
    \quad\Longleftrightarrow\quad
    \bigl[(C^2-w_i\overline{w_j})k(x_i,x_j)\bigr]_{i,j=0}^r
    \succeq0.
  \end{equation}
\end{corollary}

\begin{proof}
  The realization follows from Corollary~\ref{cor:rkhs-multipliers}.
  If \(M_\varphi\) interpolates \((w_i)_i\) and has norm at most \(C\), then we have
  \[ M_\varphi^*k_{x_i}=\overline{w_i}k_{x_i}. \]
  On \(\operatorname{span}\{k_{x_i}:0\leq i\leq r\}\), the norm bound gives
  \[ C^2\norm h^2-\norm{M_\varphi^*h}^2\geq0. \]
  Expanding this quadratic form gives the matrix in \eqref{eq:pick-certificate}.
  Conversely, for \(C>0\), its positivity and the complete Pick property give an interpolating multiplier of norm at most \(C\) \cite[Section~2]{AglerMcCarthy2000} (see also \cite{AglerMcCarthy2002}).
  If \(C=0\), its diagonal entries force \(w_i=0\) for every \(i\), and the zero multiplier suffices.
  Now use \eqref{eq:rkhs-quotient-minimum}.
\end{proof}

If \(\mathcal H\) is a nonzero separable scalar complete Pick space, fix distinct points \(x_0,\ldots,x_r\) with \(k(x_i,x_i)>0\), and let \((W_i)_{i=0}^r\) be attainable values of a matrix-valued multiplier from Corollary~\ref{cor:matrix-rkhs-multipliers}.
Writing \(\norm{(W_i)_{i=0}^r}\) for the quotient norm in \eqref{eq:rkhs-quotient-minimum}, we obtain from the matrix-valued complete Pick property \cite[Section~2]{AglerMcCarthy2000}, for every \(C\geq0\),
\[ \norm{(W_i)_{i=0}^r}\leq C \quad\Longleftrightarrow\quad \bigl[k(x_i,x_j)(C^2I_q-W_iW_j^*)\bigr]_{i,j=0}^r\succeq0. \]
For \(C=0\), positivity of the diagonal blocks gives \(-k(x_i,x_i)W_iW_i^*\succeq0\) and hence \(W_i=0\) for every \(i\), so the zero multiplier gives the required interpolant.

This includes \(H^\infty(\mathbb D)\) and the multiplier algebras \(\operatorname{Mult}(H_d^2)\) of the Drury--Arveson spaces \cite[Section~4]{AglerMcCarthy2000}, as well as the standard Dirichlet-type complete Pick multiplier algebras \cite{AglerMcCarthy2002} (for the classical Dirichlet kernel, see \cite[Section~1]{AglerMcCarthy2000}).
For \(H^\infty(\mathbb D)\), the Taylor and point-evaluation filtrations give the following quotients.

\begin{corollary}[Taylor and point-evaluation quotients of \(H^\infty\)]
  \label{cor:hardy-oracles}
  The algebra \(H^\infty(\mathbb D)\) has the following bounded detection envelopes.
  Let \(S_r\) be the truncated shift on \(\mathbb C^{r+1}\): \(S_re_j=e_{j+1}\) for \(0\leq j<r\), and \(S_re_r=0\), where \((e_j)_{j=0}^r\) is the standard basis.
  \begin{enumerate}[label={\textup{(\roman*)}},leftmargin=27pt]
    \item For the Taylor filtration \(I_r=z^{r+1}H^\infty\), the quotient is the algebra of analytic Toeplitz matrices
    \[ \left\{\sum_{j=0}^ra_jS_r^j:a_0,\ldots,a_r\in\mathbb C\right\}, \]
    with its operator norm.
    \item If \((\lambda_n)\subseteq\mathbb D\) is pairwise distinct and has an accumulation point in \(\mathbb D\), and
    \[ I_r=\{f\in H^\infty: f(\lambda_0)=\cdots=f(\lambda_r)=0\}, \]
    then \(H^\infty/I_r\) is the finite interpolation algebra whose norm is determined, for every \(C\geq0\), by
    \[ \norm{(w_i)}\leq C \quad\Longleftrightarrow\quad \left[ \frac{C^2-w_i\overline{w_j}} {1-\lambda_i\overline{\lambda_j}} \right]_{i,j=0}^r\succeq0. \]
  \end{enumerate}
\end{corollary}

\begin{proof}
  For \textup{(i)}, write \(\widehat f(j)=f^{(j)}(0)/j!\) and put \(T_r(f)=\sum_{j=0}^r\widehat f(j)S_r^j\).
  From Taylor multiplication, we obtain
  \[ T_r(fg)=T_r(f)T_r(g),\qquad \ker T_r=z^{r+1}H^\infty,\qquad T_r(1)=I. \]
  Compression of multiplication on \(H^2\) gives \(\norm{T_r(f)}\leq\norm f_\infty\).
  By the Carath\'eodory--Fej\'er theorem, we obtain the reverse quotient inequality in its attained form \cite{Sarason1967}:
  \[ \min\{\norm f_\infty:\widehat f(j)=a_j,\ 0\leq j\leq r\} =\norm{\textstyle\sum_{j=0}^ra_jS_r^j}. \]
  Hence these are exactly the quotient norms.
  The bonding maps truncate coefficients, so Lemma~\ref{lem:detector-quotient-tower} gives the metric-quotient condition.

  A bounded compatible sequence determines coefficients \((a_j)\).
  Put
  \[ C=\sup_r\norm{\textstyle\sum_{j=0}^ra_jS_r^j}<\infty. \]
  Choose the preceding interpolants \(f_r\) with \(\norm{f_r}_\infty\leq C\).
  By Montel's theorem, we obtain a subsequence \((f_{r_k})\) converging locally uniformly to \(f\in H^\infty\).
  By Cauchy's formula, we have
  \(
  \widehat f(j)=\lim_k\widehat f_{r_k}(j)=a_j.
  \)
  For the norm of \(f\), we obtain
  \[
    C=\sup_r\norm{T_r(f)}\leq\norm f_\infty\leq C.
  \]
  Thus the bounded limit is \(H^\infty\), with its original norm, and Taylor coefficients give uniqueness and preserve multiplication.

  For \textup{(ii)}, finite polynomial interpolation makes \(f\mapsto(f(\lambda_i))_{i=0}^r\) onto \(\mathbb C^{r+1}\).
  Its product is coordinatewise, and the Nevanlinna--Pick theorem gives the displayed quotient norm, again with a norm-minimizing interpolant \cite{Sarason1967}.
  As above, the bonding maps are metric quotients.
  If \((w_i)_{i\geq0}\) satisfies \(\sup_r\norm{(w_0,\ldots,w_r)}=C<\infty\) in the finite interpolation quotient norms, choose \(f_r\) such that
  \[ f_r(\lambda_i)=w_i\quad(0\leq i\leq r), \qquad \norm{f_r}_\infty\leq C. \]
  A locally uniformly convergent subsequence has a limit \(f\) with
  \[ f(\lambda_i)=w_i\quad(i\geq0). \]
  For its norm, we obtain
  \[ C=\sup_r\norm{f+I_r}\leq\norm f_\infty\leq C. \]
  The interior accumulation point and the identity theorem give uniqueness.
  Thus this bounded limit also identifies isometrically with \(H^\infty\), and the identification is multiplicative.
\end{proof}

\begin{corollary}[Canonical free semigroupoid algebras]
  \label{cor:graph-semigroupoid}
  \label{cor:free-semigroup-algebra}
  Let \(G\) be a nonempty countable directed graph.
  Then its canonical free semigroupoid algebra \(\mathfrak L_G\), with its operator norm, is isometrically a Calkin algebra.
\end{corollary}

\begin{proof}
  Let \(\mathbb F_G^+\) be the set of finite paths in \(G\), including the vertices, and put \(H_G=\ell_2(\mathbb F_G^+)\).
  In this example, \(\mu\nu\) first traverses \(\nu\) and then \(\mu\), as in \cite[Section~3]{KribsPower2004}.
  Write \((\xi_\nu)_{\nu\in\mathbb F_G^+}\) for the standard orthonormal basis of \(H_G\).
  For \(\mu\in\mathbb F_G^+\), let
  \[
  L_\mu\xi_\nu =\begin{cases} \xi_{\mu\nu},&\text{if }\mu\nu\in\mathbb F_G^+,\\
  0,&\text{otherwise}. \end{cases}
  \]
  The algebra \(\mathfrak L_G\) is the weak-operator closed algebra generated by these operators, and hence is weak-star closed.
  It is unital, since finite sums of the vertex projections converge strongly to \(I_{H_G}\).

  Choose a sequence \((\lambda_n)\) whose range is \(\mathbb F_G^+\).
  Let \(D_n\) be the set of all right factors of \(\lambda_0,\ldots,\lambda_n\), let
  \[ H_n=\operatorname{span}\{\xi_\nu:\nu\in D_n\}, \]
  and let \(P_n\) be the orthogonal projection onto \(H_n\).
  Every finite path has only finitely many right factors.
  Hence \(H_n\) is finite dimensional, \(P_n\leq P_{n+1}\), and \(P_n\to I_{H_G}\) strongly.

  Fix \(\mu\in\mathbb F_G^+\) and \(\nu\notin D_n\).
  If \(L_\mu\xi_\nu\ne0\) and \(\mu\nu\in D_n\), then \(\nu\) is a right factor of an element of \(D_n\), and hence a right factor of one of \(\lambda_0,\ldots,\lambda_n\).
  This contradicts \(\nu\notin D_n\).
  Therefore we obtain
  \[ P_nL_\mu(I_{H_G}-P_n)=0. \]
  This corner condition is preserved by sums and products and is weak-operator closed, so we have
  \[ P_n\mathfrak L_G(I_{H_G}-P_n)=0. \]
  Lemma~\ref{lem:finite-triangular-corners} applies.
  For the graph with one vertex and \(d\) loops this recovers \(\mathcal L_d\), and, when \(d\) is finite, the usual length compressions are the noncommutative interpolation models of \cite{DavidsonPitts1998}.
\end{proof}

\begin{corollary}[Finite-dimensional \(W^*\)-correspondence Hardy algebras]
  \label{cor:finite-wstar-hardy}
  Let \(M\) be a finite-dimensional von Neumann algebra and let \(E\) be a \(W^*\)-correspondence over \(M\) which is finite dimensional as a complex vector space.
  Then the Hardy algebra \(H^\infty(E)\), with its standard operator norm, is isometrically a Calkin algebra.
\end{corollary}

\begin{proof}
  Choose a faithful normal representation \(\sigma:M\to\Bcal(K)\) on a finite-dimensional Hilbert space.
  Here \(E^{\otimes j}\) denotes the \(j\)-fold internal tensor product over \(M\), with \(E^{\otimes0}=M\), and \(\mathcal F(E)=\bigoplus_{j\geq0}E^{\otimes j}\) is the Fock module.
  On the induced Fock-space representation
  \[ \mathcal K_E =\mathcal F(E)\otimes_\sigma K =\bigoplus_{j=0}^\infty E^{\otimes j}\otimes_\sigma K, \]
  the Hardy algebra \(H^\infty(E)\) is the ultraweak closure of the algebra generated by the diagonal left action of \(M\) and the creation operators \cite[Definition~2.2]{MuhlySolel2004}.
  The induced representation is faithful and normal \cite[Lemma~2.1 and Corollary~2.3]{MuhlySolel2004}, so we identify \(H^\infty(E)\) with its image in \(\Bcal(\mathcal K_E)\).

  Let \(P_r\) be the orthogonal projection onto \(\mathcal K_{E,\leq r}=\bigoplus_{j=0}^r E^{\otimes j}\otimes_\sigma K\).
  This subspace is finite dimensional, and \(P_r\to I_{\mathcal K_E}\) strongly.
  The diagonal left action preserves tensor degree and every creation operator increases it.
  Hence every analytic polynomial \(T\) satisfies \(P_rT(I_{\mathcal K_E}-P_r)=0\).
  The map \(T\mapsto P_rT(I_{\mathcal K_E}-P_r)\) is ultraweakly continuous, so the same identity holds for every \(T\in H^\infty(E)\).
  Apply Lemma~\ref{lem:finite-triangular-corners}.
\end{proof}

\begin{corollary}[Weak-star semicrossed products]
  \label{cor:finite-semicrossed-products}
  Let \(A\subseteq\Bcal(H_0)\) be a finite-dimensional unital operator algebra on a finite-dimensional Hilbert space, and let \(\beta:A\to A\) be a unital completely contractive endomorphism.
  On \(H_0\otimes\ell_2(\N)\), define
  \[ \pi_\beta(a)(\xi\otimes e_n)=\beta^n(a)\xi\otimes e_n, \qquad V(\xi\otimes e_n)=\xi\otimes e_{n+1}. \]
  This is the Fock representation of the dynamical system \((A,\beta)\) (see \cite[Example~2.5]{DavidsonFullerKakariadis2018}).
  Let \(\mathfrak S_\beta(A)\) be the weak-star semicrossed product of \(A\) by \(\beta\), namely the weak-star closed algebra generated by \(\pi_\beta(A)\) and \(V\) \cite[Definition~2.1]{Kakariadis2012}.
  Then \(\mathfrak S_\beta(A)\), with its operator norm, is isometrically a Calkin algebra.
\end{corollary}

\begin{proof}
  Since every iterate of \(\beta\) is contractive and the degree-zero block is \(a\), we have
  \[ \norm{\pi_\beta(a)}
     =\sup_{n\geq0}\norm{\beta^n(a)}
     =\norm a. \]
  Thus the displayed formula defines a bounded isometric representation.
  The algebra \(\mathfrak S_\beta(A)\) is weak-star closed and unital.
  Let \(P_r\) be the orthogonal projection onto \(H_0\otimes\operatorname{span}\{e_0,\ldots,e_r\}\).
  These projections have finite rank and converge strongly to the identity.
  Each \(\pi_\beta(a)\) preserves the second tensor degree, while \(V\) increases it by one.
  Therefore every polynomial \(T\) in the generators satisfies \(P_rT(I-P_r)=0\).
  This corner condition is weak-star closed, and hence it holds for every \(T\in\mathfrak S_\beta(A)\).
  Apply Lemma~\ref{lem:finite-triangular-corners}.
\end{proof}

\subsection{\texorpdfstring{Harmonic analysis and quantum convolution algebras}{Harmonic analysis and quantum convolution algebras}}

The measure and quantum convolution algebras below satisfy the hypotheses of Proposition~\ref{prop:coalgebra}. Fourier multiplier algebras are treated through pointwise compactness of their norm balls.

\begin{corollary}[Compact hypergroup measure algebras]
  \label{cor:compact-measures}
  Let \(H\) be a compact metrizable hypergroup in the sense of Jewett \cite{Jewett1975}.
  Then \(M(H)\), with convolution and the total-variation norm, is isometrically a Calkin algebra.
\end{corollary}

\begin{proof}
  For \(f\in C(H)\), define
  \[ (\Delta f)(x,y) =\int_H f(z)\,d(\delta_x\star\delta_y)(z) \qquad(x,y\in H). \]
  The hypergroup axioms imply that
  \[ \Delta:C(H)\longrightarrow C(H\times H) =C(H)\otimes_{\min}C(H) \]
  is unital and positive.
  Since \(C(H)\) is commutative, \(\Delta\) is completely positive \cite[Theorem~3.11]{Paulsen2003}.
  Associativity gives \((\Delta\otimes I)\Delta=(I\otimes\Delta)\Delta\), and evaluation \(\varepsilon(f)=f(e)\) at the identity is a character satisfying the two counit identities.

  By the Peter--Weyl theorem for compact hypergroups, every continuous irreducible representation of \(H\) is finite dimensional and the span of all matrix coefficients is uniformly dense in \(C(H)\) \cite[Theorems~2.2 and~2.13]{Vrem1979}.
  Since \(H\) is metrizable, \(C(H)\) is separable.
  Choose a countable dense subset of \(C(H)\), and approximate each of its members to every accuracy \(1/m\) by finite linear combinations of matrix coefficients.
  The resulting countable family of coefficients involves only countably many representations.
  Taking their full coefficient spaces and including the trivial representation first gives a sequence \((\pi_n)_{n\geq0}\) whose coefficient spaces have dense linear span.

  Write \(d_n=\dim\pi_n\), choose an orthonormal basis for each representation space, and put
  \[
  u_{ij}^{\pi_n}(x) =\langle\pi_n(\delta_x)e_j,e_i\rangle,\qquad C_r=\operatorname{span} \{u_{ij}^{\pi_n}:0\leq n\leq r,\ 1\leq i,j\leq d_n\}.
  \]
  Then \((C_r)\) is increasing and finite dimensional, contains the constants, and has dense union.
  Since \(\pi_n\) is a representation of the measure algebra, we have
  \begin{align*}
    \Delta(u_{ij}^{\pi_n})(x,y) =\langle\pi_n(\delta_x\star\delta_y)e_j,e_i\rangle=\langle\pi_n(\delta_x)\pi_n(\delta_y)e_j,e_i\rangle =\sum_{k=1}^{d_n}u_{ik}^{\pi_n}(x)u_{kj}^{\pi_n}(y).
  \end{align*}
  Hence we have \(\Delta(C_r)\subseteq C_r\otimes C_r\).
  Under the Riesz identification \(C(H)^*=M(H)\), the dual product induced by \(\Delta\) is the usual measure convolution.
  Apply Proposition~\ref{prop:coalgebra}.
\end{proof}

This includes \(M(\mathbb T)\), \(M(\mathbb T^d)\), and \(M(SU(2))\).
The last arises from the connected non-abelian compact group \(SU(2)\).
The finite-dimensional quotient \(M(H)/C_r^\perp\) carries the quotient total-variation norm.

There are also non-group examples.
Let \(\varphi=(1+\sqrt5)/2\), and give \(\{e,\tau\}\) identity \(e\), trivial involution, and convolution \cite[Example~4.1]{Voit1996}
\[ \delta_\tau\star\delta_\tau =\varphi^{-2}\delta_e+\varphi^{-1}\delta_\tau. \]
The countable product \(H_{\rm Fib}=\prod_{n\geq1}\{e,\tau\}\), with
\[ \delta_x\star\delta_y=\bigotimes_{n\geq1}(\delta_{x_n}\star\delta_{y_n}), \]
is a compact metrizable hypergroup, as the inverse limit of its finite coordinate hypergroups \cite[Theorem~1.5]{Voit1996}.
Hence \(M(H_{\rm Fib})\) is an infinite-dimensional isometric Calkin algebra.

\begin{corollary}[Fourier--Stieltjes algebras]\label{cor:fourier-stieltjes}
  Let \(\Gamma\) be a countable discrete group.
  Then \(B(\Gamma)\), with pointwise multiplication, is isometrically a Calkin algebra.
\end{corollary}

\begin{proof}
  By Eymard's identification \cite{Eymard1964}, we have \(B(\Gamma)=C^*(\Gamma)^*\) isometrically.
  Put \(C=C^*(\Gamma)\), and write \(u_g\) for its canonical unitaries.
  The universal property of \(C^*(\Gamma)\) gives a unital \(^*\)-homomorphism \(\Delta:C\to C\otimes_{\min}C\) and a character \(\varepsilon:C\to\mathbb C\) with
  \[ \Delta(u_g)=u_g\otimes u_g, \qquad \varepsilon(u_g)=1. \]
  Coassociativity and the counit identities hold on the generators, hence on \(C\).
  Choose a sequence \((g_n)\) whose range is \(\Gamma\), and put \(C_r=\operatorname{span}\{1,u_{g_0},\ldots,u_{g_r}\}\).
  Then we have \(\Delta(C_r)\subseteq C_r\otimes C_r\), and the group algebra is dense in \(C^*(\Gamma)\).
  For \(f,g\in B(\Gamma)\), we have
  \[ (f\star g)(u_s)=(f\otimes g)(u_s\otimes u_s)=f(s)g(s) \qquad(s\in\Gamma). \]
  Thus the convolution in Proposition~\ref{prop:coalgebra} is the prescribed pointwise product.
  Apply that proposition.
\end{proof}

In particular, \(B(\mathbb F_2)\) is an isometric Calkin algebra.

\begin{corollary}[Fourier multiplier algebras]
  \label{cor:fourier-multipliers}
  Let \(\Gamma\) be a countable discrete group, and let \(A(\Gamma)\) be its Fourier algebra \cite{Eymard1964}.
  The algebras
     \[ MA(\Gamma) \qquad\text{and}\qquad M_{\mathrm{cb}}A(\Gamma)=B_2(\Gamma), \]
  with pointwise multiplication and, respectively, the multiplier norm and the completely bounded multiplier norm, are isometrically Calkin algebras of separable Banach spaces.
\end{corollary}

\begin{proof}
  For a scalar function \(\varphi\) on \(\Gamma\), write \(M_\varphi f=\varphi f\).
  Thus \(MA(\Gamma)\) consists of the functions for which \(M_\varphi:A(\Gamma)\to A(\Gamma)\) is bounded, with norm \(\norm{M_\varphi}\), while the completely bounded version uses the canonical operator-space structure of \(A(\Gamma)\). See \cite[Section~2]{Knudby2014Semigroups}.
  These are unital Banach algebras with \(\norm1=1\).
  For the point indicators \(\delta_s=\mathbf1_{\{s\}}\in A(\Gamma)\), we have
  \[
    \norm\varphi_\infty
    =\sup_{s\in\Gamma}
      \frac{\norm{M_\varphi\delta_s}_{A(\Gamma)}}{\norm{\delta_s}_{A(\Gamma)}}
    \leq\norm\varphi_{MA}
    \leq\norm\varphi_{M_{\mathrm{cb}}A}.
  \]

  We first check that closed balls of \(MA(\Gamma)\) are pointwise closed.
  Let \(\varphi_i\to\varphi\) pointwise and \(\norm{M_{\varphi_i}}\leq R\).
  Write \(\lambda:\Gamma\to\Bcal(\ell_2(\Gamma))\) for the left regular representation.
  Finitely supported functions are dense in \(A(\Gamma)\): approximate the vectors in coefficients of the left regular representation by finitely supported vectors, using the coefficient-norm estimate \(\norm{\langle\lambda(\cdot)\xi,\eta\rangle}_{A(\Gamma)}\leq\norm\xi_2\norm\eta_2\) \cite{Eymard1964}.
  For a finitely supported \(f\), all \(\varphi_i f\) lie in the same finite-dimensional subspace of \(A(\Gamma)\), so by pointwise convergence we obtain
  \[ \norm{\varphi f}_{A(\Gamma)} =\lim_i\norm{\varphi_i f}_{A(\Gamma)} \leq R\norm f_{A(\Gamma)}. \]
  Hence multiplication by \(\varphi\) extends to an operator of norm at most \(R\) on \(A(\Gamma)\), and continuity of point evaluations identifies the extension with \(M_\varphi\).

  For the completely bounded case, put \(\widehat\varphi(s,t)=\varphi(t^{-1}s)\).
  By the Herz--Schur identification and the finite-restriction formula for the Schur-multiplier norm \cite[Section~2 and Lemma~4.1]{Knudby2014Semigroups}, we have
  \[
    \norm\varphi_{M_{\mathrm{cb}}A}
    =\norm{\widehat\varphi}_{\mathrm S}
    =\sup_{\substack{F\subseteq\Gamma\\F\text{ finite}}}
      \norm{\widehat\varphi|_{F\times F}}_{\mathrm S}.
  \]
  Each finite restriction norm is continuous in the entries, so these closed balls are also pointwise closed.
  In either algebra, the closed \(R\)-ball is therefore a closed subset of the compact product \(\prod_{s\in\Gamma}\{z\in\mathbb C:\abs z\leq R\}\), and is pointwise compact.

  Now let \(B\) denote either algebra, choose increasing nonempty finite sets \(F_r\) with union \(\Gamma\), and put
  \[ I_r=\{\varphi\in B:\varphi|_{F_r}=0\}. \]
  These are proper closed two-sided ideals of finite codimension, and \(\bigcap_r I_r=\{0\}\).
  Restriction induces a linear injection \(B/I_r\to\mathbb C^{F_r}\).
  Since its domain is finite dimensional, the inverse on its range is bounded.
  Thus there is \(C_r<\infty\) such that, for \(\varphi\in B\), we have
  \[ \max_{s\in F_r}\abs{\varphi(s)}
     \leq\norm{\varphi+I_r}
     \leq C_r\max_{s\in F_r}\abs{\varphi(s)}. \]
  Consequently \(\tau_{(I_r)}\) agrees with the topology of pointwise convergence, although the quotient norms need not be maximum norms.
  The closed balls are compact in \(\tau_{(I_r)}\), so Theorems~\ref{thm:detector-compact-surjectivity} and~\ref{thm:main} apply.
\end{proof}

In particular, both \(MA(\mathbb F_2)\) and the Herz--Schur multiplier algebra \(B_2(\mathbb F_2)\) occur with the multiplier and completely bounded multiplier norms, respectively.
The latter strictly contains \(B(\mathbb F_2)\), since \(B(\Gamma)=B_2(\Gamma)\) holds exactly for amenable discrete groups \cite[Proposition~5.6]{Knudby2014Semigroups}.

For a compact quantum group \(\mathbb G\), write \(\operatorname{Pol}(\mathbb G)\) for the Hopf \(^*\)-algebra spanned by its finite-dimensional corepresentation coefficients \cite[Theorem~1.2]{Woronowicz1998}.
Write \(C^u(\mathbb G)\) and \(C_{\mathrm{red}}(\mathbb G)\) for its universal and reduced \(C^*\)-completions, respectively \cite[Sections~2--3]{BedosMurphyTuset2001}.

\begin{corollary}[Universal compact quantum-group duals]
  \label{cor:compact-quantum}
  Let \(\mathbb G\) be a compact quantum group for which \(C^u(\mathbb G)\) is separable.
  Then \(C^u(\mathbb G)^*\), with convolution, is isometrically a Calkin algebra.
\end{corollary}

\begin{proof}
  The Hopf \(^*\)-algebra spanned by the irreducible corepresentation coefficients is dense in \(C^u(\mathbb G)\) \cite[Theorem~1.2]{Woronowicz1998}.
  On the universal completion, the coproduct extends to a unital \(^*\)-homomorphism and the counit to a character \cite[Section~3]{BedosMurphyTuset2001}.
  On the coefficients, we have
  \[ \Delta(u_{ij}^{\alpha})=\sum_k u_{ik}^{\alpha}\otimes u_{kj}^{\alpha}, \qquad \varepsilon(u_{ij}^{\alpha})=\delta_{ij}. \]
  Applying the countable selection used in the proof of Corollary~\ref{cor:compact-measures}, we obtain a sequence of complete coefficient spaces with dense span.
  Include the trivial corepresentation first, and let \(C_r\) be the sum of the first \(r+1\) spaces.
  Then we have
  \[ 1\in C_r,\qquad \Delta(C_r)\subseteq C_r\otimes C_r, \qquad \overline{\bigcup_rC_r}=C^u(\mathbb G). \]
  Proposition~\ref{prop:coalgebra} applies.
\end{proof}

Thus the universal duals associated with \(SU_q(2)\) (\(0<q<1\)) \cite{Woronowicz1987} and the free orthogonal compact quantum groups \(O_N^+\) (\(N\geq2\)) \cite{Wang1995} are noncommutative examples.
The UCP hypothesis in Proposition~\ref{prop:coalgebra} also permits coproducts that are not multiplicative.

\begin{corollary}[Compact quantum-hypergroup duals]
  \label{cor:compact-quantum-hypergroups}
  Let \(C\) be the separable unital \(C^*\)-algebra of a compact quantum hypergroup in the sense of Chapovsky--Vainerman \cite[Definition~4.1]{ChapovskyVainerman1999}.
  Then \(C^*\), with its dual norm and convolution product, is isometrically a Calkin algebra.
\end{corollary}

\begin{proof}
  The axioms provide a coassociative UCP coproduct and a character counit \cite[Definitions~1.1 and~4.1]{ChapovskyVainerman1999}.
  If the coproduct is initially given in the maximal tensor product, compose it with the canonical quotient onto the minimal tensor product.
  The Peter--Weyl theorem gives a norm-dense span of finite-dimensional corepresentation coefficients \cite[Theorem~5.11]{ChapovskyVainerman1999}, satisfying
  \[ \Delta(t_{ij}^{\alpha})=\sum_k t_{ik}^{\alpha}\otimes t_{kj}^{\alpha}. \]
  By separability and the countable selection in the proof of Corollary~\ref{cor:compact-measures}, we obtain complete coefficient spaces with dense span.
  Let \(C_r\) be the span of \(1_C\) and the first \(r+1\) such spaces.
  These spaces are increasing and finite dimensional, their union is dense in \(C\), and \(\Delta(C_r)\subseteq C_r\otimes C_r\).
  Proposition~\ref{prop:coalgebra} applies.
\end{proof}

This includes the compact quantum double-coset hypergroups of \cite[Section~3 and Example~4.3]{ChapovskyVainerman1999}.
Thus the finite-coalgebra argument covers both classical hypergroup measure algebras and their quantum counterparts.

The bounded-counit condition can be removed at the cost of standard unitization.
The following construction preserves the prescribed norm.

\begin{lemma}[Adjoining a counit]
  \label{lem:coalgebra-counitalization}
  Let \(C\) be a unital \(C^*\)-algebra with a coassociative UCP map \(\Delta:C\to C\otimes_{\min}C\).
  Suppose that \(C\) has increasing finite-dimensional subspaces \(C_r\) with dense union such that \(1_C\in C_r\) and \(\Delta(C_r)\subseteq C_r\otimes C_r\).
  No bounded counit is assumed.
  Then the standard unitization
  \[
  (C^*)^{\#}=C^*\oplus_1\mathbb C,\qquad (\mu,\lambda)(\nu,\eta) =(\mu\star\nu+\lambda\nu+\eta\mu,\lambda\eta),
  \]
  with \(\norm{(\mu,\lambda)}=\norm\mu+\abs\lambda\), is isometrically a Calkin algebra.
\end{lemma}

\begin{proof}
  Put \(\widetilde C=C\oplus_\infty\mathbb C\), with its usual \(C^*\)-algebra structure and unit \((1_C,1)\).
  In the canonical decomposition
  \[ \widetilde C\otimes_{\min}\widetilde C
     \cong(C\otimes_{\min}C)\oplus_\infty C\oplus_\infty C\oplus_\infty\mathbb C, \]
  define
  \[ \widetilde\Delta(a,\lambda)=(\Delta a,a,a,\lambda), \qquad
     \widetilde\varepsilon(a,\lambda)=\lambda. \]
  Each coordinate of \(\widetilde\Delta\) is completely positive, and the displayed formula preserves the unit, so \(\widetilde\Delta\) is UCP.
  The functional \(\widetilde\varepsilon\) is a character, and slicing either tensor factor by it gives \((a,\lambda)\).
  Under the eight-component decomposition of \(\widetilde C^{\otimes_{\min}3}\), ordered lexicographically with \(C\) before \(\mathbb C\) in each factor, the components of
  \((\widetilde\Delta\otimes I)\widetilde\Delta(a,\lambda)\) are
  \[ \bigl((\Delta\otimes I)\Delta a,\Delta a,\Delta a,a,\Delta a,a,a,\lambda\bigr). \]
  Those of \((I\otimes\widetilde\Delta)\widetilde\Delta(a,\lambda)\) are the same, with first component \((I\otimes\Delta)\Delta a\).
  Hence \(\widetilde\Delta\) is coassociative.
  The subspaces \(\widetilde C_r=C_r\oplus\mathbb C\) satisfy
  \[ (1_C,1)\in\widetilde C_r, \qquad
     \widetilde\Delta(\widetilde C_r)\subseteq\widetilde C_r\otimes\widetilde C_r,
     \qquad \overline{\bigcup_r\widetilde C_r}=\widetilde C. \]
  Under the canonical isometric identification \(\widetilde C^*=C^*\oplus_1\mathbb C\), we have \(\langle(\mu,\lambda),(a,z)\rangle=\mu(a)+\lambda z\) and \(\norm{(\mu,\lambda)}=\norm\mu+\abs\lambda\).
  Moreover, we have
  \begin{align*}
    \langle(\mu,\lambda)\star(\nu,\eta),(a,z)\rangle
    &=(\mu\otimes\nu)(\Delta a)+\eta\mu(a)+\lambda\nu(a)+\lambda\eta z\\
    &=\langle(\mu\star\nu+\lambda\nu+\eta\mu,\lambda\eta),(a,z)\rangle.
  \end{align*}
  Thus this isometric identification gives exactly the unitization product.
  Apply Proposition~\ref{prop:coalgebra} to \(\widetilde C\).
\end{proof}

\begin{corollary}[Reduced and exotic quantum-group completions]
  \label{cor:quantum-completions}
  Let \(C_\mu(\mathbb G)\) be a separable completion of \(\operatorname{Pol}(\mathbb G)\) in a quantum group norm: the algebraic coproduct extends to a unital \(^*\)-homomorphism
  \[ \Delta_\mu:C_\mu(\mathbb G)\longrightarrow
       C_\mu(\mathbb G)\otimes_{\min}C_\mu(\mathbb G). \]
  If the algebraic counit is bounded in this norm, then \(C_\mu(\mathbb G)^*\), with its dual norm and convolution, is isometrically a Calkin algebra.
  Without any counit assumption, its standard unitization
  \[ \bigl(C_\mu(\mathbb G)^*\bigr)^{\#}
       =C_\mu(\mathbb G)^*\oplus_1\mathbb C \]
  is isometrically a Calkin algebra.
  In particular, the latter assertion applies to every separable reduced completion, without coamenability, and to every separable exotic quantum-group completion.
\end{corollary}

\begin{proof}
  Choose a sequence of complete corepresentation coefficient spaces with dense span in \(C_\mu(\mathbb G)\), including the trivial corepresentation first.
  Their successive sums \(C_r\) contain \(1\) and satisfy \(\Delta_\mu(C_r)\subseteq C_r\otimes C_r\), as in the proof of Corollary~\ref{cor:compact-quantum}.
  Coassociativity holds on the dense Hopf \(^*\)-algebra and therefore on the completion.
  A bounded algebraic counit extends to a character satisfying both counit identities, so Proposition~\ref{prop:coalgebra} gives the first assertion.
  Lemma~\ref{lem:coalgebra-counitalization} gives the second.
\end{proof}

The coproduct-extension condition is part of the term \emph{quantum group norm}, but it is not automatic for an arbitrary \(C^*\)-norm on \(\operatorname{Pol}(\mathbb G)\) \cite[Definition~7.1 and Section~9]{KyedSoltan2012}.
When the original counit is unbounded, the preceding direct-sum construction agrees with adjoining the neutral element in \cite[Section~8]{KyedSoltan2012}.
For a non-coamenable reduced completion, the ununitized convolution dual has no identity.
Indeed, if \(\mu\) were a convolution identity, separation by bounded functionals would give
\[ (\mu\otimes I)\Delta(a)=a
   \qquad\bigl(a\in C_{\mathrm{red}}(\mathbb G)\bigr). \]
For a finite-dimensional unitary corepresentation \(u=(u_{ij})\), put \(A=(\mu(u_{ij}))\).
The coefficient formula for \(\Delta\) then gives \(Au=u\), and multiplication by \(u^*\) yields \(A=I\).
Thus \(\mu(u_{ij})=\delta_{ij}\), so \(\mu\) would extend the algebraic counit boundedly, contrary to non-coamenability \cite[Section~2]{BedosMurphyTuset2001}.
Hence the ununitized dual cannot itself be a nonzero Calkin algebra.
For example, \(\bigl(C_{\mathrm{red}}(O_N^+)^*\bigr)^{\#}\) is realized for every \(N\geq3\), and these reduced compact quantum groups are non-coamenable \cite[Remark~8.4(3)]{KyedSoltan2012}.

\begin{corollary}[Sequentially ultramatricial \(R^*\)-algebras]\label{cor:mori-rstar}
  Let
  \[ \mathscr R_{\rm alg}=\bigcup_{r=0}^\infty\mathscr R_r \]
  be an increasing union of finite-dimensional \(C^*\)-algebras with a common unit, whose connecting maps \(\iota_r^s:\mathscr R_r\to\mathscr R_s\) are unital injective \(^*\)-homomorphisms.
  Suppose that these inclusions preserve coassociative unital \(^*\)-homomorphic coproducts \(\Delta_r:\mathscr R_r\to \mathscr R_r\otimes_{\min}\mathscr R_r\) and character counits \(\varepsilon_r:\mathscr R_r\to\mathbb C\), in the sense that
  \[ \Delta_s\iota_r^s =(\iota_r^s\otimes\iota_r^s)\Delta_r, \qquad \varepsilon_s\iota_r^s=\varepsilon_r \qquad(r\leq s). \]
  If \(R=\overline{\mathscr R_{\rm alg}}\), then its Banach dual \(R^*\), with convolution, is isometrically a Calkin algebra.
\end{corollary}

\begin{proof}
  If \(x\in\mathscr R_r\), then we have
  \[ \norm{\Delta_r x}\leq\norm x, \qquad \abs{\varepsilon_r(x)}\leq\norm x. \]
  The compatibility identities make these values independent of the stage containing \(x\).
  Thus contractivity extends the finite coproducts to a coassociative unital \(^*\)-homomorphism
  \[ \Delta:R\longrightarrow R\otimes_{\min}R, \]
  and the compatible counits extend to a character on \(R\).
  The spaces \(C_r=\mathscr R_r\) satisfy the hypotheses of Proposition~\ref{prop:coalgebra}, which gives the claimed isometric realization of \(R^*\).
  The algebra \(\mathscr R_{\rm alg}\) is a sequentially ultramatricial \(R^*\)-algebra in Mori's sense \cite[Example~3.5]{Mori2023}.
\end{proof}

An arbitrary \(R^*\)-algebra, or an arbitrary AF completion, does not carry a canonical multiplication on its Banach dual.

\begin{corollary}[Profinite monoids]\label{cor:profinite-monoids}
  Let \(S\) be a compact metrizable profinite monoid with identity, meaning that \(S\) is an inverse limit of finite discrete monoids \cite[Section~3.1]{Almeida2005}.
  Then \(M(S)\), with convolution, is isometrically a Calkin algebra.
\end{corollary}

\begin{proof}
  By the finite-quotient description of profinite semigroups \cite[Theorem~3.1]{Almeida2005} and metrizability, we obtain finite quotient monoids \(S_r\) and quotient homomorphisms
  \[ q_r:S\longrightarrow S_r, \qquad q_r=\pi_r^{r+1}q_{r+1}, \]
  which separate the points of \(S\).
  Here \(\pi_r^{r+1}:S_{r+1}\to S_r\) is the bonding homomorphism.
  To obtain a countable separating family, consider, for each finite quotient \(q\), the open set of pairs \((s,t)\) with \(q(s)\ne q(t)\).
  These sets cover the complement of the diagonal in \(S\times S\).
  Since this complement is second countable, a countable subfamily still covers it.
  Include the trivial quotient and repeat quotients if the family is finite.
  Replacing the first \(r+1\) quotients by their joint image gives the nested sequence above.
  Write \(q_r^*:C(S_r)\to C(S)\) for the pullback map \(q_r^*f=f\circ q_r\), and put \(C_r=q_r^*C(S_r)\).
  The spaces \(C_r\) are increasing and finite dimensional.
  Their union is a unital self-adjoint algebra separating points, and is therefore dense by Stone--Weierstrass.
  Define
  \[ (\Delta f)(s,t)=f(st), \qquad \varepsilon(f)=f(e). \]
  Continuity and associativity of multiplication make \(\Delta\) a coassociative unital \(^*\)-homomorphism, and \(\varepsilon\) is a character counit.
  For \(f=h\circ q_r\in C_r\), we have
  \[ (\Delta f)(s,t)=h\bigl(q_r(s)q_r(t)\bigr). \]
  Since \(S_r\) is finite, this is a function in \(C_r\otimes C_r\).
  Under \(C(S)^*=M(S)\), the dual norm is the total-variation norm and the induced product is measure convolution.
  Apply Proposition~\ref{prop:coalgebra}.
\end{proof}

\begin{corollary}[Compact matrix-monoid measure algebras]
  \label{cor:compact-matrix-monoids}
  Let \(S\subseteq M_d(\mathbb C)\) be a compact submonoid containing the identity matrix.
  Then \(M(S)\), with convolution and the total-variation norm, is isometrically a Calkin algebra.
\end{corollary}

\begin{proof}
  For \(1\leq i,j\leq d\), let \(z_{ij}(a)=a_{ij}\) for \(a\in S\).
  Let \(C_r\) be the restrictions to \(S\) of the complex polynomials in the functions \(z_{ij}\) and \(\overline{z_{ij}}\) of total degree at most \(r\).
  Each \(C_r\) is finite dimensional, the spaces are increasing, and \(1\in C_r\).
  Their union is a unital self-adjoint algebra which separates the points of \(S\).
  The Stone--Weierstrass theorem therefore gives \(\overline{\bigcup_rC_r}=C(S)\).

  Define \((\Delta f)(a,b)=f(ab)\) and \(\varepsilon(f)=f(I_d)\).
  Then \(\Delta:C(S)\to C(S)\otimes_{\min}C(S)\) is a unital \(^*\)-homomorphism and \(\varepsilon\) is its counit.
  Associativity of matrix multiplication makes \(\Delta\) coassociative.
  Moreover, we have
  \[ \Delta z_{ij}=\sum_{k=1}^d z_{ik}\otimes z_{kj}, \qquad \Delta\overline{z_{ij}} =\sum_{k=1}^d\overline{z_{ik}}\otimes\overline{z_{kj}}. \]
  For a monomial of degree \(m\leq r\), multiplicativity of \(\Delta\) gives a finite sum of tensors whose two factors each have degree \(m\).
  Thus both factors belong to \(C_r\).
  Hence we obtain
  \[ \Delta(C_r)\subseteq C_r\otimes C_r. \]
  Under the isometric identification \(C(S)^*=M(S)\), we have
  \[ (\mu\star\nu)(f)=\int_S\int_S f(ab)\,d\mu(a)\,d\nu(b), \]
  which is the prescribed measure convolution.
  Proposition~\ref{prop:coalgebra} applies.
\end{proof}

In particular, for \(d\geq2\),
\[ \mathbb B_d =\{a\in M_d(\mathbb C):\norm a_{\rm op}\leq1\} \]
is a compact noncommutative matrix monoid.
Thus \(M(\mathbb B_d)\) is a noncommutative isometric Calkin algebra.
In dimension one the same result contains \([0,1]\) and the closed unit disc under multiplication.


\section{\texorpdfstring{Discussion}{Discussion}}
\label{sec:discussion}

\subsection{\texorpdfstring{Possible injury constructions}{Possible injury constructions}}
\label{sec:discussion-injury}

Lemma~\ref{lem:formal-legality} ensures persistence without injury in the present construction.
In other problems, finite requirements may cease to be compatible with later choices.
An injury rule could then abandon a current realization and restart it at a larger rank; see \cite[Chapters~6--7]{Soare2016} for the recursion-theoretic background.

An old BD atom remains in the space after its realization is abandoned.
Thus a construction with injury would need estimates showing that the accumulated contribution of abandoned pieces is compact.
Finite injury to each requirement is insufficient, since infinitely many requirements may leave such pieces.
One possible approach would be to control their tail effects uniformly while preserving the evaluation and coding estimates.
An infinite-injury version would require stronger convergence estimates.

\subsection{\texorpdfstring{The algebra \(\ell_\infty/c_0\)}{The algebra ell-infinity/c-zero}}
\label{sec:discussion-quotient}

Corollary~\ref{ex:ell-infinity} realizes \(\ell_\infty\) as a Calkin algebra.
The present mechanism does not yield a Calkin realization of \(\ell_\infty/c_0\).

\subsection{\texorpdfstring{Reusing the analytic input}{Reusing the analytic input}}
\label{sec:discussion-analysis}

Theorem~\ref{thm:datum}, summarized by \textup{(G1)--(G6)}, records the construction properties used in Part~II.
A different requirement language or family of finite modules can use the same analytic argument after these properties, the parameter and coding bounds, and the finite-certificate conditions have been verified.
In particular, the finite-extension property, common-stem comparison, uniform probe shielding, and multiplier identities must be preserved.
An injury rule would additionally require the compactness control discussed above.

  \section*{Acknowledgements}
  Rui Liu and Jie Shen were partially supported by the National Natural Science Foundation of China (Grant Nos. 12471131 and 12071230). 
  Rui Liu is grateful for the opportunities to make several extended visits to the Departments of Mathematics at Texas A\&M University and at The University of Texas at Austin, during which he benefited greatly from the guidance of Thomas Schlumprecht and Edward Odell and from many helpful discussions with them.
  Rui Liu and Jie Shen also thank Longyun Ding and Su Gao for their valuable advice and discussions on set theory, and for their encouragement and support over the years. 
  Jie Shen is grateful to Liang Yu for his lectures on recursion theory at the 2025 Fudan Logic Summer School. 
  He also thanks Zhaokuan Hao, Ruizhi Yang, Ningyuan Yao, and Will Johnson, as well as the School of Philosophy at Fudan University, for organizing and hosting the summer school.

  \section*{AI assistance, provenance, and authors' responsibility}
  \phantomsection\label{sec:ai-assistance}
  \begin{enumerate}[label={\textup{(\arabic*)}},leftmargin=*]
    \item \textbf{The authors' original design.}
    The authors formulated the main theorem and designed the finite-detection construction.
    Our original work also includes the bounded detection envelope and its formulation through finite-dimensional quotient towers.
    We also established the characterization by countably many jointly faithful weak-star-to-norm continuous finite-dimensional representations outlined in Section~\ref{sec:introduction}.
    We incorporated oracle Turing machines and recursion-theoretic methods into the AH-type construction developed here.
    The design includes the finite detector modules, covariant packets, finite requirements and certificates, and the unit and core probes.
    We also developed the vector-valued Bourgain--Delbaen construction, the multiplier representation of the bounded inverse limit, and the reduction of operator classification to local orbit estimates.
    The paired probes then make the finite-dimensional detector coefficients compatible and yield one coefficient in \(\widehat A\).
    These descriptions concern the specific design used in this paper.

    \item \textbf{Research and writing support.}
    We used GPT-5.6 Sol (OpenAI) and GPT-6 Astra (OpenAI) to check notation, formulas, and the domains and codomains of maps, look for gaps in arguments, suggest clearer wording, generate figures and tables, and help find references.
    These systems helped us locate related work in the Bourgain--Delbaen, Gowers--Maurey, and Argyros--Haydon literature, and we checked the relevant statements in the original sources.
    The technical notes in Sections~\ref{sec:ris-method}--\ref{sec:globalization} identify earlier methods for auxiliary averages, off-weight estimates, stopping trees, boundary-incidence counting, and RIS reduction and synchronization.
    The references are given at the points where these methods are used.

    \item \textbf{AI contributions to the proofs.}
    The stopping-event formulation in Lemma~\ref{lem:quantitative-auxiliary} was prompted by an AI suggestion.
    Subsequent literature checks showed that the underlying counting and tree-truncation techniques were already established.

    The contribution to the carrier-labelled argument in Lemma~\ref{lem:frontier} and Proposition~\ref{prop:dependent-mean} was more substantial.
    We confirmed that the remaining gap we had identified in the proof concerned the dependent-sequence estimate for evaluations with low level.
    GPT-5.6 Sol (OpenAI) then supplied the carrier-labelled stopping-tree construction and its counting proof that we used to close this gap.
    We rewrote this argument in our notation and checked it independently against the definitions and assumptions of the construction.
    Later source checks confirmed earlier uses of weight and depth stopping, boundary-incidence counting, and summation over disjoint intervals.
    These sources are cited beside the relevant arguments.

    \item \textbf{Calkin realization examples.}
    Our initial list of applications comprised unitizations associated with boundedly complete unconditional bases, finite-dimensional algebras, the algebras \((\oplus_n F_n)_{\ell_\infty}\) with finite-dimensional \(F_n\), and the convolution-dual examples associated with the special \(R^*\)-algebras in Corollary~\ref{cor:mori-rstar}.
    GPT-5.6 Sol (OpenAI) and GPT-6 Astra (OpenAI) suggested many of the concrete examples developed in Section~\ref{sec:examples}, both within these initially planned classes and beyond them.
    Representative suggestions include finite incidence algebras, lower-finite \(c_0\)-incidence algebras, weighted monoid and path algebras, H\"older algebras, atomic nest algebras, complete Pick multiplier algebras, free semigroupoid algebras, correspondence Hardy algebras, weak-star semicrossed products, Fourier--Stieltjes and Fourier multiplier algebras, and classical and quantum convolution duals.
    Concrete examples include \(H^\infty(\mathbb D)\), \(\operatorname{Mult}(H_2^2)\), \(\mathcal L_2\), \(M_{\mathrm{cb}}A(\mathbb F_2)\), and \(C^u(O_3^+)^*\).
    We checked that each included example satisfies the hypotheses of the realization results used.

    \item \textbf{Authors' responsibility.}
    We independently checked the AI-assisted arguments included in the paper and made all final decisions.
    The authors take full responsibility for the mathematical statements, proofs, citations, and final text.
    No AI system is treated as an author or as an authority for mathematical correctness.
  \end{enumerate}

  \bibliographystyle{amsplain}
  \bibliography{cite}

\end{document}